\documentclass[reqno]{amsart}
\usepackage[margin=1in]{geometry}
\usepackage{bbm}
\usepackage{bm}
\usepackage{enumitem}
\usepackage{amsfonts}
\usepackage{latexsym, amssymb, amsmath, amscd, amsthm, amsxtra}
\usepackage{mathtools}
\usepackage[all]{xy}
\usepackage{mathrsfs}
\usepackage{fancyhdr}
\usepackage{listings}
\usepackage{hyperref}
\usepackage{cleveref}
\usepackage{soul}
\usepackage{xcolor}
\usepackage{dsfont}
\usepackage{stmaryrd}
\usepackage{comment}

\usepackage{tikz}
\usetikzlibrary{arrows}
\usetikzlibrary{arrows.meta}
\usetikzlibrary{decorations.pathmorphing}
\usetikzlibrary{calc}
\allowdisplaybreaks

\DeclareMathOperator{\re}{{\mathrm{Re}}}
\DeclareMathOperator{\im}{{\mathrm{Im}}}

\hypersetup{
     colorlinks=true
}
\def\P{\mathbb{P}}

\def\E{\mathbb{E}}

\def\Z{\mathbb{Z}}
\def\R{\mathbb{R}}
\def\N{\mathbb{N}}

\def\11{{\mathbf{1}}}
\def\Im{\im }

\newcommand{\fd}{{\mathfrak d}}

\newcommand{\bu}{{\bf{u}}}
\newcommand{\bv}{{\bf{v}}}
\newcommand{\bw}{{\bf{w}}}

\newcommand{\sE}{{\mathsf{E}}}

\newcommand{\C}{\mathbb C}
\newcommand{\cB}{\mathcal B}

\newcommand{\cK}{\mathcal K}
\newcommand{\cL}{\mathcal L}
\newcommand{\cM}{\mathcal M}
\newcommand{\cW}{\mathcal W}
\newcommand{\OK}{\mathcal O_{\mathcal K}}
\newcommand{\lenk}{l_{\mathcal K}}
\newcommand{\hell}{{\hat\ell}}

\renewcommand{\bar}{\overline}
\newcommand{\wt}{\widetilde}
\newcommand{\wh}{\widehat}

\newcommand{\al}{\alpha}

\newcommand{\qq}[1]{\langle{#1}\rangle}
\newcommand{\qqq}[1]{\llbracket{#1}\rrbracket}

\newcommand{\LK}{{L\to n}}
\newcommand{\Zn}{\widetilde{\mathbb Z}_n^d}
\newcommand{\ZL}{{\mathbb Z}_L^d}
\newcommand{\Gc}{{\mathring G}}

\newcommand{\fa}{{\mathfrak a}}

\newcommand{\fc}{{\mathfrak c}}

\DeclareMathOperator{\OO}{O}
\DeclareMathOperator{\oo}{o}

\newcommand{\sig}{\sigma}
\newcommand{\bsig}{{\boldsymbol{\sigma}}}
\newcommand{\bchi}{{\boldsymbol{\chi}}}

\newcommand{\cut}{\mathrm{Cut}}
\newcommand{\cutL}{(\mathrm{Cut}_L)}
\newcommand{\cutR}{(\mathrm{Cut}_R)}

\newcommand{\qll}[1]{[\![{#1}]\!]}
\newcommand{\rep}[1]{({#1})}

\DeclareMathOperator{\tr}{Tr}
\DeclareMathOperator{\var}{Var}

\newcommand{\be}{\begin{equation}}
\newcommand{\ee}{\end{equation}}
\newcommand{\ii}{\mathrm{i}}
\newcommand{\dd}{\mathrm{d}}

\newcommand{\e}{{\varepsilon}}
\newcommand{\cal}{\mathcal}

\newcommand{\cor}{\color{red}}

\newcommand{\nc}{\normalcolor}

\newcommand{\fn}{{\mathfrak n}}
\newcommand{\fm}{{\mathfrak m}}

\newcommand{\ba}{{\mathbf{a}}}
\newcommand{\bfb}{{\mathbf{b}}}
\newcommand{\thn}{\vartheta}

\newcommand{\dthn}{\boldsymbol{\Theta}}

\makeatletter

\DeclarePairedDelimiter{\@p}{\lparen}{\rparen}
\DeclarePairedDelimiter{\@br}{\lbrack}{\rbrack}
\DeclarePairedDelimiter{\@avg}{\langle}{\rangle}
\DeclarePairedDelimiter{\@bbr}{\llbracket}{\rrbracket}
\DeclarePairedDelimiter{\@abs}{\lvert}{\rvert}
\DeclarePairedDelimiter{\@norm}{\lVert}{\rVert}
\DeclarePairedDelimiter{\@set}{\{}{\}}

\newcommand{\p}{\@ifstar{\@p}{\@p*}}
\newcommand{\br}{\@ifstar{\@br}{\@br*}}
\newcommand{\avg}{\@ifstar{\@avg}{\@avg*}}
\newcommand{\bbr}{\@ifstar{\@bbr}{\@bbr*}}
\newcommand*{\abs}{\@ifstar{\@abs}{\@abs*}}
\newcommand*{\norm}{\@ifstar{\@norm}{\@norm*}}
\newcommand{\set}{\@ifstar{\@set}{\@set*}}

\newcommand{\Sigmagen}{\widehat{\Sigma}}
\newcommand{\suchthat}{\:\middle\vert\:}

\DeclareMathOperator{\TSP}{TSP}
\DeclareMathOperator{\slice}{SLICE}

\makeatother
\theoremstyle{plain} %plain, definition, remark
\newtheorem{theorem}{Theorem}[section]
\newtheorem*{theorem*}{Theorem}
\newtheorem{lemma}[theorem]{Lemma}
\newtheorem{assumption}[theorem]{Assumption}
\newtheorem*{lemma*}{Lemma}
\newtheorem{corollary}[theorem]{Corollary}
\newtheorem*{corollary*}{Corollary}

\newtheorem*{proposition*}{Proposition}
\newtheorem{claim}[theorem]{Claim}
\newtheorem*{claim*}{Claim}

\newtheorem{definition}[theorem]{Definition}
\newtheorem*{definition*}{Definition}
\theoremstyle{remark}
\newtheorem{example}[theorem]{Example}
\newtheorem*{example*}{Example}
\newtheorem{remark}[theorem]{Remark}

\newtheorem*{remark*}{Remark}
\newtheorem*{remarks*}{Remarks}

\allowdisplaybreaks
\def\@setthanks{\vspace{-\baselineskip}\def\thanks##1{\@par##1\@addpunct.}\thankses}

\numberwithin{equation}{section}
\usepackage{environ}
\title{Delocalization of random band matrices at the edge}

\author{Fan Yang$^\star$}
\author{Jun Yin$^\dagger$}

\thanks{$^\star$Yau Mathematical Sciences Center, Tsinghua University, and Beijing Institute of Mathematical Sciences and Applications, \href{mailto:fyangmath@mail.tsinghua.edu.cn}{fyangmath@mail.tsinghua.edu.cn}. %Supported in part by the National Key R\&D Program of China (No. 2023YFA1010400). 
}
\thanks{$^\dagger$Department of Mathematics, University of California, Los Angeles, \href{mailto:jyin@math.ucla.edu}{jyin@math.ucla.edu}. 
}

\begin{document}

\begin{abstract} 
We consider $N\times N$ Hermitian random band matrices $H=(H_{xy})$, whose entries are centered complex Gaussian random variables. The indices $x,y$ range over the $d$-dimensional discrete torus $(\mathbb Z/L\mathbb Z)^d$, where $d\in \{1,2\}$ and $N=L^d$. The variance profile $S_{xy}=\mathbb E|h_{xy}|^2$ exhibits a banded structure: specifically, $S_{xy}=0$ whenever the distance $|x-y|$ exceeds a band width parameter $W\le L$. Let $W=L^\alpha$ for some exponent $0<\alpha\le 1$. 
We show that as $\alpha$ increases from \smash{$\mathbf 1_{d=1}/2$} to $1-d/6$, the range of energies corresponding to delocalized eigenvectors gradually expands from the bulk toward the entire spectrum. More precisely, we prove that eigenvectors associated with energies $E$ satisfying $2 - |E| \gg N^{-c_{d,\alpha}}$ are delocalized, where the exponent $c_{d,\alpha}$ is given by $c_{d,\alpha} = 2\alpha - 1$ in dimension 1 and $c_{d,\alpha} = \alpha$ in dimension 2.
Furthermore, when $\alpha>1-d/6$, all eigenvectors of $H$ become delocalized. We further establish quantum unique ergodicity for delocalized eigenvectors, as well as a rigidity estimate for the eigenvalues. 
Our findings extend previous results---established in the bulk regime for one-dimensional (1D) and two-dimensional (2D) random band matrices \cite{Band1D,Band2D}---to the entire spectrum, including the spectral edges. They also complement the results of \cite{Sod2010,Band_Edge123}, which concern the edge eigenvalue statistics for 1D and 2D random band matrices.

%In this regime,We also establish quantum unique ergodicity for delocalized eigenvectors, along with a rigidity estimate for eigenvalues. 
%To obtain these results, we establish (almost) sharp local semicircle laws and quantum diffusion estimates for all spectral parameters $z = E + \ii \eta$, with $E = \mathrm{O}(1)$ and $\eta$ down to the optimal local scale. 
%Our results also complement the results in \cite{Sod2010,Band_Edge123} on the edge eigenvalue statistics for 1D and 2D random band matrices.
%Our result extends the results established previously within the bulk regime for 1D and 2D random band matrices \cite{Band1D,Band2D} to the entire spectrum, including the vicinity of the spectral edges
\end{abstract}

\maketitle

{
\hypersetup{linkcolor=black}
\tableofcontents
}

%\newpage 

%\vspace{2in}

\section{Introduction}

The $d$-dimensional \emph{random band matrix} (RBM) model \cite{ConJ-Ref1, ConJ-Ref2, fy}, also known as the Wegner orbital model \cite{Wegner1, Wegner2, Wegner3}, describes a broad class of random Hamiltonians on a $d$-dimensional lattice, where random hopping occurs only within a band of width $W$. 
In this paper, we consider an RBM $H = (H_{xy})$ defined on a large $d$-dimensional discrete torus $\ZL := \{1, 2, \ldots, L\}^d$ with $N = L^d$ lattice sites. The entries of $H$ are independent (up to the Hermitian symmetry $H_{xy} = \overline{H}_{yx}$), centered complex Gaussian random variables with a banded variance profile $S_{xy}:=\mathbb{E}|H_{xy}|^2$, which vanishes whenever the distance $|x - y|$ exceeds $W$. 
We normalize the row sums of the variance matrix $S = (S_{xy})$ to be 1, under which the global eigenvalue distribution of $H$ converges weakly to the Wigner semicircle law, supported on the interval $[-2, 2]$ \cite{Wigner}.

The RBM, or Wegner orbital model, can be viewed as a natural interpolation between the celebrated Anderson model \cite{Anderson} and the Wigner ensemble \cite{Wigner}, as $W$ varies. In particular, similar to the Anderson model, RBMs also exhibit a sharp Anderson metal–insulator transition depending on the band width $W$ or the energy level $E$.  
More precisely, numerical simulations \cite{ConJ-Ref1, ConJ-Ref2, ConJ-Ref4, ConJ-Ref6} and non-rigorous arguments based on supersymmetry \cite{fy} suggest that within the \emph{bulk} of the limiting spectrum $[-2, 2]$, the localization length of one-dimensional (1D) RBMs is of order $W^2$. In the two-dimensional (2D) case, the localization length is conjectured to grow exponentially as $\exp(\OO(W^2))$. In particular, as $W$ increases and the localization length exceeds the system size $L$, the RBM undergoes a transition from a localized phase to a delocalized phase. For RBMs in dimensions $d \geq 3$, the localization length of bulk eigenvectors becomes infinite once $W$ exceeds a large constant, indicating complete delocalization in this regime.
At the spectral edges $\pm 2$, a sharp phase transition in the edge eigenvalue statistics of 1D RBMs has been rigorously established in \cite{Sod2010} via an intricate moment method, occurring as $W$ crosses the threshold $L^{5/6}$. 
This approach was later extended to higher dimensions ($2\le d \le 4$), revealing a similar transition at $W = L^{1 - d/6}$ \cite{Band_Edge123}. These results naturally suggest that the corresponding localization–delocalization transition for edge eigenvectors also occurs at $W = L^{1 - d/6}$ for $d \geq 1$.

%It is important to 
Note that the critical thresholds for $W$ differ between the bulk and the edge of the spectrum. Therefore, it is natural to conjecture that when $W$ lies between these two thresholds—i.e., when $L^{\mathbf{1}_{d=1}/2} \ll W \ll L^{1 - d/6}$—there exist mobility edges within the spectrum $[-2, 2]$ separating localized and delocalized phases. Specifically, eigenvectors near the spectral edges are expected to be localized, while those corresponding to energies deeper in the bulk are delocalized, undergoing a transition across the mobility edges.

%Note that the critical thresholds of the band width $W$ differ within the bulk and at the edge. Hence, it is natural to conjecture that for $W$ between these two thresholds (i.e., $L^{\mathbf 1_{d=1}/2}\ll W \ll L^{1-d/6}$), there should exist mobility edges within the spectrum $[-2,2]$ that separate the localized and delocalized phases---near the spectral edges, eigenvectors are localized, but upon crossing the mobility edges into the bulk of the spectrum, the eigenvectors become delocalized.

The above conjectures regarding the localization–delocalization transition of RBMs and the existence of mobility edges closely mirror those for the Anderson model (or random Schr{\"o}dinger operators). Heuristically, the RBM with band width $W$ and the Anderson model with disorder strength $\lambda$ (where larger $\lambda$ corresponds to stronger disorder) are believed to exhibit qualitatively similar behavior under the correspondence $\lambda^{-1} \asymp W$.
The localization phenomenon in the 1D Anderson model is by now well understood; see, for example, \cite{GMP, KunzSou, Carmona1982_Duke, Damanik2002}, among many other references. In dimensions $d \geq 2$, Anderson localization was first rigorously established by Fröhlich and Spencer \cite{FroSpen_1983} via a multi-scale analysis. A simpler and influential alternative approach was later introduced by Aizenman and Molchanov \cite{Aizenman1993} using the fractional moment method. 
Since then, numerous significant results have been developed concerning Anderson localization in higher dimensions; see, for instance, \cite{FroSpen_1985, Carmona1987, SimonWolff, Aizenman1994, ASFH2001, Bourgain2005, Germinet2013, DingSmart2020, LiZhang2019}. 
In contrast, our understanding of Anderson delocalization is far more limited. %significantly less complete. 
To the best of our knowledge, the existence of a delocalized phase (and in particular, of mobility edges) has been rigorously established only for the Anderson model on the Bethe lattice \cite{Bethe_PRL, Bethe_JEMS, Bethe-Anderson}, which may be viewed as an $\infty$-dimensional tree-like graph, and for a block-structured variant of the Anderson model \cite{RBSO, RBSO1D}, where the diagonal potential is replaced by a diagonal block matrix consisting of independent Gaussian blocks.

%The above conjectures concerning the localization-delocalization transition of RBM and the existence of mobility edges closely resemble those of the Anderson model (or random Schr{\"o}dinger operators). In fact, heuristically, the RBM with band width $W$ and Anderson model with coupling strength $\lambda$ (where larger $\lambda$ indicates stronger disorder) are expected to exhibit similar qualitative behavior under the correspondence $\lambda^{-1}\asymp W$. On the other hand, our understanding of the Anderson delocalization is far more limited. To the best of our knowledge, the existence of a delocalized phase (and mobility edges) has only been established for the Anderson model on the Bethe lattice \cite{Bethe_PRL, Bethe_JEMS,Bethe-Anderson}, which is an $\infty$-dimensional lattice, and a block variant of the Anderson model \cite{RBSO,RBSO1D}, in which the diagonal potential is replaced by a diagonal block matrix consisting of Gaussian blocks.

Compared to the Anderson model, the study of delocalization for RBMs is relatively more tractable, largely due to the significantly higher number of random entries in the matrix---$NW$ for RBMs versus $N$ for the Anderson model. As a result, RBMs have emerged as one of the most prominent models for investigating the Anderson delocalization conjecture. 
There has been significant progress in understanding the localization/delocalization for 1D RBMs  \cite{BaoErd2015,BGP_Band,BouErdYauYin2017,PartII,PartI,CS1_4,CPSS1_4,ErdKno2013,ErdKno2011,delocal,Semicircle,HeMa2018,Wegner,Sch2009,SchMT,Sch3,1Dchara,Sch1,Sch2014,Sch2,Sod2010,Band1D_III,Band1D}, as well as for RBMs in higher dimensions $d \geq 2$ \cite{DL_2D,DisPinSpe2002,Band2D,ErdKno2013,ErdKno2011,delocal, HeMa2018,BandI,BandII,Band1D_III,BandIII,Band_Edge123}. 
Currently, the strongest localization result for 1D RBMs is given in \cite{CS1_4,CPSS1_4}, where localization of bulk eigenvectors was established under the condition $W\ll N^{1/4}$. 
On the delocalization side, bulk eigenvector delocalization %(under the assumption that the matrix entries of $H$ are Gaussian) 
has been proved in dimension 1 \cite{Band1D} under the sharp condition $W\gg L^{1/2}$; in dimensions 2 \cite{Band2D} and $d\ge 7$ \cite{BandI, BandII, BandIII}, delocalization has been established under the weaker condition $W \geq L^\varepsilon$ for any arbitrarily small constant $\varepsilon > 0$.
While these works have focused primarily on the bulk regime of RBMs, results concerning the edge regime are far more limited. To the best of our knowledge, only the eigenvalue statistics near the spectral edges $\pm 2$ have been analyzed, specifically in \cite{Sod2010, Band_Edge123}. However, rigorous results concerning the localization or delocalization of edge eigenvectors remain absent from the literature.

%While these works have focused on the bulk regime of RBM, results regarding the edge regime are much scarcer in the literature. To the best of our knowledge, only the eigenvalue statistics in the vicinity of the spectral edges $\pm 2$ have been studied in \cite{Sod2010, Band_Edge123}, while results about the localization or delocalization of edge eigenvectors remain absent from the literature.

In this paper, we present the \emph{first rigorous proof establishing the predicted lower bounds on the localization lengths of all eigenvectors across the entire spectrum} for random band matrices, including the vicinity of the spectral edges. Our results provide a definitive and comprehensive answer to the delocalization conjecture for RBMs in 1D and 2D. In particular, this work represents a significant step toward a complete resolution of the long-standing Anderson metal–insulator transition conjecture for one of the most fundamental non-mean-field models in mathematical physics. As a further contribution, we also provide a precise prediction for the locations of the mobility edges in 1D and 2D---a result that is particularly novel, as even such a conjecture had not previously appeared in the literature.

To elaborate, we investigate the delocalization of 1D and 2D RBMs by extending the results of \cite{Band1D, Band2D} to the entire spectrum, including the neighborhood of the spectral edges. On one hand, our findings are consistent with those of \cite{Band1D, Band2D}, confirming the delocalization of bulk eigenvectors under the condition $W \gg L^{\mathbf{1}_{d=1}/2}$. On the other hand, they align with the results of \cite{Sod2010, Band_Edge123}, demonstrating that eigenvectors near the spectral edges are delocalized when $W \gg L^{1 - d/6}$.
Moreover, in the intermediate regime where $W=L^\al$ for a constant exponent \smash{$\mathbf{1}_{d=1}/2 <\al \le 1 - d/6$}, our main result, \Cref{thm:supu}, shows that only those eigenvectors corresponding to energies $E \in [-2, 2]$ satisfying $2 - |E| \gg N^{-c_{d,\al}}$ are delocalized, where the exponent $c_{d,\alpha}$ is given by $c_{d,\alpha} = 2\alpha - 1$ in 1D, and $c_{d,\alpha} = \alpha$ in 2D. This suggests the existence of mobility edges located at $\pm(2 - N^{ - c_{d,\al}})$. To fully establish the existence of such mobility edges, one must also show that eigenvectors associated with energies near the spectral edges---specifically those satisfying $|2 - |E|| \ll N^{ - c_{d,\al}}$---are indeed localized. In fact, we provide a more comprehensive result in \Cref{cor:localizationlength}, which strengthens \Cref{thm:supu} by giving explicit lower bounds on the localization lengths of these (potentially) localized eigenvectors. We believe these lower bounds to be sharp, although establishing matching upper bounds remains a challenging problem.

\subsection{The model and overview of main results}

For definiteness, throughout this paper, we assume that $L=nW$ for some $n,W\in 2\N+1$. Then, we choose the center of the lattice as $0$. However, our results still hold for even $n$ or $W$, as long as we choose a different center for the lattice. Consider a one or two-dimensional lattice in \(\mathbb{Z}^d\), $d\in \{1,2\}$, with $N=L^d$ lattice points, i.e., $\Z_L^d:=\qll{ -(L-1)/2 , (L-1)/2}^2$. 
% \be\label{ZLd}
% \Z_L^d:=\qll{ -(L-1)/2 , (L-1)/2}^2 . 
% \ee
Hereafter, for any $a,b\in \R$, we denote $\llbracket a, b\rrbracket: = [a,b]\cap \Z$. 
We will view $\Z_L^d$ as a torus and denote by $\rep{x-y}_L$ the representative of $x-y$ in $\ZL$, i.e.,  
\be\label{representativeL}\rep{x-y}_L:= \left((x-y)+L\Z^d\right)\cap \Z_L^d.\ee
Now, we impose a block structure on $\Z_L^d$ with blocks of side length $W$.

\begin{definition}[Block structure] \label{def: BM2}
For $d\in \{1,2\}$, suppose \be\label{eq:LnW} L=nW,\quad N=L^d,
\ee
for some integers $n, W\in 2\N+1$. We divide $\mathbb Z_L^d$ into $n^d$ many blocks of linear size $W$, such that the central one is $\qll{ -(W-1)/2, (W-1)/2}^d$. Given any $x\in \Z_L^d$, denote the block containing $x$ by $[x]$. Denote the lattice of blocks $[x]$ by \smash{$\Zn$}. We will view \smash{$\Zn$} as a torus and let $\rep{[x]-[y]}_n$ denote the representative of $[x]-[y]$ in \smash{$\Zn$}. For convenience, we will regard $[x]$ both as a vertex of the lattice \smash{$\Zn$} and a subset of vertices on the lattice $\Z_L^d$. Given any $\ZL\times \ZL$ matrix $A$, let $A|_{[x][y]}$ denote the $([x],[y])$-th block of $A$, which is a $W^d\times W^d$ matrix. 
\end{definition}

%Clearly, $\|x-y\|_L:=| \rep{x-y}_L |$ is a periodic distance on $\ZL$. 
For definiteness, we use the $L^1$-norm in this paper, i.e., $\|x-y\|_L:=\|\rep{x-y}_L\|_1$, which is the (periodic) graph distance on $\ZL$. 
Similarly, we also define the periodic $L^1$-distance $\|\cdot\|_n$ on \smash{$\Zn$}. For simplicity of notations, throughout this paper, we will abbreviate
\begin{align}\label{Japanesebracket} |x-y|\equiv \|x-y\|_L,\quad &\langle x-y \rangle \equiv \|x-y\|_L + W,\quad \text{for} \ \ x,y \in \ZL, \\
\label{Japanesebracket2} |[x]-[y]|\equiv \|[x]-[y]\|_n,\quad &\langle [x]-[y] \rangle \equiv \|[x]-[y]\|_n + 1, \quad \text{for} \ \  x,y \in \Zn.
\end{align}
We use $x\sim y$ to mean that $x$ and $y$ are neighbors on $\ZL$, i.e., $|x-y|=1$. Similarly, $[x]\sim [y]$ means that $[x]$ and $[y]$ are neighbors on \smash{$\Zn$}. 
%For any subset $\mathbb S\subset \ZL$, we denote $\dist(x,\mathbb S):=\min_{y\in \mathbb S}|x-y|$. 
We now outline the precise assumptions for our model. 

\begin{definition}[Random band matrices] \label{def: BM}
Fix $d\in\{1,2\}$ and a coupling parameter $\lambda$ satisfying $C^{-1}\le \lambda\le C$ for a constant $C>0$. Let $V$ and $\Psi$ both denote $N\times N$ complex Hermitian random block matrices, whose entries are independent Gaussian random variables up to the Hermitian symmetry $V_{xy}=\overline V_{yx}$ and $\Psi_{xy}=\overline \Psi_{yx}$. 
$V$ is a diagonal block matrix consisting of i.i.d.~GUE blocks, that is, the off-diagonal entries (in the diagonal blocks) of $V$ are complex Gaussian random variables:
\be\label{bandcw0}
V_{xy}\sim {\cal N }_{\C}(0, s_{xy})\quad \text{with}\quad  s_{xy}:=W^{-d} {\bf 1}\left( [x] =[y] \right), \quad \text{for} \ \  x\ne y,
\ee
while the diagonal entries of $V$ are real Gaussian random variables distributed as ${\cal N }_{\R}(0, W^{-d})$. $\Psi$ is a random matrix that introduces hoppings between neighboring blocks, where the blocks \smash{$\Psi|_{[x][y]}$} with $[x]\sim [y]$ are i.i.d.~$W^d\times W^d$ complex Ginibre matrices up to the Hermitian symmetry \smash{$\Psi|_{[x][y]} = (\Psi|_{[y][x]})^*$}. In other words, the entries of $\Psi$ are independent complex Gaussian random variables:
\be\label{bandcw1}
\Psi_{xy}\sim {\cal N }_{\C}(0, s'_{xy}),\quad \text{with}\quad  s'_{xy}:=W^{-d} {\bf 1}\left( [x] \sim [y] \right).
\ee
Then, we define a class of random band matrices (or called Wegner orbital models) of the form  
\be\label{eq:WO}
H:=(1+2d\lambda^2)^{-1/2}(\lambda\Psi + V),
\ee
where the normalization $(1+2d\lambda^2)^{-1/2}$ is chosen such that the total variance of the entries of $H$ within each row and column is equal to 1.
\end{definition}

%\subsection{Overview of the main results}
Define the Green's function (or resolvent) of the Hamiltonian $H$ as
\be\label{def_Green}
G(z):=(H-z)^{-1} ,  \quad z\in \C_+.
\ee
Assuming $W\ge L^\delta$ for a constant $\delta>0$, we establish the following results: %for the above model: 
\begin{itemize}
\item {\bf Local law} (\Cref{thm_locallaw,thm_locallaw_out}). 
For all energies $E$ with $|E| \le C$, where $C > 2$ is an arbitrarily large constant, we prove a sharp local law for the Green's function $G(z)$, with $z = E + \ii \eta$, down to an almost optimal local scale $\eta \gg \eta_*(E)$. Here, $\eta_*(E)$ depends both on $E$ and the scale $W$; see \eqref{eq:defeta*} and \eqref{eq:defeta*2} for its precise definition. Outside the support $[-2,2]$ of the semicircle law, the local laws hold down to a smaller scale $\eta \gg \eta_\circ(E)$, where $\eta_\circ(E)$ is defined in \eqref{eq:defeta*3}.

%For all energies $E$ satisfying $|E|\le C$ for an arbitrary large constant $C>2$, we establish a sharp local law for the Green's function $G(z)$, where $z=E+\ii \eta$ with $\eta$ down to an almost optimal local scale $N^\e \eta_*(E)$ (where $\eta_*(E)$ depends on both $E$ and the order of $W$; see \eqref{eq:defeta*} and \eqref{eq:defeta*2} for its definition).  

\item {\bf Delocalization} (\Cref{thm:supu}). 
With the local laws, we show that for $W = L^\alpha$ with \smash{$\frac{\mathbf{1}_{d=1}}{2} < \alpha \le 1 - \frac{d}{6}$}, the eigenvectors of $H$ corresponding to energies $E$ satisfying $2 - |E| \gg N^{ - c_{d,\al}}$ are delocalized with high probability. Moreover, in the supercritical regime \smash{$W \gg L^{1-d/6}$}, all eigenvectors of $H$ are delocalized. 

\item {\bf Eigenvalue rigidity} (\Cref{thm:rigidity}). 
As another consequence of the local laws, we show that when \smash{$W \gg L^{1 - d/6}$}, the eigenvalues of $H$ concentrate around their classical locations given by the quantiles of the semicircle law.

%when $W \gg L^{1-d/6}$, we show that the eigenvalues of $H$ will concentrate around their corresponding quantiles of the semicircle law. % (defined in \eqref{eq:gammak} below).

\item {\bf Quantum diffusion} (\Cref{thm_diffu}). The evolution of the quantum particle exhibits quantum diffusion on spatial scales $\gg W$ and for times $t \gg 1$.

%follows a quantum diffusion on scales $\gg W$ and for times $t\gg 1$. 

\item {\bf Quantum unique ergodicity} (\Cref{thm:QUE}). 
As a consequence of quantum diffusion, we prove that in the supercritical regime \smash{$W \gg L^{1 - d/6}$}, with probability $1-\oo(1)$, every bulk eigenvector is nearly flat on all scales $\Omega(W)$.
\end{itemize}
We refer the reader to \Cref{sec:main} below for precise statements of our main results. 

Our results can be readily extended to settings where the entries of $H$ are real Gaussian variables or follow more general Gaussian divisible distributions. In particular, as in \cite{Band1D,Band2D}, we employ a flow argument that evolves both the matrix $H_t$ and the spectral parameter $z_t$ from time $t=0$ to $t=1$. Along this flow, we apply It\^o's formula to derive a system of equations for $G$-loops (as defined in \Cref{Def:G_loop}), which allows us to transfer the $G$-loop estimates to progressively larger times $t$.
In fact, instead of initializing the flow at time $t=0$, we may alternatively start at some later time $t_0=1-\oo(1)$, with an initial matrix $H_{t_0}$ having non-Gaussian entries. In this case, the desired $G$-loop estimates at time $t_0$ can be established directly using standard techniques such as cumulant expansions. The subsequent proof then proceeds as in the current paper, allowing us to obtain the same results for Gaussian divisible RBMs.
%In this case, rather than establishing the desired $G$-loop estimates through the flow argument, we can use standard tools in the literature (such as cumulant expansions) to establish these estimates at time $t_0$. 
%We also note that, very recently, the work \cite{Band1D} was also extended in \cite{BandZZ} to 1D RBMs with non-Gaussian entries and more general variance profiles. We believe our results in 1D can also be extended to the non-Gaussian case with the improved Green's function comparison technique there, but the 2D case may require new ideas. 
We also note that, very recently, the results of \cite{Band1D} were extended in \cite{BandZZ} to 1D RBMs with non-Gaussian entries and more general variance profiles. We believe that our results in the 1D setting can similarly be extended to the non-Gaussian case by employing the improved Green's function comparison technique developed in that work. However, extending the results to the 2D case appears to require new ideas beyond the current methods.
%We also note that in 1D, our flow argument can be combined with a standard Green's function comparison technique (GFT) at each step of the induction to extend the results from Gaussian divisible ensembles to more general non-Gaussian RBMs. This strategy was recently demonstrated in \cite{BandZZ} for 1D RBMs in the bulk, but extending it to 2D requires new ideas. %which appeared online after the first version of this paper was posted on arXiv.  However, considering the length of this paper and the fact that the GFT idea does not work well in 2D, we focus here on providing the first resolution of the delocalization conjecture for a general class of random band matrices in both one and two dimensions, across the entire spectrum.
Finally, we note that our results can be further extended to a broader class of random block Schr{\"o}dinger operators and RBMs with coupling parameter $\lambda \ll 1$, as demonstrated for the bulk regime in \cite{RBSO1D}. A detailed investigation of these extensions is left for future work.

\subsection{Difficulties and new ideas}\label{sec:idea}
Compared to the analysis of random band matrices in the bulk regime, extending the associated arguments to the spectral edge presents several fundamental challenges. Some of these difficulties are already known from the study of Wigner matrices near the spectral edge; see, e.g., \cite{ErdYauYin2012Rig,Semicircle}. However, these issues become even more pronounced in the context of RBMs. For instance, the delocalization of RBMs in the bulk regime for high dimensions ($d\ge 7$) has been established through a series of works \cite{BandI,BandII,BandIII}. Yet, the methods developed in those works break down entirely outside the bulk regime. To date---even with the new techniques introduced in this paper---delocalization of edge eigenvectors in these dimensions remains an open problem.

A main obstacle arises from the instability of the self-consistent equations for the diagonal entries of the resolvent near the spectral edge (see equations \eqref{eq:self0} and \eqref{eq:self}). As we will discuss around equation \eqref{eq:self4}, this instability leads to estimates on the diagonal entries of $G$ that are too weak for our purposes. As a result, we cannot extend the local law for $G(z)$ down to the optimal scale in $\im z$, which in turn prevents us from obtaining sharp delocalization estimates for the eigenvectors of $H$. 
In the case of Wigner matrices \cite{ErdYauYin2012Rig,Semicircle}, this instability is resolved by decomposing the vector of diagonal resolvent entries into two components: the projection onto the unstable direction (i.e., the direction $(1,\ldots, 1)^\top\in \C^N$ in our setting), and its orthogonal complement. The projection onto the unstable direction reduces to a scalar quadratic self-consistent equation, which can be solved explicitly. For the orthogonal projection, the corresponding self-consistent equation becomes stable due to the spectral gap between the Perron–Frobenius eigenvalue of the variance matrix (defined in \eqref{eq:SWO}) and the rest of the spectrum.
However, for RBMs with $W\ll L$, this spectral gap is too small---on the order of $W^2/L^2$---and thus the instability persists. In \Cref{subsec_pf_lem_G<T}, we introduce a flow-based argument to estimate the diagonal resolvent entries. This approach avoids the use of self-consistent equations in establishing a weak local law for $G(z)$ and thereby circumvents the instability issue. Nonetheless, we emphasize that the self-consistent equation remains useful when upgrading this weak local law to a strong one later in the analysis.

The second conceptual difficulty lies in identifying the contribution of edge eigenvalues and eigenvectors to the resolvent $G(z)$. Heuristically, the imaginary part of $G(z)$ is primarily influenced by the local eigenvalues and eigenvectors of $H$ near $z$, whereas the real part reflects contributions from the entire spectrum of $H$. However, near the spectral edge, the eigenvalue density is significantly lower than in the bulk. As a consequence, for $z$ close to the spectral edge, $\im G_{xx}(z)$ is typically one order of magnitude smaller than $\re G_{xx}(z)$. %in sharp contrast to the bulk regime, where both are of the same order.
This imbalance necessitates a more delicate analysis. Unlike in the bulk case, we must carefully track the imaginary parts of the resolvent entries throughout the proof. It is crucial to ensure that there are sufficiently many imaginary parts, which will allow us to derive sufficiently strong estimates and help cancel diverging factors associated with the instability of the self-consistent equations. %Nevertheless, there are still situations where the real parts of the resolvent entries cannot be avoided, or where taking imaginary parts does not yield meaningful improvement in the bounds. In such cases, more technical analysis is needed. 
%precise local law estimates for the real part of $G(z)$ are required to maintain control over the analysis.

%The second conceptual difficulty is that the contribution of the edge eigenvalues and eigenvectors to the resolvent $G(z)$ is hard to identify. Heuristically speaking, the imaginary part of $G(z)$ depends mainly on the local eigenvalues and eigenvectors of $H$ near $z$, while the real part will depend on the whole spectrum of $H$. However, compared to the bulk, the eigenvalue density near the spectral edge is significantly lower. As a consequence, for $z$ near the spectral edges, $\im G_{xx}(z)$ is one order smaller than $\re G_{xx}(z)$---in contrast to the bulk regime where they are of the same order. Thus, unlike the proof for the bulk regime, we need to keep track of the imaginary parts of the resolvent entries carefully during the proof. We need to get as many imaginary parts as possible to attain strong enough local laws and also to cancel the diverging factors associated with the instability of the self-consistent equations. However, in many places, the real parts of resolvent entries cannot be avoided, or taking imaginary parts does not help to improve the bound. In that case, we need to have a fine enough local law estimate for the real part of $G(z)$.

Another fundamental difficulty arises from the problematic behavior of same-colored propagators. Roughly speaking, the propagators \smash{$\Theta^{(\sig_1,\sig_2)}$, $(\sig_1,\sig_2)\in \{+,-\}^2$}, defined in \Cref{def_Theta} below characterize the asymptotic limits of pairs of resolvent entries. In this framework, the opposite-colored propagator \smash{$\Theta^{(+,-)}$} corresponds to $|G_{xy}|^2$, while the same-colored propagators \smash{$\Theta^{(+,+)}$ and $\Theta^{(-,-)}$} correspond to $G_{xy}G_{yx}$ and $G^*_{xy}G^*_{yx}$, respectively. 
In all previous proofs of delocalization for RBMs, a key technical challenge has been handling the slow decay of \smash{$\Theta^{(+,-)}_{xy}$}. In contrast, same-colored propagators \smash{$\Theta^{(+,+)}_{xy}$ and $\Theta^{(-,-)}_{xy}$} have always been harmless in the bulk, where they exhibit exponential decay for $|x-y|$ beyond the local scale $W$. This favorable behavior significantly simplifies the graphical structure of the associated expressions: graphs involving edges representing resolvent entries and propagators can be organized into \emph{molecules}, that is, short-scale structures formed by vertices connected via same-colored propagators, as in \cite{Band1D,BandI,BandII,BandIII}. This simplification allows the analysis to focus primarily on the difficulties posed by the opposite-colored propagator. 
However, near the spectral edges, the situation changes drastically. The decay of the same-colored propagators \smash{$\Theta^{(+,+)}$ and $\Theta^{(-,-)}$} can become significantly slower---sometimes matching the slow decay of opposite-colored ones. It is this issue that invalidates the molecular decomposition and renders the proof strategy in \cite{BandI,BandII,BandIII} ineffective beyond the bulk regime. 

In our proof, the slow decay of the same-colored propagators introduces two main issues: weaker upper bounds for primitive loops, and poor control over the $(L^\infty\to L^\infty)$-norm of the evolution kernels involving same-colored propagators. Both challenges necessitate a more refined and delicate analysis of primitive loops near the spectral edge, which we carry out in detail in \Cref{sec:prim}. 
In particular, due to the insufficient bounds on the evolution kernels, we require a new approach to handle ``pure' primitive loops---those constructed entirely from same-colored propagators. In the bulk regime, bounding such loops is almost trivial, thanks to the rapid decay of same-colored propagators. However, near the spectral edge, these bounds become too weak for our purposes. Furthermore, the method developed in \cite{Band1D} for treating non-pure loops is not applicable here, as it relies on Ward’s identity, which does not hold for pure loops. To overcome this difficulty, we develop a novel argument based on Cauchy's integral formula (see \Cref{sec:inductive_step}) and a new identity for primitive loops (see \Cref{lem:pure_sum}). We also provide a proof of this identity via the tree representation formula for primitive loops, presented in \Cref{sec:pf_new_identity}.

Finally, a technical difficulty arises in the continuity estimates. The continuity argument used in the bulk \cite{Band1D} relies on $G$-loop estimates propagated along a single characteristic flow line. However, near the spectral edge, the slopes of these flow lines become very small---meaning the real part of the spectral parameter changes much more rapidly than the imaginary part. As a result, applying the bulk continuity argument in this setting introduces an additional diverging factor, given by the reciprocal of the flow line's slope, which is too large to be tolerated in the subsequent analysis. Heuristically, this divergence could be compensated by potential improvements obtained from taking the imaginary part of the $G$-loop. However, realizing such an improvement is challenging: the continuity step only provides a priori estimates, which are far from optimal, and the subsequent flow step tends to degrade the imaginary part by introducing contributions from the real part of $G$.

%Finally, there is a technical difficulty associated with the continuity estimates. The continuity argument in the bulk \cite{Band1D} relies on $G$-loop estimates at earlier times along a single characteristic flow line. However, the flow lines near the edge have very small slopes, meaning the real part of the spectral parameter flow changes much faster than the imaginary part. As a result, adopting the previous continuity argument leads to an additional diverging factor given by the reciprocal of the slope of the flow line, which is intolerable for our subsequent analysis. Heuristically, this factor should be compensated by potential improvements from considering the imaginary part of the $G$-loop. However, achieving the desired improvement is challenging because the continuity step itself only provides some a priori estimates and does not yield nearly optimal estimates, while the following flow step tends to break the imaginary part and introduce contributions from the real part of $G$.

To overcome this issue, we develop a new argument (see \Cref{subsec_pf_lem_ConArg}) that propagates continuity estimates along a family of multiple flow lines. Specifically, to obtain an a priori bound on the $G$-loops at time $t$ along a flow line associated with a spectral parameter $z$, we instead use information from a different flow line at an earlier time $s<t$, where the flow penetrates ``deeper" into the bulk of the spectrum. This allows us to derive improved continuity estimates. Although these estimates remain non-optimal—due to the presence of an unavoidable diverging factor—they suffice to support the step-by-step analysis of the loop hierarchy.
This approach ultimately yields a sharp local law down to the optimal scale of $\im z$ within the spectrum $[-2,2]$. However, outside $[-2,2]$, the scale of $\im z$ remains non-optimal due to the limitations of the continuity estimate. To establish eigenvalue rigidity, we require exclusion estimates for eigenvalues outside $[-2,2]$, which in turn demand a sharp local law down to the optimal scale of $\im z$. To achieve this, we employ an additional continuity argument outside the spectrum, combined with a standard approach based on self-consistent equations; see \Cref{subsec:outspec}.

%To circumvent this issue, we develop a new argument (see \Cref{subsec_pf_lem_ConArg}) that propagates the continuity estimates along a class of multiple flow lines. Specifically, to obtain an a priori bound on the $G$-loops at time $t$ for a flow line corresponding to a spectral parameter $z$, we utilize a different flow line at time $s<t$ that goes ``much deeper" into the bulk of the spectrum. This approach allows us to achieve a better continuity estimate. While this estimate is still non-optimal (due to the presence of an undesired diverging factor), it enables us to perform later analysis of the loop hierarchy step by step. This argument will lead to a sharp local law down to the optimal scale of $\im z$ within the spectrum $[-2,2]$. However, the scale of $\im z$ outside the spectrum remains non-sharp due to the non-optimal continuity estimate. To establish the rigidity of eigenvalues, we require an exclusion estimate for the eigenvalues outside $[-2,2]$, which necessitates a sharp local law down to the optimal scale of $\im z$. To achieve this, we apply another continuity argument outside $[-2,2]$ along with a conventional argument based on self-consistent equations again; see \Cref{subsec:outspec}.

\subsection*{Organization of the remaining text}
%check words 
In \Cref{sec:main}, we present the main results of this paper. \Cref{sec:tools} introduces several key tools that will be utilized in the proofs, including the flow framework, the $G$-loops, and the primitive loops. \Cref{sec:prim} is dedicated to proving some fundamental properties for the primitive loops. The proofs of the main results are provided in \Cref{sec:proof}, which rely on some crucial $G$-loop estimates established in the key theorem, \Cref{lem:main_ind}. This theorem offers an inductive framework for extending the $G$-loop estimates along the flow to larger times $t$, step by step. 
%They are based on some key $G$-loop estimates established in the key theorem, \Cref{lem:main_ind}, which provides an inductive framework for extending the $G$-loop estimates along the flow to larger times $t$ step by step.  
The proof of \Cref{lem:main_ind} is the primary focus of \Cref{Sec:Step1,Sec:Stoflo}. In \Cref{Sec:Step1}, we establish a continuity estimate for $G$-loops, present a new proof for the weak local law of the Green's function, and extend the local law outside the support $[-2,2]$ of the semicircle law down to the optimal scale of $\im z$. Finally, in \Cref{Sec:Stoflo}, we conduct a detailed analysis of the loop hierarchy for the $G$-loop estimates. 
%cite \cite{Bandedge} here
The proof of an upper bound (\Cref{ML:Kbound}) for primitive loops will be given in \Cref{subsec:estimates}, and additional details for the proofs in \Cref{Sec:Stoflo} are deferred to \Cref{sec:main_appd}, due to their similarity with the argument in \cite[Section 5]{Band1D}.

%Readers may refer to the arXiv version of this paper \cite{RBSO2} for an appendix that provides supplementary material related to the main proofs.
%Due to the similarity between our proof of \Cref{lem:main_ind} and the argument in \cite[Section 5]{Band1D}, most technical details in \Cref{Sec:Stoflo} are deferred to \Cref{sec:main_appd}. 

%\subsection{Notations} 
To facilitate the presentation, we introduce some necessary notations that will be used throughout this paper. We will use the set of natural numbers $\N=\{1,2,3,\ldots\}$ and the upper half complex plane $\C_+:=\{z\in \C:\im z>0\}$.  
In this paper, we are interested in the asymptotic regime with $N\to \infty$. When we refer to a constant, it will not depend on $N$ or $W$. Unless otherwise noted, we will use $C$, $D$ etc.~to denote large positive constants, whose values may change from line to line. Similarly, we will use $\e$, $\delta$, $\tau$, $c$, $\fc$, $\fd$ etc.~to denote small positive constants. 
For any two (possibly complex) sequences $a_N$ and $b_N$ depending on $N$, $a_N = \OO(b_N)$, $b_N=\Omega(a_N)$, or $a_N \lesssim b_N$ means that $|a_N| \le C|b_N|$ for some constant $C>0$, whereas $a_N=\oo(b_N)$ or $|a_N|\ll |b_N|$ means that $|a_N| /|b_N| \to 0$ as $N\to \infty$. 
We say that $a_N \asymp b_N$ if $a_N = \OO(b_N)$ and $b_N = \OO(a_N)$. For any $a,b\in\R$, we denote $\llbracket a, b\rrbracket: = [a,b]\cap \Z$, $\qqq{a}:=\qqq{1,a}$, $a\vee b:=\max\{a, b\}$, and $a\wedge b:=\min\{a, b\}$. For an event $\Xi$, we let $\mathbf 1_\Xi$ or $\mathbf 1(\Xi)$ denote its indicator function.  
For any graph (or lattice), we use $x\sim y$ to mean two vertices $x,y$ are neighbors. 
Given a vector $\mathbf v$, $|\mathbf v|\equiv \|\mathbf v\|_2$ denotes the Euclidean norm and $\|\mathbf v\|_p$ denotes the $L^p$-norm. 
Given a matrix $\cal A = (\cal A_{ij})$, $\|\cal A\|$, $\|\cal A\|_{p\to p}$, and $\|\cal A\|_{\infty}\equiv \|\cal A\|_{\max}:=\max_{i,j}|\cal A_{ij}|$ denote the operator (i.e., $L^2\to L^2$) norm,  $L^p\to L^p$ norm (where we allow $p=\infty$), and maximum (i.e., $L^\infty$) norm, respectively. We will use $\cal A_{ij}$ and $ \cal A(i,j)$ interchangeably in this paper. Moreover, we introduce the following simplified notation for trace: $ \left\langle \cal A\right\rangle=\tr (\cal A) .$

%Moreover, we will use the simplified notation of generalized matrix entries: given any vectors $\mathbf u,\bv$, we denote $ \cal A_{\mathbf u\mathbf v}:= \bu^* \cal A\bv$ and $\cal A_{x\mathbf v}:= \mathbf e_x^* \cal A\bv$, where $\mathbf e_x$ is the standard basis unit vector along the $x$-th direction. 

\medskip

\noindent{\bf Acknowledgement.}  
Fan Yang is supported in part by the National Key R\&D Program of China (No.~2023YFA1010400). We would like to thank Horng-Tzer Yau for fruitful discussions.

\section{Main results} \label{sec:main}

Denote the eigenvalues of $H$ by $\lambda_1\le \lambda_2\le \cdots\le \lambda_N$ and the corresponding eigenvectors by $\bu_1,\ldots, \bu_N$. 
Our first main result concerns the delocalization of eigenvectors. In particular, we observe the following phenomenon in dimension $d = 1$: when $N \gg W^{1/2}$, the bulk eigenvectors are delocalized, as established in \cite{Band1D}. As $W$ increases, the ``mobility edges" move closer to the spectral edges $\pm 2$, causing some eigenvectors in the transition regime (between the bulk and the edges) to become delocalized, depending on the scaling of $W$. Finally, when \smash{$W \gg N^{5/6}$}, all eigenvectors of $H$ become delocalized. A similar transition occurs in dimension $d = 2$ as $W$ increases from $W \gg 1$ to $W \gg N^{2/3}$.

\begin{theorem}[Delocalization]\label{thm:supu} 
In the setting of \Cref{def: BM} with $d\in\{1,2\}$, 
suppose $W\ge L^{\al}$ for a constant $\mathbf 1_{d=1}/2 < \al \le 1-d/6$. Then, the following delocalization estimate holds for any constants $\e,
\tau, D>0$:
\begin{align}\label{eq:delocalmax2}  
\P\Big(\max_{k: |\lambda_k | \leq 2- N^{ - c_{d,\al} +\e} } \|\bu_k\|_\infty^2 \le  N^{-1+\tau} \Big)\ge 1-N^{-D} 
\end{align}
if $N$ is sufficiently large, where $c_{d,\al}$ is defined as 
\[c_{d,\al}:=\begin{cases}
    2\al-1, \ & \text{if} \ d=1\\
    \al, \ & \text{if} \ d=2
\end{cases}.\]
%2\al-1$ in dimension $d=1$, and $c_{d,\al}:=\al$ in dimension $d=2$. 
%\[{\cor N^{2-d+\e}/W^{2}: N^{1-2\al+\e}, \ d=1; \ \ N^{-\al+\e}, \ d=2}\]
% and in dimension $d=2$,
% \begin{align}\label{eq:delocalmax3}  
% \sup_{k: |\lambda_k | \leq 2-W^{-2-c}} \|\bu_k\|_\infty^2 \prec  N^{-1} \, . 
% \end{align}
%where $a_{d,\al}$ is defined as $c_{d,\al}:=N^{1-2\al}$ when $d=1$, and $c_{d,\al}:=N^{-\al}$ when $d=2$.
Furthermore, if the constant $\al$ satisfies $\alpha > 1-d/6$, then all eigenvectors are delocalized: for any constants $\tau,D>0$, 
\begin{align}
\P\Big(\max_{k=1}^N \|\bu_k\|_\infty^2 \le N^{-1+\tau}\Big) \ge 1-N^{-D} \, \label{eq:delocalmax} 
\end{align}
if $N$ is sufficiently large.
\end{theorem}

The result \eqref{eq:delocalmax2} suggests that, for $L^{\mathbf 1_{d=1}/2}\ll W\ll L^{1-d/6}$, there should be a sharp phase transition between delocalization and localization as $2-|\lambda_k|$ crosses $N^{-c_{d,\al}+\e}$. This indicates the existence of ``mobility edges" near $\pm(2-N^{-c_{d,\al}})$, which separate the delocalized phase within the bulk from the localized phases at the edge. To establish this conjecture, further work is needed to establish the localization of the edge eigenvectors of $H$ in the subcritical regime $W\ll L^{1-d/6}$. 
We also remark that, as an extension of \Cref{thm:supu}, we will present a more complete---though slightly more technical---result in \Cref{cor:localizationlength} below. That result provides upper bounds on the $L^\infty$-norm of \emph{all eigenvectors} of $H$, uniformly across the \emph{entire parameter regime $0<\al\le 1$}.  In particular, it includes edge eigenvectors corresponding to energies outside the interval $[-2+N^{-c_{d,\al}+\e},2-N^{-c_{d,\al}+\e}]$ in the regime $\al\le 1-d/6$. These upper bounds imply corresponding lower bounds on the \emph{localization length} of the eigenvectors of $H$, which characterizes the typical length scale over which most of the $L^2$-mass of an eigenvector is concentrated.

We can further establish a \emph{quantum unique ergodicity} (QUE) estimate for all the delocalized eigenvectors of $H$. Previously, such QUE estimates have only been established in the bulk regime \cite{BandIII,Band1D,Band2D,RBSO1D}. %we extend them here to the entire spectrum that corresponds to the delocalized eigenvectors. in the supercritical regime with $W\gg L^{1-d/6}$. 

\begin{theorem}[Quantum unique ergodicity]\label{thm:QUE}
In the setting of \Cref{def: BM} with $d\in\{1,2\}$, 
suppose $W\ge L^{\al}$ for a constant $\mathbf 1_{d=1}/2 < \al \le 1-d/6$. Given any $|E|\le 2-N^{-c_{d,\al}+\e}$ for a constant $\e>0$, define the subset \({\cal I}_E\equiv {\cal I}_E(\fd):=\left\{x: |x-E|\le W^{-\fd}\eta_{0}(E)\right\}\) for an arbitrary constant $\fd >0$, where 
$$ \eta_{0}(E):=\begin{cases}
W/N^{3/2}, & \text{if} \ d=1\\
W/{N}, & \text{if} \ d=2\\
\end{cases}.$$ 
Then, for each $d\in\{1,2\}$, there exists a small constant $c$ depending on $\e$ and $\fd$ such that the following estimate holds for large enough $L$:
\begin{equation}\label{Meq:QUE_weak}
\max_{[a]\in \Zn} \P\bigg(\max_{i,j:\lambda_i, \lambda_j \in {\cal I}_E} 
 \bigg| \sum_{x\in[a]}\overline \bu_i(x)\bu_j(x)-\frac{W^{d}}{N}\delta_{ij}  \bigg|^2 \ge \frac{W^{d-c}}{N} \bigg) \le  W^{-c} \, .
\end{equation}
More generally, for any subset $A\subset \Zn$, we have
\begin{equation}\label{Meq:QUE2_weak}
\mathbb{P}\bigg(\max_{k: \lambda_k\in {\cal I}_E} 
\bigg|\sum_{[a]\in A}\sum_{x\in [a]}\left|\bu_k(x)\right |^2 -\frac{W^d}{N}|A|\bigg| \ge  \frac{W^{d-c}|A|}{N}  \bigg) \le W^{-c}.
\end{equation}  
Furthermore, if $W\ge L^{1-d/6+\e_0}$ for a constant $\e_0>0$, then there exists a constant $c>0$ (depending on $\e_0$) such that the following estimates hold for large enough $N$:
\begin{equation}\label{Meq:QUE}
\max_{k=1}^N  
\max_{ [a]\in \Zn} \P\bigg( 
 \bigg| \sum_{x\in[a]}|\bu_k(x)|^2 -\frac{W^{d}}{N} \bigg|^2 \ge \frac{W^{d-c}}{N} \bigg) \le  W^{-c},
\end{equation}
and for any subset $A\subset \Zn$,
\begin{equation}\label{Meq:QUE2}
\max_{k=1}^N 
\max_{ A\subset\ZL}\mathbb{P}\bigg(
\bigg|\sum_{[a]\in A}\sum_{x\in [a]}\left|\bu_k(x)\right |^2 -\frac{W^d}{N}|A|\bigg| \ge  \frac{W^{d-c}|A|}{N}  \bigg) \le W^{-c}.
\end{equation}  
\end{theorem}

The above QUE estimates indicate that with probability $1-\oo(1)$, every delocalized eigenvector of $H$ is asymptotically uniformly distributed (in the sense of $L^2$-mass) across all microscopic scales larger than $W$. In particular, this implies that the localization length of such an eigenvector is indeed of order $\Omega(L)$. 
We also note that the QUE estimates \eqref{Meq:QUE} and \eqref{Meq:QUE2} are slightly stronger than \eqref{Meq:QUE_weak} and \eqref{Meq:QUE2_weak}, as well as the corresponding results in \cite{Band1D, Band2D} for the bulk regime. Specifically, our results establish QUE for each individual eigenvector of $H$, whereas \eqref{Meq:QUE_weak} and \eqref{Meq:QUE2_weak} require additional assumptions on the locations of the eigenvalues $\lambda_k$. 
This improvement is made possible by our next main result, which establishes a rigidity estimate for the eigenvalues of $H$ in the supercritical regime $W \gg L^{1 - d/6}$. 

%Note that the QUE estimates \eqref{Meq:QUE} and \eqref{Meq:QUE2} are slightly stronger than \eqref{Meq:QUE_weak} and \eqref{Meq:QUE2_weak}, as well as those in \cite{Band1D, Band2D} within the bulk regime. Specifically, we establish QUE for each individual eigenvector of $H$, while \eqref{Meq:QUE_weak} and \eqref{Meq:QUE2_weak} require additional assumptions on the locations of the eigenvalues $\lambda_k$. This improvement is primarily due to our next main result, which provides a rigidity estimate for the eigenvalues of $H$ in the supercritical regime $W \gg L^{1 - d/6}$.

For $k\in \qqq{N}$, we define the $k$-th quantile $\gamma_k$ for the semicircle law as the unique real solution to the following equation:
\begin{equation}\label{eq:gammak}
\int_{-2}^{\gamma_k}\rho_{sc}(x)\dd x= \frac{k-1/2}{N},  \quad \text{where}\quad \rho_{sc}(x)=\frac{\sqrt{4-x^2}}{2\pi}\mathbf 1_{x\in [-2,2]}\, .
\end{equation}
By definition, it is easy to see that 
\be\label{eq:gammak_location} |\gamma_k+2|\asymp \p{k/N}^{2/3},\quad |\gamma_k-2|\asymp \p{(N+1-k)/N}^{2/3} . \ee
We show that all the eigenvalues of $H$ concentrate around their corresponding quantiles in the supercritical regime. For $k\in \qqq{N}$, we denote
\[\wh k:=\min(k,N+1-k).\]

\begin{theorem}[Eigenvalue rigidity]\label{thm:rigidity}
In the setting of \Cref{def: BM} with $d\in\{1,2\}$, suppose $W\ge L^{1-d/6+\e_0}$ for a constant $\e_0>0$. Then, the following estimates hold for all $k\in \qqq{N}$ and any constants $\tau, D>0$ if $N$ is sufficiently large: when $d=1$, we have
\be\label{eq:rigidity}
\P\left(\left|\lambda_k - \gamma_k\right|\le  N^{- 2/ 3+\tau}\wh k^{-1/3} + W^{-1+\tau} N^{1/6}\wh k^{-1/6}\right)\ge 1-N^{-D} \, ,
\ee
and when $d=2$, we have
\be\label{eq:rigidityd=2}
\P\left(\left|\lambda_k - \gamma_k\right|\le  N^{-2/3+\tau}\wh k^{-1/3} + W^{-2+\tau}\right)\ge 1-N^{-D} \, .
\ee
\end{theorem}  

%It has been shown in \cite{ErdYauYin2012Rig,Semicircle} that, 
To establish the above results, \Cref{thm:supu,thm:QUE,thm:rigidity}, we will derive a sharp local law for the Green's function of $H$, defined as in \eqref{def_Green}. As $W\to \infty$, $G(z)$ converges to the following Stieltjes transform of the Wigner semicircle law: 
\be\label{self_mWO}
m(z)\equiv m_{sc}(z):=\frac{-z+\sqrt{z^2-4}}{2},\quad M(z):=m_{sc}(z)I_N.
\ee
For any $E\in \R$, we denote by $\kappa_E:=|2-|E||=|E-2|\wedge |E+2|$ the distance of $E$ from the spectral edges $\pm 2$. Through direct calculations, we can derive the following identities 
\be\label{L2M} 
z=-m(z)-\frac{1}{m(z)},\quad |m(z)|^2=\frac{\im m(z)}{\eta + \im m(z)},
\ee
and the following basic estimates uniformly in all $z=E+\ii\eta$ with $|z|\le c^{-1}$ and $\eta > 0$ (where $c>0$ is an arbitrarily small constant):  \begin{equation}\label{square_root_density}
\im m(z)\asymp \begin{cases}
                    \sqrt{\kappa_E+\eta}, & \text{ if }\ E\in [-2,2]\\
                    {\eta}/{\sqrt{\kappa_E+\eta}}, & \text{ if }\ E\notin [-2,2]
                \end{cases},
\quad \text{and}\quad  |1-m^2(z)|\asymp \sqrt{\kappa_E+\eta}.
        \end{equation} 
% We also introduce the following parameter:
% \be\label{eq:defomega} 
% \omega(z):=\eta/\im m(z) \, .
% \ee
Given a small constant $\fc>0$, for $E\in [-2,2]$, we define $\eta_*(E)\equiv \eta_*(E,\fc)>0$ as 
\be\label{eq:defeta*}
\eta_*(E):= 
\begin{dcases}
    \inf\big\{\eta>0: N\eta\sqrt{\kappa_E+\eta}\ge 1,\ W^2\eta (\kappa_E+\eta)^{3/2}\ge 1\big\}, & \text{if}\ d=1 \\
    \inf\left\{\eta>0: N\eta\sqrt{\kappa_E+\eta}\ge 1,\ W^2 (\kappa_E+\eta) \ge W^\fc\right\}, & \text{if}\ d=2
\end{dcases},
\ee
while for $E\notin [-2,2]$ and $d\in\{1,2\}$, we define $\eta_*(E)>0$ as
\be\label{eq:defeta*2}
\eta_*(E):= \inf\left\{\eta>0: N\eta^2/\sqrt{\kappa_E+\eta}\ge 1,\ W^d\eta^2/(\kappa_E+\eta)^{\frac{1}{2}+\frac{d}{4}}\ge 1\right\}.
\ee
We have a sharp local law for the Green's function $G(z)$ for $\im z$ down to the scale $W^{\fd}\eta_*(E)$, where $\fd>0$ is an arbitrarily small constant. Given any large constant $C_0>2$ and small constant $\fd>0$, we define the spectral domain:
\be\label{eq:defDC0} \mathbf D_{C_0,\fd}\equiv \mathbf D_{C_0,\fd}(\fc):=\left\{z=E+\ii\eta\in \C_+: |E|\le C_0, W^\fd \eta_*(E,\fc)\le \eta \le 1\right\}.\ee

\begin{theorem}[Local law]\label{thm_locallaw}
For the model in \Cref{def: BM} with $d\in\{1,2\}$, assume that $W\ge L^\delta$ for a constant $\delta>0$. For any constants $C_0,\fd, \fc, \tau, D>0$, the following events hold with probability $\ge 1-N^{-D}$ for large enough $N$:  
\begin{align}\label{locallaw}
& %\bigcap_{|E|\le \fd^{-1}} \bigcap_{W^{\fd}\eta_*(E,\fc)\le \eta\le 1} 
\bigcap_{z=E+\ii\eta\in \mathbf D_{C_0,\fd}(\fc)} 
\bigg\{\|G(z)-M(z)\|_{\max}  \le  W^\tau \sqrt{\frac{\im m(z)}{W^d\ell(z)^d \eta}}%+\frac{W^\e}{W^d\ell(z)^d \eta} 
\bigg\}, 
\\
\label{locallaw_aver}
&\bigcap_{z=E+\ii\eta\in \mathbf D_{C_0,\fd}(\fc)} \bigg\{\max_{[a]} \Big|W^{-d}\sum_{x\in [a]}G_{xx}(z) - m(z)\Big| \le \frac{W^\tau\im m(z)}{W^d\ell(z)^d \eta \sqrt{\kappa_E+\eta}}\bigg\} , 
\end{align}
where we introduced the following notation: 
\be\label{eq:elleta}
%\sN(z):=W^d\ell(z)^d ,\quad \text{with}\quad 
\ell(z):= \min\left(L/W,\sqrt{{\im m(z)/{\eta}}}\right) . 
\ee
\end{theorem}

Outside the support $[-2,2]$ of the semicircle law, the above local laws actually hold down to a smaller scale $\eta_\circ(E)$, defined as
\be\label{eq:defeta*3}
\eta_{\circ}(E):= \inf\left\{\eta>0: N\eta \sqrt{\kappa_E+\eta}\ge 1,\ W^d\eta (\kappa_E+\eta)^{\frac{1}{2}-\frac{d}{4}}\ge 1\right\}.
\ee
Given any constants $c_0,\fd>0$, we define the spectral domain $\mathbf D^{\mathrm{out}}_{c_0,\fd}$ as 
\[\mathbf D^{\mathrm{out}}_{c_0,\fd}:=\big\{z=E+\ii\eta \in \C_+: 2\le |E|\le c_0^{-1},  \kappa_E\ge W^{c_0}\eta_*(E),  W^\fd\eta_\circ(E)\le \eta\le 1\big\}.\]

\begin{theorem}[Local law outside the support]\label{thm_locallaw_out}
In the setting of \Cref{thm_locallaw}, for any constants $c_0,\fd,\tau,D>0$, the following events hold with probability $\ge 1-N^{-D}$ for large enough $N$:
\begin{align}
&\bigcap_{z=E+\ii\eta\in \mathbf D^{\mathrm{out}}_{c_0,\fd}} \bigg\{\|G(z)-M(z)\|_{\max}  \le  W^\tau\sqrt{ \frac{1}{W^d\ell(z)^d \sqrt{\kappa_E+\eta}}} \bigg\},\label{locallaw_entry_out} \\
&\bigcap_{z=E+\ii\eta\in \mathbf D^{\mathrm{out}}_{c_0,\fd}}  \bigg\{\max_{[a]} \Big|W^{-d}\sum_{x\in [a]}G_{xx}(z)-m(z) \Big| \le \frac{W^\tau}{W^d\ell(z)^d\p{\kappa_E+\eta}}\bigg\}. \label{locallaw_aver_out} 
\end{align}
 
%holds with probability $\ge 1-N^{-D}$ for large enough $N$. 
\end{theorem}

We also extend the \emph{quantum diffusion} of random band matrices,  previously established in the bulk regime \cite{BandI,Band1D,Band2D,RBSO1D}, to the entire spectrum. This extension will play a key role in deriving the QUE estimates in \Cref{thm:QUE}. To state the result precisely, we first introduce the relevant matrices $\Theta$ and $S^\pm$. 

%We introduce the variance matrix $S$ on $\ZL$ and its ``projection" to $\Zn$.   
\begin{definition}\label{def:projlift}
Define the variance matrix \smash{$S=(S_{xy})$} as
\be\label{eq:SWO}
S_{xy}=\var(H_{xy})= \frac{W^{-d}}{1+2d\lambda^2}\left(I_n  + \lambda^2\Lambda_n\right)_{[x][y]}, %\otimes \bE,
\quad \forall x,y\in \ZL,
\ee
where $I_n$ and $\Lambda_n$ are respectively the identity and adjacency matrices defined on \smash{$\Zn$}. Then, we define the following matrices adopting the notations in \cite{BandI}:
\be \label{def:Theta}
\Theta(z): = \frac{|m(z)|^2}{1-|m(z)|^2 S},\quad S^+(z)%=(S^-(z))^*
:= \frac{m(z)^2}{1-m(z)^2 S}\, .
\ee
\end{definition}

\begin{theorem}[Quantum diffusion]\label{thm_diffu}
In the setting of Theorem \ref{thm_locallaw}, for any constants $C_0, \fd, \fc, \tau, D>0$, the following events hold with probability $\ge 1-N^{-D}$ for large enough $N$: 
\begin{align}
%&\bigcap_{|E|\le \fd^{-1}}\bigcap_{W^{\fd}\eta_*(E,\fc)\le \eta\le 1}
&\bigcap_{z=E+\ii\eta\in \mathbf D_{C_0,\fd}(\fc)}
\bigg\{\max_{[a],[b]} \bigg|%\frac{1}{W^{2d}}
\sum_{x\in[a],y\in[b]}\left(|G_{xy}(z)|^2 -\Theta_{xy}(z)\right)\bigg| \le\frac{W^\tau}{[\ell(z)^d\eta]^2}\bigg\}\, ,
\label{eq:diffu1}\\
&\bigcap_{z=E+\ii\eta\in \mathbf D_{C_0,\fd}(\fc)}\bigg\{\max_{[a],[b]} \bigg|%\frac{1}{W^{2d}}
\sum_{x\in[a],y\in[b]}\left(G_{xy}(z)G_{yx}(z) -S^+_{xy}(z)\right)\bigg| \le\frac{W^\tau}{[\ell(z)^d\eta]^2}\bigg\} \, .
\label{eq:diffu2}
\end{align}
Moreover, stronger bounds hold in the sense of expectation for each \(z\in \mathbf D_{C_0,\fd}(\fc)\): %with $|E|\le \fd^{-1}$ and $W^{\fd}\eta_*(E,\fc)\le \eta\le 1$: 
\begin{align}
%\frac{1}{W^{2d}}
\max_{[a],[b]}\bigg|\sum_{x\in[a],y\in[b]}\E\left(|G_{xy}(z)|^2 -\Theta_{xy}(z)\right) \bigg| &\le \frac{W^{2d}}{\im m(z)} \cdot \frac{W^\tau}{[W^d\ell(z)^d\eta]^3},\label{eq:diffuExp1}\\ 
%\frac{1}{W^{2d}}
\max_{[a],[b]}\bigg|\sum_{x\in[a],y\in[b]}\E\left(G_{xy}(z)G_{yx}(z) -S^+_{xy}(z)\right)\bigg|&\le \frac{W^{2d}}{\im m(z)} \cdot \frac{W^\tau}{[W^d\ell(z)^d\eta]^3}.\label{eq:diffuExp2}
\end{align}
\end{theorem}

\Cref{thm:supu,thm:rigidity} follow immediately from the local laws in \Cref{thm_locallaw,thm_locallaw_out}. For clarity of presentation, we will adopt the following notion of stochastic domination introduced in \cite{EKY_Average} in the following proof of \Cref{thm:supu,thm:rigidity} and throughout the remainder of the paper to simplify our presentation. 
%It will be used frequently throughout the remainder of the paper to simplify the statements of our main results and their proofs.

\begin{definition}[Stochastic domination and high probability event]\label{stoch_domination}
	%\begin{itemize}
	%\item[(i)] 
	{\rm{(i)}} Let
	\[\xi=\left(\xi^{(N)}(u):N\in\mathbb N, u\in U^{(N)}\right),\hskip 10pt \zeta=\left(\zeta^{(N)}(u):N\in\mathbb N, u\in U^{(N)}\right),\]
	be two families of non-negative random variables, where $U^{(N)}$ is a possibly $N$-dependent parameter set. We say $\xi$ is stochastically dominated by $\zeta$, uniformly in $u$, if for any fixed (small) $\tau>0$ and (large) $D>0$, 
	\[\mathbb P\bigg(\bigcup_{u\in U^{(N)}}\left\{\xi^{(N)}(u)>N^\tau\zeta^{(W)}(u)\right\}\bigg)\le N^{-D}\]
	for large enough $N\ge N_0(\tau, D)$, and we will use the notation $\xi\prec\zeta$. 
	If for some complex family $\xi$ we have $|\xi|\prec\zeta$, then we will also write $\xi \prec \zeta$ or $\xi=\OO_\prec(\zeta)$. 
	
	\vspace{5pt}
	\noindent {\rm{(ii)}} As a convention, for two \emph{deterministic} non-negative quantities $\xi$ and $\zeta$, we will write $\xi\prec\zeta$ if and only if $\xi\le N^\tau \zeta$ for any constant $\tau>0$. 
	
	\vspace{5pt}
	\noindent {\rm{(iii)}} We say that an event $\Xi$ holds with high probability (w.h.p.) if for any constant $D>0$, $\mathbb P(\Xi)\ge 1- N^{-D}$ for large enough $N$. More generally, we say that an event $\Omega$ holds $w.h.p.$ in $\Xi$ if for any constant $D>0$,
	$\P( \Xi\setminus \Omega)\le N^{-D}$ for large enough $N$.
	%\end{itemize}
\end{definition}

\begin{proof}[\bf Proof of \Cref{thm:supu}]
The delocalization estimates \eqref{eq:delocalmax2} and \eqref{eq:delocalmax} follow directly from the entrywise local law \eqref{locallaw} via the bound 
\be\label{eq:ukx}|\bu_k(x)|^2 \le \eta\im G_{xx}(\lambda_k + \ii \eta),\quad \forall \eta>0\, .\ee 
We first prove \eqref{eq:delocalmax} under the condition $\al\ge 1-d/6 +\e_0$ for a constant $\e_0>0$. By \eqref{eq:rigidity} and \eqref{eq:rigidityd=2}, all eigenvalues of $H$ lie in the interval $I_\e:=[-2-N^{-2/3+\e},2+N^{-2/3+\e}]$ with high probability for any constant $\e>0$. Then, we take 
\be\label{etaE}\eta(E)=N^{-1+\fd}(\kappa_E+N^{-2/3})^{-1/2}\, \ee  
for a small constant $0<\fd <\e_0/2$. We can check directly that if $0<\e <\fd/2$, then $\eta(E)\ge N^{\fd/4} \eta_*(E)$ for all $E\in I_\e$. 
Furthermore, with the estimate \eqref{square_root_density}, we can verify that $\ell(z)\ge n$ for $z=E+\ii \eta$ with $E\in I_\e$ and $\eta=\eta(E)$. Thus, by \eqref{locallaw}, we have the following estimate uniformly in $E\in I_\e$:
\[ \max_x \eta(E) \im G_{xx}(z_E) \prec  \eta(E)\im m(z_E)+ \eta(E)\sqrt{\frac{\im m(z_E)}{N\eta(E)}} \le \frac{N^{2\fd}}{N}, \]
where $z_E$ denotes $z_E=E+\ii \eta(E)$, and we used \eqref{square_root_density} again in the second step. Since $\fd$ can be arbitrarily small, this concludes the estimate \eqref{eq:delocalmax}. 

The proof of \eqref{eq:delocalmax2} is similar by applying the local law \eqref{locallaw} at $z=E+\ii \eta$ with $|E|\le 2- N^{-c_{d,\al}+\e}$ and $\eta=\eta(E)$. Here, the condition $2- N^{-c_{d,\al}+\e}$ is required to guarantee that $\eta(E)\ge  W^c\eta_*(E)$ for a constant $c>0$ under the assumption $\mathbf 1_{d=1}/2 < \al \le 1-d/6$.
\end{proof}

\begin{proof}[\bf Proof of \Cref{thm:rigidity}]
%\section{Proof of eigenvalue rigidity}\label{sec:rigidity}
The rigidity of eigenvalues in \Cref{thm:rigidity} follows essentially from the averaged local law \eqref{locallaw_aver}. However, we first need to bound the largest and smallest eigenvalues of $H$; specifically, we need to show that 
\be\label{eq:rigid_largest}
\P\p{\lambda_1\ge -2-N^{-2/3+\e}, \  \lambda_N\le 2+N^{-2/3+\e}}\ge 1-N^{-D},
\ee
for any constants $\e,D>0$. To this end, we apply the averaged local law \eqref{locallaw_aver_out} outside the support of the semicircle law and obtain that % the following averaged local law 
\begin{align}\label{locallaw_aver_out2}
&{N}^{-1}\tr G(z) - m(z) \prec \big[W^d\ell(z)^{d}\kappa_E\big]^{-1}   
\end{align}
uniformly for all $z=E+\ii \eta$ with\footnote{Note that under the assumption $W\ge L^{1-d/6+\e_0}$, if we choose $\fd<(\e\wedge\e_0)/4$, then $\kappa_E\ge W^{\e/2}\eta_*(E)$ and $\eta  \ge W^{\e/4} \eta_{\circ}(E)$ by the definitions \eqref{eq:defeta*2} and \eqref{eq:defeta*3}.} 
\be\label{eq;Eeta}
E\in \big[2+W^{\e}N^{-2/3}, \e^{-1}\big],\quad \text{and}\quad \eta=W^{-\fd}\p{\sqrt{\kappa_E}/N}^{1/2}.\ee

For $z=E+\ii \eta$ satisfying \eqref{eq;Eeta}, by \eqref{eq:elleta} and \eqref{square_root_density}, we have that
$$  \im m(z)\asymp \frac{\eta}{\sqrt{\kappa_E}} \le \frac{W^{-2\fd}}{N\eta}, \quad \frac{1}{W^d\ell(z)^d\kappa_E}\lesssim \frac{1}{W^d\kappa_E^{1-d/4}} \le \frac{W^{-\frac{\e}{4}}}{N\eta} \, .$$
Combining these facts with \eqref{locallaw_aver_out2} and applying the same argument as in \cite[Section 11.1]{erdHos2017dynamical}, we can show that, with high probability, $H$ has no eigenvalues in the regime $[2+W^{\e}N^{-2/3}, \e^{-1}]$ for any constant $\e>0$. By symmetry, $H$ also has no eigenvalues in the regime $[-\e^{-1}, -2-W^{\e}N^{-2/3}]$ with high probability.
Furthermore, it is known that with high probability, the operator norm of $H$ is at most $\e^{-1}$ for small enough constant $\e>0$ (see e.g., the bound  \eqref{eq:largest_eig} below). This establishes \eqref{eq:rigid_largest}. 

% Now, using \eqref{locallaw_aver}, \eqref{locallaw_aver_out}, and the condition $W\ge L^{1-d/6+\e_0}$, we obtain the averaged local laws 
% \be\label{aver1} 
% \avg{G(z)}-m(z)\prec (N\eta)^{-1}
% \ee
% uniformly in $z=E+\ii\eta \in \mathbf D'_{\fd,c}$ for any small constants $\fd \in (0,\e_0/10)$ and $c\in(0,\fd/10)$, where 
% $$\mathbf D'_{\fd,c}:=\left\{E+\ii \eta: E\in [-2-N^{-2/3+c},2+N^{-2/3+c}], \frac{N^\fd}{N\sqrt{\kappa_E+N^{-2/3}}} \le \eta \le 1\wedge  \frac{\sqrt{\kappa_E+n^{-4}}}{n^2}\right\}.$$ 
Combining \eqref{eq:rigid_largest} with the averaged local law \eqref{locallaw_aver}, and applying the arguments in \cite[Section 5]{ErdYauYin2012Rig} or \cite[Section 8]{EKYY_ER1}, we can derive the rigidity of eigenvalues stated in \eqref{eq:rigidity} and \eqref{eq:rigidityd=2}. Since the argument is standard, we omit the details. \end{proof}

As an extension of \Cref{thm:supu}, in the subcritical regime $1\ll W \le L^{1-d/6}$, the local laws \eqref{locallaw} and \eqref{locallaw_entry_out} allow us to derive upper bounds on the $L^\infty$-norms of all eigenvectors of $H$. These bounds provide insight into the localization lengths of eigenvectors near the spectral edges. We summarize these estimates in the following corollary. For any $\kappa>0$, we introduce the intervals 
\begin{align*}
    I_{\mathrm{in}}(\kappa)&:=\{E\in \R: |E|< 2, \kappa\le \kappa_E \le 2\kappa\},\\
    I_{\mathrm{out}}(\kappa)&:=\{E\in \R: |E|> 2, \kappa\le \kappa_E \le 2\kappa\}.
\end{align*}

\begin{corollary}\label{cor:localizationlength}
In the setting of \Cref{def: BM} with $d\in\{1,2\}$,  if $W\ge L^{\delta}$ for a constant $\delta>0$, then the following estimates holds: when $d=1$, we have 
\be\label{eq;localvvvd=1}
\begin{aligned}
\max_{k: |2-|\lambda_k||\leq N^{-2/3} + W^{-4/5}} \|\bu_k\|_\infty^2& \prec N^{-1} + W^{-6/5} ,\\
\max_{k: \lambda_k \in I_{\mathrm{in}}(\kappa)} \|\bu_k\|_\infty^2 &\prec N^{-1} + \big(W^2\kappa\big)^{-1} ,\\
\max_{k: \lambda_k \in I_{\mathrm{out}}(\kappa)} \|\bu_k\|_\infty^2 &\prec N^{-1} + \big(W^4\kappa\big)^{-3/8},
\end{aligned}
\ee
for any $\kappa$ satisfying $N^{-2/3} + W^{-4/5}\le \kappa\le C$ for some constant $C>0$; %if $d=2$ and $N^{-\frac 23+\e} + W^{-\frac{4d}{6-d}+\e}\le a_N\le C$ for some constants $\e,C>0$, 
when $d=2$, the following estimates hold for any constant $\e>0$:
\be\label{eq;localvvvd=2}
\begin{aligned}
&\max_{k: |\lambda_k|\le  2-(N^{-2/3} + W^{-2+\e})} \|\bu_k\|_\infty^2 \prec N^{-1} ,\\
&\max_{k:  |\lambda_k|\ge 2-(N^{-2/3} + W^{-2+\e})} \|\bu_k\|_\infty^2 \prec N^{-1} + W^{-3} .
%\quad \max_{k: \lambda_k \in I_{\mathrm{out}}(a_N)} \|\bu_k\|_\infty^2 \prec N^{-1} + W^{-3} .
\end{aligned} 
\ee
%where $N^{-2/3} + W^{-2+\e}\le a_N\le C$ for some constants $\e,C>0$.
% \begin{align*}
% \|\mathbf u_k\|_{\infty}^2 &\prec N^{-1} + W^{-\frac{6d}{6-d}} ,\quad \text{if}\quad \kappa_E\le N^{-\frac23}+ W^{-\frac{4d}{6-d}},\\
% \|\mathbf u_k\|_{\infty}^2 &\prec N^{-1} + \p{{W^d\kappa_E^{\frac{1}{2}-\frac{d}{4}}}}^{-3/2} ,\quad \text{if}\quad \kappa_E> N^{-\frac23}+ W^{-\frac{4d}{6-d}}, |E|\ge 2\\
% \|\mathbf u_k\|_{\infty}^2 &\prec N^{-1} + \p{W^2\kappa_E}^{-1}\mathbf 1_{d=1} ,\quad \text{if}\quad \kappa_E> N^{-\frac23}+ W^{-\frac{4d}{6-d}}, |E|<2
% \end{align*}
\end{corollary}
\begin{proof}
As in the proof of \Cref{thm:supu}, these estimates follow directly from the inequality \eqref{eq:ukx}, the local laws \eqref{locallaw} and \eqref{locallaw_entry_out}, the definitions \eqref{eq:defeta*}, \eqref{eq:defeta*2}, and \eqref{eq:defeta*3}, and the estimate \eqref{square_root_density}. We omit the details. 
\end{proof}

The proofs of QUE, local laws, and quantum diffusion  (\Cref{thm_locallaw,thm_diffu,thm:QUE}) will be presented in \Cref{sec:proof}. By combining all the main results established above---including the local laws, eigenvalue rigidity, eigenvector delocalization, and QUE---we can prove that the edge eigenvalue statistics of $H$ asymptotically match those of GUE in the supercritical regime $W \gg L^{1 - d/6}$. This follows from a Green's function comparison argument developed in \cite{BandIII} for the proof of Theorem 1.3 there. Similar methods have also been applied in \cite{Band1D,Band2D} to establish the universality of bulk eigenvalue statistics for 1D and 2D Gaussian random band matrices. Extending this method to the edge regime is straightforward using the tools developed in this paper. (For instance, a related argument was used in \cite{QC_FSYY} to establish the edge universality for a random block matrix model.) However, the edge eigenvalue statistics---particularly the Tracy–Widom law for the extreme eigenvalues---have already been established for 1D random band matrices in \cite{Sod2010} and for 2D random band matrices in \cite{Band_Edge123} using an intricate moment method. Therefore, we do not pursue this direction further in the present work.

%Combining all the main results established above (including the local laws, rigidity of eigenvalues, delocalization of eigenvectors, and the QUE estimates), we can prove that the edge eigenvalue statistics of $H$ match those of GUE asymptotically by employing a Green's function comparison argument developed in \cite{BandIII}. This argument is analogous to that used in the proof of \cite[Theorem 1.3]{BandIII}, and a similar approach has also been implemented in \cite{Band1D,Band2D} to show the universality of bulk eigenvalue statistics for 1D and 2D Gaussian random band matrices. The extension of the method to the edge regime is also straightforward, with the results we have established. For example, such an argument has been applied in \cite{QC_FSYY} to establish the edge universality for a different random block matrix model. However, the edge eigenvalue statistics, particularly the Tracy-Widom law of the extreme eigenvalues, have been established for 1D random band matrices in \cite{Sod2010} and 2D random band matrices in \cite{Band_Edge123} using an intricate moment method. Hence, we refrain from pursuing this direction in this paper. 

\section{Main tools}\label{sec:tools}

In this section, we introduce some key tools and convenient notations for our proof. First, the following classical Ward's identity, which follows from a simple algebraic calculation, will be used tacitly throughout the proof. 

\begin{lemma}[Ward's identity]\label{lem-Ward}
Given any Hermitian matrix $\cal A$, define its resolvent as $R(z):=(\cal A-z)^{-1}$ for a $z= E+ \ii \eta\in \C_+$. Then, we have 
    \be\label{eq_Ward0}
    \begin{split}
\sum_x \overline {R_{xy'}}  R_{xy} = \frac{R_{y'y}-\overline{R_{yy'}}}{2\ii \eta},\quad
\sum_x \overline {R_{y'x}}  R_{yx} = \frac{R_{yy'}-\overline{R_{y'y}}}{2\ii \eta}.
\end{split}
\ee
As a special case, if $y=y'$, we have
\be\label{eq_Ward}
\sum_x |R_{xy}( z)|^2 =\sum_x |R_{yx}( z)|^2 = \frac{\im R_{yy}(z) }{ \eta}.
\ee
\end{lemma}

%here and throughout the rest of this paper, we will use the following convenient notion of stochastic domination introduced in \cite{EKY_Average}. 

\subsection{Flows}\label{subsec:flow}

In this subsection, we introduce the flow framework for our proof, which extends that employed for RBMs within the bulk regime \cite{Band1D,Band2D}. Consider the following matrix Brownian motion: 
\begin{align}\label{MBM}
\dd (H_{t})_{xy}=\sqrt{S_{xy}}\dd (B_{t})_{xy}, \quad H_{0}=0, \quad \forall x,y\in \ZL.
\end{align}
Here, $(B_{t})_{xy}$ are independent complex Brownian motions up to the Hermitian symmetry \smash{$(B_{t})_{xy}=\overline {(B_{t})_{yx}}$}, i.e., \smash{$t^{-1/2}B_t$} is an $N\times N$ GUE whose entries have zero mean and unit variance; $S=(S_{xy})$ is the variance matrix defined in \Cref{def:projlift}. Correspondingly, we define the deterministic flow as follows.  

\begin{definition}[Deterministic flow]\label{def_flow}
For any $\sE \in \mathbb R$ and $t\in [0,1]$, denote $m{(\sE)}\equiv m_{sc}(\sE+\ii 0_+)$. Then, we define the flow $z_t$ by 
\be\label{eq:zt}
z_t(\sE) = \sE + (1-t) m(\sE),\quad t\in [0, 1]. 
\ee
Throughout the proof, we will refer to $\sE$ as the {\bf flow parameter}. 
Next, for any $t\in [0,1]$ and $z\in \C_+$, define $m_t(z)$ as the unique solution to 
 \be\label{self_mt}
 m_t(z)=- \left(z + tm_t(z)\right)^{-1}
\ee
such that $\im m_t(z)>0$ for $z\in \C_+$. 
In other words, $m_t(z)$ is the Stieltjes transform of the semicircle law (with an extra scaling $\sqrt{t}$):
\[m_t(z)=\int\frac{\rho_t(x)\dd x}{x-z}=\frac{-z+\sqrt{z^2-4t}}{2t},\quad \text{with}\quad \rho_t(x)=\frac{\sqrt{4t-x^2}}{2\pi t}\mathbf 1_{x\in [-2\sqrt{t},2\sqrt{t}]} \, .\]
Note that under the above flow \eqref{eq:zt}, we have 
\be\label{eq:mt_stay} m_t(z_t(\sE))\equiv m(\sE),\quad \forall t\in [0,1].
\ee 
We will denote $z_t(\sE)=E_t(\sE) + \ii \eta_t(\sE)$ with 
\begin{align}\label{eta}
E_t(\sE)=\sE+(1-t)\re m(\sE),\quad \eta_t(\sE) = (1-t)  \im m(\sE).
\end{align}
\iffalse
Furthermore, we denote 
\be\label{eq:defxit} \xi_t(\sE):=\big(2\sqrt{t}-E_t(\sE)\big)\wedge \big(2\sqrt{t}+E_t(\sE)\big),\quad \kappa_t(\sE)\equiv \kappa_{E_t(\sE)}:=|\xi_t(\sE)|.\ee
Note that $\kappa_t$ denotes the distance of $E_t(\sE)$ to the spectral edges of the limiting spectrum $[-2\sqrt{t},2\sqrt{t}]$ of $H_t$, while a positive (resp.~negative)  $\xi_t(\sE)$ indicates that $E_t(\sE)$ lies inside (resp.~outside) $[-2\sqrt{t},2\sqrt{t}]$. \fi
\end{definition}

Next, we define the stochastic flow for the Green's function with $H_t$ given by \eqref{MBM} and $z_t$ given by \eqref{eq:zt}.

\begin{definition}[Stochastic flow]\label{Def:stoch_flow}
Consider the matrix dynamics $H_t$ evolving according to \eqref{MBM}. Then, we denote Green's function of $H_{t}$ as
\be\label{self_Gt} G_t(z):=(H_t-z)^{-1},\quad z\in \C_+,\ee
and define the resolvent flow as 
\begin{align}
G_{t,\sE}\equiv G_t(z_t(\sE)):=(H_{t}-z_t{(\sE)})^{-1}.
\end{align}
By It{\^o}'s formula, $G_{t,\sE}$ satisfies the following SDE
\begin{align}\label{eq:SDE_Gt}
\dd G_{t,\sE}=-G_{t,\sE}(\dd H_{t})G_{t,\sE}+G_{t,\sE}\{\mathcal{S}[G_{t,\sE}]-m_t(z_t(\sE))\}G_{t,\sE}
\dd t,
\end{align}
where $\mathcal{S}:M_{N}(\C)\to M_{N}(\C)$ is a linear operator defined as
\begin{align}\label{eq:opS}
\mathcal{S}[X]_{xy}:=\delta_{xy}\sum_{y=1}^{N}S_{xy}X_{yy},\quad \text{for}\quad X\in M_{N}(\C).
\end{align}
\end{definition}

For any spectral parameter $z\in \mathbf D_{C_0,\fd}$ (recall \eqref{eq:defDC0}), we are interested in the original resolvent $G(z)=(H-z)^{-1}$. This can be achieved through the stochastic flow by carefully choosing the parameter $\sE$.

\begin{lemma}\label{zztE}
Fix any $z=E+\ii \eta\in \mathbf D_{C_0,\fd}$. We choose 
\be\label{eq:t0E0}t_0\equiv t_0(z)=|m(z)|^2=\frac{\im m(z)}{\im m(z)+ \eta},\quad \sE\equiv \sE(z)=-2\frac{\re m(z)}{|m(z)|}\, .\ee
Then, we have  
 \begin{equation}\label{eq:zztE}
\sqrt{t_0}m(\sE)=m(z), \quad     z_{t_0}(\sE) =\sqrt{t_0} z \, , 
 \end{equation} 
and the following equality in distribution (denoted by ``$\stackrel{d}{=}$"):
   \begin{equation}\label{GtEGz}
    G(z) \stackrel{d}{=} \sqrt{t_0} G_{t_0,\sE} .
   \end{equation}
Under the choices of parameters in \eqref{eq:t0E0}, abbreviating $\kappa\equiv \kappa\p{\sE}$, we have the following estimates:
\begin{align}\label{eq:E0bulk}
&\kappa \asymp (\im m(z))^2 \asymp\begin{dcases}
    \kappa_E + \eta, & \text{ if } \ E\in [-2,2]\\
    {\eta^2}/(\kappa_E+\eta) , & \text{ if } \ E\notin [-2,2]
\end{dcases},\quad 1-t_0\asymp \frac{\eta}{\sqrt{\kappa}}\, . \end{align}
Furthermore, for the parameters in \eqref{eta}, we have that  
\begin{align}
\label{eq:kappat}
%&{\cor \xi_t(\sE) =\frac{1}{2}(1+t)\kappa -(1-\sqrt{t})^2} ,
E_t(\sE)=\frac{1}{2}(1+t)\sE, 
\quad \eta_t(\sE) = (1-t)\im m(\sE) \asymp  (1-t)\sqrt{\kappa}, \quad \forall t\in[0,t_0].
\end{align}
\end{lemma}
\begin{proof}
The identities in  \eqref{eq:zztE} have been proved in \cite[Lemma 2.8]{Band1D}. We repeat it for the convenience of readers. With \eqref{eq:t0E0} and the definition in \eqref{self_mWO}, we get
$$ m(\sE)=\frac{-\sE+\sqrt{\sE^2-4}}{2}=\frac{m(z)}{|m(z)|}=\frac{m(z)}{\sqrt{t_0}} , $$
which gives the first identify in \eqref{eq:zztE}. 
Next, with the first identity in \eqref{L2M}, we get
$$ z=-\sqrt{t_0}m(\sE) -\frac{1}{\sqrt{t_0}m(\sE)}=\frac{1}{\sqrt{t_0}}\left(-t_0 m(\sE) + \sE +m(\sE)\right)=\frac{z_{t_0}(\sE)}{\sqrt{t_0}},$$
which concludes the second identify in \eqref{eq:zztE}. Using \eqref{eq:zztE} and the fact that \smash{$H_{t_0}\stackrel{d}{=}\sqrt{t_0}H$}, we get 
\begin{align*}
    G_{t_0,\sE} \stackrel{d}{=}\left(\sqrt{t_0}H  - z_{t_0}(\sE)\right)^{-1}=t_0^{-1/2} G(z),
\end{align*}
which concludes \eqref{GtEGz}. 
For $\sE$ in \eqref{eq:t0E0}, we can write $\kappa$ as  
$$\kappa= \bigg(2-2\frac{\re m(z)}{\sqrt{|\re m(z)|^2+|\im m(z)|^2}}\bigg) \wedge \bigg(2+2\frac{\re m(z)}{\sqrt{|\re m(z)|^2+|\im m(z)|^2}}\bigg). $$
Combining it with the estimate on $\im m(z)$ in \eqref{square_root_density}, we obtain the first estimate in \eqref{eq:E0bulk}. The second estimate in \eqref{eq:E0bulk} then follows directly from \eqref{eq:t0E0}. 
The first equation in \eqref{eq:kappat} follows from the fact that $\re m(\sE)=-\sE/2$. 
%following expression and that $\kappa=(2-\sE)\wedge(\sE+2)$: \be\label{eq:EtE} E_t(\sE)=\sE + (1-t)\re m(\sE) =\sE - \frac{1}{2}(1-t)\sE = \frac{1}{2}(1+t)\sE. \ee
Finally, the second estimate in \eqref{eq:kappat} follows from the estimate on $\im m(\sE)$ in \eqref{square_root_density} since $|\sE|\le 2$. 
\end{proof}

In the proof, we will fix a target spectral parameter $z=E+\ii \eta\in \mathbf D_{C_0,\fd}$ for an arbitrarily small constant $\fd>0$. 
Accordingly, we choose $t_0$ and $\sE$ as specified in \eqref{eq:t0E0}. For clarity, unless we want to emphasize their dependence on $\sE$, we will often omit this variable from various notations---particularly from the notations $z_t(\sE)$, $E_t(\sE)$, $\eta_t(\sE)$, $m(\sE)$, $M(\sE)$, and $G_{t,\sE}$.

\subsection{\texorpdfstring{$G$}{G}-loops and primitive loops}\label{subsec:Gloop}

Our focus will be on the dynamics of $G_{t}\equiv G_{t,\sE}$ and the corresponding $G$-loops defined in \Cref{Def:G_loop}. % for $t\in [0,t_0]$ with $1-t_0\gtrsim \eta/\sqrt{\kappa} \ge N^\fd\eta_*/\sqrt{\kappa}$. 

\begin{definition}[$G$-loop]\label{Def:G_loop}
For $\sigma\in \{+,-\}$, we denote 
 \begin{equation}\nonumber
  G_{t}(\sigma):=\begin{cases}
       (H_t-z_t)^{-1}, \ \ \text{if} \ \  \sigma=+,\\
        (H_t-\bar z_t)^{-1}, \ \ \text{if} \ \ \sigma=-.
   \end{cases}
 \end{equation}
In other words, we let $G_{t}(+)\equiv G_{t}$ and $G_{t}(-)\equiv G_{t}^*$. Denote by $I_{[a]}$ and $E_{[a]}$, \smash{$[a]\in \Zn$}, the block identity and rescaled block identity matrices, respectively:
\be\label{def:Ia} (I_{[a]})_{ij}= \delta_{ij}\cdot \mathbf 1_{i\in [a]} , \quad E_{[a]}=W^{-d}I_{[a]}.\ee
For any $\fn\in \N$, for $\bsig=(\sigma_1, \cdots \sigma_\fn)\in \{+,-\}^\fn$ and $\ba=([a_1], \ldots, [a_\fn])\in (\Zn)^\fn$, we define the $\fn$-$G$ loop by 
\begin{equation}\label{Eq:defGLoop}
    {\cal L}^{(\fn)}_{t, \boldsymbol{\sigma}, \ba}=\left \langle \prod_{i=1}^\fn \left(G_{t}(\sigma_i) E_{[a_i]}\right)\right\rangle .
\end{equation}
Furthermore, we denote %(recall \eqref{eq:mt_stay})
\begin{equation}\label{def_mtzk}
m (\sigma ):= \begin{cases}
    m(\sE),  &\text{if} \ \ \sigma  =+ \\
    \bar m(\sE),  &\text{if} \ \ \sigma = -
\end{cases},\quad M (\sigma ):= \begin{cases}
    M(\sE) ,  &\text{if} \ \ \sigma  =+ \\
    M(\sE)^*,  &\text{if} \ \ \sigma = -
\end{cases}.
\end{equation}
% \begin{equation}\label{def_mtzk}
% m (\sigma ):= \begin{cases}
%     m(\sE) \equiv m_t(z_t(\sE)),  &\text{if} \ \ \sigma  =+ \\
%     \bar m(\sE) \equiv m_t(\bar z_t(\sE)),  &\text{if} \ \ \sigma = -
% \end{cases},\quad M (\sigma ):= \begin{cases}
%     M(\sE) \equiv M_t(z_t(\sE)),  &\text{if} \ \ \sigma  =+ \\
%     M(\sE)^* \equiv M_t(\bar z_t(\sE)),  &\text{if} \ \ \sigma = -
% \end{cases}.
% \end{equation}
\end{definition}

To represent the loop hierarchy for the $G$-loops, we introduce the following operations as in \cite{Band1D}.

\begin{definition}[Loop operations]\label{Def:oper_loop}
For the $G$-loop in \eqref{Eq:defGLoop}, we define the following operations on it. 
% Recall the $G$ loops ${\cal L}_{t, \boldsymbol{\sigma}, \ba}$ defined in Eq.~\eqref{Eq:defGLoop} of Definition~\ref{Def:G_loop}. Here we define some basic operators: cutting and gluing for these $G$ loops.  Assume that
% \begin{equation}\label{taahCal}
%    {\cal L}_{t, \boldsymbol{\sigma}, \ba} = \left\langle \prod_{i=1}^n G_t(\sigma_i) E_{a_i} \right\rangle, \quad \sigma_i \in \{+, -\}, \quad 1 \le i \le n 
% \end{equation}

 \medskip
 
\noindent \emph{(1)} 
 For $k \in \qqq{\fn}$ and $[a]\in \Zn$, we define the first type of cut-and-glue operator ${\cut}^{[a]}_{k}$ as follows:
\be\label{eq:cut1}
    {\cut}^{[a]}_{k} \circ {\cal L}^{(\fn)}_{t, \boldsymbol{\sigma}, \ba}:= \left \langle \prod_{i<k}  \left(G_{t}(\sigma_i) E_{[a_i]}\right)\left( G_{t}(\sigma_k) E_{[a]} G_{t}(\sigma_k)E_{[a_k]}\right)\prod_{i>k}  \left(G_{t}(\sigma_i) E_{[a_i]}\right)\right\rangle.
\ee
In other words, it is the $(\fn+1)$-$G$ loop obtained by replacing $G_t(\sigma_k)$ as $G_{t}(\sigma_k) E_a G_{t}(\sigma_k)$. Graphically, the operator \smash{${\cut}^{[a]}_{k}$} cuts the $k$-th $G$ edge $G_t(\sigma_k)$ and glues the two new ends with $E_{[a]}$. This operator can also be considered as an operator on $(\boldsymbol{\sigma},\ba)$: 
\begin{align*}{\cut}^{[a]}_{k} (\boldsymbol{\sigma}, \ba) =  \big( & (\sigma_1,\ldots, \sigma_{k-1}, \sigma_k,\sigma_k ,\sigma_{k+1},\ldots, \sigma_\fn ),\\
& ([a_1],\ldots, [a_{k-1}], [a],[a_k],[a_{k+1}],\ldots, [a_\fn] )\big).\end{align*}
Hence, we will also express \eqref{eq:cut1} as 
    $${\cut}^{[a]}_{k} \circ {\cal L}^{(\fn)}_{t, \boldsymbol{\sigma}, \ba}\equiv 
     {\cal L}^{(\fn+1)}_{t, \;  {\cut}^{[a]}_{k} (\boldsymbol{\sigma}, \ba)}.
    $$

%\medskip

 \noindent 
 \emph{(2)} For $k < l \in \qqq{\fn}$, we define the second type of cut-and-glue operator ${\cutL}^{[a]}_{k,l}$ from the left (``L") of $k$ as: 
 \begin{align}\label{eq:cutL}
 {\cutL}^{[a]}_{k,l} \circ {\cal L}^{(\fn)}_{t, \boldsymbol{\sigma}, \ba}:= \left \langle \prod_{i<k}  \left[G_{t}(\sigma_i) E_{[a_i]}\right]\left( G_{t}(\sigma_k) E_{[a]} G_{t}(\sigma_l)E_{[a_l]}\right)\prod_{i>l}  \left[G_{t}(\sigma_i) E_{[a_i]}\right]\right\rangle, 
 \end{align}
and the third type of cut-and-glue operator ${\cutR}^{[a]}_{k,l}$ from the right (``R") of $k$ as: 
 \begin{align}\label{eq:cutR}
 {\cutL}^{[a]}_{k,l} \circ {\cal L}^{(\fn)}_{t, \boldsymbol{\sigma}, \ba}:= \left \langle \prod_{k\le i <l}  \left[G_{t}(\sigma_i) E_{[a_i]}\right]\cdot \left( G_{t}(\sigma_l) E_{[a]}\right)\right\rangle. 
 \end{align}
In other words, the second type operator cuts the $k$-th and $l$-th $G$ edges $G_t(\sigma_k)$ and $G_t(\sigma_l)$, and creates two chains: the left chain to the vertex $[a_k]$ is of length $(\fn+k-l+1)$ and contains the vertex $[a_\fn]$, while the right chain to the vertex $[a_k]$ is of length $(l-k+1)$ and does not contain the vertex $[a_\fn]$.  
Then, \eqref{eq:cutL} (resp.~\eqref{eq:cutR}) gives a $(\fn+k-l+1)$-loop (resp.~$(l-k+1)$-loop) obtained by gluing the left chain (resp.~right chain) at the new vertex $[a]$.
Again, we can also consider the two operators to be defined on the indices $\boldsymbol{\sigma},\ba$: 
\begin{align*}
    &{\cutL}^{[a]}_{k,l} (\boldsymbol{\sigma}, \ba) = \left((\sigma_1,\ldots, \sigma_k,\sigma_l ,\ldots, \sigma_\fn ),([a_1],\ldots, [a_{k-1}], [a],[a_l],\ldots, [a_\fn] )\right),\\
    &{\cutR}^{[a]}_{k,l} (\boldsymbol{\sigma}, \ba) = \left((\sigma_k,\ldots, \sigma_l),([a_k],\ldots, [a_{l-1}], [a])\right).
\end{align*}
Hence, we will also express \eqref{eq:cutL} and \eqref{eq:cutR} as 
    $${\cutL}^{[a]}_{k,l} \circ {\cal L}^{(\fn)}_{t, \boldsymbol{\sigma}, \ba}\equiv 
     {\cal L}^{(\fn+k-l+1)}_{t, \;  {\cutL}^{[a]}_{k,l} (\boldsymbol{\sigma}, \ba)},\quad {\cutR}^{[a]}_{k} \circ {\cal L}^{(\fn)}_{t, \boldsymbol{\sigma}, \ba}\equiv 
     {\cal L}^{(l-k+1)}_{t, \;  {\cutR}^{[a]}_{k,l} (\boldsymbol{\sigma}, \ba)}.
    $$
\end{definition}

Given a matrix $\cal A$ defined on $ \ZL$, we define its ``projection" $\cal A^{L\to n}$ to $\Zn$ as
\be\nonumber %\label{eq:defLton}
\cal A^{\LK}_{[x][y]}:=  W^{-d}\, \sum_{x'\in [x]}\sum_{y'\in [y]}\cal A_{x'y'}. 
\ee
Under this definition, the projection of the variance matrix in \eqref{eq:SWO} is given by
\be\label{eq:SLton}S^{\LK}=(1+2d\lambda^2)^{-1}\left( I_n + \lambda^2\Lambda_n\right).\ee 
For $(x,y)\in (\ZL)^2$, we denote $\partial_{xy}:=\partial_{(H_t)_{xy}}$. Then, by It\^o's formula and equation \eqref{eq:SDE_Gt}, it is easy to derive the following SDE satisfied by the $G$-loops.  

\begin{lemma}[The loop hierarchy] \label{lem:SE_basic}
An $\fn$-$G$ loop satisfies the following SDE, called the ``loop hierarchy": 
\begin{align}\label{eq:mainStoflow}
    \dd \mathcal{L}^{(\fn)}_{t, \boldsymbol{\sigma}, \ba} &=\dd  
    \mathcal{B}^{(\fn)}_{t, \boldsymbol{\sigma}, \ba} 
    +\mathcal{W}^{(\fn)}_{t, \boldsymbol{\sigma}, \ba} \dd t 
    \nonumber\\
    &+  W^d\sum_{1 \le k < l \le \fn} \sum_{[a], [b]} \left( \cutL^{[a]}_{k, l} \circ \mathcal{L}^{(\fn)}_{t, \boldsymbol{\sigma}, \ba} \right) S^{\LK}_{[a][b]} \left( \cutR^{[b]}_{k, l} \circ \mathcal{L}^{(\fn)}_{t, \boldsymbol{\sigma}, \ba} \right) \dd t, 
\end{align}
 where the martingale term $\mathcal{B}^{(\fn)}_{t, \boldsymbol{\sigma}, \ba}$ and the ``weight" term $\mathcal{W}^{(\fn)}_{t, \boldsymbol{\sigma}, \ba}$ are defined by 
 \begin{align} \label{def_Edif}
\dd\mathcal{B}^{(\fn)}_{t, \boldsymbol{\sigma}, \ba} :  = & \sum_{x,y\in \ZL} 
  \left( \partial_{xy}  {\cal L}^{(\fn)}_{t, \boldsymbol{\sigma}, \ba}  \right)
 \cdot \sqrt{S _{xy}}
  \left(\dd B_t\right)_{xy}, \\\label{def_EwtG}
\mathcal{W}^{(\fn)}_{t, \boldsymbol{\sigma}, \ba}: = &  {W}^d \sum_{k=1}^\fn \sum_{[a], [b]\in \Zn} \; 
 \left \langle (G_t(\sigma_k)-M(\sig_k)) E_{[a]} \right\rangle
   S^{\LK}_{[a][b]} 
  \left( {\cut}^{[b]}_{k} \circ {\cal L}^{(\fn)}_{t, \boldsymbol{\sigma}, \ba} \right) . 
\end{align}
%Here, following \eqref{Eq:defwtG}, we use $\Gc_t(\sigma_k)$ to represent $G_t(\sig)-M(\sig)$.
%Note that for the block Anderson model, $S^{\LK}$ is an identity matrix, i.e., $S^{\LK}_{[a][b]}=\delta_{[a][b]}.$
\end{lemma}
  
We will see that this loop hierarchy is well-approximated by the primitive loops, defined as follows. 

\begin{definition}[Primitive loops]\label{Def_Ktza}
We define the primitive loop of length 1 as: 
\be\label{eq:1loop}{\cal K}^{(1)}_{t,\sigma,[a]}=m(\sigma),\quad \forall t\in [0,1],\ \sigma\in \{+,-\}.\ee 
For $\fn\ge 2$, we define the function ${\cal K}^{(\fn)}_{t, \boldsymbol{\sigma}, \ba}$ (of $t\in[0,1]$, $\bsig\in \{+,-\}^k$, and $\ba\in (\Zn)^\fn$) to be the unique solution to the following system of differential equation: 
 \begin{align}\label{pro_dyncalK}
       \partial_t{\cal K}^{(\fn)}_{t, \boldsymbol{\sigma}, \ba} 
       =  
       W^d \sum_{1\le k < l \le \fn} \sum_{[a], [b]} \left( \cutL^{[a]}_{k, l} \circ \mathcal{K}^{(\fn)}_{t, \boldsymbol{\sigma}, \ba} \right) S^{\LK}_{[a][b]} \left( \cutR^{[b]}_{k, l} \circ \mathcal{K}^{(\fn)}_{t, \boldsymbol{\sigma}, \ba} \right) ,
    \end{align}
with the following initial condition at $t=0$: 
\be\label{eq:initial_K}
{\cal K}^{(k)}_{0, \boldsymbol{\sigma}, \ba} =  {\cal M}^{(k)}_{\boldsymbol{\sigma}, \ba} ,\quad \forall k\in \N, \ \bsig\in \{+,-\}^k,\ \ba\in (\Zn)^k,\ee
where ${\cal M}^{(k)}_{\boldsymbol{\sigma}, \ba}$ is defined as
\be\label{eq:KMloop} {\cal M}^{(k)}_{\boldsymbol{\sigma}, \ba}:=\left \langle \prod_{i=1}^k \left(M(\sigma_i) E_{[a_i]}\right)\right\rangle = W^{-(k-1)d}\prod_{i=1}^k m(\sig_i) \mathbf 1([a_1]=\cdots=[a_k]).\ee
In equation \eqref{pro_dyncalK}, the operators $\cutL$ and $\cutR$ act on ${\cal K}^{(\fn)}_{t, \boldsymbol{\sigma}, \ba}$ through the actions on indices, that is, 
  \be\label{calGonIND}
     {\cutL}^{[a]}_{k,l}  \circ {\cal K}^{(\fn)}_{t, \boldsymbol{\sigma}, \ba} := {\cal K}^{(\fn+k-l+1)}_{t,  {\cutL}^{[a]}_{k,l}  (\boldsymbol{\sigma}, \ba)} , \ \  {\cutR}^{[b]}_{k,l}  \circ {\cal K}^{(\fn)}_{t, \boldsymbol{\sigma}, \ba} := {\cal K}^{(l-k+1)}_{t,  {\cutR}^{[b]}_{k,l}  (\boldsymbol{\sigma}, \ba)}.  
       \ee
    We will call ${\cal K}^{(\fn)}_{t, \boldsymbol{\sigma}, \ba}$ a primitive loop of length $\fn$ or an $\fn$-$\cK$ loop.  
\end{definition}

Note equation \eqref{eq:mainStoflow} involves $G$-loops of length larger than $\fn$, and hence represents a ``hierarchy" rather than a ``self-consistent equation" for the $G$-loops. Conversely, the primitive equation \eqref{pro_dyncalK} indeed gives a self-consistent equation that can be solved inductively, as demonstrated in \cite{Band1D}. We will present an explicit representation of the primitive loops in \Cref{sec:prim}, which will serve as the deterministic limits of the $G$-loops. 
For clarity of presentation, we will also call $G$-loops and primitive loops as $\cL$-loops and $\cK$-loops, respectively. We will also call \smash{$(\cL-{\cal K})^{(\fn)}_{t, \boldsymbol{\sigma}, \ba}\equiv \cL^{(\fn)}_{t, \boldsymbol{\sigma}, \ba}-\cK^{(\fn)}_{t, \boldsymbol{\sigma}, \ba}$} an $(\cL-\cK)$-loop. 
Finally, by replacing the spectral parameter $z_t(\sE)$ with $z$ in the definitions of the $\cL$ and $\cK$-loops (i.e., we replace $G_{t,\sE}$ and $m(\sE)$ with $G_t(z)$ and $m_t(z)$, respectively), we can define the more general notations \smash{$\cL^{(\fn)}_{t, \boldsymbol{\sigma}, \ba}(z)$} and \smash{$\cK^{(\fn)}_{t, \boldsymbol{\sigma}, \ba}(z)$}.  

Applying Ward's identity in \Cref{lem-Ward} to $G$, we can show that the $G$-loops satisfy the following identity \eqref{WI_calL}, which we will also refer to as a ``Ward's identity". In \cite{Band1D}, it shows that a similar Ward's identity \eqref{WI_calK} holds for the $\cal K$-loops. 

\begin{lemma}[Ward's identity for $\cL$-loops and $\cK$-loops]\label{lem_WI_K} 
For an $\fn$-$G$ loop ${\cal L}^{(\fn)}_{t, \boldsymbol{\sigma}, \ba}$ with $\fn\ge 2$ and $\sigma_1=-\sig_{\fn}$, 
we have the following identities, which are called Ward's identities at the vertex $[a_\fn]$: 
\begin{align}\label{WI_calL}
&\sum_{[a_\fn]}{\cal L}^{(\fn)}_{t, \boldsymbol{\sigma}, \ba}(z)=
\frac{1}{2\ii W^d\eta}\left( {\cal L}^{(\fn-1)}_{t,  \wh\bsig^{(+,\fn)}, \wh\ba^{(\fn)}}(z)- {\cal L}^{(\fn-1)}_{t,  \wh\bsig^{(-,\fn)} , \wh\ba^{(\fn)}}(z)\right),  \\
\label{WI_calK}
& \sum_{[a_\fn]}{\cal K}^{(\fn)}_{t, \boldsymbol{\sigma}, \ba}(z)=
\frac{1}{2\ii W^d\eta}\left( {\cal K}^{(\fn-1)}_{t,  \wh\bsig^{(+,\fn)}, \wh\ba^{(\fn)}}(z)- {\cal K}^{(\fn-1)}_{t,  \wh\bsig^{(-,\fn)}, \wh\ba^{(\fn)}}(z)\right) , 
\end{align}
where $\eta=\im z$, $\wh\bsig^{(\pm,\fn)}$ is obtained by removing $\sigma_\fn$ from $\boldsymbol{\sigma}$ and replacing $\sigma_1$ with $\pm$, i.e., $\wh\bsig^{(\pm,\fn)}:=(\pm, \sigma_2, \cdots \sigma_{\fn-1})$, and  $\wh\ba^{(\fn)}$ is obtained by removing $[a_\fn]$ from $\ba$, i.e., 
$\wh\ba^{(\fn)}:=([a_1], [a_2],\cdots, [a_{\fn-1}]).$ 
\end{lemma}
\begin{proof}
\eqref{WI_calL} follows from \eqref{eq_Ward0}, while \eqref{WI_calK} is proved in \cite[Lemma 3.6]{Band1D}. 
\end{proof}

We will also need an additional property of \emph{pure primitive loops}, which have only one type of charge. Its proof will be given in \Cref{sec:pf_new_identity}.

\begin{lemma}\label{lem:pure_sum}
Let $\bm{\sigma}$ be a pure loop such that $\sigma_1 = \sigma_2 = \cdots = \sigma_\fn=+$. Then, we have the following identity for the derivatives of $m_t(z)$ (recall the definition \eqref{self_mt}):
\be\label{eq:pure_sum}  
\frac{1}{(\fn-1)!}\frac{\dd^{\fn-1}}{\dd z^{\fn-1}}m_t(z)= \int\frac{\rho_t(x)\dd x}{(x-z)^\fn}= W^{(\fn-1)d}\sum_{[a_2],\ldots,[a_\fn]} \cK^{(\fn)}_{t,\bsig,\ba}(z).\ee
\end{lemma}

For simplicity of presentation, we also define the following concepts of generalized $G$-loops and primitive loops by introducing a new charge ``$\im$". 

\begin{definition}[Generalized $G$-loops and primitive loops]\label{Def:G_loopgen}
We introduce a new charge $\im$, where
 \begin{equation}\nonumber
  G_{t}(\im):= \im G_t = \frac{1}{2\ii}\p{G_t(+)-G_t(-)}.
 \end{equation}
For $\bchi=(\chi_1, \cdots \chi_\fn)\in \{+,-,\im\}^\fn$ and $\ba=([a_1], \ldots, [a_\fn])\in (\Zn)^\fn$, we define the \emph{generalized $\fn$-$G$ loop} by 
\begin{equation}\label{Eq:defGLoop_gen}
    {\cal L}^{(\fn)}_{t, \bchi, \ba}=\left \langle \prod_{i=1}^\fn \left(G_{t}(\chi_i) E_{[a_i]}\right)\right\rangle ,
\end{equation}
where, with a slight abuse of notation, we continue to use $\mathcal{L}$ to denote the $G$-loops. 
Let $v_\bchi$ denote the subset of indices that correspond to the charge $\im$: 
\be\label{eq:vbchi}v_\bchi=\{i_1,\ldots, i_{|v_\bchi|}\}:=\{i\in\qqq{\fn}: \chi_i=\im\}.\ee  
Then, we can expand ${\cal L}^{(\fn)}_{t, \bchi, \ba}$ as
\begin{align*}
    {\cal L}^{(\fn)}_{t, \bchi, \ba}=\frac{1}{(2\ii)^{|v_\bchi|}}\sum_{\bsig\in\{+,-\}^\fn}\mathbf 1(\sig_{i}=\chi_i, \forall i\notin v_\bchi) \cdot (-1)^{k_{-}(\bsig)} {\cal L}^{(\fn)}_{t, \bsig, \ba},
\end{align*}
where $k_{-}(\bsig):=\#\{i\in v_\bchi: \sig_i=-1\}$. Correspondingly, we define the \emph{generalized primitive loop} by 
\begin{align}\label{Eq:defKLoop_gen}
    {\cal K}^{(\fn)}_{t, \bchi, \ba}:=\frac{1}{(2\ii)^{|v_\bchi|}}\sum_{\bsig\in\{+,-\}^\fn}\mathbf 1(\sig_{i}=\chi_i, \forall i\notin v_\bchi) \cdot (-1)^{k_{-}(\bsig)} {\cal K}^{(\fn)}_{t, \bsig, \ba}.
\end{align}
In the proof, we will use generalized $G$-loops and primitive loops with only one charge equal to $\Im$.
\end{definition}

\subsection{Propagators} \label{sec:propagator}

The primitive loops will be expressed with the $\Theta$-propagators defined as follows. 
 
\begin{definition}[$\Theta$-propagator]\label{def_Theta}
Given $t\in[0,1]$ and $\sigma_1,\sigma_2\in \{+,-\}$, the $\Theta$-propagator $\Theta_t^{(\sigma_1,\sigma_2)}$ is an $n^d\times n^d$ matrix defined as (recall \smash{$S^{\LK}$} given by \eqref{eq:SLton}):
\begin{equation}\label{def_Thxi}
\Theta_{t}^{(\sigma_1,\sigma_2)} := \big[1 - t m(\sigma_1)m(\sigma_2)S^{\LK}\big]^{-1} \, .
\end{equation}
%where $S^{\LK}$ is given by \eqref{eq:SLton}. 
We will denote its entries by $\Theta_{t}^{(\sigma_1,\sigma_2)}([a],[b])$ or $\Theta_{t,[a][b]}^{(\sigma_1,\sigma_2)}$. 
\end{definition}

We now state some fundamental properties of  \smash{$\Theta_{t}^{(\sigma_1,\sigma_2)}$} in \Cref{lem_propTH}, which has been proved in \cite{Band1D,Band2D,RBSO1D}. 

\begin{lemma}\label{lem_propTH}
In the flow framework given by \Cref{zztE}, define
\be\label{eq:ellt}
 \ell_t:= \min\big(|1-t|^{-\frac 1 2}, n\big),\ \ \hell_{t}:=\min\big(\omega_t^{-\frac 1 2}, n\big),\ \ \text{where}\ \ \omega_t:=|1-t|+\sqrt{\kappa}.
% \left((|1-t|+\sqrt{\kappa})^{-1/2},\; n\right). 
\ee
Then, the $\Theta$-propagators satisfy the following properties for $t\in [0,1)$ and $\sigma_1,\sigma_2\in \{+,-\}$:
\begin{enumerate}
    \item {\bf Transposition}: We have $\Theta_{t}^{(\sigma_1,\sigma_2)}=\Theta_{t}^{(\sigma_2,\sigma_1)}= (\Theta_{t}^{(\sigma_1,\sigma_2)})^\top .$
    
\item {\bf Symmetry}: For any $[x],[y],[a]\in \Zn$, we have
\be\label{symmetry}
\begin{aligned}
    & \Theta^{(\sigma_1,\sigma_2)}_{t}([x]+[a],[y]+[a])= \Theta^{(\sigma_1,\sigma_2)}_{t}([x],[y]),\\
   &  \Theta^{(\sigma_1,\sigma_2)}_{t}(0,[x])= \Theta^{(\sigma_1,\sigma_2)}_{t}(0,-[x]).
\end{aligned}   \ee

\item {\bf Exponential decay on length scale $\ell_t$}: For any large constant $D>0$, there exists a constant $c>0$ such that the following estimate holds for all $\sigma_1,\sigma_2\in\{+,-\}$: 
\begin{equation}\label{prop:ThfadC}     
\Theta^{(\sigma_1,\sigma_2)}_{t}(0,[x])\prec \frac{e^{-c |[x]|/ {\ell}_t}}{|1-t|  {\ell}_t^d}+ W^{-D}.  
\end{equation}
When $\sigma_1=\sigma_2$, we have that for any constants $\tau,D>0$, %a better estimate: 
\begin{equation}\label{prop:ThfadC_short}  \Theta^{(\sigma_1,\sigma_2)}_{t}(0,[x])\prec \frac{1}{\omega_t\hell_t^d}\mathbf 1\left(|[x]|\le W^\tau \hell_t\right)+ W^{-D}.
\end{equation}

% \begin{equation}\label{prop:ThfadC_short}  \Theta^{(\sigma_1,\sigma_2)}_{t}(0,[x])\prec \frac{e^{-c |[x]|/ \hell_t}}{\omega_t \hell_t^d}+ W^{-D}.
% \end{equation}
% There is some problem with the proof of this estimate in \cite{Band2D}. But this estimate is only used in the proof of the $\cK$-loop bounds, so our proof only requires a weaker bound, which can be proved easily using the method of Fourier transform:
% \begin{equation}\label{prop:ThfadC_short}  \Theta^{(\sigma_1,\sigma_2)}_{t}(0,[x])\prec \frac{W^\tau}{\omega_t\hell_t^d}\mathbf 1(|[x]|\le W^\tau \hell_t)+ W^{-D},
% \end{equation}
% for any constants $\tau,D>0$. 

\item {\bf First-order finite difference}: The following estimate holds for all $[x],[y]\in \Zn$ and $\sigma_1, \sigma_2\in\{+,-\}$:  \begin{equation}\label{prop:BD1}     
    \left| \Theta^{(\sigma_1,\sigma_2)}_{t}(0, [x])-\Theta^{(\sigma_1,\sigma_2)}_{t}(0, [y])\right|\prec \frac{|[x]-[y]|}{\langle [x]\rangle^{d-1}+\qq{[y]}^{d-1}}.
     \end{equation}

\item {\bf Second-order finite difference}: The following estimate holds for all $[x],[y]\in \Zn$ and $\sigma_1, \sigma_2\in\{+,-\}$: 
\begin{equation}\label{prop:BD2}
\qquad \Theta^{(\sigma_1,\sigma_2)}_{t} (0,[x]+[y]) + \Theta^{(\sigma_1,\sigma_2)}_{t} (0,[x]-[y])-  2\Theta^{(\sigma_1,\sigma_2)}_{t} (0,[x]) 
\prec \frac{|[y]|^2}{\langle [x]\rangle^{d}} .
 \end{equation}
 
\end{enumerate}
\end{lemma}
\begin{proof}
The properties (1) and (2) follow directly from the definition of $\Theta_{t}^{(\sigma_1,\sigma_2)}$. The estimate \eqref{prop:ThfadC} has been established in \cite[Lemma 2.14]{Band1D} and \cite[Lemma 2.14]{Band2D}, and the estimates \eqref{prop:BD1} and \eqref{prop:BD2} have been proved in \cite[Lemma 3.10]{RBSO1D}. Finally, noting that $|1-tm^2|\asymp \omega_t$, the estimate \eqref{prop:ThfadC_short} can be proved using the Fourier series representation for \smash{$\Theta^{(\sigma,\sigma)}_{t}$}, along with a summation by parts argument, as illustrated in \cite[Appendix E]{RBSO}. We omit the details of the proof. 
\end{proof}

We can solve the primitive equation \eqref{pro_dyncalK} explicitly in the cases $\fn \in\{2, 3\}$.

\begin{example} %[$\mathcal{K}_{t, \boldsymbol{\sigma}, \ba}$ for $n=2$]
As shown in \cite{Band1D}, the $2$-$\cK$ and $3$-$\cK$ loops are given by  
\begin{align}    \label{Kn2sol}\cK^{(2)}_{t,\boldsymbol{\sigma}, \ba}&=W^{-d}m(\sig_1)m(\sig_2) \Theta_{t,[a_1][a_2]}^{(\sigma_1,\sigma_2)}   ,  \\
    \label{Kn3sol} \mathcal{K}^{(3)}_{t, \boldsymbol{\sigma}, \ba} %&= \sum_{[b_1],[b_2],[b_3]}\Theta_{t}^{(\sig_1,\sig_2)}([a_1],[b_1])\Theta_{t}^{(\sig_2,\sig_3)}([a_2],[b_2])\Theta_{t}^{(\sig_3,\sig_1)}([a_3],[b_3]) \cM^{(3)}_{\bsig,\mathbf b}\\
    &= \frac{m(\sig_1)m(\sig_2)m(\sig_3)}{W^{2d}}\sum_{[b]}\Theta_{t,[a_1][b]}^{(\sig_1,\sig_2)} \Theta_{t,[a_2][b]}^{(\sig_2,\sig_3)}\Theta_{t,[a_3][b]}^{(\sig_3,\sig_1)} ,
\end{align}
where $\bsig=(\sig_1,\ldots,\sig_\fn)$ and $\ba=([a_1],\ldots,[a_\fn])$ for $\fn\in \{2,3\}$. % and we recall the $M$-loop defined in \eqref{eq:KMloop}. 

% \begin{remark}
In the setting of \Cref{thm_diffu}, given $z=E+\ii \eta\in \mathbf D_{C_0,\fd}$, we can choose the deterministic flow as described in \Cref{zztE} such that \eqref{eq:zztE} and \eqref{GtEGz} hold. Using the notations introduced earlier, we can express the quantities in \Cref{thm_diffu} as
\be\label{eq:avg_GTheta}
\begin{split}
&W^{-2d}\sum_{x\in[a],y\in[b]}|G_{xy}|^2= t_0 \cL^{(2)}_{t_0,(-,+),([a],[b])},\\ &W^{-2d}\sum_{x\in[a],y\in[b]}G_{xy}G_{yx}= t_0\cL^{(2)}_{t_0,(+,+),([a],[b])},\\
&W^{-2d}\sum_{x\in[a],y\in[b]}\Theta_{xy}=t_0\cK^{(2)}_{t_0,(-,+),([a],[b])},\\ &W^{-2d}\sum_{x\in[a],y\in[b]}S^{+}_{xy}=t_0\cK^{(2)}_{t_0,(+,+),([a],[b])}.
\end{split}\ee
Here, the third and fourth identities follow directly from the definitions \eqref{def:Theta}, \eqref{def_Thxi}, and \eqref{Kn2sol}. Hence, \Cref{thm_diffu} essentially states that the $2$-$G$ loops converge to the 2-$\cK$ loops as $N\to \infty.$
%\end{remark}
\end{example}

In \Cref{sec:prim}, we will describe the tree representation for general $\fn$-$\cK$ loops with $\fn\ge 4$. Using this representation, together with Ward's identity \eqref{WI_calK}, the estimates in \Cref{lem_propTH}, and a key \emph{sum zero property}, we establish the following upper bound \eqref{eq:bcal_k} for $\mathcal{K}$-loops of length $\fn\ge 2$. 
%cite \cite{Bandedge} here
The proof follows a strategy similar to that of \cite[Lemma 3.11]{Band1D} in the bulk, with appropriate modifications to handle the edge regime. For the reader’s convenience, we defer the proof of \Cref{ML:Kbound} to \Cref{subsec:estimates}.

%Its proof follows a similar strategy to that of \cite[Lemma 3.11]{Band1D} in the bulk regime, with some modifications in the details. For readers' convenience, we defer the proof of \Cref{ML:Kbound} to \Cref{subsec:estimates}.

\begin{lemma}[Estimates of $\cK$-loops]\label{ML:Kbound}
In the flow framework given by \Cref{zztE}, the primitive loops satisfy the following estimate for each $t\in [0,t_0]$ and fixed $\fn\ge 2$: %, $\bsig\in \{+,-\}^\fn$, $\ba\in (\Zn)^\fn$, 
\begin{equation}\label{eq:bcal_k}
\max_{\bsig,\ba}\left|{\cal K}^{(\fn)}_{t, \boldsymbol{\sigma}, \ba}\right|  \prec \frac{\omega_t}{(W^d \ell_t^d |1-t| \cdot \omega_t )^{\fn-1}} \, .
\end{equation}
For the $1$-$\cK$ loop defined in \eqref{eq:1loop}, it is easy to see that 
$$ {\cal K}^{(1)}_{t,\sigma,[a]}=m(\sigma) =\OO(1) \quad \text{for} \ \ \sig\in\{+,-\},\quad \text{and}\quad  {\cal K}^{(1)}_{t,\im,[a]}=\im m =\OO(\sqrt{\kappa}).$$
\end{lemma}

\section{Properties of primitive loops}\label{sec:prim}

This section is devoted to the proofs of \Cref{lem:pure_sum,ML:Kbound}. For this purpose, we first define a \emph{tree representation} for the primitive loops as discovered in \cite{Band1D}. 

%\subsection{Canonical partitions with $M$-loops}\label{subsec:partition}

\subsection{Tree Representation}\label{sec_tree}
 
The tree representation is built upon the concept of \emph{canonical partitions of polygons} introduced in \cite[Definition 3.1]{Band1D}. Essentially, canonical partitions of an oriented polygon $\mathcal{P}_{\ba}$ are partitions in which each edge of the polygon is in one-to-one correspondence with each region in the partition. Throughout this section, given two vertices $[a]$ and $ [b]$ in a graph, we will use $\p{[a],[b]}$ to represent an \emph{undirected edge}. In particular, we always identify $\p{[a],[b]}$ with $\p{[b],[a]}$.
%More specifically, we will build upon the (equivalence class) of trees of the canonical partitions, which is the set $\TSP$ as in Definition 3.1 of \cite{Band1D}. We refer readers to Figure 4 of \cite{Band1D} for an illustration of the canonical partition of a polygon.

\begin{definition}[Canonical partitions]\label{def:canpnical_part}
Fix $\fn\ge 3$ and let $\cal P_{\ba}$ be an oriented polygon with vertices $\ba=([a_1],[a_2], \ldots ,[a_\fn])$ arranged in counterclockwise order. We adopt the cyclic convention that $[a_{i}]=[a_{j}]$ if and only if $i=j\mod \fn$. Then, we use $([a_{k-1}],[a_k])$ to denote the $k$-th side of $\cal P_{\ba}$. A planar partition of the polygonal domain enclosed by $\cal P_{\ba}$ is called {\bf canonical} if the following properties hold:
\begin{itemize}
\item Every sub-region in the partition is also a polygonal domain. 
\item There is a one-to-one correspondence between the edges of the polygon and the sub-regions, such that each side $([a_{k-1}],[a_k])$ belongs to exactly one sub-region, and each sub-region contains exactly one side of $\cal P_{\ba}$. Denote the sub-region containing $([a_{k-1}],[a_k])$ by $R_k$. The sub-region $R_k$ is assigned a charge $\sig_k$, corresponding to the charge of the edge $([a_{k-1}],[a_k])$. 

\item Every vertex $a_k$ of $\cal P_{\ba}$ belongs to exactly two regions, $R_k$ and $R_{k+1}$ (with the convention $R_{\fn+1}=R_1$). 
  
\end{itemize}
Note that given a canonical partition, by removing the $\fn$ sides of the polygon ${\cal P}_{\ba}$, the remaining interior edges form a tree, with the leaves being the vertices of ${\cal P}_{\ba}$. Following \cite{Band1D}, we define the equivalence classes of all such trees under graph isomorphism, and denote the collection of equivalence classes by $\TSP({\cal P}_{\ba})$. Subsequently, we will consider each element of $\TSP({\cal P}_{\ba})$ as an abstract tree structure rather than as an equivalence class, and call it a \emph{canonical tree partition}.

\end{definition}

%\begin{notation}
In the following discussion, we will typically use the letter $[a]$ to refer to the vertices of the polygon $\mathcal{P}_{\ba}$, which we call \emph{external vertices}, while we will use the letter $[b]$ to denote the remaining interior vertices, referred to as \emph{internal vertices}. 
We will call an edge containing exactly one external vertex as an \emph{external edge}, and an edge between internal vertices as an \emph{internal edge}. The sides of $\mathcal{P}_{\ba}$ will be referred to as  \emph{boundary edges}. (We will always draw boundary edges as \emph{dashed lines}, since they are irrelevant to the graph values as we will see.)
In a canonical tree partition $\Gamma \in \TSP({\cal P}_{\ba})$, we say that two regions $R_k$ and $R_l$ are \emph{adjacent} or \emph{neighbors} if they share a common side, which may be either an external or an internal edge. In the case of an external edge, we must have $k-l\mod \fn =\pm 1$, and we refer to $R_k$ and $R_l$ as \emph{trivial neighbors}. If the shared edge is internal, we call them as \emph{nontrivial neighbors}. 
We refer readers to the left picture of \Cref{example} below for an illustration of a canonical tree partition $\Gamma\in \TSP({\cal P}_{\ba})$ of a polygon with six vertices, where $R_4$ and $R_6$ are nontrivial neighbors. We remark that the figures in this section are taken from \cite{RBSO1D}, since our graphs exhibit similar structures to those presented therein.

We assign a value to a canonical tree partition according to the following rule. 
 
\begin{definition}\label{M-graph-value-definition}
Given a polygon $\mathcal{P}_{\ba}$ with $\fn$ vertices $\ba=([a_1],\ldots, [a_\fn])$, let $\Gamma \in \TSP\p{\mathcal{P}_{\ba}}$ and denote %by $\mathcal{E}\p{\Gamma_{\ba}}$ the subset of non-boundary edges in $\Gamma_{\ba}$.
  \begin{equation}\label{unlabeled-edge-set}
    \mathcal{E}\p{\Gamma}
      \coloneqq \set{e \suchthat e\text{ is a non-boundary edge in }\Gamma}.
  \end{equation}
For an arbitrary $\bsig\in\{+,-\}^\fn$, we define the value function $f_\bsig \colon \mathcal{E}(\Gamma) \to \C$ in the following way.
  \begin{enumerate}
    \item If $e = \p{[a_k], [b]}$ is an external edge lying between regions $R_k$ and $R_{k+1}$ (with the cyclic convention that $R_{1}=R_{\fn+1}$), then we define 
      \begin{equation}\label{f-external}
        f_{t,\bsig}\p{e} \coloneqq \Theta^{(\sig_k,\sig_{k+1})}_{t}([a],[b]).
      \end{equation}
    \item If $e = \p{[b_1], [b_2]}$ is an internal edge lying between regions $R_k$ and $R_{l}$, then 
      \begin{equation}\label{f-internal}
      \begin{aligned}
        f_{t,\bsig}\p{e}
          &\coloneqq \big(\Theta^{(\sig_k,\sig_{l})}_{t}-1\big) \p{[b_1], [b_2]}\\
          &= t m(\sig_k)m(\sig_l)\big(S^{\LK}\Theta^{(\sig_k,\sig_{l})}_{t}\big) \p{[b_1], [b_2]} .
          \end{aligned}
      \end{equation}
  \end{enumerate}
Then, we assign a value $\Gamma^{(\fn)}_{t,\bsig,\ba}$ to  $\Gamma$ as:
\begin{equation}\label{M-graph-value-unsummed}
\Gamma^{(\fn)}_{t,\bsig,\ba} \coloneqq \p{\prod_{i=1}^\fn m(\sigma_i)} \cdot \sum_{\mathbf b} \prod_{e \in \mathcal{E}(\Gamma )} f_{t,\bsig}\p{e}    \, ,
\end{equation}
where $\mathbf b=([b_1],\ldots,[b_{\fm}])$ denotes the internal vertices in $\Gamma$. 
The boundary edges are irrelevant to the graph values---one can think that each of them has a value of 1.
%Note $\Gamma$ represents a graph, while $\Gamma^{(\fn)}_{t,\bsig,\ba}$ is a complex function of $t$, $\bm{\sigma}$, and $\ba$.
\end{definition}

Recall that the primitive loops of length 2 and 3 take the forms \eqref{Kn2sol} and \eqref{Kn3sol}. The next lemma gives the tree representation formula for general primitive loops of length $\fn\ge 4$, expressed as a sum of \smash{$\Gamma^{(\fn)}_{t,\bsig,\ba}$} for all possible canonical tree partitions.

\begin{lemma}[Lemma 3.4 of \cite{Band1D}]\label{tree-representation}
For any $\fn\ge 4$, $t\in[0,1)$, $\bsig\in \{+,-\}^\fn$, and $\ba\in (\Zn)^\fn$, we have the following representation formula for primitive loops:
  \begin{equation}\label{eq:tree_rep}
\cK_{t,\bm{\sigma},\ba}^{(\fn)}
      =W^{-d(\fn-1)} \sum_{\Gamma \in \TSP\p{\mathcal{P}_{\ba}}} \Gamma^{(\fn)}_{t,\bm{\sigma},\ba}.
  \end{equation}
\end{lemma}

%cite \cite{Bandedge} here
Both proofs of \Cref{lem:pure_sum,ML:Kbound} rely on the tree representation formula for $\cK$-loops presented above. In the remainder of this section, we focus on proving the new result, \Cref{lem:pure_sum}, and postpone the proof of \Cref{ML:Kbound} to \Cref{subsec:estimates}.

\subsection{Proof of \Cref{lem:pure_sum}}\label{sec:pf_new_identity}

For the proof of \Cref{lem:pure_sum}, we use an equivalent representation of primitive loops introduced in \cite[Section 4]{RBSO1D}. 
This representation is slightly more complicated than that presented in \Cref{sec_tree}, but it has a clearer structure, which enhances the clarity of our proof. We first introduce an extension of \Cref{def:canpnical_part}, which incorporates loops consisting of \emph{$M$-edges}. As the name suggests, these edges correspond to the entries of $M_t(z)=m_t(z)I_N$, whereas the remaining non-boundary edges in our graphs will be \emph{unlabeled} with the external and internal edges representing the entries of $\Theta_t$ and $tS^{\LK}\Theta_t$, respectively.

\begin{definition}[Canonical partitions with $M$-loops]\label{m-loop-tsp}
Let $\Gamma \in \TSP\p{\mathcal{P}_{\ba}}$ be a canonical tree partition of the oriented polygon $\mathcal{P}_{\ba}$, and denote the internal vertices of $\Gamma$ by $\mathbf b = \p{[b_1], \ldots, [b_\fm]}$. We define the graph {$\Gamma_{M}$} by replacing each $b_i$ with an $M$-loop in the following way.
%\begin{enumerate}[label=\arabic*.]
%\item 

\medskip 

\noindent 1. Consider the subgraph $(V^{\p{b_i}}, E^{\p{b_i}})$ of all vertices in $\Gamma$ connected to $[b_i]$. More precisely, we let
      \begin{equation}\label{b_i-subgraph}
        V^{\p{b_i}} \coloneqq \set{[b_i],[c_1], \ldots, [c_{\delta_i}]},
        \quad \text{and}\quad E^{\p{b_i}} = \set{\p{[c_1], [b_i]}, \ldots, \p{[c_{\delta_i}], [b_i]}}
      \end{equation}
      denote the subsets of all vertices (including $[b_i]$) and edges connected to $[b_i]$. Reordering the $[c_k]$'s if necessary, we can ensure that $\p{[c_1], \ldots, [c_{\delta_i}]}$ form a loop without any crossing edges and with vertices arranged in counterclockwise order.
      
\medskip 

\noindent 2. Next, we construct a new graph $(\widetilde{V}^{\p{b_i}}, \widetilde{E}^{\p{b_i}})$ with vertices and edges
      \begin{align*}
          \widetilde{V}^{\p{b_i}} &= \set{[b_{i,1}], \ldots, [b_{i,\delta_i}], [c_1], \ldots, [c_{\delta_i}]},\\
           \widetilde{E}^{\p{b_i}}
     &=\set{\p{[c_j], [b_{i,j}]} :  j \in \qqq{\delta_i}} \cup \set{\p{[b_{i,j}], [b_{i,j+1}]; M}: j \in \qqq{\delta_i}}, %\p{[b_{i,2}], [b_{i,3}]; M},      \ldots, \p{[b_{i,\delta_i}], [b_{i,1}]; M}
      \end{align*}
   where $e = \p{e_i, e_f; M}$ refers to an edge labeled with $M$, and we adopt the cyclic convention that $[b_{i,\delta_i+1}]=[b_{i,1}]$. We will call $\p{e_i, e_f; M}$ as an \emph{$M$-edge}. Relabeling the $[b_{i,j}]$'s if needed, we ensure that $\p{[b_{i,1}] \to [b_{i,2}] \to \cdots \to [b_{i,\delta_i}]}$ are also arranged in counterclockwise order. %has the same orientation as $\p{a_1 \to a_2 \to \cdots \to a_n}$.
      % For convenience, we denote
      % \begin{equation}\label{b_i-cycle}
      %   C^{\p{b_i}} \coloneqq \p{b_{i,1}, b_{i,2}, \ldots, b_{i,\delta_i}},
      % \end{equation}
      % which is a polygon obtained by ``splitting" $b_i$ into $\delta_i$ new vertices and connecting them by $M$-edges. 

\medskip 

\noindent 3. Lastly, we replace the subgraph $(V^{\p{b_i}}, E^{\p{b_i}})$ with $(\widetilde{V}^{\p{b_i}}, \widetilde{E}^{\p{i}})$.
%  \end{enumerate}

  \begin{figure}[h]\label{b_i-loop-replacement}
    \centering
     \scalebox{0.9}{
    \begin{tikzpicture}
      \coordinate (bi) at (0, 0);
      
      \fill (bi) circle (1pt) node[left=4pt]{$[b_i]$};
      \foreach \i/\t in {1/72, 2/144, 3/216, 4/288, 5/360} {
        \fill (\t:2) circle (1pt);
        \draw (bi) -- (\t:2);
        \node at (\t+72:2.4) {$[c_{\i}]$};
      }
    \end{tikzpicture}
    $\quad$
    \begin{tikzpicture}
      \node at (-3, 0) {$\to$};
      \node at (0, 0) {$M$};
      \foreach \i/\t in {1/72, 2/144, 3/216, 4/288, 5/360} {
        \fill (\t:2) circle (1pt);
        \fill (\t:1) circle (1pt);
        \draw (\t:1) -- (\t:2);
        \draw (\t:1) -- (\t+72:1);
        \node at (\t+72:2.4) {$[c_{\i}]$};
        \node at (\t+92:1.3) {$[b_{i,\i}]$};
      }
    \end{tikzpicture}}
    \caption{Replacing $[b_i]$ with an $M$-loop with 5 sides.}\label{Fig:Mloop}
  \end{figure}

  In \Cref{Fig:Mloop}, we illustrate the above procedure for an example with $\delta_i=5$. Repeating these steps for each internal vertex $b_i$, $i\in\qqq{\fm}$, we get a graph \smash{$\Gamma_{M}$}, which we will refer to as the \emph{$M$-graph corresponding to $\Gamma$}. Note that the order in which we replace $b_i$'s by $M$-loops does not matter, and every boundary edge $\p{a_{k-1}, a_k}$ still belongs to exactly one polygonal region in the $M$-graph \smash{$\Gamma_{M}$}. 
With a slight abuse of notation, we still use $R_k$ to denote the sub-region containing $\p{a_{k-1}, a_k}$ in \smash{$\Gamma_{M}$}, and assign the charge $\sig_k$ to $R_k$. 
% see \Cref{Fig:Mloop2} for an example.
% \begin{figure}[h]\label{m-graph-region}
%     \centering
%     \begin{tikzpicture}
%       \fill (-1, 0) circle (1pt);
%       \fill (1, 0) circle (1pt);
%       \fill (0, 2) circle (1pt);
      
%       \draw (-1, 0) node[below]{$a_{k-1}$}
%         -- (1, 0) node[below]{$a_k$}
%         -- (0, 2) node[above]{$b_i$}
%         -- cycle;
%       \node at (0, 0.75) {$R_k$};
%       \node at (0, -1) {$\Gamma_{\ba}$};
%     \end{tikzpicture}
%     \hspace{1in}
%     \begin{tikzpicture}
%       \fill (-1, 0) circle (1pt);
%       \fill (1, 0) circle (1pt);
%       \fill (0.5, 2) circle (1pt);
%       \fill (-0.5, 2) circle (1pt);
      
%       \draw (-1, 0) node[below]{$a_{k-1}$}
%         -- (1, 0) node[below]{$a_k$}
%         -- (0.5, 2)  node[above right]{$b_{i,j}$}
%         -- (-0.5, 2) node[above left]{$b_{i,j-1}$}
%                      node[pos=0.5, above]{$M$}
%         -- cycle;
%       \node at (0, 0.75) {$R_k$};
%       \node at (0, -1) {$\Gamma^{\p{M}}_{\ba}$};
%     \end{tikzpicture}
%     \caption{$R_k$ refers to the region containing $\protect\p{a_{k-1}, a_k}$ in both $\Gamma_{\ba}$ and its corresponding $M$-graph \smash{$\Gamma^{\protect\p{M}}_{\ba}$}. %Moreover, this also illustrates that the $M$-edge $\protect\p{b_{i,j-1}, b_{i,j}}$ belongs to the region shared by $\protect\p{a_{k-1}, b_i}$ and $\protect\p{a_k, b_i}$.
%     }\label{Fig:Mloop2}
%   \end{figure}
\end{definition}

We refer readers to \Cref{example} for an example of a canonical tree partition $\Gamma$ and its $M$-graph $\Gamma_{M}$. 
  \begin{figure}[h]
    \centering
     \scalebox{0.9}{
    \begin{tikzpicture}
      \coordinate (b1) at (-1, 0);
      \coordinate (b2) at (1, 0);
      
      \fill (b1) circle (1pt);
      \fill (b2) circle (1pt);

      \foreach \i/\t in {1/60, 2/120, 3/180, 4/240, 5/300, 6/360} {
        \coordinate (a\i) at (\t+90:2.5);
        \fill (a\i) circle (1pt);
        \draw (a\i) [dashed]-- (\t+150:2.5);
        \node at (\t+90:2.85) {$[a_{\i}]$};
        \node at (\t+60:2.5) {$\sigma_{\i}$};
      }
      
      \draw (a6) -- (b1);
      \draw (a1) -- (b1);
      \draw (a2) -- (b1);
      \draw (a3) -- (b1);
      
      \draw (a4) -- (b2);
      \draw (a5) -- (b2);
      
      \draw (b1) -- (b2);

      \node at (0, -3.5) {$\Gamma$};
    \end{tikzpicture}
    \hspace{1in}
    \begin{tikzpicture}
      \foreach \i/\t in {1/60, 2/120, 3/180, 4/240, 5/300, 6/360} {
        \coordinate (a\i) at (\t+90:2.5);
        \fill (a\i) circle (1pt);
        \draw (a\i) [dashed]-- (\t+150:2.5);
        \node at (\t+90:2.85) {$[a_{\i}]$};
        \node at (\t+60:2.5) {$\sigma_{\i}$};
      }
      
      \coordinate (c11) at ($(b1)!0.45!(a6)$);
      \coordinate (c12) at ($(b1)!0.45!(a1)$);
      \coordinate (c13) at ($(b1)!0.45!(a2)$);
      \coordinate (c14) at ($(b1)!0.45!(a3)$);
      \coordinate (c15) at ($(b1)!0.45!(b2)$);
      
      \coordinate (c21) at ($(b2)!0.25!(b1)$);
      \coordinate (c22) at ($(b2)!0.45!(a4)$);
      \coordinate (c23) at ($(b2)!0.45!(a5)$);
      
      \foreach \i/\nc in {1/5, 2/3} {
        \foreach \j in {1, 2, ..., \nc} {
          \fill (c\i\j) circle (1pt);
        }
      }
      
      \foreach \i/\nc in {1/5, 2/3} {
        \draw (c\i1)
        \foreach \j in {1, ..., \nc} {
           -- (c\i\j)
        }
          -- cycle;
      }
      
      \draw (c11) -- (a6);
      \draw (c12) -- (a1);
      \draw (c13) -- (a2);
      \draw (c14) -- (a3);
      
      \draw (c21) -- (c15);
      
      \draw (c22) -- (a4);
      \draw (c23) -- (a5);
      
      \node at ($0.2*(c11) + 0.2*(c12) + 0.2*(c13) + 0.2*(c14) + 0.2*(c15)$) {$M$};
      \node at ($0.333*(c21) + 0.333*(c22) + 0.333*(c23)$) {$M$};
      
      \node at (0, -3.5) {$\Gamma_{{M}}$};
    \end{tikzpicture}}
    \caption{Example of $\Gamma \in \TSP\protect\p{\mathcal{P}_{\ba}}$ and its corresponding $M$-graph $\Gamma_{M}$.}\label{example}
  \end{figure}

Given a $M$-graph \smash{$\Gamma_{M}$}, set
  \begin{equation}\label{unlabeled-edge-set2}
    \mathcal{E}\p{\Gamma_{M}}
      \coloneqq \set{e  \suchthat e\text{ is a non-boundary edge in } {\Gamma_{M}}}.
  \end{equation}
Given $\bsig\in\{+,-\}^\fn$ and $z\in \C_+$, define the value function $g_{\bsig,z} \colon \mathcal{E}(\Gamma_M) \to \C$ as: % follows. %in the following way.
  \begin{enumerate}
    \item If $e = \p{[a_k], [b]}$ is an external edge lying between regions $R_k$ and $R_{k+1}$, then %we define
      \begin{equation}\label{f-external2}
        g_{\bsig,z}\p{e} \coloneqq \Theta^{\p{\sig_{k}, \sig_{k+1}}}_{t,[a_k][b]}(z) \, .
      \end{equation}
   Here, similar to \eqref{def_Thxi}, we define 
\[\Theta_{t}^{(\sigma_1,\sigma_2)}(z):= \big[1 - t m_t(z,\sigma_1)m_t(z,\sigma_2)S^{\LK}\big]^{-1} \, ,\]
with $m_t(z,+)\equiv m_t(z)$ and $m_t(z,-)\equiv \overline m_t(z)$ following the convention in \eqref{def_mtzk}. Note that under this notation, we can express \eqref{def_Thxi} as \smash{$\Theta_{t}^{(\sigma_1,\sigma_2)}(z_t)$}.

    \item If $e = \p{[b_1], [b_2]}$ is an internal unlabeled edge lying between regions $R_k$ and $R_{l}$, then %we define
      \begin{equation}\label{f-internal2}
        g_{\bsig,z}\p{e}
          \coloneqq \big(tS^{\LK}\Theta^{(\sig_k,\sig_{l})}_{t}(z)\big)_{[b_1][b_2]} .
      \end{equation}

    \item If $e = \p{[b_1], [b_2]}$ is an internal $M$-edge belonging to the region $R_k$, then %we define
      \begin{equation}\label{f-internalM}
        g_{\bsig,z}\p{e}
          \coloneqq m_t(z,\sig_k)\delta_{[b_1][b_2]}\, .
      \end{equation}
      
  \end{enumerate}
Now, we assign a value $\Gamma^{(\fn)}_{M;t,\bsig,\ba}(z)$ to $\Gamma_M$ as:
\begin{equation}\label{M-graph-value-unsummed2}
\Gamma^{(\fn)}_{M;t,\bsig,\ba}(z) \coloneqq \sum_{\mathbf b}\prod_{e \in \mathcal{E}(\Gamma_M)} g_{\bsig,z}\p{e}  \,  ,
\end{equation}
where $\mathbf b$ denotes the internal vertices in $\Gamma$. 
Counting the $m_t(z_t,\sig_i)=m(\sig_i)$ factors, it is not hard to see \smash{$\Gamma^{(\fn)}_{M;t,\bsig,\ba}(z)$} is equal to \smash{$\Gamma^{(\fn)}_{t,\bsig,\ba}$} defined in \eqref{M-graph-value-unsummed} when we take $z=z_t$. Hence, the tree representation formula \eqref{eq:tree_rep} can be rewritten as:
\begin{equation}\label{eq:tree_rep2}
\cK_{t,\bm{\sigma},\ba}^{(\fn)}
=W^{-d(\fn-1)} \sum_{\Gamma \in \TSP\p{\mathcal{P}_{\ba}}} \Gamma^{(\fn)}_{M;t,\bm{\sigma},\ba}(z_1,\ldots, z_\fn),
  \end{equation}
where $z_i\in \{z_t,\overline z_t\}$ labels all the $m(\sig_i)=m_t(z_i,\sig_i)$ factors with charge $\sig_i$ for $i\in \qqq{\fn}$. 

For each $k \in \qqq{\fn}$, the region $R_k$ (in $\Gamma$ or $\Gamma_M$) is a polygon and hence has exactly two paths from $[a_{k-1}]$ to $[a_k]$; one is $\p{[a_{k-1}] \to [a_k]}$, and we label the other as $p_k \coloneqq \p{[d_{k,1}] \to [d_{k,2}] \to \cdots \to [d_{k,l_k}]}$ with $[d_{k,1}] = [a_{k-1}]$ and $[d_{k,l_k}] = [a_k]$. We call it a \emph{canonical directed path} (in $\Gamma$ or $\Gamma_M$). Then, we define the following operations on graphs $\Gamma \in \TSP\p{\mathcal{P}_{\ba}}$. % $\TSP$ graphs.

\begin{definition}\label{Def:slice}
%For $\ba = \p{[a_1], \ldots, [a_\fn]}$ and fix 
For $\Gamma \in \TSP\p{\mathcal{P}_{\ba}}$. We define the \textbf{slices} of $\Gamma$ as follows.
  \begin{enumerate}
    \item Let $p_1$ be the canonical directed path in $R_1$. 
    Suppose $p_1$ can be written as 
    \be\label{eq:p1}p_1 = \p{[a_\fn] \to [b_1] \to \cdots \to [b_{k-1}] \to [a_1]}.\ee 

    \item For each edge of the form $\p{[b_{i-1}], [b_i]}$, $1\le i \le k$ (with the convention $[b_0] = [a_\fn]$ and $[b_k] = [a_1]$), we define $\Gamma_{\p{[b_{i-1}], [b_i]}} \in \TSP(\mathcal{P}_{\p{\ba,[a_{\fn+1}]}})$ as the graph obtained by adding new vertices $[a_{\fn+1}], [b']$ to $\Gamma$ and replacing $\p{[b_{i-1}], [b_i]}$ with the subgraph consisting of three edges $\p{[b_{i-1}], [b']}$, $\p{[b'], [b_i]},$ and $ \p{[a_{\fn+1}], [b']}$.
    \item For each vertex $[b_i]$, $1 \leq i \leq k - 1$ (i.e., we exclude the vertices $[a_1], [a_\fn]$), we define $\Gamma_{\p{[b_i]}} \in \TSP(\mathcal{P}_{\p{\ba, [a_{\fn+1}]}})$ as the graph obtained by simply adding the vertex $[a_{\fn+1}]$ and the edge $\p{[a_{\fn+1}], [b_i]}$.
  \end{enumerate}
  See \Cref{Fig:slice} for illustrations of the above slicing operations. We then define the collection of slices of $\Gamma$ by
  \begin{equation}
    \slice\p{\Gamma}
      \coloneqq \set{\Gamma_{\p{[b_{i-1}],[b_i]}} \suchthat 1 \leq i \leq k} \cup \set{\Gamma_{\p{[b_i]}} \suchthat 1 \leq i \leq k-1}. \label{gamma-slices}
  \end{equation}
  \begin{figure}[h]
    \centering
    \scalebox{0.85}{
    \begin{tikzpicture}
      \coordinate (a1) at (180:2);
      \coordinate (an) at (90:2);
      \coordinate (b) at (0, 0);
      
      \fill (a1) circle (1pt) node[left]{$[a_1]$};
      \fill (an) circle (1pt) node[above]{$[a_\fn]$};
      \fill (b) circle (1pt) node[below right]{$[b_1]$};
      
      \draw (an) [dashed]-- (a1);
      \draw (a1) -- (b);
      \draw (b) -- (an);
      
      \node at (-1, -0.5) {$\Gamma$};
    \end{tikzpicture}
    \begin{tikzpicture}
      \coordinate (a1) at (180:2);
      \coordinate (an) at (90:2);
      \coordinate (an+1) at (135:2);
      \coordinate (b) at (0, 0);
      \coordinate (c) at (90:1);
      
      \fill (a1) circle (1pt) node[left]{$[a_1]$};
      \fill (an) circle (1pt) node[above]{$[a_\fn]$};
      \fill (b) circle (1pt) node[below right]{$[b_1]$};
      \fill (an+1) circle (1pt) node[above left]{$[a_{\fn+1}]$};
      \fill (c) circle (1pt);
      
      \draw (an) [dashed]-- (an+1);
      \draw (an+1) [dashed]-- (a1);
      \draw (a1) -- (b);
      \draw (b) -- (c) -- (an);
      \draw (an+1) -- (c);
      
      \node at (-1, -0.5) {$\Gamma_{\p{[a_\fn],[b_1]}}$};
      \node at (-3, 1) {$\to$};
    \end{tikzpicture}
    $\quad$
    \begin{tikzpicture}
      \coordinate (a1) at (180:2);
      \coordinate (an) at (90:2);
      \coordinate (an+1) at (135:2);
      \coordinate (b) at (0, 0);
      
      \fill (a1) circle (1pt) node[left]{$[a_1]$};
      \fill (an) circle (1pt) node[above]{$[a_\fn]$};
      \fill (b) circle (1pt) node[below right]{$[b_1]$};
      \fill (an+1) circle (1pt) node[above left]{$[a_{\fn+1}]$};
      
      \draw (an) [dashed]-- (an+1);
      \draw (an+1) [dashed]-- (a1);
      \draw (a1) -- (b);
      \draw (b) -- (an);
      \draw (an+1) -- (b);
      
      \node at (-1, -0.5) {$\Gamma_{\p{[b_1]}}$};
    \end{tikzpicture}
    $\quad$
    \begin{tikzpicture}
      \coordinate (a1) at (180:2);
      \coordinate (an) at (90:2);
      \coordinate (an+1) at (135:2);
      \coordinate (b) at (0, 0);
      \coordinate (c) at (180:1);
      
      \fill (a1) circle (1pt) node[left]{$[a_1]$};
      \fill (an) circle (1pt) node[above]{$[a_\fn]$};
      \fill (b) circle (1pt) node[below right]{$[b_1]$};
      \fill (an+1) circle (1pt) node[above left]{$[a_{\fn+1}]$};
      \fill (c) circle (1pt);
      
      \draw (an) [dashed]-- (an+1);
      \draw (an+1) [dashed]-- (a1);
      \draw (a1) -- (c);
      \draw (c) -- (b) -- (an);
      \draw (an+1) -- (c);
      
      \node at (-1, -0.5) {$\Gamma_{\p{[b_1],[a_1]}}$};
    \end{tikzpicture}
    }
    \caption{$\slice\protect\p{\Gamma}$ is created by ``slicing" up $R_1$ into two pieces.}\label{Fig:slice}
  \end{figure}
\end{definition}

By the above definition, it is easy to observe that the $\TSP$ graphs on $(\fn + 1)$ vertices can be constructed by taking slices of the $\TSP$ graphs on $\fn$ vertices. 
\begin{claim}\label{slice-decomposition}
%For $\ba = \p{a_1, \ldots, a_\fn}$, 
$\TSP(\mathcal{P}_{\p{\ba, [a_{\fn+1}]}})$ can be expressed as a disjoint union
  \[
    \TSP\p{\mathcal{P}_{\p{\ba, [a_{\fn+1}]}}}
      = \bigsqcup_{\Gamma \in \TSP\p{\mathcal{P}_{\ba}}} \slice\p{\Gamma}.
  \]
\end{claim}

Now, we are prepared to give the proof of \Cref{lem:pure_sum} using the representation \eqref{eq:tree_rep2} and \Cref{slice-decomposition}. 
\begin{proof}[\bf Proof of \Cref{lem:pure_sum}]
First, \eqref{eq:pure_sum} trivial holds when $\fn=1$. Second, using the self-consistent equation $m_t(z)(z+tm_t(z))+1=0$ for $m_t$, we can derive that
\be\label{eq:mt'z} \partial_z m_t(z)= \frac{m_t(z)^2}{1-t m_t(z)^2}.\ee
Combining it with the definition of $2$-$\cK$ loops in \eqref{Kn2sol} (with $z_t$ replaced by $z$), we see that \eqref{eq:pure_sum} also holds when $\fn=2$.

Suppose we have shown that \eqref{eq:pure_sum} holds for some $\fn \ge 2$. We now show that 
\be\label{eq:pure_sum2}  
W^{-\fn d} \frac{\dd^{\fn}}{\dd z^{\fn}}m_t(z)= \fn!\sum_{[a_2],\ldots,[a_{\fn+1}]} \cK^{(\fn+1)}_{t,\bsig,\ba}(z).\ee
With the induction hypothesis, the representation \eqref{eq:tree_rep2}, and the symmetry under permutations of $z_1,\ldots, z_{\fn}$ (since $\bsig$ is a pure loop), we see that it suffices to showing the following identity (with $\ba':=(\ba,[a_{\fn+1}])$):
\be\label{eq:pure_sum_equiv} 
\sum_{\Gamma \in \TSP\p{\mathcal{P}_{\ba}}} \partial_{z_1}\Gamma^{(\fn)}_{M;t,\bm{\sigma},\ba}(z_1,\ldots, z_\fn) = \sum_{[a_{\fn+1}]}\sum_{\Gamma \in \TSP(\mathcal{P}_{\ba'})} \Gamma^{(\fn+1)}_{M;t,\bm{\sigma},\ba}(z_1,\ldots, z_{\fn+1}).
\ee

To show \eqref{eq:pure_sum_equiv}, we pick an arbitrary $\Gamma \in \TSP\p{\mathcal{P}_{\ba}}$. Denote its internal vertices by $\mathbf b$ and suppose the canonical directed path in region $R_{1}$ of \smash{$\Gamma$} takes the form \eqref{eq:p1}. Then, we consider the canonical directed path in region $R_{1}$ of \smash{$\Gamma_{M}$}, which contains all edges that are associated with the charge $\sig_1$. Suppose the product of these values takes the form
  \be\label{eq:calGz1}
  \begin{aligned}
    \mathcal{G}(z_1)
      \coloneqq&~ \Theta^{\p{\sigma_\fn,\sigma_1}}_{t,[a_\fn] [b_{1,1}]}
        \p{\prod_{i=2}^{k-1} M^{\p{\sigma_1}}_{[b_{i-1,1}][b_{i-1,2}]} \p{tS^{\LK}\Theta_t^{(\sig_{r_i},\sigma_1)}}_{[b_{i-1,2}][b_{i,1}]}} \\
&~\times     M^{\p{\sigma_1}}_{[b_{k-1,1}][b_{k-1,2}]}
        \Theta^{\p{\sigma_1,\sigma_2}}_{t,[a_1][b_{k-1,2}]},
        \end{aligned}  \ee
  where, as defined in \Cref{m-loop-tsp}, $b_{i,j}$ denotes vertices on the $M$-loops obtained from the replacements of the vertices $b_i$ in $\Gamma$, $M^{(\sigma_1)}$ represents $m_t(z_1)I_N$, and $r_i$ labels the other domain that is adjacent to the edge $([b_{i-1,2}],[b_{i,1}])$. Note that these are the only factors depending on $z_1$, and that they are listed in the order in which their indices appear in the canonical directed path.  
  %$\tau_i:=q(b_{i-1,2})$ for $2\le i \le k-1$, and $b_{i,j}$ for $1\le i \le k-1$ to represent the vertices on the $M$-loops obtained from the replacements of the vertices $b_i$ in $\Gamma_{\ba}$. 
 %  \begin{figure}[h]
 %    \centering
 %    \begin{tikzpicture}
 %      \coordinate (an) at (0, 3);
 %      \coordinate (a1) at (-3, 0);
 %      \coordinate (b11) at (0, 2);
 %      \coordinate (b12) at (0, 1);
 %      \coordinate (b21) at (-1, 0);
 %      \coordinate (b22) at (-2, 0);
      
 %      \fill (an) circle (1pt) node[above]{$a_\fn$};
 %      \fill (a1) circle (1pt) node[left]{$a_1$};
 %      \fill (b11) circle (1pt) node[right]{$b_{1,1}$};
 %      \fill (b12) circle (1pt) node[right]{$b_{1,2}$};
 %      \fill (b21) circle (1pt) node[below]{$b_{2,1}$};
 %      \fill (b22) circle (1pt) node[below]{$b_{2,2}$};
      
 %      \draw[blue, dashed, arrows={->[scale=1.5]}, shorten >= 2pt] (an) -- (a1) node[black, pos=0.5, above left]{$+$};
 %      \draw (an) -- (b11);
 %      \draw[->, blue, arrows={->[scale=1.5]}, shorten >= 2pt] (b11) -- (b12);
 %      \draw (b12) -- (b21);
 %      \draw[->, blue, arrows={->[scale=1.5]}, shorten >= 2pt] (b21) -- (b22);
 %      \draw (b22) -- (a1);
 %    \end{tikzpicture}
 % \caption{An example of $\cal G(z_1)$ in \eqref{eq:calGz1} with $k = 3$. The blue edges represent the $M$-edges and boundary edges.}\label{calG-k=3}
 %  \end{figure}
   
We now study the derivative of $\cal G(z_1)$ with respect to $z_1$. First, suppose the derivative acts on an $M$-edge \smash{\(M^{(\sig_1)}_{[b_{i,1}][b_{i,2}]}=m_t(z_1)\delta_{[b_{i,1}][b_{i,2}]}\)}. With \eqref{eq:mt'z}, we get that 
\begin{align*}
  \partial_{z_1}m_t(z_1)\delta_{[b_{i,1}][b_{i,2}]}&= \frac{m_t(z_1)^2}{1-tm_t(z_1)^2}\delta_{[b_{i,1}][b_{i,2}]}\\
  &=\sum_{[a_{\fn+1}],[b']} m_t(z_1)\delta_{[b_{i,1}][b']} \cdot m_t(z_1)\delta_{[b'][b_{i,2}]} \cdot \Theta_{t,[a_{\fn+1}][b']}^{(+,+)} \, .
\end{align*}
Referring to Figure \ref{ward-even-i}, we readily see that the RHS corresponds to the $M$-graph of the slice $\Gamma_{([b_i])}$. In other words, we have that
  \[\Gamma^{\p{\fn}}_{M;t, \bm{\sigma},\ba}(z) \cdot \frac{\partial_{z_1}m_t(z_1)}{m_t(z_1)}
      = \sum_{[a_{\fn+1}]} \p{\Gamma_{\p{[b_{i}]}}}^{\p{\fn+1}}_{M;t, \p{\bm{\sigma},+}, (\ba,[a_{\fn+1}])},\quad \forall i=1,\ldots, k-1 \, .
  \]
  \begin{figure}[h]
    \centering
    \scalebox{0.9}{
    \begin{tikzpicture}
      \coordinate (an) at (0, 3);
      \coordinate (a1) at (-3, 0);
      \coordinate (b11) at (0, 1);
      \coordinate (b12) at (-1, 0);
      
      \fill (an) circle (1pt) node[above]{$[a_\fn]$};
      \fill (a1) circle (1pt) node[left]{$[a_1]$};
      \fill (b11) circle (1pt) node[right]{$[b_{i,1}]$};
      \fill (b12) circle (1pt) node[below]{$[b_{i,2}]$};
      
      \draw[dashed, arrows={->[scale=1.5]}, shorten >= 2pt] (an) -- (a1) node[black, pos=0.5, above left]{$\sigma_1$};
      \draw[dotted] (an) -- (b11);
      \draw[->, shorten >=2pt] (b11) -- (b12);
      \draw[dotted] (b12) -- (a1);
      
      \node at (135:0.25) {$M$};
    \end{tikzpicture}
    \begin{tikzpicture}
      \coordinate (an) at (0, 3);
      \coordinate (a1) at (-3, 0);
      \coordinate (an+1) at (135:3);
      \coordinate (b11) at (0, 1);
      \coordinate (b12) at (-1, 0);
      \coordinate (b') at (135:1);
      
      \fill (an) circle (1pt) node[above]{$[a_\fn]$};
      \fill (a1) circle (1pt) node[left]{$[a_1]$};
      \fill (an+1) circle (1pt) node[above left]{$[a_{\fn+1}]$};
      \fill (b11) circle (1pt) node[right]{$[b_{i,1}]$};
      \fill (b12) circle (1pt) node[below]{$[b_{i,2}]$};
      \fill (b') circle (1pt) node[above, xshift=2pt, yshift=2pt]{$[b']$};
      
      \draw[dashed, red, arrows={->[scale=1.5]}, shorten >= 2pt] (an) -- (an+1) node[black, pos=0.5, above left]{$+$};
      \draw[dashed, blue, arrows={->[scale=1.5]}, shorten >= 2pt] (an+1) -- (a1) node[black, pos=0.5, above left]{$+$};
      \draw[->, red, arrows={->[scale=1.5]}, shorten >= 2pt] (b11) -- (b');
      \draw[dotted] (an) -- (b11);
      \draw[->, blue, arrows={->[scale=1.5]}, shorten >= 2pt] (b') -- (b12);
      \draw[dotted] (b12) -- (a1);
      \draw (an+1) -- (b');
      
      \node at (135:0.25) {$M$};
      \node[anchor=east] at (-3.5, 1.5) {$\to \sum_{[a_{\fn+1}], [b']}$};
    \end{tikzpicture}}
    \caption{Derivative of an $M$-edge.}\label{ward-even-i}
  \end{figure}

  Next, suppose $\partial_{z_1}$ acts on an internal unlabeled edge \smash{\(\big(tS^{\LK}\Theta_t^{(\sig_{r_i},\sigma_1)}\big)_{[b_{i-1,2}][b_{i,1}]}\)}. Then, with \eqref{eq:mt'z}, we can calculate that 
  \begin{align*}
  &\frac{\dd}{\dd z_1}\left(\frac{tS^{\LK}}{1-tm_t(z_1)m_t(z_{r_i})S^{\LK}}\right)_{[b_{i-1,2}][b_{i,1}]} =\left(tS^{\LK}\Theta_{t}^{(+,+)} \frac{m_t(z_{r_i}) m_t(z_1)^2}{1-t m_t(z_1)^2}tS^{\LK}\Theta_{t}^{(+,+)} \right)_{[b_{i-1,2}][b_{i,1}]} \\
  &= \sum_{[a_{\fn+1}],[b'],[x],[y]}\left(tS^{\LK}\Theta_t^{(+,+)}\right)_{[b_{i-1,2}][x]}\left(tS^{\LK}\Theta_t^{(+,+)}\right)_{[y][b_{i,1}]} \cdot \p{m(z)^3\delta_{[x][b']} \delta_{[b'][y]}  \delta_{[y][x]}} \cdot \Theta^{(+,+)}_{t,[a_{\fn+1}][b']}.
  \end{align*}
Referring to Figure \ref{ward-odd-i}, we see that the RHS corresponds to the $M$-graph of the slice \smash{$\Gamma_{([b_{i-1}],[b_i])}$}.
With a similar argument, we can check that the derivatives of the two external edges {$\Theta^{\p{\sigma_\fn,\sigma_1}}_{t,[a_\fn] [b_{1,1}]}$ and $\Theta^{\p{\sigma_1,\sigma_2}}_{t,[a_1][b_{k-1,2}]}$} with respect to $z_1$ give rise to graphs that correspond to the $M$-graphs of the slices $\Gamma_{([a_\fn],[b_1])}$ and $\Gamma_{([b_{k-1}],[a_1])}$, respectively.
 
  \begin{figure}[h]
    \centering
    \scalebox{0.9}{
    \begin{tikzpicture}
      \coordinate (an) at (0, 3);
      \coordinate (a1) at (-3, 0);
      \coordinate (b12) at (0, 1);
      \coordinate (b21) at (-1, 0);
      
      \fill (an) circle (1pt) node[above]{$[a_\fn]$};
      \fill (a1) circle (1pt) node[left]{$[a_1]$};
      \fill (b12) circle (1pt) node[right]{$[b_{i-1,2}]$};
      \fill (b21) circle (1pt) node[below]{$[b_{i,1}]$};
      
      \draw[dashed, arrows={->[scale=1.5]}, shorten >= 2pt] (an) -- (a1) node[black, pos=0.5, above left]{$\sigma_1$};
      \draw[dotted] (an) -- (b12);
      \draw[arrows={->[scale=1.5]}, shorten >= 2pt] (b12) -- (b21);
      \draw[dotted] (b21) -- (a1);
      
      \node at (135:0.25) {$R_{r_i}$};
    \end{tikzpicture}
    \begin{tikzpicture}
      \coordinate (an) at (0, 3);
      \coordinate (a1) at (-3, 0);
      \coordinate (an+1) at (135:3);
      \coordinate (b11) at (0, 1);
      \coordinate (b12) at (-1, 0);
      \coordinate (b') at (135:2);
      \coordinate (x) at (-0.5, 1.25);
      \coordinate (y) at (-1.25, 0.5);
      
      \fill (an) circle (1pt) node[above]{$[a_\fn]$};
      \fill (a1) circle (1pt) node[left]{$[a_1]$};
      \fill (an+1) circle (1pt) node[above left]{$[a_{\fn+1}]$};
      \fill (b11) circle (1pt) node[right]{$[b_{i-1,2}]$};
      \fill (b12) circle (1pt) node[below]{$[b_{i,1}]$};
      \fill (b') circle (1pt) node[above, xshift=3pt]{$[b']$};
      \fill (x) circle (1pt) node[above, xshift=2pt]{$[x]$};
      \fill (y) circle (1pt) node[left, yshift=-2pt]{$[y]$};
      
      \draw[dashed, red, arrows={->[scale=1.5]}, shorten >= 2pt] (an) -- (an+1) node[black, pos=0.5, above left]{$+$};
      \draw[dashed, blue, arrows={->[scale=1.5]}, shorten >= 2pt] (an+1) -- (a1) node[black, pos=0.5, above left]{$+$};
      \draw (b11) -- (x);
      \draw[->, red, shorten >=2pt] (x) -- (b');
      \draw[->, blue, shorten >=2pt] (b') -- (y);
      \draw[->, shorten >=2pt] (y) -- (x);
      \draw[dotted] (an) -- (b11);
      \draw (y) -- (b12);

      \draw[dotted] (b12) -- (a1);
      \draw (an+1) -- (b');
      
      \node at (135:1.5) {$M$};
      \node at (135:0.75) {$R_{r_i}$};
      \node[anchor=east] at (-3.5, 1.5) {$\to \displaystyle  \sum_{[a_{\fn+1}], [b'], [x], [y]}$};
    \end{tikzpicture}}
    \caption{Derivative of an unlabeled internal edge.}\label{ward-odd-i}
  \end{figure}

  Putting everything together and applying \Cref{slice-decomposition}, we conclude \eqref{eq:pure_sum_equiv}, which implies \eqref{eq:pure_sum2} and completes the induction step.  
\end{proof}

\section{$G$-loop estimates}\label{sec:proof}

The local laws, \Cref{thm_locallaw}, and quantum diffusion, \Cref{thm_diffu}, follow directly from the following three key theorems about $G$-loop estimates.  

\begin{theorem}[$G$-loop estimates]\label{ML:GLoop}
In the setting of \Cref{thm_locallaw}, fix any $z=E+\ii \eta\in \mathbf D_{C_0,\fd}(\fc)$ and consider the flow framework in \Cref{zztE}. For each fixed $\fn\in \N$, the following estimate holds uniformly in $t\in [0,t_0]$: 
\be\label{Eq:L-KGt}
 \max_{\boldsymbol{\sigma}, \ba}\left|{\cL}^{(\fn)}_{t, \boldsymbol{\sigma}, \ba}-{\cal K}^{(\fn)}_{t, \boldsymbol{\sigma}, \ba}\right|\prec \p{W^d\ell_t^d\eta_t}^{-\fn} .  
 \ee
% Furthermore, for the generalized loops in \cref{Def:G_loopgen}, we have a slightly stronger estimate:
% \be\label{Eq:L-KGtIm}
% \max_{\bchi\in\{+,-,\im\}^{\fn}:|v_{\bchi}|=1}\max_{\ba}\left|{\cL}^{(\fn)}_{t, \bchi, \ba}-{\cal K}^{(\fn)}_{t, \bchi, \ba}\right|\prec \p{\sqrt{\kappa}/{\omega_t}}\cdot \p{W^d\ell_t^d\eta_t}^{-\fn}. 
% \ee 
\end{theorem}

Since $\eta_t\lesssim |1-t|\omega_t$, combining \eqref{Eq:L-KGt} with \eqref{eq:bcal_k}, we obtain the following estimate for any fixed $\fn\ge2$: % and $\bchi\in \{+,-,\im\}^\fn$ with $|v_\bchi|=1$:
\be
%\max_{\boldsymbol{\sigma}, \ba}\left|{\cal L}^{(\fn)}_{t, \boldsymbol{\sigma}, \ba} \right|\prec \frac{1}{W^d \ell_t^d |1-t|}\frac{1}{(W^d \ell_t^d |1-t| \eta_t )^{\fn-2}},\quad  
\max_{\bsig,\ba} \left| {\cal L}^{(\fn)}_{t, \bsig, \ba} \right| \prec   \frac{\sqrt{\kappa}}{(W^d \ell_t^d \eta_t )^{\fn-1}}. \label{Eq:L-KGt2} \ee
To see why \eqref{Eq:L-KGt} implies \eqref{Eq:L-KGt2}, we apply the estimates in \eqref{eq:E0bulk} and \eqref{eq:kappat} to obtain that for each $z=E+\ii \eta\in \mathbf D_{C_0,\fd}(\fc)$ with $E\in[-2,2]$ and $t\in [0,t_0(z)]$,
\begin{align}
W^d \ell_t^d \eta_t \sqrt{\kappa} \gtrsim \min\p{N \eta \sqrt{\kappa_E+\eta},W^d\eta^{1-d/2}(\kappa_E+\eta)^{1/2+d/4}}\gtrsim W^{(\fd/2)\wedge \fc} ,\label{eq:small_para}
\end{align}
where we used $\eta\ge W^\fd\eta_*(E)$ in the second step with $\eta_*(E)$ defined in \eqref{eq:defeta*}. Similarly, if $E\notin[-2,2]$ and $\eta\ge W^\fd\eta_*(E)$ with $\eta_*(E)$ defined in \eqref{eq:defeta*2},  we have that for $t\in [0,t_0(z)]$,
\begin{align}
%W^d \ell_t^d |1-t| \omega_t^2 \ge   
W^d \ell_t^d \eta_t \sqrt{\kappa} \gtrsim \min\p{N \eta^2/\sqrt{\kappa_E+\eta},W^d\eta^{2}/(\kappa_E+\eta)^{1/2+d/4}}\gtrsim W^{\fd}.\label{eq:small_para2}
\end{align}

\begin{theorem}[$2$-$G$ loop estimates]\label{ML:GLoop_expec}
In the setting of \Cref{ML:GLoop}, the expectation of a $2$-$G$ loop satisfies a better bound uniformly in $t\in [0,t_0]$: 
\begin{equation}\label{Eq:Gtlp_exp}
\max_{\boldsymbol{\sigma}, \ba}\left|\mathbb E{\cal L}^{(2)}_{t, \boldsymbol{\sigma}, \ba}-{\cal K}^{(2)}_{t, \boldsymbol{\sigma}, \ba}\right|\prec \kappa^{-1/2}(W^d\ell_t^d\eta_t)^{-3}
 . 
\end{equation}
Moreover, for $\boldsymbol{\sigma}=(+,-)$ and $ \ba=([a_1], [a_2])$, we have the following decay estimate uniformly in $t\in [0,t_0]$: for any large constant $D>0$, 
\begin{equation}\label{Eq:Gdecay}
 \left| {\cal L}^{(2)}_{t, \boldsymbol{\sigma}, \ba}-{\cal K}^{(2)}_{t, \boldsymbol{\sigma}, \ba}\right|\prec \frac{1}{\p{W^d\ell_t^d\eta_t}^{2}}\exp \left(-\left|\frac{ [a_1]-[a_2] }{\ell_t}\right|^{1/2}\right)+W^{-D}. 
\end{equation}
\end{theorem}

We define a deterministic control parameter $\Psi_t(\sE)$ as  
\[\Psi_t(\sE):=\sqrt{\frac{\sqrt{\kappa(\sE)}}{W^d\ell_t^d\eta_t(\sE)}} . %+ \frac{\sqrt{\kappa(\sE)}}{W^d\ell_t^d\eta_t(\sE)\omega_{t}(\sE)}
\]
 Note that under \eqref{eq:small_para} and \eqref{eq:small_para2}, we have \be\label{eq;smallpsi}\Psi_t(\sE)\lesssim W^{-(\fd\wedge \fc)/4}\sqrt{\kappa(\sE)}\asymp W^{-(\fd\wedge \fc)/4}\im m(\sE).\ee  
\begin{theorem}[Local law]\label{ML:GtLocal}
In the setting of \Cref{ML:GLoop}, the following local laws hold uniformly in $t\in [0,t_0]$: 
\begin{align}\label{Gt_bound}
 \|G_{t,\sE}-M(\sE)\|_{\max} &\prec \Psi_t(\sE), \\
 \max_{[a]}\left| \avg{\p{G_{t,\sE}-M(\sE)}E_{[a]}}\right|&\prec \Psi_t(\sE)^2/\omega_t(\sE).\label{Gt_avgbound}
 \end{align} 
Note that the averaged local law \eqref{Gt_avgbound} gives a slightly stronger bound than the one-loop estimate in \eqref{Eq:L-KGt} with $\fn=1$, improved by a factor of $\sqrt{\kappa}/\omega_t(\sE)$. 
 % $$\sqrt{\frac{\sqrt{\kappa}}{W^d\ell_t^d\eta_t(\sE)}}+\frac{\sqrt{\kappa}}{W^d\ell_t^d\eta_t \omega_t} $$
\end{theorem}

Before proceeding to the proofs of Theorems \ref{ML:GLoop}--\ref{ML:GtLocal}, we use them to complete the proofs of our main results, \Cref{thm_locallaw,thm_diffu,thm:QUE}.
%\begin{proof}[\bf Proof of \Cref{thm_locallaw} and \Cref{thm_diffu}] 

\subsection{Proof of \Cref{thm_locallaw,thm_diffu,thm:QUE}}

For a fixed $z=E+\ii \eta \in \in \mathbf D_{C_0,\fd}(\fc)$, we can choose the deterministic flow as in \Cref{zztE} such that \eqref{eq:zztE} and \eqref{GtEGz} hold. 
By \eqref{square_root_density}, \eqref{eq:E0bulk}, and \eqref{eq:kappat}, we notice that $\eta_{t_0(z)}(\sE)\asymp\eta$, $\omega_{t_0}(\sE)\asymp \sqrt{\kappa_E+\eta}$, \smash{$\sqrt{\kappa(\sE)}\asymp \im m(z)$}, and the function $\ell(z)$ defined in \eqref{eq:elleta} is of the same order as $\ell_{t_0(z)}$ defined in \eqref{eq:ellt}.  
With these estimates, we observe that \eqref{Gt_bound} gives the entrywise local law \eqref{locallaw}, \eqref{Gt_avgbound} implies the averaged local law \eqref{locallaw_aver}, \eqref{ML:GLoop} (with $\fn=2$) provides the quantum diffusion estimates \eqref{eq:diffu1} and \eqref{eq:diffu2}, and \eqref{Eq:Gtlp_exp} yields the expected quantum diffusion estimates \eqref{eq:diffuExp1} and \eqref{eq:diffuExp2}.
These estimates are initially established at each \emph{fixed $z$}. To extend them uniformly over all $z$, we apply a standard argument based on an $N^{-C}$-net, union bound, and perturbation estimates, whose details are omitted here. This completes the proofs of \Cref{thm_locallaw,thm_diffu}.
%\end{proof}

%\begin{proof}[\bf Proof of \Cref{thm:QUE}] 

\Cref{thm:QUE} is an immediate consequence of \Cref{thm_diffu}. We first prove \eqref{Meq:QUE} under the condition \smash{$W\ge L^{1-d/6+\e_0}$}. For $k\in \qqq{N}$, we choose $z=\gamma_k+\ii \eta$, where $\eta=W^\fd N^{-2/3}$ for a small constant $0<\fd<\e_0/2$. 
Using \eqref{eq:gammak_location}, we can check that $\eta \ge W^{\fd/4}\eta_*(\gamma_k)$.
Furthermore, with the estimate \eqref{square_root_density}, we can check that $\ell(z)\ge n$ under the condition $W\ge L^{1-d/6+\e_0}$, so we have $W^d\ell(z)^d=N$.
Now, with the spectral decomposition of $G(z)$ and the eigenvalue rigidity estimate \eqref{eq:rigidity} or \eqref{eq:rigidityd=2} for $\lambda_k$, we obtain that
\begin{align}\label{ssfa2}
&  \left|\bu_k^*\left(E_{[a]}-N^{-1}\right) \bu_k\right|^2 \prec \eta^4 \sum_{i, j} \frac{\left|\bu_i^*\left(E_{[a]}-N^{-1}\right) \bu_j\right|^2}{\left|\lambda_i-z\right|^2\left|\lambda_j-z\right|^2}\nonumber\\
& = 
\eta^2\tr \left[\operatorname{Im} G(z)\left(E_{[a]}-N^{-1}\right) \operatorname{Im} G(z)\left(E_{[a]}-N^{-1}\right)\right].
\end{align}
The expectation of the RHS can be written as 
\begin{align}\label{que0} 
&~\frac{\eta^2}{n^{2d}}\sum_{[b],[b']\in \Zn}
\mathbb{E}  
 \tr \left[\operatorname{Im} G\left(E_{[a]}-E_{[b]}\right) \operatorname{Im} G\left(E_{[a]}-E_{[b']}\right)\right] \nonumber
\\ 
\lesssim &~ \eta^2\max_{[a],[b],[b']}
\Big|\mathbb{E}\tr \left(\p{\operatorname{Im} G} E_{[a]}  \p{\operatorname{Im} G} E_{[b]} \right)
-\mathbb{E}\tr \left(\p{\operatorname{Im} G} E_{[a]}  \p{\operatorname{Im} G} E_{[b']}\right) \Big|. 
\end{align}
Expanding $\im G$ as $(G-G^*)/(2\ii)$ and recalling \eqref{eq:avg_GTheta} and \eqref{Kn2sol}, we can bound \eqref{que0} by 
\begin{align*}
I_1+I_2:=&~\eta^2 \max_{[a],[b]}\max_{\sig_1,\sig_2} \left|\E (\cL-\cK)^{(2)}_{t_0,(\sig_1,\sig_2),([a],[b])}\right| \\
&~+ \frac{\eta^2}{W^d}  \max_{[a],[b],[b']}\max_{\sig_1,\sig_2} \left|\Theta_{t_0,[a][b]}^{(\sig_1,\sig_2)}-\Theta_{t_0,[a][b']}^{(\sig_1,\sig_2)}\right| . 
\end{align*}

Applying the QUE estimates \eqref{eq:diffuExp1} and \eqref{eq:diffuExp2}, we can bound $I_1$ as follows for any small constant $\tau>0$:  
\be\label{eq:I1} I_1 \le \frac{W^\e \eta^2}{(W^d\ell(z)^d\eta)^3 \im m(z)} \lesssim \frac{1}{N^2}\cdot \frac{W^\tau}{N\eta \sqrt{\kappa_{\gamma_k}+\eta}} 
\lesssim \frac{W^{-3\fd/2+\tau}}{N^2} \, ,
\ee
where in the second step we used \eqref{square_root_density} and $W^d\ell(z)^d=N$. To bound $I_2$, we apply the estimate \eqref{prop:BD1} (by choosing the flow as described in \Cref{zztE} and applying this estimate at time $t_0$). 
Using \eqref{prop:BD1} and $W\ge L^{1-d/6+\e_0}$, we can get the following bounds: when $d=2$,
\be\label{eq:diffsig1}
\frac{\eta^2}{W^d}  \max_{[a],[b],[b']}\max_{\sig_1,\sig_2} \left|\Theta_{t_0,[a][b]}^{(\sig_1,\sig_2)}-\Theta_{t_0,[a][b']}^{(\sig_1,\sig_2)}\right| \prec \frac{\eta^2}{W^{d}} \lesssim \frac{W^{-2\e_0+2\fd}}{N^2} \, ;
\ee
when $d=1$, we have that for any constant $D>0$,
\be\label{eq:diffsig2}
\frac{\eta^2}{W^d}  \max_{[a],[b],[b']}\max_{\sig_1,\sig_2} \left|\Theta_{t_0,[a][b]}^{(\sig_1,\sig_2)}-\Theta_{t_0,[a][b']}^{(\sig_1,\sig_2)}\right|  \prec \frac{\eta^2 n}{W} \lesssim \frac{W^{-2\e_0+2\fd}}{N^2} \, .
\ee
Combining \eqref{eq:I1}--\eqref{eq:diffsig2}, we get that $ I_1+I_2 \prec W^{-\fd}/N^2$ since $\tau$ can be arbitrarily small. Together with \eqref{ssfa2} and \eqref{que0}, it gives that 
$$ \E  \bigg|\frac{N}{W^{d}}\sum_{x\in[a]}|\bu_k(x)|^2 -1 \bigg|^2= \E \left|N\cdot \bu_k^*\left(E_{[a]}-N^{-1}\right) \bu_k\right|^2 \prec  W^{-\fd} .
$$
Then, applying Markov's inequality and taking $c$ sufficiently small depending on $\fd$, we conclude \eqref{Meq:QUE}. The proof of \eqref{Meq:QUE2} follows similarly---we simply replace $E_{[a]}$ in \eqref{ssfa2} with \smash{$|A|^{-1}\sum_{[a]\in A} E_{[a]}$}, after which all subsequent arguments remain valid. 

The proofs of the estimates \eqref{Meq:QUE_weak} and \eqref{Meq:QUE2_weak} follow a similar strategy and are nearly identical to those of \cite[Theorem 2.4]{Band1D} and \cite[Theorem 2.2]{RBSO1D}. Therefore, we omit the details. This concludes the proof of \Cref{thm:QUE}.

%\end{proof}

\subsection{Strategy for the proofs of \Cref{ML:GLoop,ML:GLoop_expec,ML:GtLocal}}\label{subsec:mainproofs}

We now outline the strategy for the proofs of Theorems \ref{ML:GLoop}--\ref{ML:GtLocal}. First, these theorems have already been established in \cite{Band1D,Band2D} uniformly for all flow parameters in the bulk regime, i.e., for $\sE\in [-2+c,2-c]$, where $c>0$ is an arbitrarily small constant. Hence, our focus will be on the edge regime, where $|\sE-2|\ll 1$. Moreover, the proof for the case $\sE\notin [-2,2]$ proceeds in exactly the same way as for $\sE\in [-2,2]$, with the only difference being that we invoke \eqref{eq:small_para2} in place of \eqref{eq:small_para} at various points. Finally, by symmetry, all results established for $\sE\ge 0$ also hold for $\sE\le 0$. Thus, for clarity of presentation and without loss of generality, we adopt the following simplifying assumptions throughout the proof.

%Thus, in the following, we adopt the assumptions in \Cref{simple_assm} again to simplify the presentation.

\begin{assumption}\label{simple_assm}
We assume that $\sE\ge 0$ and $\kappa=|\sE-2|\le c_\sE$ for a small constant $ c_\sE \in (0,10^{-2})$. 
Additionally, we assume that the target spectral parameter $z=E+\ii\eta$ lies inside the support $[-2,2]$ of the semicircle law, so that the estimate in \eqref{eq:E0bulk} corresponding to the $E\in [-2,2]$ case holds.
\end{assumption}

At $t=0$, we have $G_{0}(\sigma)=M(\sigma)$ for $\sig\in\{+,-\}$. Together with Definitions \ref{Def:G_loop} and \ref{Def_Ktza}, it implies that for any fixed $\fn\in \N$:   
$${\cL}^{(\fn)}_{0, \boldsymbol{\sigma},\ba}= {\cK}^{(\fn)}_{0, \boldsymbol{\sigma},\ba} = \cM^{(\fn)}_{ \boldsymbol{\sigma},\ba},\quad \forall \; \boldsymbol{\sigma}\in \{+,-\}^\fn,\ \ \ba\in (\Zn)^\fn. 
$$
To extend the proof to $t\in [0,t_0]$ and to the edge regime, where $\sE \in [2 - c_{\sE}, 2]$ (recall that $c_{\sE}$ is the constant in \Cref{simple_assm}), we establish a key theorem, \Cref{lem:main_ind}. This theorem allows us to extend the estimates \eqref{Eq:L-KGt},  \eqref{Eq:Gtlp_exp}, \eqref{Eq:Gdecay}, \eqref{Gt_bound}, and \eqref{Gt_avgbound} inductively along the flow to all $t\in[0,t_0]$, and from the bulk regime to the edge. We first define a lattice for the flow parameter $\sE$: given $\sE_0\in [2-c_\sE,2]$, we set
\be\label{eq:lattice}
\sE_{k}:=\sE_0 -  c_\sE k /N^{10},\quad k=0,1,\ldots, N^{10} \, .
\ee
We will analyze the loop hierarchy \eqref{eq:mainStoflow} along the flows defined by these discretized values of the flow parameter. 
In the proof, if a function (e.g., $\kappa$, $z_t$, $E_t$, $\eta_t$, $\omega_t$, $G_t$, $m$, $M$, \smash{$\hell_t$}, $\cal L$, and $\cK$) contains an argument $\sE$, it is understood to be defined along the flow with flow parameter $\sE$. At times, to simplify the notation, we will also place $\sE$ as a superscript or subscript.
%Sometimes, we will also put  $\sE_k$ in the super or subscripts. 

\begin{theorem}\label{lem:main_ind}
Suppose the assumptions of \Cref{ML:GLoop} hold.
In addition, suppose the simplified assumptions in \Cref{simple_assm} hold. 
Assume in addition that the estimates \eqref{Eq:L-KGt}, \eqref{Eq:Gtlp_exp}, \eqref{Eq:Gdecay}, \eqref{Gt_bound}, and \eqref{Gt_avgbound} hold at some fixed $s\in [0,t_0]$, uniformly for all $\sE\in[0,2-c_\sE] \cup\{\sE_k:k=0,1,\ldots, N^{10}\}$, that is:
\begin{itemize}
\item[(a)] {\bf $G$-loop estimate}: For each fixed $\fn\in \N$, we have
\begin{align}\label{Eq:L-KGt+IND}
 \max_{\boldsymbol{\sigma}, \ba}\left|{\cal L}^{(\fn)}_{s, \boldsymbol{\sigma}, \ba}(\sE)-{\cal K}^{(\fn)}_{s, \boldsymbol{\sigma}, \ba}(\sE)\right|&\prec \big(W^d\ell_{s}^d\eta_{s}(\sE)\big)^{-\fn}.
\end{align}

\item[(b)] {\bf 2-$G$ loop estimate}: 
For $\boldsymbol{\sigma}\in\{(+,-),(-,+)\}$ and $ \ba=([a_1],[a_2]),$ and for any large constant $D>0$,
\be\label{Eq:Gdecay+IND}
\qquad \left| {\cal L}^{(2)}_{s, \boldsymbol{\sigma}, \ba}(\sE)-{\cal K}^{(2)}_{s, \boldsymbol{\sigma}, \ba}(\sE)\right|\prec \frac{e^{- (\left|[a_1]-[a_2]\right|/\ell_s)^{1/2}}}{\p{W^d\ell_s^d\eta_s(\sE)}^{2}}+W^{-D} . 
\ee
\item[(c)] {\bf Local law}: The following entrywise and averaged local laws hold:
\begin{align} \label{Gt_bound+IND}
 \|G_{s,\sE}-M(\sE)\|_{\max} &\prec \Psi_s(\sE)\, ,\\
  \max_{[a]}\left| \avg{\p{G_{s,\sE}-M(\sE)}E_{[a]}}\right| &\prec \Psi_s(\sE)^2/\omega_s(\sE)\, .\label{Gt_avgbound+IND}
\end{align}

\item[(d)] {\bf Expected $2$-$G$ loop estimate}:
 \be \label{Eq:Gtlp_exp+IND}
\max_{\boldsymbol{\sigma}, \ba}\left|\mathbb E{\cal L}^{(2)}_{s, \boldsymbol{\sigma}, \ba}(\sE)-{\cal K}^{(2)}_{s, \boldsymbol{\sigma}, \ba}(\sE)\right|\prec \kappa(\sE)^{-1/2}\big({W^d\ell_s^d\eta_{s}(\sE)}\big)^{-3} \, .
\ee
\end{itemize}
Then, for any $t\in [s,t_0]$ satisfying that
\begin{equation}\label{con_st_ind}
%\p{W^d\ell_t^d\eta_{t}(\sE_0)\sqrt{\kappa(\sE_0)}}
W^{-\frac{\fd\wedge\fc}{100}} \le  \frac{1-t}{1-s} < 1, 
\end{equation}
the estimates \eqref{Eq:L-KGt}, \eqref{Eq:Gtlp_exp}, \eqref{Eq:Gdecay}, \eqref{Gt_bound}, and \eqref{Gt_avgbound} hold uniformly for all $\sE=\sE_k$ with $k \in \qqq{0, N^{10}}$. 
%In addition, if we do not assume \eqref{Eq:Gtlp_exp+IND}, then the estimates \eqref{Eq:L-KGt}, \eqref{Eq:L-KGtIm}, \eqref{Eq:Gdecay}, \eqref{Gt_bound}, and \eqref{Gt_avgbound} still hold at $t$ under \eqref{con_st_ind}.
\end{theorem}

With \Cref{lem:main_ind}, we can obtain Theorems \ref{ML:GLoop}, \ref{ML:GLoop_expec}, and \ref{ML:GtLocal} easily by induction in $t$. 
The proof of \Cref{lem:main_ind} will be divided into six steps, where the details for each step will be provided in Section \ref{Sec:Stoflo}.
We remark that the following estimates, established at each step, hold uniformly in $u\in[s,t]$ and for $\sE=\sE_k$ with $k \in \qqq{0, N^{10}}$ (recall \Cref{stoch_domination}). 
This uniformity follows from a standard $\e$-net argument combined with a {\bf deterministic} perturbation technique (c.f.~the proof of \Cref{claim:uniform} below). 
For simplicity of presentation, we will not always state this uniformity explicitly in the statements and proofs that follow.
%at every step of the proof. 

\medskip 
\noindent 
\textbf{Step 1} (A priori $G$-loop bound): We will show that $\fn$-$G$ loops ($\fn\ge 2$) satisfy the a priori bound 
 \begin{equation}\label{lRB1}
{\cal L}^{(\fn)}_{u,\boldsymbol{\sigma}, \ba}(\sE) \prec \bigg(\frac{\ell_u^d}{\ell_s^d}\bigg)^{\fn-1} 
  \frac{\sqrt{\kappa(\sE)}} {\p{W^d\ell_u^d\eta_{u}(\sE)}^{\fn-1}},\quad  \forall s\le u\le t.
\end{equation}
% if $t\le 1 - \sqrt{\kappa}$, then for each $\bchi\in \{+,-,\im\}^\fn$, we have
%  \begin{equation}\label{lRB1gen}
% {\cal L}^{(\fn)}_{u,\boldsymbol{\sigma}, \ba}\prec \bigg(\frac{\ell_u^d}{\ell_s^d}\bigg)^{(\fn-1)} 
%   \frac{\omega_u} {\p{W^d\ell_u^d|1-u|\omega_u}^{\fn-1}} \left( \frac{\sqrt{\kappa}}{1-u}\right)^{|v_\bchi|} ,\quad  \forall s\le u\le t.
% \end{equation}
Furthermore, the following entrywise local law holds: 
\begin{equation}\label{Gtmwc}
\|G_{u,\sE}-M(\sE)\|_{\max}\prec  \big({{\ell_u^d}/{\ell_s^d}}\big)\cdot \Psi_u(\sE),\quad \forall s\le u\le t \, .
\end{equation}

\smallskip
\noindent 
\textbf{Step 2} (Sharp local law and a priori $2$-$G$ loop estimate): 
The following sharp local laws hold for all $u\in[s,t]$: 
\begin{align}\label{Gt_bound_flow}
     \|G_{u,\sE}-M(\sE)\|_{\max}&\prec \Psi_u(\sE) \, ,\\
     \max_{[a]}\left| \avg{\p{G_{u,\sE}-M(\sE)}E_{[a]}}\right| &\prec \Psi_u(\sE)^2/\omega_u(\sE) \, .\label{Gt_avgbound_flow}
\end{align} 
Hence, the local laws \eqref{Gt_bound} and  \eqref{Gt_avgbound} hold at time $t$. Note that \eqref{Gt_avgbound_flow} directly implies the $G$-loop estimate \eqref{Eq:L-KGt} when $\fn=1$:
\begin{align}\label{Gt_bound_loop1}
\max_{\sig\in\{+,-\}}\max_{[a]}\left|{\cL}^{(1)}_{u, {\sigma}, [a]}-{\cal K}^{(1)}_{u, {\sigma}, [a]}\right|\prec \big(W^d\ell_u^d\eta_u\big)^{-1} \, .
\end{align}
Furthermore, for $\boldsymbol{\sigma}\in\{(+,-),(-,+)\}$ and $ \ba=([a_1],[a_2]),$ and for any large constant $D>0$, the following estimate holds uniformly for all $u\in[s,t]$: 
\begin{equation}\label{Eq:Gdecay_w}
\left| {\cal L}^{(2)}_{u, \boldsymbol{\sigma}, \ba}(\sE)-{\cal K}^{(2)}_{u, \boldsymbol{\sigma}, \ba}(\sE)\right| \prec \left(\frac{1-s}{1-u}\right)^4\frac{e^{- \left(\left|[a_1]-[a_2]\right|/{\ell_u}\right)^{1/2}}}{(W^d\ell_u^d\eta_{u}(\sE))^{2}}+W^{-D} \, .
\end{equation}

   \medskip
 \noindent 
\textbf{Step 3} (Sharp $G$-loop bound): 
 The following bound on $\fn$-$G$ loops holds for each fixed $\fn\ge 2$: 
\begin{equation}\label{Eq:LGxb}
\max_{\bsig\in\{+,-\}^\fn}\max_{\ba}\left| {\cal L}^{(\fn)}_{u, \boldsymbol{\sigma}, \ba}(\sE) \right|
\prec 
  \frac{\sqrt{\kappa(\sE)}}{(W^d\ell_u^d\eta_{u}(\sE))^{\fn-1}} ,\quad \forall s\le u\le t .
\end{equation}

\medskip
 \noindent 
\textbf{Step 4}  (Sharp $(\cL-\cK)$-loop estimate): 
The following estimate on $ ({\cal L}-{\cal K})_u$ holds for each fixed $\fn\ge 2$: 
\begin{align}\label{Eq:L-KGt-flow}
 \max_{\bsig\in\{+,-\}^\fn}\max_{\ba}\left|{\cal L}^{(\fn)}_{u, \boldsymbol{\sigma}, \ba}(\sE)-{\cal K}^{(\fn)}_{u, \boldsymbol{\sigma}, \ba}(\sE)\right|&\prec \big(W^d\ell_u^d\eta_{u}(\sE)\big)^{-\fn},\quad \forall u\in [s,t]\, .
\end{align}
Hence, the $G$-loop estimate \eqref{Eq:L-KGt} holds at time $t$.

\medskip
 \noindent 
\textbf{Step 5}  (Sharp 2-$G$ loop estimate): For $\boldsymbol{\sigma}\in\{(+,-),(-,+)\}$ and $ \ba=([a_1],[a_2]),$ the following estimate holds for any large constant $D>0$:
\begin{equation}\label{Eq:Gdecay_flow}
\left| {\cal L}^{(2)}_{u, \boldsymbol{\sigma}, \ba}(\sE)-{\cal K}^{(2)}_{u, \boldsymbol{\sigma}, \ba}(\sE)\right| \prec   \frac{e^{- \left(\left|[a_1]-[a_2]\right|/{\ell_u}\right)^{1/2}}}{\p{W^d\ell_u^d\eta_{u}(\sE)}^{2}}+W^{-D}. 
\end{equation}
It shows that the estimate \eqref{Eq:Gdecay} holds at time $t$.

\medskip
\noindent 
\textbf{Step 6} (Expected 2-$G$ loop estimate): 
The following estimate holds: 
\begin{equation}\label{Eq:Gtlp_exp_flow}
\max_{\boldsymbol{\sigma}, \ba}\left|\mathbb E{\cal L}^{(2)}_{u, \boldsymbol{\sigma}, \ba}(\sE)-{\cal K}^{(2)}_{u, \boldsymbol{\sigma}, \ba}(\sE)\right|\prec \kappa\p{\sE}^{-1/2} \big(W^d\ell_u^d\eta_{u}\p{\sE}\big)^{-3},\quad  
 \forall s \le u \le t .
\end{equation}
Hence, the estimate \eqref{Eq:Gtlp_exp} holds at time $t$. We remark that in Steps 1--5, the induction hypothesis \eqref{Eq:Gtlp_exp+IND} will not be used, i.e., the estimates \eqref{Eq:L-KGt}, \eqref{Eq:Gdecay},  \eqref{Gt_bound}, and \eqref{Gt_avgbound} still hold at $t$ without assuming  \eqref{Eq:Gtlp_exp+IND}.

For the remainder of the paper, we focus on the proof of \Cref{lem:main_ind}. Compared to the proofs in the bulk regime \cite{Band1D,Band2D}, the primary challenges and distinctions arise in Steps 1 and 3. In Step 1, we develop a new approach to establish continuity estimates for $G$-loops and address the key difficulty posed by the instability of the self-consistent equations near the spectral edge, as mentioned in \Cref{sec:idea}. This step is the main focus of \Cref{Sec:Step1}. In Step 3, we deal with complications related to same-colored propagators \smash{$\Theta_t^{(\sig,\sig)}$} and pure $G$-loops, which will be discussed in detail in \Cref{sec:inductive_step}.

%For the rest of the paper, we focus on the proof of \Cref{lem:main_ind}. Compared with the proofs within the bulk regimes \cite{Band1D,Band2D}, the main difficulties and differences lie in Steps 1 and 3. In Step 1, we need to develop a novel argument to establish the continuity estimates for $G$-loops and to handle the major issue associated with the instability of self-consistent equations near the spectral edges as discussed in \Cref{sec:idea}. This will be main focus of \Cref{Sec:Step1} below. In Step 3, we need to handle issues regarding same-colored propagators pure $G$-loops, which will be discussed in \Cref{sec_sumzero}. 

\subsection{Resolvent entry estimates}

Our proof of \Cref{lem:main_ind} depends crucially on the following two lemmas, which provide estimates on the resolvent entries via bounds on 2-$G$ loops. Specifically, \Cref{lem G<T} establishes bounds on both the entrywise and averaged differences between $G$ and $M$ in the $\max$-norm sense, while \Cref{lem_GbEXP} provides a finer estimate on the off-diagonal resolvent entries, which allows us to derive the decay of the off-diagonal resolvent entries from the decay of the 2-$G$ loops. We denote $M_t(z):=m_t(z)I_N$, where $m_t(z)$ is defined in \eqref{self_mt}, and recall that $G_t(z)$ is defined in \eqref{self_Gt}. 
% We define the $T$-variables 
% $$T_{xy}(t,z):=\sum_{\al}S_{x\al}|(G_{t})_{\al y}(z)|^2,\quad \wt T_{xy}(t,z):=\sum_{\al}|(G_{t})_{x\al}(z)|^2S_{\al y}.$$ 
% Moreover, we denote the $\fn$-$G$ loops formed by $G_t(z)$ and $G_t(\overline z)$ as ${\cal L}^{(\fn)}_{t, \boldsymbol{\sigma}, \ba}(z)$. 

\begin{lemma}\label{lem G<T2}
In the setting of \Cref{thm_locallaw}, fix an arbitrary $z=E+\ii \eta\in \mathbf D_{C_0,\fd}\cup \mathbf D^{\mathrm{out}}_{c_0,\fd}$ and a small constant $\e_0>0$. Let $W^{-d/2}\le \phi_t \le W^{-\e_0}$ be a deterministic control parameter. For any $t\in [0,1]$, under the assumptions
\be\label{initialGT2} 
\|G_{t}(z) - M_t(z) \|_{\max}\prec W^{-\e_0}|1-t(m_t(z))^2|,\ \max_{[a],[b]\in \Zn} \cL^{(2)}_{t,(-,+),([a],[b])}(z) \prec \phi_t^2 \, , 
\ee
the following entrywise and averaged local laws hold:
\begin{align}
\|G_{t}(z)-M_t(z)\|_{\max} &\prec \phi_t + {\phi_t^2}/{|1-t(m_t(z))^2|}  \, ,\label{GiiGEX2}  \\
\max_{[a]} \left|\left\langle  \left(G_{t}(z)-M_t(z)\right)E_{[a]}\right\rangle \right| &\prec {\phi_t^2}/{|1-t(m_t(z))^2|} \, . \label{GavLGEX}
\end{align}
Note that when $z=z_t(\sE)$, we have \[|1-t[m_t(z_t(\sE))]^2|=|1-t[m(\sE)]^2|\asymp (1-t)+\sqrt{\kappa(\sE)}=\omega_t(\sE).\] 
% \begin{align}
% \|G_{t}(z)-M_t(z)\|_{\max} &\prec \Phi_t + \Phi_t^2/|1-t(m_t(z))^2| \, ,\label{GiiGEX2}  \\
% \max_{[a]} \left|\left\langle  \left(G_{t}(z)-M_t(z)\right)E_{[a]}\right\rangle \right| &\prec \Phi_t\wedge \left( \frac{\Phi_t^2}{|1-t(m_t(z))^2|} + \frac{\Phi_t^4}{|1-t(m_t(z))^2|^3}\right)\, . \label{GavLGEX}
% \end{align}
% \begin{comment}
% More generally, the above estimates hold for $z\in \mathbf D_{u,\sE}:=\{z\in \C_+: |z-z_u(\sE)|\le \eta_u(\sE)/\log W\}$. Specifically, let $M_t(z):=m_t(z)I_N$ with $m_t(z)$ defined in \eqref{self_mt} and recall $G_t(z)$ defined in \eqref{self_Gt}. 
% Under the assumptions 
% \be\label{initialGT3} 
% \|G_{t}(z) - M_t(z) \|_{\max}\prec W^{-\e_0},\ \max_{[a],[b]\in \Zn} \avg{G_t(z)E_{[a]}G(\overline z)E_{[b]}} \prec \Phi_t^2 \, , 
% \ee
% the following estimates hold for any $z\in \mathbf D_{u,\sE}$:
% \begin{align}
% \|G_{t}(z)-M_t(z)\|_{\max} &\prec \Phi_t + \Phi_t^2/\omega_{t}(\sE) \, ,\label{GiiGEX2-add}  \\
% \max_{[a]} \left|\left\langle  \left(G_{t}(z)-M_t(z)\right)E_{[a]}\right\rangle \right| &\prec \Phi_t\wedge \left( \frac{\Phi_t^2}{\omega_{t}(\sE)} + \frac{\Phi_t^4}{\omega_t(\sE)^3}\right)\, . \label{GavLGEX-add}
% \end{align}
%\end{comment}
\end{lemma}
%\subsection{Proof of \Cref{lem G<T2}}\label{sec:pf_G<T2}
\begin{proof}
The proofs of the local laws \eqref{GiiGEX2} and \eqref{GavLGEX} for random band matrices in the literature typically rely on standard resolvent identities, combined with large deviation estimates for linear and quadratic forms of high-dimensional random vectors. This approach is demonstrated in \cite[Sections 5 and 6]{Semicircle} and \cite[Section 4]{Band1D}.
In addition, the proof of the averaged local law \eqref{GavLGEX} requires a fluctuation averaging mechanism; see, for example, \cite[Proposition 3.3]{EKY_Average} and \cite[Theorems 4.6 and 4.7]{Semicircle}.
For the reader’s convenience---and more importantly, to inform our later discussion of the instability issue associated with the self-consistent equations---we now provide further details.

Given the bound on 2-$G$ loops in \eqref{initialGT2}, the off-diagonal entries of $G_t$ can be readily bounded using the techniques in \cite{Semicircle}. In particular, the following estimate was established in equation (4.11) of \cite{Band1D}, on the event $\Omega_{t,\e}(z):=\{\|G_{t}(z)- M_t(z)\|_{\max}\le W^{-\e}\}$, where $\e$ is an arbitrary constant:
\begin{align}
\mathbf 1\p{\Omega_{t,\e}(z)}\cdot \max_{x\ne y}\left|(G_{t}(z))_{xy}\right|^2 \prec \max_{\ba}\big|{\cal L}^{(2)}_{t, (+,-), \ba}(z)\big| + W^{-d}\prec \phi_t^2 \, . \label{GijGEXoff}
\end{align}
Next, using \eqref{GijGEXoff} and the techniques in \cite{Semicircle}, we can show that the diagonal entries of $G_t$ satisfy the following system of self-consistent equations on $\Omega_{t,\e}(z)$:
\be\label{eq:self0}
\frac{1}{(G_t)_{xx}}=- z - t \sum_y S_{xy}(G_t)_{yy}+\OO_\prec(\phi_t)\, .
\ee
Subtracting this system from the self-consistent equation for $m_t\equiv m_t(z)$ yields:
\be\label{eq:self}
\mathbf 1(\Omega_{t,\e}(z)) \sum_y \left(1- t m_t^2 S\right)_{xy}[(G_t)_{yy}-m_t] \prec \phi_t + \Lambda_t(z)^2 \, ,
\ee
where \smash{$\Lambda_t\equiv \Lambda_t(z):=\max_{x}|(G_t(z))_{xx}-m_t(z)|$}. Using a similar estimate to that in \eqref{prop:ThfadC_short}, we find  
\be\label{eq:inftoinf}\big\|\big(1-tm_t^2S\big)^{-1}\big\|_{\infty\to\infty} =\big\|\big(1-tm_t^2S^{\LK}\big)^{-1}\big\|_{\infty\to\infty} \prec |1-tm_t^2|^{-1}.\ee 
Thus, by solving the equation \eqref{eq:self}, we obtain that 
\be\label{eq:self1}
\mathbf 1\p{\Omega_{t,\e}(z)}\cdot  \Lambda_t \prec \phi_t/|1-t(m_t(z))^2| + \Lambda_t^2/|1-t(m_t(z))^2|. \ee
Incorporating the first assumption in \eqref{initialGT2}, we derive from equation \eqref{eq:self1} that 
\be\label{eq:diagonal_bad} 
\Lambda_t \prec \phi_t/|1-t(m_t(z))^2| + W^{-\e_0}\Lambda_t \implies \Lambda_t \prec \phi_t/|1-t(m_t(z))^2| .\ee
(Note the first assumption in \eqref{initialGT2} is not used in the derivation leading to \eqref{eq:self1}.)

To improve the bound \eqref{eq:diagonal_bad} on the diagonal resolvent entries, we employ the fluctuation averaging estimate from \cite[Proposition 3.3]{EKY_Average}. This allows us to derive the following improved estimate for the \emph{average} of equation \eqref{eq:self}:
\be\label{eq:self_improve}
\sum_y \left(1- t m_t^2 S^{\LK}\right)_{[a][b]}\avg{(G_t-M_t)E_{[b]}} \prec \phi_t^2 + \Lambda_t(z)\theta_t(z), \quad \forall [a]\in \Zn,
\ee
where \smash{$\theta_t(z):=\max_{[a]}|\langle(G_t-M_t)E_{[a]}\rangle|$}. Solving equation \eqref{eq:self_improve}, and using  \eqref{eq:inftoinf} and the first assumption in \eqref{initialGT2}, we derive that 
\(\theta_t(z) \prec \phi_t^2/|1-t(m_t(z))^2| ,\) which concludes the proof of \eqref{GavLGEX}. Finally, plugging \eqref{GavLGEX} into \eqref{eq:self0} leads to the bound \(\Lambda_t(z) \prec \phi_t+\phi_t^2/|1-t(m_t(z))^2|\) for the diagonal entries. Combined with \eqref{GijGEXoff}, this completes the proof of \eqref{GiiGEX2}.   
%This lemma can be proved by utilizing a fluctuation averaging mechanism in the proof of the averaged local law (see, e.g., \cite[Proposition 3.3]{EKY_Average} and \cite[Theorems 4.6 and 4.7]{Semicircle}). As the argument is standard given the assumptions \eqref{initialGT2} and the estimates \eqref{eq:TPsi} (which replaces the role of  Ward's identity in the original argument) and \eqref{GijGEXoff} established in \Cref{lem G<T}, we omit the details. 
\end{proof}

%Our proof relies crucially on the following lemma, which bounds the entries of the Green's function through 2-$G$ loops. This can be viewed as a generalization of the estimate \eqref{GiiGEX2}, which also allows us to derive a decay of the Green's function entries from the decay of 2-$G$ loops.

\begin{lemma}\label{lem_GbEXP}
In the setting of \Cref{lem G<T2}, suppose \eqref{initialGT2} holds. Let $0<\Phi_t([a],[b])\le W^{-\e_0}$ be a set of deterministic control parameters such that 
\be\label{eq:def_Phit} 
{\cal L}^{(2)}_{t, (-,+), ([a], [b])}(z)\prec \Phi_t([a],[b]),\quad \forall [a],[b]\in \Zn \, .\ee
For all $[a]\ne [b]\in \Zn$, the following estimate holds for any large constant $D>0$:
\begin{align}\label{GijGEX}
\max_{x\in [a], y \in [b]} |(G_t(z))_{xy}|^{2}  \prec & \sum_{\substack{|[a']-[a]|\le 1 ,\\ |[b']-[b]|\le 1}}  \Phi_t([a'],[b'])+ W^{-d}\mathbf 1_{|[a]-[b]|\le 1} +W^{-D} \, .
\end{align}
\end{lemma}
\begin{proof}
    It is proven as Lemma 4.1 in \cite{Band1D} for 1D random band matrices using resolvent identities; however, the same arguments apply in 2D. %Additionally, a different approach based on Gaussian integration by parts, as discussed in \cite[Section 8]{RBSO1D}, would also work.
\end{proof}

\section{Continuity estimates and local law}\label{Sec:Step1}

%This section is devoted to the completion of Step 1 in the proof of \Cref{lem:main_ind}, as well as the proofs of \Cref{thm_locallaw_out} and \Cref{lem G<T2}. Our proofs rely on a new continuity argument for $G$-loops and a flow argument for establishing the local laws, which we detail below.

This section is devoted to completing Step 1 of the proof of \Cref{lem:main_ind}, as well as to proving \Cref{thm_locallaw_out}. Our approach is based on a novel continuity argument for $G$-loops, along with a flow-based method for establishing the local laws, both of which will be described in detail below.

\subsection{Step 1 in the proof of \Cref{lem:main_ind}}
First, the continuity estimate \eqref{lRB1} for $G$-loops follows directly from the following lemma, whose proof will be given in \Cref{subsec_pf_lem_ConArg}.

%estimates \eqref{lRB1} and \eqref{Gtmwc} under the assumption \eqref{Eq:L-KGt+IND} rely on the following key lemma, which provides the desired continuity estimate \eqref{lRB1} for $G$-loops. The proof of \Cref{lem_ConArg} will be postponed to \Cref{subsec_pf_lem_ConArg} below. 

\begin{lemma}[Continuity estimate for $G$-loops]\label{lem_ConArg}
Fix any \( c\le s \leq t< 1\) for a constant \(c > 0\). In the flow framework of \Cref{zztE}, assume that the following two bounds hold at time \(s\) for each fixed \(\fn \ge 2\), uniformly for all $\sE\in[-2+c_\sE,2-c_\sE] \cup\{\sE_k:k=0,1,\ldots, N^{10}\}$: 
\begin{equation}\label{55}
%\max_{[a]} \left|\avg{\im \p{G_{s,\sE}-M(\sE)}E_{[a]}} \right|\prec W^{-\e_0}\sqrt{\kappa(\sE)},
\left\|G_{s,\sE}-M(\sE)\right\|_{\max}\prec \Psi_s(\sE), \quad 
\max_{\bsig,\ba}\left|{\cal L}^{(\fn)}_{s, \boldsymbol{\sigma}, \ba}(\sE)\right|\prec  \frac{\sqrt{\kappa(\sE)}}{\left( W^d \ell_{s}^d \eta_{s}(\sE) \right)^{\fn-1} } \, .
\end{equation}
%where $\e_0$ is a small constant. 
% and if $t_1$ satisfies $1-t_1\ge \sqrt{\kappa}$, then for each $\bchi\in \{+,-,\im\}^\fn$, 
% \begin{equation}\label{56}
% \max_{\ba} {\cal L}^{(\fn)}_{t_1, \boldsymbol{\sigma}, \ba} \prec \frac{\omega_{t_1}}{\left( W^d \ell_{t_1}^d |1-t_1| \omega_{t_1} \right)^{\fn-1} } \left( \frac{\sqrt{\kappa}}{1-t_1}\right)^{|v_\bchi|}  \, . 
% \end{equation}
%For each $k\in \qqq{0,N^{10}}$, let $a_N>0$ be a deterministic parameter satisfying that $C^{-1}\sqrt{\kappa(\sE_k)}\le a_N \le C$ for a constant $C>0$. 
% Then, on the event 
% \be\label{eq:omegak}
% \Omega_k= \left\{\max_{[a]} \left|\avg{\p{\im G_{t,\sE_k}}E_{[a]}}\right|\leq a_N\right\},
% \ee 
Then, for any $k\in \qqq{0,N^{10}}$ and fixed $\fn\ge 2$, the following estimate holds uniformly for all $u\in [s,t]$: 
\begin{equation}\label{res_lo_bo_eta}
    %{\bf 1}(\Omega_k) \cdot
    \max_{\bsig,\ba} \left|{\cal L}_{u, \boldsymbol{\sigma}, \ba}^{(\fn)}(\sE_k)\right| %\prec  \frac{\max_{[a]} \left|\sqrt{\kappa(\sE_k)}+\avg{\p{\im G_{t,\sE_k}}E_{[a]}}\right|}{\left( W^d \ell_{s}^d \eta_{t}(\sE_k) \right)^{\fn-1} }
    = \bigg(\frac{\ell_{u}^d}{\ell_{s}^d}\bigg)^{\fn-1}  \frac{\sqrt{\kappa(\sE_k)}+\max_{[a]} \left|\avg{\p{\im G_{u,\sE_k}}E_{[a]}}\right|}{\left( W^d \ell_{u}^d \eta_{u}(\sE_k) \right)^{\fn-1} }\, . 
\end{equation}
\iffalse
Furthermore, the following rough bound on $\avg{\p{\im G_{u,\sE_k}}E_{[a]}}$ holds uniformly for all $u\in [s,t]$:
\begin{equation}\label{res_lo_bo_etagen}
\max_{[a]} \left|\avg{\p{\im G_{u,\sE_k}}E_{[a]}}\right| \prec \frac{1-s}{1-u}\cdot \frac{\br{\omega_{u}(\sE_k)}^2}{ \sqrt{\kappa(\sE_k)}}    \, .
\end{equation}
\fi
% \begin{equation}\label{res_lo_bo_etagen}
%     {\bf 1}(\Omega_k) \cdot \max_{\bsig,\ba} \left|{\cal L}_{t, \boldsymbol{\sigma}, \ba}^{(2)}(\sE_k)\right| \prec  \frac{ (1-s)^3\ell_{t}^d}{(1-t)^3\ell_{s}^d}  \frac{\br{\omega_{t}(\sE_k)}^2}{ W^d \ell_{t}^d \eta_t(\sE_k)\sqrt{\kappa(\sE_k)}}   \, .
% \end{equation}
\end{lemma}

%Given \Cref{lem_ConArg,lem G<T}, the proof of \eqref{lRB1} and \eqref{Gtmwc} is similar to the argument in Section 5.1 of \cite{Band1D}. 
%Second, to establish the local law \eqref{Gtmwc} using our \Cref{lem G<T2}, we need to verify the assumptions in \eqref{initialGT2}. This requires to establish a weak entrywise local law first as in the following lemma.  

Second, to establish the local law \eqref{Gtmwc} using \Cref{lem G<T2}, we must first verify the assumptions in \eqref{initialGT2}. This, in turn, requires proving a weak entrywise local law, as stated in the following lemma. 

\begin{lemma}\label{lem G<T}
In the setting of \Cref{lem:main_ind}, fix any $0\le s \le t \le t_0$ such that \eqref{con_st_ind} holds. Suppose the estimates in \eqref{55} hold at time $s$, and the estimate \eqref{res_lo_bo_eta} holds uniformly for all $u\in[s,t]$ for $\fn=2$ and an arbitrary $k\in\qqq{0,N^{10}}$. Then, the following local law holds uniformly for all $u\in[s,t]$:
\begin{align}
%\mathbf 1(\Omega_{u,\e_0})\cdot
\|G_{u,\sE_k}-M(\sE_k)\|_{\max} \prec \frac{1-s}{1-t}\Phi_u \, ,\label{GiiGEX}
%\\ \mathbf 1(\Omega_{u,\e_0})\cdot
%\max_{[a]} \left|\avg{\p{G_{u,\sE_k}-M(\sE_k)}E_{[a]}} \right|\prec \frac{1-s}{1-u}\Phi_u, \label{GavGEX} 
\end{align}
where we abbreviate that
\[\Phi_u^2:=\frac{\ell_u^d}{\ell_s^d}\frac{\sqrt{\kappa(\sE_k)}}{W^d\ell_u^d\eta_u(\sE_k)}.\]
%\[\Phi_u^2:=\frac{(1-s)\ell_u^d}{(1-u)\ell_s^d}\frac{\br{\omega_{u}(\sE_k)}^2}{W^d\ell_u^d\eta_u(\sE_k)\sqrt{\kappa(\sE_k)}}.\]
% fix an arbitrary $z=E+\ii \eta\in \mathbf D_{C_0,\fd}$ and a small constant $\e_0>0$. 
% Let $W^{-d/2}\le \Phi_t \le W^{-\e_0}$ be a deterministic control parameter.  
% For a fixed $t\in [0,1]$ and an arbitrarily small constant $\e\in (0,\e_0)$, we define the following event:  
% \be\label{initialGT} 
% \Omega_{t,\e_0,\e}(z):=\Big\{\|G_{t}(z) - M_t(z) \|_{\max}\le W^{-\e_0},\ \max_{[a],[b]\in \Zn} \cL^{(2)}_{t,(+,-),([a],[b])}(z) \le W^{\e}\Phi_t^2\Big\}. 
% \ee 
% Then, on $\Omega_{t,\e_0,\e}(z)$, the following estimates hold for any constants $\e'>\e$ and $D>0$:
% \begin{align}\label{eq:TPsi} 
% \cor \P\p{\|T(t,z)\|_{\max}+\|\wt T(t,z)\|_{\max}\ge W^{\e'}\Phi_t^2 ; \Omega_{t,\e_0,\e}(z)} &\le N^{-D}  \, ,\\
% \P\Big( \max_{x\ne y}\left|(G_{t}(z))_{xy}\right|\ge W^{\e'/2}\Phi_t; \Omega_{t,\e_0,\e}(z)\Big) &\le N^{-D} \, .\label{GijGEXoff}
% \end{align}
\end{lemma}

The proof of \Cref{lem G<T} will be presented in \Cref{subsec_pf_lem_G<T} below, where we address the central challenge posed by the instability of the self-consistent equation \eqref{eq:self0} for the diagonal resolvent entries. Before turning to these proofs, we first apply \Cref{lem_ConArg,lem G<T} to complete the first step in the proof of \Cref{lem:main_ind}.

%where we need to handle the key issue regarding the instability of the self-consistent equation for the diagonal resolvent entries. Before that, we first use \Cref{lem_ConArg,lem G<T} to complete the first step in the proof of \Cref{lem:main_ind}.

\begin{proof}[\bf Step 1: Proof of \eqref{lRB1} and \eqref{Gtmwc}] 
We begin by establishing an upper bound on $\|H\|$ for the matrix $H$ defined in \eqref{eq:WO}. Using the rigidity of eigenvalues of Wigner matrices \cite{ErdYauYin2012Rig} and the rigidity of singular values of sample covariance matrices \cite{isotropic}, we can bound both the operator norms of the diagonal blocks of $H$ (denoted by $H|_{[x][x]}$) and the off-diagonal blocks (denoted by $H|_{[x][y]}$ for $[x] \sim [y]$) as follows:  
$$ \|H|_{[x][x]}\|\le \frac{2}{\sqrt{1+2d\lambda^2}}+\e ,\quad \|H|_{[x][y]}\|\le \frac{2\lambda}{\sqrt{1+2d\lambda^2}}+\e, $$
with high probability for any constant $\e>0$. Combining these bounds, we conclude that $\|H\| \le C_{d,\lambda}$ with high probability for some constant $C_{d,\lambda} > 2$. In particular, it implies that for all $t\in [0,1]$,
\be\label{eq:largest_eig}
\|H_t\|\le C_{d,\lambda}\sqrt{t} \quad \text{with high probability} \, .
\ee
Since $z_0(\sE)\asymp 1$, by \eqref{eq:largest_eig}, there exists a small enough constant $c_0>0$ such that 
\be\label{eq:largest_eig2}
\|H_t\|\le |z_t(\sE)|-c_0 \ \ \ \text{with high probability for all }  t\in[0, c_0] \, .
\ee
Then, we divide the proof into the following three cases. For simplicity of presentation, we will omit the argument $\sE=\sE_k$ from our notations in the proof. 

\medskip

\noindent {\bf Case 1:} 
$t> s \ge c_0$. Note that under \eqref{eq:small_para}, \eqref{eq:small_para2}, and the condition \eqref{con_st_ind}, we have  
\be\label{eq:verifyT1}\frac{1-s}{1-t}\Phi_{u}\le W^{-(\fd\wedge\fc)/5}\sqrt{\kappa(\sE)},\quad \forall u\in [s,t].\ee  
Then, combining \eqref{GiiGEX} and \eqref{res_lo_bo_eta}, we immediately obtain \eqref{lRB1}.
Next, with \eqref{GiiGEX} and \eqref{lRB1} (when $\fn=2$) verifying the assumptions in \eqref{initialGT2} with $\phi_t=\Phi_t$, we conclude \eqref{Gtmwc} with \Cref{lem G<T2}.  

\medskip
\noindent {\bf Case 2:} $0\le s<t \le c_0$. By \eqref{eq:largest_eig2}, we have that $\|G_{u,\sE}\|\lesssim 1$ with high probability, uniformly for all $u \in [s, t]$. This immediately implies that
\begin{equation}\label{lRB2}
{\cal L}^{(\fn)}_{u,\boldsymbol{\sigma}, \ba} \prec W^{-d(\fn-1)},\quad  \forall s\le u\le t,
\end{equation}
and hence \eqref{lRB1} follows. Then, applying \Cref{lem G<T}, \eqref{lRB2}, and \Cref{lem G<T2} gives 
\be\label{eq:Gtmwcstrong}
\|G_u-M\|_{\max}\prec W^{-d/2},\quad  \forall s\le u\le t,
\ee
which concludes \eqref{Gtmwc}. % as in Case 1.

% Using \eqref{lRB2} (in the case of $\fn=2$), and applying \Cref{lem G<T,lem G<T2} and along with a similar continuity argument as explained in Case 1, we can obtain that 
% \begin{equation}\label{Gtmwc2}
% \|G_u -M\|_{\max} \prec W^{-d/2},\quad  \forall s\le u\le t,
% \end{equation}
% which concludes \eqref{Gtmwc}.

\medskip
\noindent {\bf Case 3:} If $0\le s <c_0 < t$, then we can first show that \eqref{lRB2} and \eqref{eq:Gtmwcstrong} hold for $u\in [s,c_0]$ as in Case 2. Then, applying the argument in Case 1 with $s$ replaced by $c_0$, we can further show that \eqref{lRB1} and \eqref{Gtmwc} hold for $u\in [c_0,t]$. 
\end{proof}

\subsection{Proof of \Cref{lem_ConArg}}\label{subsec_pf_lem_ConArg}
There are some key differences between the proof of \Cref{lem_ConArg} and the argument presented in \cite[section 6]{Band1D}. Therefore, we will present the full details here. To prepare for the proof, we first claim that under the assumptions of \Cref{lem_ConArg}, the estimates in \eqref{55} hold uniformly for all $\sE\in [-2+c_{\sE},\sE_0]$. 
\begin{claim}\label{claim:uniform}
For any $\sE\in [-2+c_{\sE},\sE_0]$, let $\sE'$ denote the closest point to $\sE$ in the set $[-2+c_\sE,2-c_\sE] \cup\{\sE_k:k=0,1,\ldots, N^{10}\}$. Fix an arbitrary constant $\e>0$ and $\fn\in 2\N$. On the events 
\[\Omega_1:=\big\{\|G_{s,\sE'}\|_{\max}\le C_0\big\},\quad \Omega_2:=\big\{\max_{[a]} \left|\avg{\p{\im G_{s,\sE'}}E_{[a]}}\right|\le C_0\sqrt{\kappa(\sE')}\big\}\] 
for a constant $C_0>1$, we have 
\be\label{eq:claim_perturb2}\mathbf 
 1(\Omega_1)\cdot \left\|G_{s,\sE}\right\|_{\max}\le C_0+1 , \quad 1(\Omega_2)\cdot \max_{[a]} \left|\avg{\p{\im G_{s,\sE}}E_{[a]}}\right|\le (C_0+1)\sqrt{\kappa(\sE)} \, .\ee
On the event 
\be\label{eq:wtLboundE'}\Omega:= \bigcap_{2\le k \le 2\fn} \left\{  \max_{\bsig,\ba}\left|{\cal L}^{(k)}_{s, \boldsymbol{\sigma}, \ba}(\sE')\right|\le  \frac{W^\e\sqrt{\kappa(\sE')}}{\left( W^d \ell_{s}^d \eta_{s}(\sE') \right)^{k-1} }\right\}\ee
for a constant $C>0$, we have that 
\be\label{eq:claim_perturb}\mathbf 1(\Omega_1\cap \Omega)\cdot \max_{\bsig,\ba}\left|{\cal L}^{(k)}_{s, \boldsymbol{\sigma}, \ba}(\sE)\right|\le  \frac{2W^\e\sqrt{\kappa(\sE)}}{\left( W^d \ell_{s}^d \eta_{s}(\sE)\right)^{k-1}},\quad \forall 2\le k \le \fn .\ee
\end{claim}

Note that this claim is completely \emph{deterministic}; in particular, there is no loss of probability. To show this claim, we use the following resolvent identity: 
\begin{equation}\label{GGZGG0}
  G_{s}(z_s(\sE))= G_{s}(z_s(\sE'))+(z_s(\sE)-z_s(\sE')) G_{s}(z_s(\sE)) G_{s}(z_s(\sE')) \, ,
\end{equation}
where, using the expressions of $z_s(\sE)$ and $z_s(\sE_k)$ in \eqref{eq:zt}, we can verify that 
\begin{align}\label{z-zs}
  |z_s(\sE)-z_s(\sE')| = |(\sE-\sE') + (1-s)(m(\sE)-m(\sE'))| \lesssim \frac{|\sE-\sE'|}{\sqrt{\kappa(\sE)}}\le N^{-9}.
\end{align}
With \eqref{GGZGG0}, \eqref{z-zs}, and the trivial bound $\|G_s\|\le N^{-1}$, we immediately derive \eqref{eq:claim_perturb2}. To establish \eqref{eq:claim_perturb}, we will adopt a similar argument to that used in the proof of \eqref{res_lo_bo_eta}, which we will present below. 
Thus, we will postpone the proof of \eqref{eq:claim_perturb} until we complete the proof of \Cref{lem_ConArg}.

%\begin{proof}[\bf Proof of \Cref{lem_ConArg}]%\label{sec:Inh_LE}
%The proof of \Cref{lem_ConArg} is similar to that of \cite[Lemma 5.1]{Band1D}, but the argument is more intricate at the edge. In particular, we need to divide the proof into two cases according to whether $1-t\ge \sqrt{\kappa}$ or $1-t\le \sqrt{\kappa}$. 
%Using \eqref{GGZGG0}, \eqref{z-zs}, and a similar argument as that for the proof of \eqref{res_lo_bo_eta} as we will show below, we can establish the claim. Since the proof is much simpler than the argument below, we omit the details.  
For the proof of \eqref{res_lo_bo_eta}, we fix a $\sE=\sE_k$ for some $k\in\qqq{0,N^{10}}$ and consider the resolvent $G_{t,\sE}=(H_t-z_t(\sE))^{-1}$. 
Then, we choose 
\be\label{wtsE}\wt \sE= \frac{(1+t)\sqrt{s}}{(1+s)\sqrt{t}}\sE,\ee
and consider the resolvent
\be\label{eq:wtGt1} G_{s,\wt\sE}=\big(H_{s}-z_{s}(\wt\sE)\big)^{-1} \stackrel{d}{=}\sqrt{\frac{t}{s}}\p{H_{t}-\sqrt{\frac{t}{s}} z_{s}(\wt\sE)}^{-1} \, .
%,\quad G_{t}=(H_{t}-z_{t})^{-1} .
\ee
For clarity of presentation, we will use the following notations in the proof:
\be\label{eq:abbrv}
\begin{aligned}
z\equiv z_{t}(\sE),\quad  
&\wt z:= \sqrt{{t}/{s}} \cdot z_{s}(\wt\sE),\quad G\equiv G_{t,\sE},\quad \widetilde{G}:=\p{H_{t}-\wt z}^{-1},\\ &{\cal L}_{t,\bsig,\ba}^{(\fn)}\equiv {\cal L}_{t,\bsig,\ba}^{(\fn)}(\sE), \quad 
{\cal L}_{s,\bsig,\ba}^{(\fn)}\equiv {\cal L}_{s,\bsig,\ba}^{(\fn)}(\wt\sE).
\end{aligned}
\ee
Further, we also define the $\fn$-$G$ loops \(\widetilde{\cal L}_{s,\bsig,\ba}^{(\fn)}\)  with resolvents \smash{\({G}_{s,\wt\sE}(\sig_i)\)} in ${\cal L}_{s,\bsig,\ba}^{(\fn)}(\wt \sE)$ replaced by \smash{\(\widetilde{G}(\sig_i)\)} for $i\in\qqq{\fn}$. Note that by \eqref{eq:wtGt1}, we have 
\begin{equation}\label{calLcalL}
\wt{\cal L}_{s,\bsig,\ba}^{(\fn)} \stackrel{d}{=}\left({s}/{t}\right)^{\fn/2} {\cal L}_{s,\bsig,\ba}^{(\fn)}.  
\end{equation}
Note that \(\wt\sE<\sE\) under the choice \eqref{wtsE}, so ${\cal L}_{s,\bsig,\ba}^{(\fn)}$ satisfies the assumption \eqref{55} by \Cref{claim:uniform}. Consequently, \smash{$\wt{\cal L}_{s,\bsig,\ba}^{(\fn)}$} satisfies the same estimate as that in \eqref{55}. 

Our proof is based on expanding the  $\cL_t$-loops in terms of the \smash{$\wt{\cL}_s$} loops, using a resolvent identity similar to \eqref{GGZGG0}: 
\begin{equation}\label{GGZGG}
  G=\widetilde{G}+(z-\wt {z})  G \widetilde{G} \, .
\end{equation}
Under \eqref{wtsE}, using \eqref{eq:kappat}, we can check that 
\be\label{eq:Imz} 
\begin{aligned}
&\re \wt z= \sqrt{\frac{t}{s}}\frac{1+s}{2}\wt\sE=\frac{1+t}{2}\sE = \re z,\\ 
&\im z=\eta_t(\sE)\asymp (1-t)\sqrt{\kappa(\sE)},\quad \im \wt z \asymp \eta_s(\wt\sE)\asymp (1-s)\sqrt{\kappa(\wt\sE)}.
\end{aligned}\ee
Moreover, we have 
\be\label{eq:kappawtsE} \kappa(\wt\sE) = \kappa(\sE)+(\sE-\wt\sE)= \kappa(\sE)+\frac{(\sqrt{t}-\sqrt{s})(1-\sqrt{ts})}{(1+s)\sqrt{t}}\sE\asymp \kappa(\sE)+(t-s)(1-s),\ee
which, combined with \eqref{eq:Imz}, implies that 
\begin{align}
   |z-\widetilde{z}|\lesssim \im \wt z \asymp (1-s)\sqrt{\kappa(\sE) +(t-s)(1-s)} \lesssim (1-s)\omega_{s,\sE}.\label{eq:z-z}
\end{align}
%where we used the identity $m(\sE)(\sE+m(\sE))+1=0$ in the second step and the estimate \eqref{square_root_density} in the second step. 
Besides \eqref{GGZGG}, our proof will also use the following linear algebra fact: given $k\in \N,$ let $\bv$ and $\bw_i$, $i \in \qqq{k}$, be a sequence of vectors in a Hilbert space and defined the $k\times k$ Hermitian matrix $A$ with $A_{ij} = (\bw_i, \bw_j)$. Then, for any $p \in 2\N$, we have that 
\be\label{lem_L2Loc}
\sum_i |(\bv, \bw_i)|^2 \le \|\bv\|_2^2 \|A\|_{2\to 2}\le \|\bv\|_2^2 (\operatorname{tr} A^p)^{1/p}.
\ee
The proof of this bound is straightforward; readers may also refer to \cite[Lemma 6.1]{Band1D} for details.
 
%\begin{proof}[Proof of Lemma \ref{lem_ConArg}]

%We first focus on the proof of \eqref{res_lo_bo_eta}. 
Now, we are ready to give the proof of \eqref{res_lo_bo_eta}. Note that for any $k\in \N$, the $(2k+1)$-$G$ loop can be bounded by $G$-loops of lengths $2k$ and $(2k+2)$ by using Cauchy-Schwarz inequality: 
\begin{align}\label{adummzps}
\left|\mathcal{L}^{(2k+1)}_{t, \bsig, \textbf{a}}\right|^2 \le \left|\mathcal{L}^{(2k)}_{t, \bsig_1, \ba_1} \right| \left| \mathcal{L}^{(2k+2)}_{t, \bsig_2, \ba_2}\right|, \quad \forall \bsig\in \{+,-\}^{2k+1}, \ \ba\in (\Zn)^{2k+1} ,
  \end{align}
where $\bsig_1$, $\ba_1$, $\bsig_2$, and $\ba_2$ are defined by  
\begin{align*}
&\bsig_1=(\sig_1,\ldots,\sig_{k},-\sig_{k},\ldots,-\sig_1), \ \ \bsig_2=(-\sig_{2k+1},\ldots,-\sig_{k+1},\sig_{k+1},\ldots,\sig_{2k+1}),\\ 
&\ba_1=([a_1],\ldots, [a_{k-1}],[a_{k}],[a_{k-1}],\ldots, [a_1],[a_{2k+1}]),\\
&\ba_2=([a_{2k}],\ldots, [a_{k+1}],[a_{k}],[a_{k+1}],\ldots, [a_{2k}],[a_{2k+1}]).
\end{align*}
Thus, we only need to prove \eqref{res_lo_bo_eta} for $G$-loops of even lengths. We will perform the proof inductively. 
For clarity of presentation, given $r\in \N$, $\bsig=(\sig_1,  \ldots, \sig_r)\in \{+,-\}^r,$ and \smash{$\ba = ([a_1], \ldots, [a_{r-1}])\in (\Zn)^{r-1}$}, we introduce the following notation of \emph{$G$-chain} of length $r$:
\begin{equation}\label{defC=GEG}
    {\cal C}^{(r)}_{t, \bsig, \ba} := \prod_{i=1}^{r-1}\left(G(\sig_i) E_{[a_i]}\right)\cdot  G(\sig_r),\quad \wt{\cal C}^{(r)}_{s, \bsig, \ba} := \prod_{i=1}^{r-1}\big(\wt G(\sig_i) E_{[a_i]}\big)\cdot  \wt G(\sig_r) .
\end{equation}

Given $\fn=2k\in 2\N$, consider a loop $\cL^{(2k)}_{t,\bsig_0,\ba_0}$ where 
$$\bsig_0=(\sig_1,\ldots, \sig_k, \sig_1',\ldots, \sig_k'),\quad \ba_0=([a_1],\ldots,[a_{k-1}],[a],[a_1'],\ldots,[a'_{k-1}],[a']).$$ 
Applying the Cauchy-Schwarz inequality, we obtain a similar inequality as in \eqref{adummzps}:
\begin{align}\label{eq;CS-chain}
\left|\mathcal{L}^{(2k)}_{t, \boldsymbol{\sigma}, \textbf{a}}\right|^2 \le  
\frac{1}{W^{2d}}\sum_{x\in [a'],y\in [a]}\left|\left({\cal C}^{(k)}_{t,\bsig,\ba}\right)_{xy}\right|^2\cdot \frac{1}{W^{2d}}\sum_{x\in [a'],y\in [a]}\left|\left({\cal C}^{(k)}_{t,\bsig',\ba'}\right)_{yx}\right|^2, \end{align}
where $\bsig$, $\ba$, $\bsig'$, and $\ba'$ are defined as $\bsig:=(\sig_1,\ldots,\sig_{k})$, $\ba=([a_1],\ldots, [a_{k-1}])$, and $\bsig'=(\sig'_{1},\ldots, \sig'_{k})$, $\ba'=([a'_{1}],\ldots, [a'_{k-1}])$. 
We only need to bound the first average on the RHS of \eqref{eq;CS-chain} involving the chain \smash{${\cal C}^{(k)}_{t,\bsig,\ba}$}, while the second average involving the chain \smash{${\cal C}^{(k)}_{t,\bsig',\ba'}$} satisfies exactly the same bound. 
Using the resolvent identity \eqref{GGZGG} and the following algebraic identity (where $A_i$ and $B_i$ are arbitrary matrices) 
\begin{equation}\label{abform1}
        \prod_{i=1}^{k} (A_i + B_i) = \prod_{i=1}^k A_i + \sum_{l=1}^k \left(\prod_{j=1}^{l-1} (A_j + B_j)\right) B_l \left(\prod_{j=l+1}^k A_j\right),
\end{equation}
we can expand \smash{${\cal C}^{(k)}_{t,\bsig,\ba}$} as:
\begin{align}\label{eq:onebyone}
{\cal C}^{(k)}_{t,\bsig,\ba} = & \wt{\cal C}^{(k)}_{s,\bsig,\ba} + (z - \widetilde{z}) \sum_{l=1}^k \prod_{r=1}^{l-1}\p{G(\sig_r) E_{[a_r]}} G(\sig_l)\wt G(\sig_l) \prod_{r=l}^{k-1} \p{E_{[a_r]}\wt G(\sig_{r+1})}.
\end{align}
From \eqref{eq:onebyone}, we obtain that for any $x\in[a']$ and $y\in[a]$, 
\begin{align}\label{GwGtC}
\left| \left({\cal C}^{(k)}_{t,\bsig,\ba}\right)_{xy} \right|^2 \lesssim  & \left| \left(\wt{\cal C}^{(k)}_{s,\bsig,\ba}\right)_{xy} \right|^2 + \left|z - \widetilde{z}\right|^2 \sum_{l=1}^k  \left|\left(\bv^{(l)}_x, \bw^{(l)}_y\right)\right|^2,
\end{align}
where we introduce the vectors $\bv_x^{(l)} ,    \bw^{(l)}_y\in \mathbb C^{N}$, defined as follows for \(\al\in \ZL\):
$$
\bv_x^{(l)}(\al) = \p{\prod_{r=1}^{l-1}\p{G(\sig_r) E_{[a_r]}} G(\sig_l)}_{x\al}, \ \ 
\overline{  \bw^{(l)}_y(\al)} = \p{\wt G(\sig_l)  \prod_{r=l}^{k-1} \p{E_{[a_r]}\wt G(\sig_{r+1})}}_{\al y},
$$
Then, using \eqref{lem_L2Loc}, we can bound the average of $|(\bv_x^{(l)}, \bw^{(l)}_y)|^2$ over $y\in[a]$ by 
\begin{align}\label{jziwmas}
\frac{1}{W^d}\sum_{y \in [a]} \left|\left(\bv_x^{(l)}, \bw^{(l)}_y\right)\right|^2 
    \le \big\|\bv_x^{(l)}\big\|_2^2 \left( \operatorname{tr} \big( A^{(l)} \big)^p \right)^{1/p} 
\end{align}
for any fixed $p\in 2\N$, where the matrix $A^{(l)}$ is defined as
\begin{align*}
    A^{(l)}_{yy'} := \frac{1}{W^d}\left( \bw^{(l)}_y, \bw^{(l)}_{y'} \right)=\sum_{\al\in \ZL}\bw^{(l)}_y(\al)\overline{\bw^{(l)}_{y'}(\al)} =\frac{1}{W^d\im \wt z} \left(\wt{\cal C}^{(2(k-l)+1)}_{s,\bchi_1,\ba_1}\right)_{yy'}
\end{align*}
for $y, y' \in [a]$. Here, we applied Ward's identity to $\wt G(-\sig_l)\wt G(\sig_l)$ in the third step, and $\bchi_1$ and $\ba_1$ are defined as \(\bchi_1=(-\sig_k,\ldots, -\sig_{l+1},\im,\sig_{l+1},\ldots, \sig_k)\) and \(\ba_1=([a_{k-1}],\ldots,[a_l],[a_l],\ldots,[a_{k-1}])\). Note that $ \tr[(\im \wt z\cdot  A^{(l)})^p]$ is a (generalized) $\wt\cL$-loop of length $2p(k-l)+p$. Thus, by the assumption \eqref{55}, we can bound it by 
\be\nonumber
 \left\{ \tr\left[(A^{(l)})^p\right]\right\}^{1/p} \prec \frac{1}{\im \wt z}\frac{\kappa(\wt\sE)^{1/(2p)}}{\big( W^d \ell_{s}^d \eta_{s}(\wt\sE) \big)^{2(k-l)+1-1/p} } \, .
\ee
Plugging this estimate into \eqref{jziwmas} and taking $p$ arbitrarily large, we get
\be\label{eq:boundAp}
 \frac{1}{W^d}\sum_{y \in [a]} \left|\left(\bv_x^{(l)}, \bw^{(l)}_y\right)\right|^2 
    \prec \big\|\bv_x^{(l)}\big\|_2^2 \cdot \p{\im \wt z}^{-1} \big( W^d \ell_{s}^d  \eta_{s}(\wt\sE) \big)^{-2(k-l)-1}  \,  .
\ee
Next, using Ward's identity, we can express the average of \smash{$\|\bv_x^{(l)}\|_2^2$} over $x\in[a']$ as
%using the induction hypothesis for $\cal L_{t}$-loops: 
\begin{align} \label{eq:boundAp2}
 &   \frac{1}{W^d} \sum_{x \in [a']}\big\|\bv_x^{(l)}\big\|_2^2  = \p{\im z}^{-1}  \mathcal{L}^{(2l-1)}_{t, \bchi_2, \ba_2},   
\end{align} 
%where we applied Ward's identity again to $G(\sig_l)G(-\sig_l)$ and 
where $\bchi_2$ and $\ba_2$ are defined as 
\(\bchi_2:=(\sig_1,\ldots, \sig_{l-1},\im, -\sig_{l-1},\ldots, -\sig_1)\) and \(\ba_2:=([a_1],\ldots, [a_{l-1}],[a_{l-1}],\ldots, [a_1],[a']).\) 
With \eqref{eq:boundAp} and \eqref{eq:boundAp2}, we can bound the average of \eqref{GwGtC} over $x\in[a']$ and $y\in[a]$ as 
\begin{align}
    \frac{1}{W^{2d}}\sum_{x\in [a'],y\in [a]}\left|\left({\cal C}^{(k)}_{t,\bsig,\ba}\right)_{xy}\right|^2 \prec  
    \wt{\cL}^{(2k)}_{s,\bsig_3,\ba_3}+ \frac{\eta_s(\wt\sE)}{\eta_t(\sE)} \sum_{l=1}^k \frac{\mathcal{L}^{(2l-1)}_{t, \bchi_2, \ba_2}}{\big( W^d \ell_{s}^d \eta_{s}(\wt\sE)\big)^{2(k-l)+1}}   ,\label{LLVWL}
\end{align}
where $\bsig_3:=(\sig_1,\ldots, \sig_k,-\sig_k,\ldots, -\sig_1)$, $\ba_3:=([a_1],\ldots,[a_{k-1}],[a],[a_{k-1}],\ldots, [a'])$, and we also used \eqref{eq:z-z} in the derivation. A similar bound holds for the second average on the RHS of \eqref{eq;CS-chain} involving the chain \smash{${\cal C}^{(k)}_{t,\bsig',\ba'}$}. Thus,  from \eqref{eq;CS-chain}, we obtain the following estimate (recall $v_\bchi$ defined in \eqref{eq:vbchi}):  
\begin{align}\label{eq;CS-chain2}
\max_{\bsig,\ba}\left|\mathcal{L}^{(2k)}_{t, \boldsymbol{\sigma}, \textbf{a}}\right| \prec   \max_{\bsig,\ba}\left|\wt{\cL}^{(2k)}_{s,\bsig,\ba}\right|+ \frac{\eta_s(\wt\sE)}{\eta_t(\sE)} \sum_{l=1}^k \frac{\max_{\bchi:|v_{\bchi}|=1}\max_{\ba}\big|\mathcal{L}^{(2l-1)}_{t, \bchi, \ba}\big|}{\big( W^d \ell_{s}^d \eta_{s}(\wt\sE)\big)^{2(k-l)+1}}   . \end{align}

We now complete the proof using the induction relation from \eqref{eq;CS-chain2}. First, when $\fn=2k=2$, using the assumption \eqref{55}, we obtain from \eqref{eq;CS-chain2} that 
\begin{align*}
 \max_{\bsig,\ba}\left|\mathcal{L}^{(2)}_{t, \boldsymbol{\sigma}, \textbf{a}}\right| \prec \frac{\big(\kappa(\wt\sE)\big)^{1/2}}{W^d \ell_{s}^d \eta_{s}(\wt\sE)} + \frac{\max_{[a]} \left|\avg{\p{\im G_{t,\sE}}E_{[a]}}\right|}{W^d \ell_{s}^d \eta_{t}(\sE)} \lesssim  \frac{a_t}{W^d \ell_{s}^d \eta_{t}(\sE)} \, ,
\end{align*}
where we abbreviate $a_t:=\sqrt{\kappa(\sE)}+\max_{[a]} \left|\avg{\p{\im G_{t,\sE}}E_{[a]}}\right|$, and we also applied \eqref{eq:kappat} in the second step. 
%Obviously, the above argument applies to any $u\in [s,t]$. 
This verifies \eqref{res_lo_bo_eta} for $\fn=2$ (and at $u=t$), thus completing the first step of the induction argument.

Next, given $\fn=2k\ge 4$, suppose we have shown that \eqref{res_lo_bo_eta} holds for all $\fn=2l$ where $1\le l <k$. By the Cauchy-Schwarz inequality in \eqref{adummzps}, we know that \eqref{res_lo_bo_eta} also holds for all $\fn=2l-1$ where $2\le l <k$. 
Using the assumption \eqref{55} and the induction hypothesis, we obtain from the relation \eqref{eq;CS-chain2} that 
\begin{align}\label{eq:induc_2k}
 \max_{\bsig,\ba}\left|\mathcal{L}^{(2k)}_{t, \boldsymbol{\sigma}, \textbf{a}}\right| \prec   \frac{a_t}{\left( W^d \ell_{s}^d \eta_{t}(\sE) \right)^{2k-1} }+ \frac{\eta_s(\wt\sE)}{\eta_t(\sE)} \frac{\max_{\bsig,\ba}\big|\mathcal{L}^{(2k-1)}_{t, \bsig, \ba}\big|}{W^d \ell_{s}^d \eta_{s}(\wt\sE)}   . 
\end{align}
Applying the Cauchy-Schwarz inequality in \eqref{adummzps} again, we get
\begin{align*}
\max_{\bsig,\ba}\big|\mathcal{L}^{(2k-1)}_{t, \bsig, \ba}\big| \le  \max_{\bsig_1,\ba_1}\left|\mathcal{L}^{(2k-2)}_{t, \bsig_1, \ba_1} \right|^{1/2}\cdot \max_{\bsig_2,\ba_2} \left| \mathcal{L}^{(2k)}_{t, \bsig_2, \ba_2}\right|^{1/2}\, . 
\end{align*}
Plugging it into \eqref{eq:induc_2k} and using the induction hypothesis \eqref{res_lo_bo_eta} with $\fn=2k-2$, we obtain that 
\begin{align}
 \max_{\bsig,\ba}\left|\mathcal{L}^{(2k)}_{t, \boldsymbol{\sigma}, \textbf{a}}\right| \prec  \frac{a_t}{\left( W^d \ell_{s}^d \eta_{t}(\sE) \right)^{2k-1} } +  \max_{\bsig,\ba} \big| \mathcal{L}^{(2k)}_{t, \bsig, \ba}\big|^{1/2} \cdot \frac{a_t^{1/2}}{\left( W^d \ell_{s}^d \eta_{t}(\sE) \right)^{k-1/2} }.\label{eq:ind_prf1}
\end{align}
This implies \eqref{res_lo_bo_eta} for $\fn=2k$ (and at $u=t$), thus completing the induction argument.

The above arguments establish \eqref{res_lo_bo_eta} for each fixed $u\in [s,t]$. To extend this uniformly to all $u\in [s,t]$, we again apply a standard perturbation argument combined with a union bound over an $N^{-C}$-net. This completes the proof of \Cref{lem_ConArg}.
 
%\end{proof}

\begin{proof}[\bf Proof of \Cref{claim:uniform}]
As mentioned below \eqref{z-zs}, it remains to prove \eqref{eq:claim_perturb}. Our approach follows the argument used in the proof of \Cref{lem_ConArg}. Adopting the notations from that proof, we abbreviate $\wt z\equiv z_s(\sE')$, $z\equiv z_s(\sE)$, \smash{$\wt G\equiv G_{s,\sE'}$}, and \smash{$G\equiv G_{s,\sE}$}, and denote the $G$-loops formed with $G$ and \smash{$\wt G$} by $\cal L$ and \smash{$\wt{\cal L}$}, respectively. 
On the event $\Omega$, the \smash{$\wt{\cal L}$}-loops satisfy the bounds in \eqref{eq:wtLboundE'}. 
Then, for $2\le 2k \le \fn$, by repeating the arguments between \eqref{lem_L2Loc} and \eqref{eq;CS-chain2} with $p=2$, we get that on $\Omega$,
\begin{align}\label{eq;CS-chain4}
\max_{\bsig,\ba}\left|\mathcal{L}^{(2k)}_{s, \boldsymbol{\sigma}, \textbf{a}}\right| \le   \max_{\bsig,\ba}\left|\wt{\cL}^{(2k)}_{s,\bsig,\ba}\right|+ \frac{CW^\e |z-\wt z|^2}{\im z \cdot \im \wt z} \sum_{l=1}^k \frac{\max_{\bchi:|v_{\bchi}|=1}\max_{\ba}\big|\mathcal{L}^{(2l-1)}_{t, \bchi, \ba}\big|}{\big( W^d \ell_{s}^d \eta_{s}(\sE')\big)^{2(k-l)+1/2}} 
\end{align}
for a constant $C>0$ depending on $k$.
Next, on the event $\Omega_1$, we have 
\[\max_{x}|\im (G_{s,\sE})_{xx}|\le C_0+1\]
by the first estimate in \eqref{eq:claim_perturb2}. Using the above two estimates, along with the facts \eqref{z-zs}, $\im z\ge N^{-1}$, and $\im \wt z\ge N^{-1}$, we can repeat the inductive proof below \eqref{eq;CS-chain2} to conclude the proof of \eqref{eq:claim_perturb}.
\end{proof}

\subsection{Proof of \Cref{lem G<T}}\label{subsec_pf_lem_G<T}

For simplicity of presentation, we abbreviate $\sE\equiv \sE_k$ and $G_u\equiv G_{u,\sE}$ in the following proof. The proof of the weak local law estimate \eqref{GiiGEX} in the literature typically follows the strategy used in our proof of \Cref{lem G<T2}. This step is relatively straightforward in the bulk of the spectrum, as demonstrated in \cite[Section 4]{Band1D}, since $\|(1-u m^2S)^{-1}\|_{\infty\to\infty}\asymp [\omega_u(\sE)]^{-1}\lesssim 1$ in the bulk. This ensures that the self-consistent equation \eqref{eq:self0} for the diagonal resolvent entries remains stable. However, this method encounters significant difficulties near the spectral edges, where \smash{$[\omega_t(\sE)]^{-1}$} diverges and the self-consistent equation becomes unstable. We now explain this issue in greater detail.

%and it proves to be straightforward and effective for establishing the local laws within the bulk of the spectrum, as shown in \cite[Section 4]{Band1D}. 
%However, it fails near the edge regimes due to the instability of the self-consistent equations for the diagonal resolvent entries as we will explain now. 
%Additionally, a fluctuation averaging mechanism is employed in the proof of the averaged local law, as seen in \cite[Proposition 3.3]{EKY_Average} and \cite[Theorems 4.6 and 4.7]{Semicircle}. 

Using the assumption \eqref{res_lo_bo_eta}, we obtain that $\max_{\bsig,\ba}|{\cal L}^{(2)}_{u, \boldsymbol{\sigma}, \ba}(\sE)|\prec\zeta_u^2$ uniformly for $u\in [s,t]$, where $\zeta_u$ is a \emph{random} control parameter defined as  
\[
\zeta_u= \frac{\ell_{u}^d}{\ell_{s}^d}  \frac{\sqrt{\kappa(\sE)}+\max_{[a]} \left|\avg{\p{\im G_{u}}E_{[a]}}\right|}{ W^d \ell_{u}^d \eta_{u}(\sE)  } \ge \Phi_u\, . 
\]
In particular, if $\left|\avg{\p{\im G_{u}}E_{[a]}}\right|\lesssim\sqrt{\kappa(\sE)}$, then $\zeta_u$ can be bounded by the \emph{deterministic} control parameter $\Phi_u$. As a result, the off-diagonal entries of $G_u$ satisfies the estimate in \eqref{GijGEXoff}, with $t$ and $\phi_t$ replaced by $u$ and $\zeta_u$, respectively, and the self-consistent equation \eqref{eq:self1} transforms to:
% With this bound on $2$-$G$ loops, the off-diagonal entries of $G_u$ can be readily bounded using the techniques in \cite{Semicircle}. 
% In particular, the following estimate was established in equation (4.11) of \cite{Band1D}:
% \begin{align}
% \mathbf 1\p{\|G_{u}- M\|_{\max}\le W^{-\e_0}}\cdot \max_{x\ne y}\left|(G_{u})_{xy}\right|^2 \prec \max_{\bsig,\ba}|{\cal L}^{(2)}_{u, \boldsymbol{\sigma}, \ba}(\sE)| + W^{-d}\lesssim \zeta_u^2 \, , %\label{GijGEXoff}
% \end{align}
% where $\e_0>0$ is a small constant.
% Using \eqref{GijGEXoff} and the techniques in \cite{Semicircle}, we can show that the diagonal entries of $G_u$ satisfy a system of self-consistent equations on the event $\{\|G_{u}- M\|_{\max}\le W^{-\e_0}\}$:
% \be %\label{eq:self0}
% \frac{1}{(G_u)_{xx}}=- z_u(\sE) - u \sum_y S_{xy}(G_u)_{yy}+\OO_\prec(\zeta_u)\, .
% \ee
% Subtracting this equation from the self-consistent equation for $m\equiv m_u(z_u(\sE))$ yields:
% \be \label{eq:self4}
% \mathbf 1(\|G_u-M\|_{\max}\le W^{-\e}) \sum_y \left(1- um^2 S\right)_{xy}[(G_u)_{yy}-m] \prec \zeta_u + \Lambda_u^2 \, ,
% \ee
% where \smash{$\Lambda_u:=\max_{x}|(G_u)_{xx}-m|$}. By \eqref{prop:ThfadC_short}, we find that  
% \smash{$\|(1-um^2S)^{-1}\|_{\infty\to\infty} \prec \omega_u^{-1}$}. Thus, solving the equation \eqref{eq:self}, we obtain that 
\be\label{eq:self4}
\mathbf 1\p{\|G_{u}- M\|_{\max}\le W^{-\e}}\cdot  \Lambda_u \prec \zeta_u/\omega_u + \Lambda_u^2/\omega_u, \ee
where \smash{$\Lambda_u:=\max_{x}|(G_u)_{xx}-m|$}. On the other hand, we have the initial bound $\Lambda_s\prec \Psi_s(\sE)\lesssim \Phi_s(\sE)$ by \eqref{55}. Within the framework developed in \cite{Semicircle}, one attempts to propagate this estimate from $s$ to $t$ by applying \eqref{eq:self4} along with a perturbative argument on an $N^{-C}$-net of $[s,t]$. This method transfers control over $\Lambda_u$ step by step along the lattice points. 
However, due to the $\Lambda_u^2/\omega_u$ term on the RHS of \eqref{eq:self4}, this continuity argument hinges critically on the smallness condition $\Phi_u/\omega_u\ll \omega_u$, which effectively replaces the role of the first assumption in \eqref{initialGT2}. Unfortunately, this condition is not always satisfied under our flow. Given \eqref{eq:small_para} and \eqref{eq:small_para2}, we can only ensure the weaker condition $\Phi_u/\omega_u\ll 1$.

%In the framework developed by \cite{Semicircle}, we attempt to apply \eqref{eq:self4} along with a perturbation argument on an $N^{-C}$-net of $[s,t]$ to transfer the estimate on $\Lambda_u$ from $s$ to $t$ step by step along the lattice points. However, due to the $\Lambda_u^2/\omega_u$ term on the RHS of \eqref{eq:self4}, this process only works when $\Phi_u^2/\omega_u^2\ll \omega_u$ (which actually replaces the role of the first assumption in \eqref{initialGT2}). Unfortunately, this condition may not always hold under our flow; under \eqref{eq:small_para} and \eqref{eq:small_para2}, we can only guarantee that $\Phi_u/\omega_u\ll 1$.

To deal with the above issue, we adopt a dynamic argument to establish the local law estimate \eqref{GiiGEX} for the diagonal resolvent entries along the flow. For this purpose, we introduce the stopping times
\begin{align}\label{def:T}
T:=\inf \Big\{u\ge s: \|G_{u}(\sE) - M(\sE) \|_{\max} \ge W^{-\e_0}\sqrt{\kappa(\sE)}\Big\}
\end{align}
% \begin{align}\label{def:T}
% T:=\inf \Big\{u\ge s: \|G_{u}(\sE) - M(\sE) \|_{\max}\le W^{-\e_0}, \  \max_{0\le \ell \le n}   \theta_u \ge W^{-\e_0}\sqrt{\kappa(\sE)}\Big\}
% \end{align}
for a small constants $0< \e_0 < (\fd\wedge\fc)/10$. 
Note that for $u\le t\wedge T$, we have \be\label{eq:zetau}
\avg{(\im G_u)E_{[a]}} = \im m(E)+\avg{\im (G_u-M)E_{[a]}}\asymp \sqrt{\kappa(\sE)},\quad \text{and}\quad 
\zeta_u\lesssim\Phi_u.
\ee
From the equation \eqref{eq:SDE_Gt}, we obtain that for any $w\in \ZL$,
\begin{align}\label{eq:SDE_Gt2}
(G_{t\wedge T})_{ww}-m=&~\br{(G_{s})_{ww}-m}-\int_s^{t\wedge T} (G_u)_{w x} (G_{u})_{yw} \sqrt{S_{xy}}\dd (B_{u})_{xy} \nonumber\\
&~+\int_s^{t\wedge T} W^d\sum_{[b]}\avg{(G_u-M)E_{[b]}}\p{G_{u} E_{[b]}G_{u}}_{ww}\dd u \, .
\end{align}
% \begin{align}\label{eq:SDE_Gt2}
% \avg{(G_{t\wedge T}-M)E_{[a]}}=&~\avg{(G_{s}-M)E_{[a]}}-\int_s^{t\wedge T} \sum_{x,y}(G_u E_{[a]}G_{u})_{yx} \sqrt{S_{xy}}\dd (B_{u})_{xy} \nonumber\\
% &~+\int_s^{t\wedge T} W^d\sum_{[b]}\avg{(G_u-M)E_{[b]}}\avg{G_{u} E_{[b]}G_{u}E_{[a]}}\dd u \, .
% \end{align}
We first bound the quadratic variation of the martingale term on the RHS of \eqref{eq:SDE_Gt2}, defined as:
\begin{align}\label{eq:quad_var}
\int_s^{t\wedge T} \sum_{x,y}S_{xy} |(G_u)_{wx}|^2|(G_u)_{wy}|^2  \dd u.
\end{align}
Using \eqref{GijGEXoff} and \eqref{eq:zetau}, we can derive that for $u\in [s,t\wedge T]$,
\begin{align}\label{eq:boundmartg}
\sum_{x,y}S_{xy} |(G_u)_{wx}|^2|(G_u)_{wy}|^2 &\prec \p{\Phi_u^2+W^{-d}} \sum_y |(G_u)_{wy}|^2 \nonumber\\
&\lesssim \Phi_u^2 \frac{\im (G_u)_{ww}}{\eta_u} \lesssim \frac{\Phi_u^2}{1-u} , 
\end{align}
where we used Ward's identity in the second step and \eqref{eq:kappat} in the third step.\footnote{With a standard perturbation and union bound argument on an $N^{-C}$-net, we can guarantee that the estimate \eqref{eq:boundmartg} holds uniformly for $u\in [s,t\wedge T]$.} 
So, we can bound \eqref{eq:quad_var} as:
\begin{align*}
%\mathbf 1(\Omega_{t,\e_0})\cdot \int_s^{t\wedge T} \sum_{x,y}S_{xy} (G_u^* E_{[a]}G_{u}^*)_{xy}(G_u E_{[a]}G_{u})_{yx} \dd u 
\int_s^{t\wedge T} \sum_{x,y}S_{xy} |(G_u)_{wx}|^2|(G_u)_{wy}|^2  \dd u \prec \int_s^{t\wedge T} \frac{1}{1-u}\Phi_u^2\dd u \prec \Phi_{t\wedge T}^2.
\end{align*}
Applying the Burkholder-Davis-Gundy inequality, we obtain that 
\be\label{eq:martingaleterm}
\int_s^{t\wedge T} (G_u)_{w x} (G_{u})_{yw} \sqrt{S_{xy}}\dd (B_{u})_{xy}\prec \Phi_{t\wedge T} .
\ee
For the last term on the RHS of \eqref{eq:SDE_Gt2}, recall that \smash{$\Lambda_u:=\max_{x}|(G_u)_{xx}-m|$}. Then, using the Cauchy-Schwarz inequality and Ward's identity again, we obtain that 
\begin{align}\label{eq:driftterm}
W^d\sum_{[b]}\big|\avg{(G_u-M)E_{[b]}}\big| \big|\p{G_{u} E_{[b]}G_{u}}_{ww}\big| &\le \Lambda_u \frac{\im (G_u)_{ww}}{\eta_u}= \Lambda_u \frac{\im m + W^{-\e_0}\sqrt{\kappa(\sE)}}{\eta_u} \nonumber\\
&= (1+\OO(W^{-\e_0}))\frac{\Lambda_u}{1-u}
\end{align}
for $u\le t\wedge T$, where we used the definition of $T$ in the second step and the definition of $\eta_u$ in \eqref{eta} in the last step. Plugging \eqref{eq:martingaleterm} and \eqref{eq:driftterm} into \eqref{eq:SDE_Gt2}, and taking the maximum over $w$, we obtain that with high probability, 
\begin{align*}
\Lambda_{t\wedge T}\le (1+\OO(W^{-\e_0}))\int_{s}^{t\wedge T}\frac{\Lambda_u }{1-u}\dd u + \OO_\prec \left( \Psi_s(\sE)+\Phi_{t\wedge T}\right).   
\end{align*} 
Applying Gr{\"o}nwall's inequality, we derive that %the following bound holds uniformly in $u\in [s,t\wedge T]$:
\begin{align}\label{eq:thetatt}
\Lambda_{t\wedge T}\prec \frac{1-s}{1-t}\left[\Psi_s(\sE)+\Phi_{t\wedge T} \right] \lesssim  \frac{1-s}{1-t}\Phi_{t\wedge T} .   
\end{align} 
% Recalling \eqref{eq:verifyT1}, we have that
% \be\label{eq:verifyT3}\frac{1-s}{1-t}\Phi_{t\wedge T}\le W^{-2\e_0}\sqrt{\kappa(\sE)}.\ee 
% Moreover, plugging the averaged local law \eqref{eq:thetatt} back to \eqref{eq:self0} gives that
% \be\label{eq:Psitt}
% (G_{t\wedge T})_{xx}=\br{- z_{t\wedge T}(\sE) - ({t\wedge T}) m +\OO_\prec\p{\frac{1-s}{1-u}\Phi_{t\wedge T}}}^{-1} =m+\OO_\prec\p{\frac{1-s}{1-t}\Phi_{t\wedge T}}.
% \ee
Together with \eqref{GijGEXoff} and \eqref{eq:zetau}, it gives the entrywise local law 
\be\label{eq:verifyT2} \|G_{t\wedge T}-m\|_{\max} \prec \frac{1-s}{1-t}\Phi_{t\wedge T} .\ee
From \eqref{eq:verifyT1} and \eqref{eq:verifyT2}, we see that $T\ge t$ with high probability, which concludes the proof of \Cref{lem G<T}. % by \eqref{eq:thetatt}. 

\subsection{Proof of \Cref{thm_locallaw_out}}\label{subsec:outspec}

% The rigidity of eigenvalues in \Cref{thm:rigidity} follows essentially from the averaged local law \eqref{locallaw_aver}. However, we first need to bound the largest and smallest eigenvalues of $H$; specifically, we need to prove that 
% \be\label{eq:rigid_largest}
% \P\p{\lambda_1\ge -2-N^{-2/3+\e}, \  \lambda_N\le 2+N^{-2/3+\e}}\ge 1-N^{-D},
% \ee
% for any small constant $\e>0$ and large constant $D>0$. 
% To this end, we need to extend the averaged local law \eqref{locallaw_aver_out} to $z=E+\ii\eta$ with $E\notin [-2-N^{-2/3+\e},2+N^{-2/3+\e}]$ and $\eta$ slightly below $\eta_*(E)$. 

%Note that if $\kappa_E \le \eta_*(E)\log W$, then we have $\eta_{\circ}(E) \ge \eta_*(E)/(\log W)$, and \Cref{thm_locallaw_out} follows directly from \Cref{thm_locallaw}. It remains to deal with the case where $\kappa_E > \eta_*(E)\log W$. 

The proof of \Cref{thm_locallaw_out} follows a similar strategy to that of \Cref{lem_ConArg}. Specifically, given $E\notin [-2,2]$, we begin by establishing a continuity estimate that extends the $G$-loop bound in \eqref{Eq:L-KGt2} from  $\eta\gg \eta_*(E)$ down to $\eta\gg \eta_{\circ}(E)$, using an argument analogous to that in the proof of \Cref{lem_ConArg}. We then extend the local laws from $\eta\gg \eta_*(E)$ down to $\eta\gg \eta_{\circ}(E)$ by combining \Cref{lem G<T2} with a simple a priori estimate for the resolvent entries. 
%via a flow argument similar to the one used in the proof of \Cref{lem G<T}. 

As in \eqref{Eq:defGLoop}, consider the $G$-loops formed by $G(z)\equiv G(+)$ and $G(\overline z)\equiv G(-)$: 
$${\cal L}^{(\fn)}_{\boldsymbol{\sigma}, \ba}(z)=\left \langle \prod_{i=1}^\fn \left(G(\sigma_i) E_{[a_i]}\right)\right\rangle .$$ 
We first claim the following continuity estimate for these loops.

\begin{lemma}\label{lem:below_threshold}
In the setting of \Cref{thm_locallaw_out}, fix any $2\le |E| \le c_0^{-1}$ and $\wt\eta\ge W^{\fd}\eta_{\circ}(E)$ for some constant $\fd>0$. Suppose that the following entrywise local law holds at \(\wt z=E+\ii \wt \eta\): 
\begin{equation}\label{55_out}
\left\|G(\wt z)-M(\wt z)\right\|_{\max}^2\prec \frac{1}{W^d\ell(\wt z)^d\sqrt{\kappa_E+\wt \eta}}, 
\end{equation}
and the following $G$-loop bound holds for each $\fn\ge 2$:
\begin{equation}\label{55_loop}
\max_{\bsig,\ba}\left|{\cal L}^{(\fn)}_{ \boldsymbol{\sigma}, \ba}(\wt z)\right|\prec  \frac{\im m(\wt z)}{\left( W^d \ell(\wt z)^d \wt \eta \right)^{\fn-1} } \, .
\end{equation}
Then, the following 2-$G$ loop bound holds uniformly in all $z=E+\ii \eta$ with $W^{\fd}\eta_\circ(E) \le \eta \le \wt \eta$: 
\begin{equation}\label{cont_smalleta}
\max_{\bsig,\ba} \left|{\cal L}_{ \boldsymbol{\sigma}, \ba}^{(2)}(z)\right| \prec  \frac{\im m(z) + \max_{[a]} \left|\avg{\p{\im G(z)}E_{[a]}}\right|}{ W^d \ell(\wt z)^d \eta} \, . 
\end{equation}
%where $a(z):=\im m(z) + \max_{[a]} \left|\avg{\p{\im G(z)E_{[a]}}}\right|$.
%where, with a slight abuse of notation, we denote \(\ell_{E}:=\min\big(\kappa_E^{-1/4},n\big)\). 
% Furthermore, denote $z'= E + \ii \eta'$, where $\eta'=W^c\eta$ for a small constant $0<c<\fd$. Let $a_N(z')>0$ be a deterministic parameter satisfying that $ \eta'/(C\sqrt{\kappa_E})\le a_N(z') \le C$. Then, we have that 
% \begin{equation}\label{cont_smalleta2}
%     {\bf 1}\left(\max_{[a]} \left|\avg{\p{\im G(z')}E_{[a]}}\right|\leq a_N(z'), \, \max_{[a]} \left|\avg{\p{\im G(z)}E_{[a]}}\right|\leq C\right) \cdot \max_{\bsig,\ba} \left|{\cal L}_{ \boldsymbol{\sigma}, \ba}^{(2)}(z)\right| \prec  \frac{\wt\eta}{\eta}\cdot \frac{ a_N(z') }{ W^d \ell_{E}^d \eta }\, . 
% \end{equation}
\end{lemma}
\begin{proof} 
%The proof of this lemma follows exactly the same argument as that in the proof of \Cref{lem_ConArg}. 
Denote \smash{$\wt G\equiv G(\wt z)$}, $G\equiv G(z)$, and the $\fn$-$G$ loops formed with $\wt G$ as \smash{$\wt{\cal L}^{(\fn)}_{\bsig,\ba}$}. Then, following the proof of \eqref{res_lo_bo_eta} in \Cref{subsec_pf_lem_ConArg}, we can obtain \eqref{cont_smalleta}. We omit the details.
\end{proof}

Next, we establish the following induction result, which allows us to extend the local laws to an arbitrary $z\in\mathbf D^{\mathrm{out}}_{c_0,\fd}$ along a sequence of spectral parameters with progressively smaller imaginary parts. 

%We introduce the spectral domain
%$$ \mathbf D^{\mathrm{out}}_{\fd,C_0,\e}:=\{z=E+\ii\eta\in \C_+: 2\le |E|\le C_0,\, \kappa_E> \eta_*(E)\log W,\, W^\fd \eta_{\circ}(E)\le \eta \le W^\e\eta_*(E)\},$$
%for some constants $C_0,\fd,\e>0$. 
%By combining \Cref{lem:below_threshold} with \Cref{lem G<T,lem G<T2} (for the case $t=1$), and following an argument analogous to the one used in the proof of \eqref{lRB1} and \eqref{Gtmwc} in Step 1 of the proof of \Cref{lem:main_ind}, we establish the following averaged local law for $\eta$, valid down to the scale $W^{-\fd}\eta_\circ(E)$.

\begin{lemma}\label{lem:below_threshold2}
In the setting of \Cref{thm_locallaw_out}, fix any $\wt z=E+\ii \wt\eta\in \mathbf D^{\mathrm{out}}_{c_0,\fd}$ with $\wt \eta \le W^\e\eta_*(E)$ for a small constant $0<\e<(\fd\wedge c_0)/20$. Suppose the entrywise local law \eqref{55_out} holds and the $2$-$G$ loop bound \eqref{cont_smalleta} holds uniformly for all $z=E+\ii \eta$ with $W^\fd\eta_{\circ}(E)\le \eta \le \wt\eta$. 
%Consider the flow framework in \Cref{zztE}, with $t_0$ and $\sE$ defined in \eqref{eq:t0E0}, and define 
%\[t_1:=t_0\wedge \inf\left\{t\ge 0: \eta_t(\sE)\le W^{\e}\eta_*\big(E_t(\sE)/\sqrt{t}\big)\right\}.\]
%Suppose the following $2$-$G$ loop bound holds uniformly for all $u\in[t_1,t_0]$:
%\begin{equation}\label{res_lo_bo_eta}
%    \max_{\bsig,\ba} \left|{\cal L}_{u, \boldsymbol{\sigma}, \ba}^{(2)}(\sE)\right| \prec W^{\e} \frac{\sqrt{\kappa(\sE)}+\max_{[a]} \left|\avg{\p{\im G_{u}}E_{[a]}}\right|}{W^d \ell_{u}^d \eta_{u}(\sE)}\, .
%\end{equation}
%Moreover, assume the following local law estimate holds at $\wt z$: 
%\begin{equation}\label{localats}
% \|G(\wt z)-M(\wt z)\|_{\max}^2\prec \frac{1}{W^d\ell(\wt z)^d\sqrt{\kappa_E+\wt\eta}}  \, .
%\end{equation}
%In the setting of \Cref{thm_locallaw_out}, fix any $\wt z=E+\ii \wt\eta$ such that $\kappa_E\ge \eta_*(E)$ and $\wt\eta = W^{\e}\eta_{*}(E)$ for a small constant $\e>0$. Suppose the local laws in \eqref{55_out} and the $2$-$G$ loop bound in \eqref{55_loop} hold. 
%Then, for any $t\in [s,t_0]$ satisfying $\omega_s(\sE)/\omega_t(\sE)\le W^{\e}$, the following estimate holds uniformly for all $u\in [s,t]$: 
%$z=E+\ii \eta$ with $W^\fd\eta_\circ (E)\le \eta \le \wt \eta$ if we choose $\e<\fd/10$:
%In the setting of \Cref{lem:below_threshold}, for any constants $\tau, D>0$, the following estimate holds uniformly for all $z=E+\ii\eta$ with $E\in [2+W^{\e}N^{-2/3}, \e^{-1}]$ and $W^{-\fd}\eta_\circ(E)\le \eta\le W^\fd \eta_*(E)$: 
Then, the following local laws hold uniformly for all $z=E+\ii \eta \in \mathbf D^{\mathrm{out}}_{c_0,\fd}$ satisfying $W^{-(\fd\wedge c_0)/20}\wt\eta \le \eta \le \wt\eta$:
\begin{align}
\|G(z)-M(z)\|_{\max}^2&\prec \frac{1}{W^d\ell(z)^d\sqrt{\kappa_E+\eta}}, \label{locallaw_entry_below}\\ 
 \max_{[a]} \big|\big\langle{(G(z)-M(z)) E_{[a]}}\big\rangle\big| &\prec \frac{1}{W^d\ell(z)^d\p{\kappa_E+\eta}} . \label{locallaw_aver_below}
\end{align}
\end{lemma}
\begin{proof}
For any $\eta\le \wt\eta\le W^\e\eta_*(E)$, we have $\kappa_E \gg \eta$ and $\ell(z)\asymp \ell_{E}:=\min(\kappa_E^{-1/4},n)$. Applying the Cauchy-Schwarz inequality and Ward's identity, we obtain that for all $0<\eta\le \wt\eta$ and $x,y\in \ZL$, 
\be\label{eq:derGGG}
\abs{\partial_\eta G_{xy}(z)}=\abs{(G^2(z))_{xy}}\le \frac{\p{\im G_{xx}(z)\cdot\im G_{yy}(z)}^{\frac 1 2}}{\eta} \le \frac{\wt \eta}{\eta^2}\max_{x}\im G_{xx}(\wt z),
\ee
where in the second step, we used the monotonicity $\eta \im G_{xx}(z) \le \wt\eta \im G_{xx}(\wt z)$. Integrating equation \eqref{eq:derGGG}, we obtain that for all $0<\eta\le \wt\eta$,
\be\label{eq:G-wtG} \max_{x,y}|G_{xy}(z)-G_{xy}(\wt z)| \le \frac{\wt\eta}{\eta}  \max_{x}\im G_{xx}(\wt z) \lesssim  \frac{\wt\eta}{\eta}  \br{\im m(\wt z)+\OO_\prec(\Phi_E)},\ee
where we applied the local law \eqref{55_out} at $\wt z$ and denote that \(\Phi_E := (W^d\ell_E^d\sqrt{\kappa_E})^{-1/2}.\) On the other hand, it is easy to check that 
\[|m(z)-m(\wt z)|\lesssim\wt\eta/\sqrt{\kappa_E+\eta}.\]
Combining it with \eqref{eq:G-wtG} and the local law \eqref{55_out} at $\wt z$, and using \eqref{square_root_density} for $\im m(\wt z)$, we get that 
\be\label{eq:G-wtG2} \|G(z)-M(z)\|_{\max}  \prec  \frac{\wt\eta}{\eta}\frac{\wt\eta}{\sqrt{\kappa_E}} +\frac{\wt\eta}{\eta}\Phi_E\le  W^{-(\fd\wedge c_0)/4}\sqrt{\kappa_E},\ee
where in the second step, we used the definition \eqref{eq:defeta*3} and the facts $\kappa_E\ge W^{c_0}\eta_*(E)$, $\wt\eta\le W^\e\eta_*(E)$, and $\wt\eta/\eta\le W^{(\fd\wedge c_0)/20}$. This verifies the first condition in \eqref{initialGT2}. Moreover, combining \eqref{eq:G-wtG2} with \eqref{cont_smalleta}, we obtain that 
\begin{equation}\label{cont_smalleta0}
	\max_{\bsig,\ba} \left|{\cal L}_{ \boldsymbol{\sigma}, \ba}^{(2)}(z)\right| \prec  \phi_0^2:=\frac{\im m(z) + W^{-(\fd\wedge c_0)/4}\sqrt{\kappa_E}}{ W^d \ell_E^d \eta}. 
\end{equation}

%\be\label{eq:parasmall}\frac{\wt\eta}{\eta} \im m(\wt z) \le W^{-c_0/2}\sqrt{\kappa_E},\quad \frac{\wt\eta}{\eta} \Phi_E \le W^{-\fd/3} .\ee
%from \eqref{eq:G-wtG} that 
%\[ \max_{x,y}|G_{xy}(z)-G_{xy}(\wt z)| \prec W^{-c_0/2}\sqrt{\kappa_E}.\]

Now, applying \Cref{lem G<T2} (at $t=1$), and using \eqref{cont_smalleta0} and  $|1-m(z)^2|\asymp\sqrt{\kappa_E}$, we can establish that
\begin{equation}\label{locallaw_aver_below_iter0}
\|G(z)-M(z)\|_{\max}\prec \phi_0,\quad \max_{[a]} \big|\big\langle{(G(z)-M(z)) E_{[a]}}\big\rangle\big| \prec  \phi_0^2/\sqrt{\kappa_E}\, .
\end{equation}
With \eqref{cont_smalleta} and \eqref{locallaw_aver_below_iter0}, we get that
\begin{equation}\label{cont_smalletause0}
	\max_{\bsig,\ba} \left|{\cal L}_{ \boldsymbol{\sigma}, \ba}^{(2)}(z)\right| \prec   \phi_1^2:=\frac{\im m(z)+\phi_0^2/\sqrt{\kappa_E}}{ W^d \ell_{E}^d \eta} \, .
\end{equation}
For $\eta\ge W^\fd\eta_{\circ}(E)$, we can check from the definition \eqref{eq:defeta*3} that $\phi_1$ is a smaller parameter than $\phi_0$: 
\[\phi_1^2 = \frac{\im m(z)}{ W^d \ell_{E}^d \eta}+\OO(W^{-\fd}\phi_0^2).\]
Applying \Cref{lem G<T2} again gives an even better parameter $\phi_2^2$. Iterating the above argument at most $k=\lceil d\fd^{-1}\rceil$ many times, we can decrease the $\phi$ parameter to $\phi_k=\im m(z)/[W^d \ell_{E}^d \eta]\asymp\Phi_E$, which concludes \eqref{locallaw_entry_below} and \eqref{locallaw_aver_below}.
%The proof of this lemma is almost the same as that of \Cref{lem G<T}, except that we use a different stopping time from \eqref{def:T}:
%\begin{align}\label{def:T2}
%\tau:=\inf \Big\{u\ge s: \|G_{u} - M \|_{\max} \ge W^{-\e}\omega_u(\sE)\Big\}.
%\end{align}
%Then, repeating the proof below \eqref{def:T}, we obtain the desired estimate \eqref{locallaw_aver_below}. We omit the details.
\end{proof}

Combining \Cref{lem:below_threshold,lem:below_threshold2}, we can complete the proof of \Cref{thm_locallaw_out}.
\begin{proof}[\bf Proof of \Cref{thm_locallaw_out}]
When $\eta \ge W^\e\eta_*(E)$ (where $\e$ is the constant in \Cref{lem:below_threshold2}), \Cref{thm_locallaw_out} follows directly from \Cref{thm_locallaw}. 

Now, given any $z=E+\ii \eta\in \mathbf D^{\mathrm{out}}_{c_0,\fd}$, we define a sequence of $z_k=E+\ii \eta_k$ with decreasing imaginary parts \smash{$\eta_k:= \max\left(W^{-k\e+\e} \eta_*(E), \eta\right)$}. 
First, with the local law \eqref{55_out} at $\wt z=z_0$ and the $G$-loop bounds \eqref{55_loop} established in \eqref{Eq:L-KGt2}, we can apply \Cref{lem:below_threshold} to conclude that \eqref{cont_smalleta} holds uniformly for all $W^\fd\eta_{\circ}(E)\le \eta \le \eta_0$. Next, suppose we have proved \eqref{locallaw_entry_below} for $z_{k-1} = E + \ii\eta_{k-1}$. Then, applying \Cref{lem:below_threshold2}, we obtain that \eqref{locallaw_entry_below} and \eqref{locallaw_aver_below} hold at $z_{k}$. By induction in $k$, we conclude that \eqref{locallaw_entry_below} and \eqref{locallaw_aver_below} hold for each fixed \smash{$z=E+\ii \eta \in \mathbf D^{\mathrm{out}}_{c_0,\fd}$} with $\eta\le W^\fd \eta_*(E)$. 
Finally, the uniformity in $z$ follows from a standard \smash{$N^{-C}$}-net, union bound, and perturbation argument, whose detail we omit.
\end{proof}

\section{Analysis of the loop hierarchy}\label{Sec:Stoflo}

%In this section, we complete the proof of \Cref{lem:main_ind} by analyzing the loop hierarchy \eqref{eq:mainStoflow} for $G$-loops, following Steps 2–6 outlined in \Cref{sec:proof}. Some parts of our proof are similar to those in \cite[Section 5]{Band1D} and \cite[Section 6]{RBSO1D}. Therefore, we will highlight only the main differences in the arguments, omitting some similar details.

%cite \cite{Bandedge} here
This section is devoted to completing the proof of \Cref{lem:main_ind}, following the steps outlined in \Cref{subsec:mainproofs}. With the a priori $G$-loop bound \eqref{lRB1} and the weak local law \eqref{Gtmwc} established in Step 1, the remainder of the argument closely follows the approach in \cite[Section 5]{Band1D} and \cite[Section 7]{RBSO1D}, building on the tools developed in the preceding sections. 
In particular, we make use of the flow (\Cref{zztE}), the loop hierarchy (\Cref{lem:SE_basic}), the deterministic estimates from \Cref{lem_propTH}, and the resolvent entry estimates from \Cref{lem G<T2,lem_GbEXP}. We also rely on a number of key structural properties of $\cK$-loops, including the upper bound in \Cref{ML:Kbound}, Ward's identity (\Cref{lem_WI_K}), the identity \eqref{eq:pure_sum} for pure loops, and the tree representation (\Cref{tree-representation}). 
Hence, we will outline only the main ideas in the body of the text, focusing on the key differences---particularly those related to the treatment of pure $G$-loops in Steps 3 and 4 (see \Cref{sec:inductive_step,sec:step4}). For the remaining technical details, we either omit them due to their similarity with \cite{Band1D,RBSO1D}, or defer certain key proofs to \Cref{sec:main_appd} for the reader’s convenience. The proof of Step 6, which is largely independent of the preceding steps, is presented separately in \Cref{sec:pf_step6}.%cite \cite{Bandedge} here

%This section is devoted to the completion of the proof of \Cref{lem:main_ind}, following the steps outlined in \Cref{subsec:mainproofs}. With the a priori $G$-loop bound \eqref{lRB1} and the weak local law \eqref{Gtmwc} established in Step 1, the remainder of the proof closely parallels the approach used in \cite[Section 5]{Band1D} and \cite[Section 7]{RBSO1D}, based on the tools developed in the preceding sections. Specifically, we rely on the flow (\Cref{zztE}), the loop hierarchy (\Cref{lem:SE_basic}), the deterministic estimates in \Cref{lem_propTH}, various properties of $\cK$-loops---including the upper bound (\Cref{ML:Kbound}), Ward's identity (\Cref{lem_WI_K}), the identity \eqref{eq:pure_sum} for pure loops, and the tree representation (\Cref{tree-representation})---as well as the resolvent entry estimates (\Cref{lem G<T2,lem_GbEXP}). As a consequence, we will outline the main ideas in the body of the text, focusing on the primary differences in the arguments, particularly those associated with handling the pure $G$-loops in Steps 3 and 4 (\Cref{sec:inductive_step,sec:step4}). For the remaining technical details in the proof, we will either omit them due to their similarity to those in \cite{Band1D,RBSO1D}, or defer some key proofs to \Cref{sec:main_appd} for the reader’s convenience. Finally, the proof of Step 6 is largely independent of the previous steps and is therefore presented separately in \Cref{sec:pf_step6} in the appendix. 

For clarity of presentation, we fix an arbitrary parameter $\sE=\sE_k$ with $k \in \qqq{0, N^{10}}$, and suppress the argument $\sE$ in various notations throughout the proof. By a union bound, all estimates established in each step hold uniformly for all $\sE=\sE_k$ with $k \in \qqq{0, N^{10}}$. Moreover, using an argument similar to that in \Cref{claim:uniform}, these estimates extend uniformly to all $\sE\in [-2+c_\sE,\sE_0]$.

\subsection{Dynamics of the $G$-loops}
We begin by deriving a representation of the $G$-loop dynamics, formulated using Duhamel’s principle and certain evolution kernels, which we introduce below.
For any fixed $\fn\in \N$, combining the loop hierarchy \eqref{eq:mainStoflow} and the equation \eqref{pro_dyncalK} for primitive loops, we obtain that
\begin{align}\label{eq_L-K-1}
       \dd(\mathcal{L} - \mathcal{K})^{(\fn)}_{t, \boldsymbol{\sigma}, \ba} 
    =&~ W^d \sum_{1 \leq k < l \leq \fn} \sum_{[a],[b]} 
   (\mathcal{L} - \mathcal{K})^{(\fn+k-l+1)}_{t, \cutL^{[a]}_{k, l}\left(\boldsymbol{\sigma},\, \ba\right)} S^{\LK}_{[a][b]}
   \mathcal{K}^{(l-k+1)}_{t,\cutR^{[b]}_{k,l}\left(\boldsymbol{\sigma} ,\ba\right)}\, \dd t \nonumber\\
    +&~ W^d \sum_{1 \leq k < l \leq \fn} \sum_{[a],[b]} \mathcal{K}^{(\fn+k-l+1)}_{t, \cutL^{[a]}_{k, \,l}\left(\boldsymbol{\sigma},\ba\right)} 
   S^{\LK}_{[a][b]}\left(\cL-\cK\right)^{(l-k+1)}_{t,\cutR^{[b]}_{k, l}\left(\boldsymbol{\sigma},\ba\right)} \, \dd t \nonumber\\
   +&~ \mathcal{E}^{(\fn)}_{t, \boldsymbol{\sigma}, \ba}\dd t + 
    \dd\mathcal{B}^{(\fn)}_{t, \boldsymbol{\sigma}, \ba} 
    +\mathcal{W}^{(\fn)}_{t, \boldsymbol{\sigma}, \ba}
    \dd t,
\end{align}
where $\mathcal{E}^{(\fn)}_{t, \boldsymbol{\sigma}, \ba}$ is defined by 
\begin{equation}\label{def_ELKLK}
\mathcal{E}^{(\fn)}_{t, \boldsymbol{\sigma}, \ba} :=
    W^d \sum_{1 \leq k < l \leq \fn} \sum_{[a],[b]} 
   (\mathcal{L} - \mathcal{K})^{(\fn+k-l+1)}_{t, \cutL^{[a]}_{k, l}\left(\boldsymbol{\sigma},\, \ba\right)} 
  S^{\LK}_{[a][b]} (\cL-\mathcal{K})^{(l-k+1)}_{t,\cutR^{[b]}_{k,l}\left(\boldsymbol{\sigma} ,\ba\right)}\,  .
\end{equation}
We can rearrange the first two terms on the RHS of \eqref{eq_L-K-1} according to the lengths of $\cK$-loops and rewrite them as a sum of \smash{$\big[\OK^{(\lenk)} (\mathcal{L} - \mathcal{K})\big]^{(\fn)}_{t, \boldsymbol{\sigma}, \ba}\dd t$} for $2\le \lenk \le \fn$, where \smash{$\OK^{(\lenk)}$} is a linear operator defined as 
\begin{align}\label{DefKsimLK}
&\left[\OK^{(\lenk)} (\mathcal{L} - \mathcal{K})\right]_{t, \boldsymbol{\sigma}, \ba}^{(\fn)}:= W^d \sum_{1\le k < l \leq \fn : l-k=\lenk-1} \sum_{[a],[b]} 
   (\mathcal{L} - \mathcal{K})^{(\fn-\lenk+2)}_{t, \cutL^{[a]}_{k, l}\left(\boldsymbol{\sigma},\, \ba\right)} 
   S^{\LK}_{[a][b]}\mathcal{K}^{(\lenk)}_{t,\cutR^{[b]}_{k,l}\left(\boldsymbol{\sigma} ,\ba\right)} \nonumber\\
&\quad +  W^d \sum_{1 \leq k < l \leq \fn:l-k=\fn-\lenk+1} \sum_{[a],[b]} \mathcal{K}^{(\lenk)}_{t, \cutL^{[a]}_{k, \,l}\left(\boldsymbol{\sigma},\ba\right)} S^{\LK}_{[a][b]}
   \left(\cL-\cK\right)^{(\fn-\lenk+2)}_{t,\cutR^{[b]}_{k, l}\left(\boldsymbol{\sigma},\ba\right)}  .
 %   \\   
 % \sum_{l_{\cal K}=2}^n  W^d \cdot \sum_{1 \leq k < l \leq n} \sum_{[a]}  
 %   \Big( 
 %   (\mathcal{L} - \mathcal{K})_{t,\;  \mathcal{G}^{(a), \,L}_{k, \,l}\left(\boldsymbol{\sigma},\, \ba\right)} 
 %   \cdot S^{(B)}_{ab}
 %   \cdot 
 %   \mathcal{K}_{t,\;  \mathcal{G}^{(b), \,R}_{k, \,l}\left(\boldsymbol{\sigma},\, \ba\right)} \cdot \textbf{1}(\text{length of $\mathcal{K}$ loop equals } l_\mathcal{K})   ,
  % + \left(\mathcal{K} \iff \mathcal{L-K} \right) \Big)     := \sum_{l_{\cal K}=2}^n\Big[\mathcal{K} \sim (\mathcal{L} - \mathcal{K})\Big]^{l_\mathcal{K}}_{t, \boldsymbol{\sigma}, \ba}.
\end{align}
Taking out the leading term with $\lenk =2$, we can rewrite \eqref{eq_L-K-1} as
\begin{align}\label{eq_L-Keee}
    \dd(\mathcal{L} - \mathcal{K})^{(\fn)}_{t, \boldsymbol{\sigma}, \ba} = &\left[\OK^{(2)} (\mathcal{L} - \mathcal{K})\right]^{(\fn)}_{t, \boldsymbol{\sigma}, \ba} \, \dd t+\sum_{\lenk=3}^\fn \left[\OK^{(\lenk)} (\mathcal{L} - \mathcal{K})\right]^{(\fn)}_{t, \boldsymbol{\sigma}, \ba}\, \dd t \nonumber\\
    &+ \mathcal{E}^{(\fn)}_{t, \boldsymbol{\sigma}, \ba}\dd t + \dd\mathcal{B}^{(\fn)}_{t, \boldsymbol{\sigma}, \ba} +\mathcal{W}^{(\fn)}_{t, \boldsymbol{\sigma}, \ba}\dd t.
\end{align} 
%Clearly, we can view $\Big[\mathcal{K} \sim (\mathcal{L} - \mathcal{K})\Big]^{l_\mathcal{K} = 2}_{t, \boldsymbol{\sigma}, \ba}$ as a linear transform of  the  tensor $(\mathcal{L} - \mathcal{K})_{t, \boldsymbol{\sigma}, \ba}$.

\begin{definition}[Evolution kernel]\label{DefTHUST}
For $t\in [0,1]$ and $\bsig=(\sigma_1,\ldots,\sig_\fn)\in \{+,-\}^\fn$, define the linear operator \smash{${\thn}^{(\fn)}_{t, \boldsymbol{\sigma}}$} acting on $\fn$-dimensional tensors \smash{${\cal A}$} as follows:
\begin{align}\label{def:op_thn}
    \left({{\thn}}^{(\fn)}_{t, \boldsymbol{\sigma}} \circ \mathcal{A}\right)_{\ba} & = \sum_{i=1}^\fn \sum_{[b_i]\in\Zn} \left(\frac{m(\sig_i)m(\sig_{i+1})}{1 - t m(\sig_i)m(\sig_{i+1})S^\LK}\right)_{[a_i] [b_i]}  \mathcal{A}_{\ba^{(i)}([b_i])},
\end{align} 
where $\ba=([a_1],\ldots, [a_\fn])\in (\Zn)^\fn$, $\ba^{(i)}([b_i])$ is defined as 
\begin{align}    
    \ba^{(i)}([b_i]):= ([a_1], \ldots, [a_{i-1}], [b_i], [a_{i+1}], \ldots, [a_\fn]),
\end{align}
and we recall the convention that $\sigma_{\fn+1}=\sig_1$. Then, the evolution kernel corresponding to \smash{${\thn}^{(\fn)}_{t, \boldsymbol{\sigma}}$} is given by 
    \begin{align}\label{def_Ustz}
        \left(\mathcal{U}_{s, t, \boldsymbol{\sigma}}^{(\fn)} \circ \mathcal{A}\right)_{\ba} = \sum_{\mathbf{b} = ([b_1], \ldots, [b_{\fn}])} \prod_{i=1}^\fn \left(\frac{1 - s \cdot m(\sig_i)m(\sig_{i+1}) S^\LK }{1 - t \cdot m(\sig_i)m(\sig_{i+1}) S^\LK}\right)_{[a_i] [b_i]} \cdot \mathcal{A}_{\mathbf{b}} . 
    \end{align}
\end{definition}

By the definition of the 2-$\cK$ loop in \eqref{Kn2sol}, we observe that 
\begin{align*}
\frac{\dd }{\dd t}\left(\mathcal{U}_{s, t, \boldsymbol{\sigma}}^{(\fn)} \circ \mathcal{A}\right)_{\ba} =     \left({\thn}_{t, \boldsymbol{\sigma}}^{(\fn)} \circ \mathcal{U}_{s, t, \boldsymbol{\sigma}}^{(\fn)} \circ \mathcal{A}\right)_{\ba} = \left[\OK^{(2)} (\mathcal{U}_{s, t, \boldsymbol{\sigma}}^{(\fn)} \circ \mathcal{A})\right]^{(\fn)}_{t, \boldsymbol{\sigma}, \ba}.
\end{align*}
With Duhamel's principle, we can derive from \eqref{eq_L-Keee} the following equation for $s\le t$: 
\begin{align}\label{int_K-LcalE}
   & (\mathcal{L} - \mathcal{K})^{(\fn)}_{t, \boldsymbol{\sigma}, \ba}  =
    \left(\mathcal{U}^{(\fn)}_{s, t, \boldsymbol{\sigma}} \circ (\mathcal{L} - \mathcal{K})^{(\fn)}_{s, \boldsymbol{\sigma}}\right)_{\ba} + \sum_{l_\mathcal{K} =3}^\fn \int_{s}^t \left(\mathcal{U}^{(\fn)}_{u, t, \boldsymbol{\sigma}} \circ \Big[\OK^{(\lenk)} (\mathcal{L} - \mathcal{K})\Big]^{(\fn)}_{u, \boldsymbol{\sigma}}\right)_{\ba} \dd u \nonumber \\
    &+ \int_{s}^t \left(\mathcal{U}^{(\fn)}_{u, t, \boldsymbol{\sigma}} \circ \mathcal{E}^{(\fn)}_{u, \boldsymbol{\sigma}}\right)_{\ba} \dd u  + \int_{s}^t \left(\mathcal{U}^{(\fn)}_{u, t, \boldsymbol{\sigma}} \circ \cW^{(\fn)}_{u, \boldsymbol{\sigma}}\right)_{\ba} \dd u + \int_{s}^t \left(\mathcal{U}^{(\fn)}_{u, t, \boldsymbol{\sigma}} \circ \dd \cB^{(\fn)}_{u, \boldsymbol{\sigma}}\right)_{\ba} .
\end{align}
Furthermore, let $T$ be a stopping time with respect to the matrix Brownian motion $\{H_t\}$ in \eqref{MBM}, and denote $\tau:= T\wedge t$. Then, we have a stopped version of \eqref{int_K-LcalE}:  
\begin{align}\label{int_K-L_ST}
&(\mathcal{L} - \mathcal{K})^{(\fn)}_{\tau, \boldsymbol{\sigma}, \ba} =
    \left(\mathcal{U}^{(\fn)}_{s, \tau, \boldsymbol{\sigma}} \circ (\mathcal{L} - \mathcal{K})^{(\fn)}_{s, \boldsymbol{\sigma}}\right)_{\ba} + \sum_{l_\mathcal{K} =3}^\fn \int_{s}^\tau \left(\mathcal{U}^{(\fn)}_{u, \tau, \boldsymbol{\sigma}} \circ \Big[\OK^{(\lenk)} (\mathcal{L} - \mathcal{K})\Big]^{(\fn)}_{u, \boldsymbol{\sigma}}\right)_{\ba} \dd u \nonumber \\
    &+ \int_{s}^\tau \Big(\mathcal{U}^{(\fn)}_{u, \tau, \boldsymbol{\sigma}} \circ \mathcal{E}^{(\fn)}_{u, \boldsymbol{\sigma}}\Big)_{\ba} \dd u  + \int_{s}^\tau \Big(\mathcal{U}^{(\fn)}_{u, \tau, \boldsymbol{\sigma}} \circ \cW^{(\fn)}_{u, \boldsymbol{\sigma}}\Big)_{\ba} \dd u + \int_{s}^\tau \Big(\mathcal{U}^{(\fn)}_{u, \tau, \boldsymbol{\sigma}} \circ \dd \cB^{(\fn)}_{u, \boldsymbol{\sigma}}\Big)_{\ba} . 
\end{align}
 %\end{lemma}
We will utilize the above two equations to estimate the $(\mathcal{L} - \mathcal{K})$-loops. 

Before the analysis, we introduce the following notation that corresponds to the quadratic variation of the martingale term (recall \eqref{def_Edif}).

\begin{definition}\label{def:CALE} 
For $t\in [0,1]$ and $\bsig=(\sigma_1,\ldots,\sig_\fn)\in \{+,-\}^\fn$, we introduce the $(2\fn)$-dimensional tensor
  \begin{align*}
 \left( \cB\otimes  \cB \right)^{(2\fn)}_{t, \boldsymbol{\sigma}, \ba, \ba'} \; := \; &
 \sum_{k=1}^\fn \left( \cB\times  \cB \right)^{(k)}_{t, \,\boldsymbol{\sigma}, \ba, \ba'},\quad \forall \ba=([a_1],\ldots, [a_\fn]) ,\ \ba'=([a_1'],\ldots, [a_\fn']), %\in (\Zn)^\fn,
 \end{align*}
 where $\left( \cB\times  \cB \right)^{(k)}_{t, \,\boldsymbol{\sigma}, \ba, \ba'}$ is defined as 
\be\label{defEOTE}
 \left( \cB\times  \cB \right)^{(k)}_{t, \boldsymbol{\sigma}, \ba, \ba'} :
 =   W^d \sum_{[b],[b']} S^{\LK}_{[b][b']}{\cal L}^{(2\fn+2)}_{t, (\boldsymbol{\sigma}\times\overline\bsig)^{(k) },(\ba\times \ba')^{(k )}([b],[b'])}.
\ee
Here, ${\cal L}^{(2\fn+2)}$ denotes a $(2\fn+2)$-loop  
%where the loop ${\cal L}_{t, \boldsymbol{\sigma}^{(k ) },\ba^{(k )}}$ is 
obtained by cutting the $k$-th edge of ${\cal L}^{(\fn)}_{t,\boldsymbol{\sigma},\ba}$ and then gluing it (with indices $\ba$) with its conjugate loop (with indices $\ba'$) along the new vertices $[b]$ and $[b']$. Formally, its expression is written as: 
\begin{align*}
   & {\cal L}^{(2\fn+2)}_{t, (\boldsymbol{\sigma}\times\overline\bsig)^{(k) },(\ba\times \ba')^{(k )}([b],[b'])}:=  \bigg \langle \prod_{i=k}^\fn \left(G_{t}(\sigma_i) E_{[a_i]}\right) \cdot \prod_{i=1}^{k-1} \left(G_{t}(\sigma_i) E_{[a_i]}\right) \cdot G_t(\sigma_k) E_{[b]} G_t(-\sigma_k) \\
    &\qquad \times \prod_{i={1}}^{k-1} \left(E_{[a_{k-i}']} G_{t}(-\sigma_{k-i}) \right)\cdot \prod_{i=k}^{\fn} \left(E_{[a_{\fn+k-i}']} G_{t}(-\sigma_{\fn+k-i}) \right) E_{[b']}\bigg\rangle ,
\end{align*}
%Hence, there are $n$ indices between $b$ and $b'$ so that the following intgrarted loop hierarchy $\ba^{(k )}(n)=b', \; \ba^{(k )}(2n )=b $, i.e., 
where the notations $(\boldsymbol{\sigma}\times\overline \bsig)^{(k) }$ and $(\ba\times \ba')^{(k )}$ represent (with $\overline\bsig$ denoting $(-\sig_1,\ldots,-\sig_\fn)$)
\begin{align}\label{def_diffakn_k}
   & (\ba\times \ba')^{(k )}([b],[b'])=( [a_k],\ldots, [a_\fn], [a_1],\ldots [a_{k-1}], [b], [a'_{k-1}],\ldots [a_1'], [a_\fn']\cdots [a'_{k}],[b']),\nonumber \\
  &  (\boldsymbol{\sigma}\times \overline\bsig)^{(k )}=(  \sigma_k, \ldots \sigma_\fn,   \sigma_1,\ldots ,\sigma_{k}, -\sigma_{k}, \ldots, -\sigma_1 ,   -\sigma_\fn, \ldots, -\sigma_{k }).
\end{align}
We remark that the symbols ``$\otimes$" and ``$\times$" in the above notations do not represent any kind of ``products". 
\end{definition}

Under the above notation, the following lemma is an immediate consequence of the Burkholder-Davis-Gundy inequality. 
\begin{lemma}[Lemma 5.5 of \cite{Band1D}]\label{lem:DIfREP}
Let $T$ be a stopping time with respect to the matrix Brownian motion $\{H_t\}$, and denote $\tau:= T\wedge t$. Then, for any fixed $p\in \N$, we have %that 
 \begin{equation}\label{alu9_STime}
   \mathbb{E} \left [    \int_{s}^\tau\left(\mathcal{U}^{(\fn)}_{u, \tau, \boldsymbol{\sigma}} \circ \dd\cB^{(\fn)}_{u, \boldsymbol{\sigma}}\right)_{\ba} \right ]^{2p} 
   \lesssim 
   \mathbb{E} \left(\int_{s}^\tau 
   \left(\left(
   \mathcal{U}^{(\fn)}_{u, \tau,  \boldsymbol{\sigma}}
   \otimes 
   \mathcal{U}^{(\fn)}_{u, \tau,  \overline{\boldsymbol{\sigma}}}
   \right) \;\circ  \;
   \left( \cB \otimes  \cB \right)^{(2\fn)}
   _{u, \boldsymbol{\sigma}  }
   \right)_{\ba, \ba}\dd u 
   \right)^{p}
  \end{equation} 
  where $\mathcal{U}^{(\fn)}_{u, \tau,  \boldsymbol{\sigma}} \otimes    \mathcal{U}^{(\fn)}_{u, \tau,  \overline{\boldsymbol{\sigma}}}$ denotes the tensor product of the evolution kernel in \eqref{def_Ustz}: 
\begin{align*}
\left[\left(
   \mathcal{U}^{(\fn)}_{u, \tau,  \boldsymbol{\sigma}}
   \otimes 
   \mathcal{U}^{(\fn)}_{u, \tau,  \overline{\boldsymbol{\sigma}}}
   \right)\circ \cal A\right]_{\ba,\ba'}&=
   \sum_{\mathbf{b},\mathbf{b}'} \prod_{i=1}^\fn \left(\frac{1 - u m(\sig_i)m(\sig_{i+1}) S^\LK}{1 - t m(\sig_i)m(\sig_{i+1}) S^\LK}\right)_{[a_i] [b_i]} \\
  & \times \prod_{i=1}^\fn \left(\frac{1 - u m(-\sig_i)m(-\sig_{i+1}) S^\LK}{1 - t m(-\sig_i)m(-\sig_{i+1}) S^\LK}\right)_{[a'_i] [b'_i]} \mathcal{A}_{\mathbf{b},\mathbf{b}'}    
\end{align*}
for any $(2\fn)$-dimensional tensor ${\cal A}$ and $\mathbf b=([b_1],\ldots, [b_\fn])$, $\mathbf b'=([b_1'],\ldots, [b_\fn'])$.
\end{lemma}

%\subsection{Evolution kernel estimates}\label{ks}

For the proofs in the subsequent steps, we need to control the terms in equations \eqref{int_K-LcalE} and \eqref{int_K-L_ST}, which requires some estimates on the $(\infty\to\infty)$-norm on the evolution kernel in \Cref{DefTHUST}. These estimates have been established in \cite{Band1D,Band2D,RBSO1D} in the setting of random band matrices (or Wegner orbital models) within the bulk regime. Similar estimates also hold in the edge regime. 

\begin{lemma}[Lemma 7.1 of \cite{Band1D}]
\label{lem:sum_Ndecay}
Let ${\cal A}: (\Zn)^{\fn}\to \mathbb C$ be an $\fn$-dimensional tensor for a fixed $\fn\in \N$ with $\fn\ge 2$. Then, for each $0\le s \le t < 1$, we have that  
\begin{align}\label{sum_res_Ndecay}
   \| {\cal U}_{s,t,\boldsymbol{\sigma}}\;\circ {\cal A}\|_{\infty} \prec \left(  \frac{1-s}{1-t}\right)^{\fn }\cdot \|{\cal A}\|_{\infty} , 
\end{align}
where the $L^\infty$-norm of ${\cal A}$ is defined as $\|{\cal A}\|_{\infty}=\max_{\ba}|\cal A_{\ba}|$.
\end{lemma}

If a tensor $\cal A$ exhibits faster-than-polynomial decay at scales larger than $\ell_s$, then we obtain a stronger bound than the $({\infty\to \infty})$-norm bound given by \eqref{sum_res_Ndecay}. This bound can be further improved if $\cal A$ satisfies certain sum zero property or symmetry.

\begin{lemma}\label{lem:sum_decay}
Let ${\cal A}: (\Zn)^{\fn}\to \mathbb C$ be an $\fn$-dimensional tensor for a fixed $\fn\in \N$ with $\fn\ge 2$. Suppose it satisfies the following property for some small constant $\e\in(0,1)$ and large constant $D>1$,  
\begin{equation}\label{deccA0}
\max_{i,j\in \Zn}|[a_i]-[a_j]|\ge W^{\e}\ell_s \ \ \text{for} \ \ \ba=([a_1],\ldots, [a_\fn])  \implies  |\cal A_{\ba}|\le W^{-D} .
\end{equation}
Fix any $0\le s \le t < 1$ such that $(1-t)/(1-s)\ge W^{-d}$. 
There exists a constant $C_\fn>0$ that does not depend on $\e$ or $D$ such that the following bound holds (note  $\eta_t\ell_t^d\lesssim\eta_s\ell_s^d$ by \eqref{eq:kappat} and the definition of $\ell_t$ in \eqref{eq:ellt}):
\begin{align}\label{sum_res_1}
    \left\|{\cal U}^{(\fn)}_{s,t,\boldsymbol{\sigma}} \circ {\cal A}\right\|_\infty \le W^{C_\fn\e}\frac{\ell_t^d }{\ell_s^d }\left(\frac{\ell_s^d |1-s|}{\ell_t^d |1-t|}\right)^{\fn} \|{\cal A}\|_\infty
    +W^{-D+C_\fn}.
\end{align}
%In addition, stronger bounds hold in the following cases:  
% \begin{itemize}
% \item[(I)] If we have $\sigma_1=\sigma_2$ for $\boldsymbol{\sigma}=(\sigma_1,\cdots, \sigma_\fn),$ then
% \begin{align}\label{sum_res_2_NAL}
%  \left\|{\cal U}^{(\fn)}_{s,t,\boldsymbol{\sigma}} \circ {\cal A}\right\|_\infty \le W^{C_\fn\e}\frac{\omega_s}{\omega_t}  \left(\frac{\ell_s^d|1-s|}{\ell_t^d|1-t|}\right)^{\fn-1}  \|{\cal A}\|_{\infty} 
%    +W^{-D+C_\fn}
% \end{align} 
% for a constant $C_\fn>0$ that does not depend on $\e$ or $D$.
In addition, a stronger bound holds if ${\cal A}$ also satisfies the following sum zero property: 
\begin{align}\label{sumAzero}
 \sum_{[a_2],\ldots,[a_\fn]\in \Zn}{\cal A}_{\ba}=0, \quad \forall [a_1]\in \Zn \, .
\end{align}
More precisely, under \eqref{sumAzero}, there exists a constant $C_\fn>0$ that does not depend on $\e$ or $D$ such that
\begin{align}\label{sum_res_2}
    \left\|{\cal U}^{(\fn)}_{s,t,\boldsymbol{\sigma}} \circ {\cal A}\right\|_\infty \le W^{C_\fn\e} \frac{\ell_t^{d-1}}{\ell_s^{d-1}} \left(\frac{\ell_s^d|1-s|}{\ell_t^d|1-t|}\right)^{\fn}  \|{\cal A}\|_{\infty} 
   +W^{-D+C_\fn}.
\end{align} 
This estimate can be further improved in the following cases:
\begin{itemize}
\item[(I)] If ${\cal A}$ satisfies \eqref{sumAzero} and the following symmetry condition:
\be\label{eq:A_zero_sym}
\qquad\ \ \cal A_{([a],[a]+[b_2],\ldots, [a]+[b_\fn])} = \cal A_{([a],[a]-[b_2],\ldots, [a]-[b_\fn])},\ \  \forall [a],[b_2],\ldots, [b_\fn]\in \Zn \, ,
\ee
then the bound \eqref{sum_res_2} can be improved to:
\begin{align}\label{sum_res_2_sym}
    \left\|{\cal U}^{(\fn)}_{s,t,\boldsymbol{\sigma}} \circ {\cal A}\right\|_\infty \le W^{C_\fn\e} \left(\frac{\ell_s^d|1-s|}{\ell_t^d|1-t|}\right)^{\fn}  \|{\cal A}\|_{\infty} 
   +W^{-D+C_\fn}.
\end{align}

\item[(II)] Given an even $\fn=2k\ge 2$, suppose ${\cal A}$ satisfies the following \emph{double sum zero property}:
\be\label{eq:doublesumzero}
\qquad \sum_{[a_2],\ldots,[a_k]} \cal A_{\ba,\mathbf b}=0,  \ \ \text{and}\ \ \sum_{[a_2],\ldots,[a_k]} \cal A_{\mathbf b, \ba}=0, \quad \forall [a_1]\in \Zn, \ \mathbf b\in (\Zn)^{k}, 
\ee
where we denote $\ba=([a_1],\ldots,[a_k])$ and $\mathbf b=([b_1],\ldots,[b_k])$. Then, the bound \eqref{sum_res_2} can be improved to:
\begin{align}\label{sum_res_2_doublezero}
	\left\|{\cal U}^{(\fn)}_{s,t,\boldsymbol{\sigma}} \circ {\cal A}\right\|_\infty \le W^{C_\fn\e} \left(\frac{\ell_s^d|1-s|}{\ell_t^d|1-t|}\right)^{\fn}  \|{\cal A}\|_{\infty} +W^{-D+C_\fn}.
\end{align} 
%where $C_\fn>0$ is a constant that does not depend on $\e$ or $D$.
 \end{itemize}
\end{lemma} 
\begin{proof}
The proof of this lemma is analogous to those for \cite[Lemma 7.3]{Band1D}, \cite[Lemma 7.3]{Band2D}, and \cite[Lemmas 7.7 and A.5]{RBSO1D}, utilizing the estimates \eqref{prop:ThfadC}, \eqref{prop:BD1}, and \eqref{prop:BD2}. We omit the details.
\end{proof}

\subsection{Step 2: Sharp local law and a priori 2-$G$ loop estimate}

In this section, we focus on the $2$-$G$ loops with $\fn=2$ and $\boldsymbol{\sigma}=(+,-)$, while the case $\boldsymbol{\sigma}=(-,+)$ also follows by taking the matrix transposition. 

Given $u\in[s,t]$, $0\le \ell\le n$, and a sufficiently large constant $D>0$, we introduce the following functions to control the tail behavior of $(\cL-\cK)^{(2)}$:
\begin{align}
 {\cal T}^{(\cal L-K)}_u (\ell)&:= \max_{[a],[b] :  |[a]-[b]|\ge \ell} \left| ({\cal L-\cK})^{(2)}_{u,  (+,-),  ( [a], [b])}\right|, 
      \label{def_WTu} \\
{\cal T} _{u, D} (\ell) &:=  (W^d\ell_u^d\eta_u)^{-2}
\exp \Big(- \big( \ell /\ell_u \big)^{1/2} \Big)+W^{-D}.\label{def_WTuD}
\end{align} 
 % Note ${\cal T}^{(\cal L-K)}_u$ and ${\cal T} _{u, D}$ are non-increasing functions in $\ell$:
 % $$ 
 % {\cal T}^{(\cal L-K)}_u (\ell_1)\ge {\cal T}^{(\cal L-K)}_u (\ell_2),\quad {\cal T} _{u,D} (\ell_1)\ge{\cal T} _{u,D}(\ell_2),\quad \text{for}\quad 0\le \ell_1\le \ell_2.
 % $$
Moreover, we denote the ratio between them by  
\begin{equation}\label{def_Ju}
     {\cal J} _{u,D} (\ell):=   {\cal T}^{(\cal L-K)}_u (\ell) \big/{\cal T}_{u,D} (\ell) +1
\end{equation} 
For the proof, let ${\cal J}^* _{u,D}$ be a \emph{deterministic} control parameter such that 
\be\label{eq:def_new_J*}
\max_{0\le \ell \le n}   {\cal J} _{u,D} (\ell) \prec {\cal J}^* _{u,D}.
\ee
Moreover, we introduce an intermediate scale parameter \smash{$\ell^{*}_u := (\log W)^{3/2}\ell_u$}. 
Note that on the scale $\ell_u^*$, the propagators $\Theta_u^{\bsig}$ is exponentially small (i.e., faster than any polynomial decay) by \Cref{lem_propTH}, whereas ${\cal T}_{u,D}$ is not (i.e., slower than any polynomial decay). To show the estimate \eqref{Eq:Gdecay_w}, we bound the terms in \eqref{int_K-L_ST} with $\fn=2$ as follows.

% \begin{claim}[Lemma 5.6 of \cite{Band1D}]\label{lem_dec_calE_0}
% For any large constant $D>0$ and small constant $\delta>0$, the following estimates hold for all $u\in [0,t]$ if $|[b] -[a]|\ge \delta \ell_t^*$:
%  \be \label{Kell*}
%    %|b -a|\ge \delta \cdot \ell_t^* &\implies 
% \left({\Theta}_t^{\bsig}\right)_{[a][b]} \le W^{-D} ,\quad    \left(\frac{1 - u m(\sig_1)m(\sig_2)S^\LK }{1 - t m(\sig_1)m(\sig_2)S^\LK }\right)_{[a] [b]} \le W^{-D},\quad \forall \bsig\in \{+,-\}^2;
% \ee
% \be
% {\cal L}^{(2)}_{u
% , \bsig,([a], [b])} \prec {\cal J}^*_{u
% ,D}\cdot{\cal T}_{u,D}(|[a]-[b]|) ,\quad \text{for}\ \ \bsig=(+,-). \label{auiwii}
% \ee
% Furthermore, for any constant $C>0$, we have 
% \begin{equation}\label{Tell*}
%     {\cal T}_{t,D} \left(\ell-C\ell_t^*\right)\prec   {\cal T}_{t,D} \left(\ell \right),\quad \forall \ell\ge 0.
% \end{equation} 
% \end{claim}

%The proof of the following lemma will be given in \Cref{subsec:pf_lem_dec_calE}.  

\begin{lemma}\label{lem_dec_calE} 
In the setting of \Cref{lem:main_ind}, suppose that \eqref{lRB1} and \eqref{Gtmwc} hold. Then, for each $\bsig=(+,-)$, \smash{$\ba=([a_1],[a_2])\in (\Zn)^2$}, and $\ba'=([a_1'],[a_2'])$ satisfying 
\begin{align} 
\max_{i=1}^2 |[a_i]-[a_i']|\le \ell_t^*\, , \label{57}
\end{align}
the following estimates hold uniformly for all $u\in [s,t]$ and for any large constant $D\ge 10$: %  $\bsig=(+,-)$, and $\ba=([a_1],[a_2])$, we have that 
\begin{align}\label{res_deccalE_lk}
&\frac{{\cal E}^{(2)}_{u, \boldsymbol{\sigma},\ba}}{ 
  {\cal T}_{t,D}(|[a_1]-[a_2]|)} \prec  \frac{1}{1-u}  \frac{ \big({\cal J}^{*}_{u,D} \big)^{2} }{ W^d\ell_u^d \eta_u\sqrt{\kappa}}
\, ,\\
& \frac{ \cW_{u, \boldsymbol{\sigma},\ba}^{(2)}}{
 {\cal T}_{t,D}(|[a_1]-[a_2]|)}
  \prec \left(\frac{\ell^d_u}{\ell^d_s}\right)^2 \frac{{\bf 1}\left(|[a_1]-[a_2]|\le \ell_u^*\right)}{1-u}  + \frac{1}{1-u} \frac{\big({\cal J}^{*}_{u,D} \big)^{3}}{\left(W^d\ell^d_u \eta_u\sqrt{\kappa} \right)^{\frac13} }
       \, ,
 \label{res_deccalE_wG} \\
\label{res_deccalE_dif}
&\frac{\left(
   \mathcal{B}\otimes  \mathcal{B} \right)^{(4)}
_{u, \bsig, \ba ,\ba'}}{  
 {\cal T}_{t,D}(|[a_1]-[a_2]|)^2}
  \prec \left(\frac{\ell^d_u}{\ell^d_s}\right)^5 \frac{{\bf 1}(|[a_1]-[a_2]|\le 4\ell_t^*)}{1-u} +  \frac{1}{1-u}\frac{\big({\cal J}^{*}_{u,D} \big)^{3}}{(W^d \ell_u^d\eta_u\sqrt{\kappa})^{\frac13} }.   
\end{align} 
\end{lemma}
\begin{proof}%[Proof of \Cref{lem_dec_calE}]
The proof of this lemma, while technical, is nearly identical to that of Lemma 5.7 in \cite{Band1D}, by utilizing the assumptions \eqref{lRB1} and \eqref{Gtmwc}, the $\cK$-loop bound \eqref{eq:bcal_k}, the conditions \eqref{con_st_ind} and \eqref{eq:small_para}, \Cref{lem G<T2,lem_GbEXP}, and the definitions \eqref{def_WTu}--\eqref{eq:def_new_J*}. %and the facts in \Cref{lem_dec_calE_0}. 
Hence, we omit the details. 
%We also remark that the proof of \eqref{res_deccalE_dif} (corresponding to (5.36) in \cite{Band1D}) originally relies on an argument based on a certain $G$-chain estimate (recall the definition of $G$-chains in \eqref{defC=GEG}). However, this estimate is not stated in the present paper. In fact, it is unnecessary and can be replaced by an application of the Cauchy-Schwarz inequality, as demonstrated in the proof of equation (6.53) in \cite{RBSO1D}.
\end{proof}

%We can now complete Step 2 of the proof of \Cref{lem:main_ind} by analyzing the equation \eqref{int_K-L_ST} using the estimates from Lemma \ref{lem_dec_calE} and the evolution kernel estimate in \Cref{lem:sum_Ndecay}.
%\Cref{lem:sum_Ndecay,TailtoTail}.
%this is, proving the estimates \eqref{Gt_bound_flow}--\eqref{Eq:Gdecay_w}. 

Using Lemma \ref{lem_dec_calE}, along with the evolution kernel estimate in \Cref{lem:sum_Ndecay}, we can complete Step 2 in the proof of \Cref{lem:main_ind}. Define the stopping time
\begin{align}\label{eq:def_TTT}
T:=\inf \Big\{u\ge s: \max_{0\le \ell \le n}   {\cal J} _{u,D} (\ell) \ge W^\e\left(|1-s| / |1-t|\right)^4\Big\} 
\end{align}
for a small constant $\e>0$. The estimates \eqref{Gt_bound_flow}---\eqref{Eq:Gdecay_w} are immediate consequences of the following lemma, whose proof is similar to those for \cite[equation (2.76)]{Band1D} and \cite[Lemma 7.11]{RBSO1D}. Therefore, the detailed proof will be deferred to \Cref{sec:pf_step2}. %cite \cite{Bandedge} here

\begin{lemma}\label{lem:pf_step2}  
In the setting of \Cref{lem:main_ind}, suppose \Cref{lem_dec_calE} holds. Then, for sufficiently small constant $\e>0$, we have that $T\ge t$ with high probability. Furthermore, we have a slightly stronger bound that will be used in Step 5: 
\begin{align}\label{53}
\left({\cal L}-{\cal K}\right)^{(2)}_{t, \bsig, \ba} \prec  {\cal T}_{t, D}(|[a_1]-[a_2]|)\cdot  \bigg[ \p{\frac{1-s}{1-t}}^{5/2}  {\bf 1}(|a_1-a_2|\le 6\ell_t^*)+1 \bigg].
\end{align}
\end{lemma}

\begin{proof}[\bf Step 2: Proof of \eqref{Gt_bound_flow}--\eqref{Eq:Gdecay_w}]
By the definition \eqref{def_Ju}, the estimate \eqref{Eq:Gdecay_w} follows readily from the fact that $T\ge t$ with high probability and $\e$ can be arbitrarily small.  
Combining the estimate \eqref{Eq:Gdecay_w} with \eqref{eq:bcal_k}, we see that for any $u\in[s,t]$,
\begin{equation}
\label{lk2safyas}
 \max_{[a], [b]}{\cal L}^{(2)}_{u, (+,-), ([a],[b])}\prec \left(\frac{1-s}{1-t}\right)^4\frac{1}{(W^d\ell_u^d\eta_u)^{2}}+\frac{\sqrt{\kappa}}{W^d\ell_u^d\eta_u}\lesssim \frac{\sqrt{\kappa}}{W^d\ell_u^d\eta_u}  , %\quad \forall u\in [s,t], 
\end{equation}
where we again used the conditions \eqref{eq:small_para} and \eqref{con_st_ind} in the second step. 
Then, applying \Cref{lem G<T2}, we conclude the local laws \eqref{Gt_bound_flow} and \eqref{Gt_avgbound_flow}. 
\end{proof}

%\subsection{Fast decay property}

\subsection{Step 3: Sharp $G$-loop bound}\label{sec:inductive_step}

In Step 2 of the proof of \Cref{lem:main_ind}, we have established an exponential decay of the 2-$G$ loops beyond the scale $\ell_u$ as shown in \eqref{Eq:Gdecay_w}. With \Cref{lem_GbEXP}, we can easily extend this decay to general $G$-loops. 

\begin{definition}[Fast decay property]\label{Def_decay}
Let ${\cal A}: (\Zn)^{\fn}\to \mathbb C$ be an $\fn$-dimensional tensor for a fixed $\fn\in \N$ with $\fn\ge 2$. Given $u\in[s,t]$ and constants $\e,D>0$, we say $\cal A$ satisfies the $(u, \e, D)$-decay property if 
\begin{equation}\label{deccA}
    \max_{i,j}|[a_i]-[a_j]|\ge W^{\e}\ell_u  \implies  {\cal A}_{\ba}=\OO(W^{-D}) \ \ \text{for}\ \ \ba=([a_1],[a_2],\ldots, [a_\fn]) .
\end{equation}
\end{definition}

%We now show that the $G$-loops satisfy the $(u, \e, D)$-decay property under the estimates \eqref{Gt_bound_flow} and \eqref{Eq:Gdecay_w}.
\begin{claim}[Lemma 5.9 of \cite{Band1D}]
\label{lem_decayLoop} 
Assume that \eqref{Gt_bound_flow} and \eqref{Eq:Gdecay_w} hold. For 
any fixed $\fn\ge 2$, constants $\e,D>0$, and $\bsig\in \{+,-\}^\fn$, the $G$-loops \smash{${\cal L}^{(\fn)}_{u,\bsig,\ba}$} and primitive loops \smash{${\cal K}^{(\fn)}_{u,\bsig,\ba}$} satisfy the $(u, \e, D)$-decay property with probability $1-\OO(W^{-D'})$ for any large constant $D'>0$, that is,
\begin{align}\label{res_decayLK}
\mathbb P\left( \max_{\boldsymbol{\sigma}}  \Big(\left|{\cal L}^{(\fn)}_{u,\boldsymbol{\sigma},\ba}\right|+\left|{\cal K}^{(\fn)}_{u,\boldsymbol{\sigma},\ba}\right|\Big)\cdot {\bf 1}_{\max_{  i, j} |[a_i]-[a_j]|\ge W^{\e}\ell_u } \ge W^{-D}\right)\le W^{-D'}. 
\end{align}
 \end{claim}
% \begin{proof}
%Under the assumptions \eqref{Gt_bound_flow} and \eqref{Eq:Gdecay_w}, using \Cref{lem_GbEXP}, we obtain that for any small constant $\e>0$ and large constants $D',D>0$,
% $$
%\mathbb P\left( \max_{x\in [a],y\in [b]}|G_{xy}|
% \cdot{\bf 1}\left( |[a]-[b]|\ge W^\e\ell_u \right) \ge W^{-D}\right)\le W^{-D'} .
% $$
% With this fact, we can easily show that the $G$-loops ${\cal L}^{(\fn)}_{u,\bsig,\ba}$ has the $(u, \e, D)$-decay property. 
% The $(u, \e, D)$-decay property for \smash{${\cal K}^{(\fn)}_{u,\bsig,\ba}$} follows from the tree representation established in \Cref{tree-representation} and the exponential decay of $\Theta_t$-edges in \eqref{prop:ThfadC}.
% \end{proof}

Due to the fast decay property of the $G$ and primitive loops, when the evolution kernels act on them, we can apply \Cref{lem:sum_decay} to bound their $(\infty\to \infty)$-norms, which leads to a crucial improvement over \Cref{lem:sum_Ndecay}. 

For any $\fn\ge 2$, let \smash{$\Xi^{({\cal L})}_{t, \fn}\ge 1$} and \smash{$\Xi^{({\cal L}-{\cal K})}_{t, \fn}\ge 1$} be \emph{deterministic} control parameters for $G$-loops and $(\cL-\cK)$-loops of length $\fn$ such that the following bounds hold: 
\begin{align}\label{def:XiL}
\wh\Xi^{({\cal L})}_{t, \fn} &:= \max_{\boldsymbol{\sigma}\in \{+, -\}^\fn}\max_{\ba\in (\Zn)^\fn} \left|{\cal L}^{(\fn)}_{t, \boldsymbol{\sigma}, \ba} \right|\cdot \kappa^{-1/2}\left(W^d\ell_t^d \eta_t\right)^{\fn-1}\prec \Xi^{({\cal L})}_{t, \fn},\\
\wh\Xi^{({\cal L}-{\cal K})}_{t, \fn} &:=\max_{\boldsymbol{\sigma}\in \{+, -\}^\fn}\max_{\ba\in (\Zn)^\fn} \left|({\cal L}-{\cal K})^{(\fn)}_{t, \boldsymbol{\sigma}, \ba} \right|\cdot \left(W^d\ell_t^d \eta_t\right)^{\fn} \prec \Xi^{({\cal L}-{\cal K})}_{t, \fn}.\label{def:XIL-K}
\end{align}
We define \smash{$\wh\Xi^{({\cal L}-{\cal K})}_{t, 1}$} in the same way as \eqref{def:XIL-K}, which can be bounded by \smash{$\Xi^{({\cal L}-{\cal K})}_{t, 1}:=1$} due to \eqref{Gt_bound_loop1}.
We can control the terms on the RHS of equation \eqref{int_K-LcalE} using these parameters, the initial estimate \eqref{Eq:L-KGt+IND},  \Cref{lem:DIfREP} for the martingale term, and the evolution kernel estimates in \Cref{lem:sum_decay} due to the fast decay property shown in \Cref{lem_decayLoop}. This leads to the following result, whose proof will be deferred to \Cref{sec:pf_STOeq_NQ} since it is similar to those for \cite[Lemma 5.11]{Band1D} and \cite[Lemma 7.14]{RBSO1D}. %cite \cite{Bandedge} here

\begin{lemma}\label{lem:STOeq_NQ} 
Under the assumptions of \Cref{lem:main_ind}, suppose the estimates \eqref{Gt_bound_flow} and \eqref{Eq:Gdecay_w} hold uniformly in $u\in[s,t]$. Then, for any fixed $\fn\ge 2$, the following estimate holds:
\begin{align}\label{am;asoiuw}
\wh\Xi^{({\cal L-\cal K})}_{t,\fn}
\prec&~ \p{\frac{\ell_t^d}{\ell_s^d}}^d \sup_{u\in [s,t]}
  \left(\max_{k\in \qqq{2,\fn-1} }\Xi^{({\cal L-\cal K})}_{u,k}+\Xi^{({\cal L})}_{u,\fn+1}+(\Xi^{(\cal L)}_{u, {2\fn+2}})^{1/2}\right)\nonumber\\
  &~+\p{\frac{\ell_t^d}{\ell_s^d}}^d \sup_{u\in [s,t]}\max_{k\in \qqq{2,\fn} }\frac{\Xi^{({\cal L-\cal K})}_{u,k} \Xi^{({\cal L-\cal K})}_{u,\fn-k+2}}{W^d\ell_u^d\eta_u\sqrt{\kappa}} \, .
\end{align} 
\end{lemma}

%\subsection{Sum zero operator}\label{sec_sumzero}

We next introduce another tool---the sum zero operator ${\cal Q}_{t}$---which enables us to utilize the improved kernel estimates \eqref{sum_res_2} and \eqref{sum_res_2_sym} to achieve improvements over \eqref{am;asoiuw}.

\begin{definition}[Sum zero operator]\label{Def:QtPt} 
Let ${\cal A}: (\Zn)^{\fn}\to \mathbb C$ be an $\fn$-dimensional tensor for a fixed $\fn\in \N$ with $\fn\ge 2$. Define the partial sum operator ${\cal P}$ as 
$$
\left( {\cal P} \circ {\cal A}\right)_{[a_1]}:= \sum_{[a_i]: i\in\qqq{2,\fn}}  {\cal A}_{\ba}. 
$$
Note a tensor $\cal A$ satisfies the sum zero property in \eqref{sumAzero} if and only if $ {\cal P} \circ {\cal A}= 0$. Given $t\in[0,1]$, we define the \emph{sum zero operator}
\be\label{eq:sumzero_op}
 \left({\cal Q}_{t}\circ {\cal A}\right)_{\ba}:={\cal A}_{\ba}-\left({\cal P} \circ {\cal A}\right)_{[a_1]}\dthn^{(\fn)}_{t, \ba},\quad \text{where}\quad 
 \dthn^{(\fn)}_{t, \ba}:=(1-t)^{\fn-1}\prod_{i=2}^\fn \Theta^{(+,-)}_{t,[a_1][a_i]}.\ee
Since $\sum_{[a_i]}\Theta^{(+,-)}_{t,[a_1][a_i]}=(1-t)^{-1}$, we can check that $ {\cal P} \circ  \dthn^{(\fn)}_{t, \ba}=1$, ${\cal P} \circ {\cal Q}_t =0$, and 
\be\label{eq:sum0PA}
{\cal P} \circ {\cal A} = 0\ \ \implies\ \ {\cal P} \circ \left(\thn^{(\fn)}_{t, \boldsymbol{\sigma}}  \circ {\cal A}\right) = 0,
\ee
where we recall the operator \(\thn^{(\fn)}_{t, \boldsymbol{\sigma}} \) defined in \Cref{DefTHUST}. In other words, if \({\cal A}\) satisfies the sum zero property \eqref{sumAzero}, then so does \smash{$\thn^{(\fn)}_{t, \boldsymbol{\sigma}}  \circ {\cal A}$}. 
\end{definition}

By definition, it is easy to see the following bound on the $(\infty\to \infty)$-norm of the sum zero operator. 

\begin{claim}[Lemma 5.13 of \cite{Band1D}]\label{lem_+Q} 
Let ${\cal A}: (\Zn)^{\fn}\to \mathbb C$ be an $\fn$-dimensional tensor for a fixed $\fn\in \N$ with $\fn\ge 2$. If $\cal A$ satisfies the $(t, \e, D)$-decay property, then we have  
\begin{align}\label{normQA}
\left\|{\cal Q}_t\circ \cal A\right\|_{\infty} \le W^{C_\fn \e} \|\cal A\|_\infty +W^{-D+C_\fn}
\end{align}
for a constant $C_\fn$ that does not depend on $\e$ or $D$. Furthermore, if $\|\cal A\|_\infty\le W^C$ for a constant $C>0$, then $\cal A_\fa-({\cal Q}_t\circ \cal A)_{\ba}$ satisfies the $(t, \e', D')$-decay property for any constants $\e',D'>0$.  
\end{claim}

We now use the sum zero operator to get improved estimates for the terms on the RHS of \eqref{int_K-LcalE}.  
%Roughly speaking, we will decompose $\cal A$ as \smash{${\cal A}_{\ba}= ({\cal Q}_{t}\circ {\cal A})_{\ba} + \left({\cal P} \circ {\cal A}\right)_{[a_1]}\dthn^{(\fn)}_{t, \ba}$} using \eqref{eq:sumzero_op}. For the first part, we can get an improvement by using \eqref{sum_res_2}, while for the second part, we can apply Ward's identity to ${\cal P} \circ {\cal A}$. 
We first deal with the partial sum term \smash{$\big[\cal P\circ (\mathcal{L} - \mathcal{K})^{(\fn)}_{t, \boldsymbol{\sigma}}\big]_{[a_1]}\dthn^{(\fn)}_{t, \ba}$}. 

\begin{lemma}\label{lem:partialsumterm}
Under the assumptions of \Cref{lem:main_ind}, suppose the estimates \eqref{lRB1}, \eqref{Gt_bound_flow}, \eqref{Gt_bound_loop1}, and \eqref{Eq:Gdecay_w} hold uniformly in $u\in[s,t]$.
If there exists an $i\in \qqq{2,\fn}$ such that $\sig_i\ne \sig_{i+1}$ (with the convention $\sig_\fn=\sig_1$), then the following estimate holds uniformly for all $u\in[s,t]$:
\begin{align}\label{jywiiwsoks}
 \left[\cal P\circ (\mathcal{L} - \mathcal{K})^{(\fn)}_{u, \boldsymbol{\sigma}}\right]_{[a_1]}  \prec \ell_u^{-d}(W^d\eta_u)^{-\fn} \Xi^{(\mathcal{L}-\mathcal{K})}_{u, \fn-1}.
\end{align} 
If $\bsig$ is a pure loop with $\sig_1=\ldots=\sig_\fn=+$, then we have that 
\begin{align}\label{eq:pure-L-sum}
\left[\cal P\circ (\mathcal{L} - \mathcal{K})^{(\fn)}_{u, \boldsymbol{\sigma}}\right]_{[a_1]}  \prec  \ell_s^{-d}(W^d\eta_u)^{-\fn} \, .
%{\frac{\ell_u^{d}}{\ell_s^d}} \cdot \frac{1}{W^d\ell_u^d\eta_u}\cdot\frac{1}{(W^d\eta_u)^{\fn-1}} .
\end{align}
\end{lemma} 
\begin{proof}
For \eqref{jywiiwsoks}, assume that $\sig_\fn\ne \sig_1$ without loss of generality. Applying \Cref{lem_WI_K} at the vertex $[a_\fn]$, we can write the LHS of \eqref{jywiiwsoks} as 
$$ \frac{1}{2\ii W^d \eta_u}\left[\cal P\circ (\mathcal{L} - \mathcal{K})^{(\fn-1)}_{u, \wh\bsig^{(+,\fn)}}-\cal P\circ (\mathcal{L} - \mathcal{K})^{(\fn-1)}_{u, \wh\bsig^{(-,\fn)}}\right]_{[a_1]}.$$
With \eqref{def:XIL-K}, the two $(\fn-1)$-$G$ loops are controlled by $$(\mathcal{L} - \mathcal{K})^{(\fn-1)}_{u, \wh\bsig^{(\pm,\fn)},\wh \ba^{(\fn)}} \prec (W^d\ell_u^d\eta_u)^{-(\fn-1)}\Xi^{(\mathcal{L}-\mathcal{K})}_{u, \fn-1}.$$ 
Moreover, due to the fast decay property of the $(\cal L-\cal K)$-loops, the partial sums over the remaining $(\fn-2)$ vertices lead to a \smash{$\ell_u^{d(\fn-2)}$} factor up to a negligible error $\OO(W^{-D})$. This leads to \eqref{jywiiwsoks}.

When $\bsig$ is a pure loops with all charges equal to $+$, we have the identity   
$$\left(\cal P\circ \mathcal{L}^{(\fn)}_{u, \boldsymbol{\sigma}}\right)_{[a_1]} = {W^{-(\fn-1)d}}\cdot \frac{1}{(\fn-1)!}\frac{\dd^{\fn-1}}{\dd z_u^{\fn-1}}\avg{G_u(z_u)E_{[a_1]}} .$$
Combining it with \eqref{eq:pure_sum} and using Cauchy's integral formula, we can express the LHS of \eqref{eq:pure-L-sum} as 
\begin{align}
 \left[\cal P\circ (\mathcal{L} - \mathcal{K})^{(\fn)}_{u, \boldsymbol{\sigma}}\right]_{[a_1]} &=  W^{-(\fn-1)d}  \cdot \frac{1}{(\fn-1)!}\frac{\dd^{\fn-1}}{\dd z_u^{\fn-1}} \avg{(G_u(z_u)-M_u(z_u))E_{[a_1]}} \nonumber\\
 &=W^{-(\fn-1)d}  \cdot \frac{1}{2\pi \ii}\oint_\gamma \frac{\avg{(G_u(z)-M_u(z))E_{[a_1]}}}{(z-z_u)^\fn}\dd z,\label{eq;Cauchy-int}
\end{align}
where $\gamma$ denotes a counterclockwise circle around $z_u$: $\gamma=\{z\in \C_+:|z-z_u|= W^{-c}\eta_u\}$ for an arbitrarily small constant $c\in (0, (\fd\wedge\fc)/10)$, and $M_u(z):=m_u(z)I_N$ with $m_u$ defined in \eqref{self_mt}. 

To estimate \eqref{eq;Cauchy-int}, we need to control $\avg{(G_u(z)-M_u(z))E_{[a_1]}}$ for $z\in \Gamma$. To this end, we need to bound the $2$-$G$ loops formed with $G_u(z)$ and apply \Cref{lem G<T2}. Our approach follows the argument used in the proof of \Cref{lem_ConArg}, specifically by applying the resolvent identity \eqref{GGZGG} with  
$$\wt G\equiv G_u(z_u),\quad G\equiv G_u(z)=(H_u-z)^{-1}.$$
Denoting $\Lambda(z):=\|G_u(z)-M_u(z)\|_{\max}$, and using \eqref{GGZGG}, we get that for any $z\in \gamma$, 
\begin{align}
\Lambda(z)&\le \|\wt G(z)-M_u(z_u)\|_{\max}+|m_u(z)-m_u(z_u)| + |z-z_u| \|G\wt G\|_{\max} \nonumber\\
&\lesssim \Psi_u(\sE)+ |z-z_u|/\omega_u + W^{-c} \kappa^{1/4}\Big(\max_{x}\im G_{xx}\Big)^{1/2} \nonumber\\
&\lesssim W^{-c}\sqrt{\kappa} + W^{-c}\kappa^{1/4}(\Lambda+\sqrt{\kappa})^{1/2}\label{eq:Lambdaz}
\end{align}
with high probability. Above, in the second step, we used the local law \eqref{Gt_bound_flow} for \smash{$\wt G$} and applied the Cauchy-Schwarz inequality along with Ward’s identity to control \smash{$\|G\wt G\|_{\max}$} as follows with high probability:
$$ \|G\wt G\|_{\max} \le \eta_u^{-1}\Big(\max_{x}\im \wt G_{xx} \Big)^{1/2}\Big(\max_{x}\im G_{xx} \Big)^{1/2} \lesssim \eta_u^{-1}\kappa^{1/4}\Big(\max_{x}\im G_{xx}\Big)^{1/2} .$$
Moreover, we have used the following fact to control $|m_u(z)-m_u(z_u)|$:  
\be\label{eq:mu'} m_u'(\xi)=\frac{m_u(\xi)^2}{1-tm_u(\xi)^2} \lesssim \omega_u^{-1}\quad \forall \xi \text{ such that } |\xi-z_u|\ll \eta_u \, .\ee 
As a consequence, it also implies that $\im m_u(z)\lesssim\sqrt{\kappa}$ for $z\in \gamma$.
In the third step of \eqref{eq:Lambdaz}, we applied \eqref{eq;smallpsi} and bounded $\im G_{xx}(z)$ by $\im m_u(z)+\Lambda=\OO(\Lambda+\sqrt{\kappa})$. From \eqref{eq:Lambdaz}, we derive that with high probability,
\be\label{eq:initialGT}
\Lambda(z)\lesssim W^{-c}\sqrt{\kappa} \implies \max_{[a]} \left|\avg{\p{\im G_u(z)E_{[a]}}}\right|\lesssim \sqrt{\kappa} \quad  \text{uniformly in } z\in \gamma \, .
\ee

Next, for notational convenience, we denote the $G$-loops formed with \smash{$\wt G$} and $G$ by $\wt{\cal L}$ and $\cal L$, respectively. By assumption, the \smash{$\wt{\cal L}$}-loops satisfy the a priori bounds in \eqref{lRB1}. 
%Let $a_N>0$ be a deterministic parameter satisfying $\sqrt{\kappa}\lesssim a_N \lesssim 1$, and define the corresponding event as $\Omega(z)= \left\{\max_{[a]} \left|\avg{\p{\im G(z)E_{[a]}}}\right|\leq a_N\right\}$. 
% \be\label{eq:omegak2}
% \Omega_k(z,\e)= \left\{\max_{[a]} \left|\avg{\p{\im G(z)E_{[a]}}}\right|\leq a_N, \ \max_{\bsig,\ba}\left|\wt{\cal L}^{(4)}_{u,\bsig,\ba}\right|\le W^\e \bigg(\frac{\ell_u^d}{\ell_s^d}\bigg)^{3} 
%   \frac{\sqrt{\kappa}} {\p{W^d\ell_u^d\eta_{u}}^{3}}\right\}.
% \ee
Then, by repeating the arguments following \eqref{lem_L2Loc} with $\fn=2$ and using \eqref{eq:initialGT}, we obtain that 
$$ \max_{\bsig,\ba}\left|\mathcal{L}^{(2)}_{u, \boldsymbol{\sigma}, \textbf{a}}(z)\right| \prec  \max_{\bsig,\ba}\left|\wt{\mathcal{L}}^{(2)}_{u, \boldsymbol{\sigma}, \textbf{a}}\right| + W^{-2c} {\frac{\ell_u^{d}}{\ell_s^d}} \cdot\frac{\sqrt{\kappa}}{W^d \ell_{u}^d \eta_{u}}  \, . $$
Combining this estimate with the $2$-$G$ loop bound \eqref{lk2safyas} for $\wt{\mathcal{L}}^{(2)}$ (which is a consequence of \eqref{Eq:Gdecay_w}), we conclude that 
\be\label{eq:cont-est_pure}
\max_{\bsig,\ba}\left|\mathcal{L}^{(2)}_{u, \boldsymbol{\sigma}, \textbf{a}}(z)\right|\prec  {\frac{\ell_u^{d}}{\ell_s^d}} \cdot \frac{\sqrt{\kappa}}{W^d\ell_u^d\eta_u }\quad \text{uniformly in } z\in \gamma \, .
\ee 
Next, with \eqref{eq:initialGT} and \eqref{eq:cont-est_pure}, applying \Cref{lem G<T2}, we obtain that 
%for $\OO(1)$ many times, we can obtain that $\max_{[a]} \left|\avg{\p{\im G(z)E_{[a]}}}\right|\leq \im m(\sE)  +\oo(1)\lesssim\sqrt{\kappa}$ with high probability, and 
\be\label{eq:cont-est_pure2}\max_{[a]}\left|\avg{\p{G_u(z)-M_u(z)}E_{[a]}}\right|\prec {\frac{\ell_u^{d}}{\ell_s^d}}\cdot \frac{1}{W^d\ell_u^d\eta_u}\quad \text{uniformly in } z\in \gamma \, .
\ee
Plugging this result into \eqref{eq;Cauchy-int} and controlling the integral yields \eqref{eq:pure-L-sum}, since $c$ is arbitrary. 
\end{proof}

To control the sum zero term ${\cal Q}_t \circ (\mathcal{L} - \mathcal{K})^{(\fn)}_{t, \boldsymbol{\sigma}, \ba}$, we derive from  \eqref{eq_L-Keee} that  
\begin{align} 
&\dd{\cal Q}_t \circ (\mathcal{L} - \mathcal{K})^{(\fn)}_{t, \boldsymbol{\sigma}, \ba} 
= {\cal Q}_t \circ \left[\mathcal{O}_{\cK}^{(2)} (\mathcal{L} - \mathcal{K})\right]^{(\fn)}_{t, \boldsymbol{\sigma}, \ba} 
    + \sum_{l_\mathcal{K}=3}^\fn {\cal Q}_t \circ \Big[\mathcal{O}_{\cK}^{(\lenk)} (\mathcal{L} - \mathcal{K})\Big]^{(\fn)}_{t, \boldsymbol{\sigma}, \ba} + {\cal Q}_t \circ \mathcal{E}^{(\fn)}_{t, \boldsymbol{\sigma}, \ba}\dd t  \nonumber
    \\
    &\qquad  \qquad 
    + {\cal Q}_t \circ \mathcal{W}^{(\fn)}_{t, \boldsymbol{\sigma}, \ba} \dd t
    + {\cal Q}_t \circ \dd\mathcal{B}^{(\fn)}_{t, \boldsymbol{\sigma}, \ba}    - \left[{\cal P} \circ \left(\mathcal{L} - \mathcal{K}\right)^{(\fn)}_{t,\bsig}\right]_{[a_1]} \partial_t \dthn_{t,\ba}^{(\fn)} \dd t.\label{zjuii1}
\end{align} 
Recalling \Cref{DefTHUST}, we can rewrite the first term on the RHS as 
\begin{align}\label{zjuii2}
{\cal Q}_t \circ \Big[\mathcal{O}_{\cK}^{(2)}(\mathcal{L} - \mathcal{K})\Big]^{(\fn)}_{t, \boldsymbol{\sigma},\ba}
& = \thn^{(\fn)}_{t, \boldsymbol{\sigma}} \circ \left[{\cal Q}_t \circ (\mathcal{L} - \mathcal{K})^{(\fn)}_{t, \boldsymbol{\sigma}}\right]_{\ba}  \nonumber\\
&+ \left[{\cal Q}_t , \thn^{(\fn)}_{t, \boldsymbol{\sigma}} \right] \circ (\mathcal{L} - \mathcal{K})^{(\fn)}_{t, \boldsymbol{\sigma}, \ba} \, ,
\end{align}
where $[{\cal Q}_t , \thn^{(\fn)}_{t, \boldsymbol{\sigma}} ] = {\cal Q}_t \circ \thn^{(\fn)}_{t, \boldsymbol{\sigma}} - \thn^{(\fn)}_{t, \boldsymbol{\sigma}}\circ \cal Q_t$ denotes the commutator between ${\cal Q}_t $ and $\thn^{(\fn)}_{t, \boldsymbol{\sigma}}$. Since \({\cal P} \circ {\cal Q}_t = 0\), we notice that the first 5 terms on the RHS of \eqref{zjuii1} satisfy the sum zero property. Using 
$${\cal P}  \circ \dthn^{(\fn)}_{t,  \ba}= \sum_{\ba'}\dthn^{(\fn)}_{t,  \ba}\equiv 1,\quad \text{where}\quad \ba'=([a_2],\cdots, [a_\fn]),$$ 
we see that the last term on the RHS of \eqref{zjuii1} also satisfies the sum zero property:
\be\label{eq:sumzero666}
 {\cal P}  \circ \left\{\left[{\cal P} \circ \left(\mathcal{L} - \mathcal{K}\right)^{(\fn)}_{t,\bsig}\right]_{[a_1]} \partial_t \dthn_{t,\ba}^{(\fn)}\right\}
=\left[{\cal P} \circ \left(\mathcal{L} - \mathcal{K}\right)^{(\fn)}_{t,\bsig}\right]_{[a_1]} {\cal P} \circ \left(\partial_t \dthn_{t,\ba}^{(\fn)}\right)=0.
\ee
Next, due to \eqref{eq:sum0PA}, the first term on the RHS of \eqref{zjuii2} also satisfies the sum zero property. 
Finally, since the LHS of \eqref{zjuii2} has sum zero property, the second term on the RHS of \eqref{zjuii2} also satisfies the sum zero property: 
\begin{align}\label{pqthlk}
   {\cal P}\circ \big[{\cal Q}_t , \thn^{(\fn)}_{t,\boldsymbol{\sigma}} \big]\circ (\mathcal{L} - \mathcal{K})^{(\fn)}_{t, \boldsymbol{\sigma}, \ba}=0 .
\end{align} 
Applying Duhamel's principle to the equation \eqref{zjuii1}, and using the improved estimate \eqref{sum_res_2} to control the resulting expression, we can derive the following result. 
The detailed proof is similar to those for \cite[Lemma 5.14]{Band1D} and \cite[Lemma 7.16]{RBSO1D}, and hence will be deferred to \Cref{sec:pf_STOeq_weak}. %cite \cite{Bandedge} here

%We can control the terms on the RHS of \eqref{int_K-L+Q} by using the improved estimate \eqref{sum_res_2}. This leads to the following lemma.  

\begin{lemma}\label{lem:STOeq_Qt_weak} 
Under the assumptions of \Cref{lem:main_ind}, suppose the estimates \eqref{lRB1}, \eqref{Gt_bound_flow}, and \eqref{Eq:Gdecay_w} hold uniformly in $u\in[s,t]$. For any fixed $\fn\ge 2$, we have that: 
\begin{align}\label{am;asoi222_weak}
 \wh\Xi^{({\cal L-\cal K})}_{t,\fn}
\prec &~ \frac{\ell_t^{d-1}}{\ell_s^{d-1}}\sup_{u\in [s,t]}
  \left(\Xi^{+}_{u} +\max_{k\in \qqq{1,\fn-1} }\Xi^{({\cal L-\cal K})}_{u,k}+\Xi^{({\cal L})}_{u,\fn+1}+ (\Xi^{(\cal L)}_{u, {2\fn+2}})^{1/2}\right) \nonumber\\
  &~+\frac{\ell_t^{d-1}}{\ell_s^{d-1}}\sup_{u\in [s,t]}\max_{k\in \qqq{2, \fn} }\frac{ \Xi^{({\cal L-\cal K})}_{u,k}\Xi^{({\cal L-\cal K})}_{u,\fn-k+2}}{W^d\ell_u^d\eta_u\sqrt{\kappa}} \, .
\end{align}
Here, $\Xi^+_{u}\ge 1$ is a deterministic control parameter satisfying that 
\be\label{eq;Xip}
(W^d\eta_u)^{\fn}\ell_u^d \max_{[a_1]}\left|\left[\cal P\circ (\mathcal{L} - \mathcal{K})^{(\fn)}_{u, \boldsymbol{\sigma}_+}\right]_{[a_1]} \right|\prec \Xi^+_{u}, \ee
where $\bsig_+$ denotes a pure loop with all charges equal to $+$. 
\end{lemma}

To complete the proof of Step 3, we still need to eliminate the factor $\ell_t^{d-1}/\ell_s^{d-1}$ in \eqref{am;asoi222_weak} in dimension 2, as shown in the following lemma. Its proof explores a CLT-type cancellation mechanism developed in \cite[Section 7]{Band2D} and \cite[Lemma 7.17]{RBSO1D}, which yields an additional $\ell_s/\ell_t$ factor that cancels the $\ell_t/\ell_s$ prefactor in \eqref{am;asoi222_weak}. We will briefly describe the proof in \Cref{subsec:STOeq_Qt_pf}. %cite \cite{Bandedge} here

\begin{lemma}\label{lem:STOeq_Qt} 
In the setting of \Cref{lem:STOeq_Qt_weak}, when $d=2$, we have that
\begin{align}\label{am;asoi222}
 \wh\Xi^{({\cal L-\cal K})}_{t,\fn}
 \prec &~\sup_{u\in [s,t]}
  	\left(\Xi^{+}_{u} +\max_{k\in \qqq{1,\fn-1} }\Xi^{({\cal L-\cal K})}_{u,k}+\Xi^{({\cal L})}_{u,\fn+1}+ (\Xi^{(\cal L)}_{u, {2\fn+2}})^{1/2}\right) \nonumber\\
    &~+\sup_{u\in [s,t]}\max_{k\in \qqq{2, \fn} }\frac{ \Xi^{({\cal L-\cal K})}_{u,k}\Xi^{({\cal L-\cal K})}_{u,\fn-k+2}}{W^d\ell_u^d\eta_u\sqrt{\kappa}} .
\end{align}
\end{lemma}

We are now ready to complete the proof of Step 3. In this step, besides the assumptions in Theorem \ref{lem:main_ind}, we have proved the a priori $G$-loop bound \eqref{lRB1} in Step 1, the sharp local laws \eqref{Gt_bound_flow} and \eqref{Gt_avgbound_flow}, the sharp $1$-$G$ loop estimate \eqref{Gt_bound_loop1} (so we can choose \smash{$\Xi_{u,1}^{(\cL-\cK)}=1$}), and the a priori $2$-$G$ loop estimate \eqref{Eq:Gdecay_w} in Step 2. 
%By the averaged local law  \eqref{Gt_avgbound_flow} and the fact \eqref{eq;smallpsi}, we see that \smash{$\wh\Xi_{u,1}^{(\cL-\cK)}\prec 1$} uniformly in $u\in[s,t]$. 
Furthermore, by \eqref{eq:pure-L-sum}, we can choose the parameter \smash{$\Xi_u^+$} in \eqref{eq;Xip} as  $\Xi_u^+=\ell_u^d/\ell_s^d$. 
With these inputs, applying \Cref{lem:STOeq_Qt}, we obtain the following estimate uniformly in $u\in [s,t]$: 
\begin{align}\label{am;asoi333}
\sup_{v\in[s,u]} \wh \Xi^{({\cal L-\cal K})}_{v,\fn}
\prec &~\ell_u^d/\ell_s^d+\sup_{v\in [s,u]}
  \left(\max_{r\in \qqq{2,\fn-1} }\Xi^{({\cal L-\cal K})}_{v,r}+ \Xi^{({\cal L})}_{v,\fn+1}+(\Xi^{(\cal L)}_{v, {2\fn+2}})^{1/2}\right)\nonumber\\
&~+\sup_{v\in [s,u]}\max_{r\in \qqq{2, \fn} } \frac{\Xi^{({\cal L-\cal K})}_{v,r}\Xi^{({\cal L-\cal K})}_{v,\fn-r+2}}{W^d\ell_v^d\eta_v \sqrt{\kappa}}.
\end{align}
(By Lemma \ref{lem:STOeq_Qt}, the estimate \eqref{am;asoi333} holds for each fixed $u\in[s,t]$. Again, using the $N^{-C}$-net and perturbation argument as in \Cref{claim:uniform} extends it uniformly to all $u\in [s,t]$.)
As in Section 5.6 of \cite{Band1D}, we will iterate the estimate \eqref{am;asoi333} to derive the sharp bound \eqref{Eq:LGxb} on $G$-loops, that is, for each fixed $\fn\in \N$, 
\be\label{eq:Xiiter}
\sup_{u\in [s,t]} \wh\Xi^{(\cL)}_{u,\fn} \prec 1.
\ee

By the $\cK$-loop bound \eqref{eq:bcal_k}, $\wh\Xi^{(\cal L)}_{u,\fn}$ and $\wh\Xi^{(\cal L-\cal K)}_{u,\fn}$ bound each other as follows:
\begin{align}\label{rela_XILXILK}
    \wh\Xi^{(\cal L)}_{u,\fn}\lesssim 1+\left(W^d\ell_u^d\eta_u\sqrt{\kappa}\right)^{-1}\cdot \wh\Xi^{({\cal L-\cal K})}_{u, \fn},\ \  \wh\Xi^{({\cal L-\cal K})}_{u, \fn}\lesssim \left(W^d\ell_u^d\eta_u\sqrt{\kappa}\right) \left(\wh\Xi^{(\cal L)}_{u,\fn}+1\right).
\end{align}
Moreover, by the a priori $G$-loop bounds in \eqref{lRB1} and \eqref{Eq:Gdecay_w}, we have the following initial estimate uniformly in $u\in[s,t]$ for any fixed $\fn\ge 2$:
\begin{align}\label{sef8w483r324}
  \wh\Xi^{(\cal L)}_{u,\fn}\prec \p{\frac{\ell_u^d}{\ell_s^d}}^{\fn-1}, \ \ 
 \wh\Xi^{({\cal L-K})}_{u, \fn}\prec \p{\frac{\ell_u^d}{\ell_s^d}}^{\fn-1}(W^d\ell_u^d\eta_u \sqrt{\kappa}),\ \ \wh\Xi^{({\cal L-K})}_{u, 2}\prec \frac{|1-s|^4}{|1-u|^4}.
\end{align}
Then, for $u\in [s,t]$, we define the control parameter
\begin{equation}\label{adsyzz0s8d6}
      \Psi_{u}(\fn,k;[s,t]):=  (W^d\ell_s^d\eta_s\sqrt{\kappa})^{\frac 1 2}+\p{\frac{\ell_t^d}{\ell_s^d}}^{\fn-1}\times
    \begin{cases}
        W^d\ell_u^d\eta_u\sqrt{\kappa} , &   \text{if}\ k=0\\
        (W^d\ell_s^d\eta_s\sqrt{\kappa})^{1-\frac k 4}, &    \text{if}\  k\ge 1 
    \end{cases}.
\end{equation} 
(Note that when $k\ge 1$, $\Psi_{u}(\fn,k;[s,t])$ does not depend on $u$.) The iterations will be performed in both $\fn$ and $k$ at the same time. We summarize the result of each iteration in the next lemma. 
\begin{lemma}\label{lem:iterations}
In the setting of \Cref{lem:main_ind}, suppose the estimates \eqref{am;asoi333} and \eqref{sef8w483r324} hold uniformly in $u\in[s,t]$. Moreover, for any fixed $(\fn,k)\in \N^2$ with $\fn\ge 2$ and $k\ge 1$, suppose the following estimate holds uniformly in $u\in[s,t]$: 
    \be\label{eq:iteration_induc}
    \sup_{v\in[s,u]}\wh \Xi^{({\cal L-K})}_{v,r}\prec   \sup_{v\in[s,u]}\Psi_v(r,l;[s,u])\ee
    for all $(r,l)$ satisfying one of the following conditions: (1) $l=k$ and $2\le r \le \fn-1$; (2) $l=k-1$ and $2\le r\le \fn+2$. Then, we have the following estimate uniformly in $u\in[s,t]$:
    \be\label{eq:iteration_improve}
    \sup_{v\in[s,u]}\wh \Xi^{({\cal L-K})}_{v,\fn}\prec   \Psi_u(\fn,k;[s,u]).\ee
\end{lemma}

\begin{proof}%[\bf Proof of \Cref{lem:iterations}]
The proof of this lemma follows a similar approach to that of equation (5.109) in \cite{Band1D}, utilizing the inductive estimate \eqref{am;asoi333}. 
%However, as noted in the proof of \Cref{lem_dec_calE}, the original argument for equation (5.108) in \cite{Band1D} relies on a $G$-chain estimate, which is not stated in the present paper. This $G$-chain estimate, however, is not essential and can be replaced by an application of the Cauchy–Schwarz inequality, as demonstrated in the proof of Lemma 6.22 in \cite{RBSO1D}. 
Given the close similarity to the argument provided there, we omit the details.
\end{proof}

%Roughly speaking, this lemma states if we have already established a ``good" bound for every shorter $G$-loop of length $r\le \fn-1$ or (case (1)) or a ``weaker" bound for $G$-loops of length $r\le \fn+2$ (case (2)), then we can derive the ``good" bound for all $G$-loops of length $\fn$. 
With \Cref{lem:iterations}, the iterative argument leading to the proof of \eqref{Eq:LGxb} in Step 3 proceeds as follows: 

\begin{proof}[\bf Step 3: Proof of \eqref{Eq:LGxb}]
By \eqref{sef8w483r324}, we initially have a weak bound for $G$-loops of arbitrarily large lengths, which shows that \eqref{eq:iteration_induc} holds with $l=0$ for every $r\in \N$. Applying \Cref{lem:iterations} once, we obtain a slightly improved bound \eqref{eq:iteration_improve} with $k=1$ and $\fn=2$. Then, continuing the iteration in $\fn$ while keeping $k=1$ fixed, we can establish the bound \eqref{eq:iteration_improve} for arbitrarily large $\fn\in \N$ with $k=1$. 
Next, applying the iteration in \Cref{lem:iterations} again yields an even stronger bound \eqref{eq:iteration_improve} with $\fn=2$ and $k=2$. Repeating the iteration in $\fn$ while keeping $k=2$ fixed, we can establish the bound \eqref{eq:iteration_improve} for arbitrarily large $\fn\in \N$ with $k=2$. This process continues, progressively improving the bound to \eqref{eq:iteration_improve} for any $\fn\in \N$ with $k=3$, and so forth.

Given any $(\fn,k)\in \N^2$, by repeating the above argument a finite number of times, we can show that the estimate \eqref{eq:iteration_improve} holds. As a special case, it yields that  
\smash{\(\sup_{u\in[s,t]}\wh \Xi^{({\cal L-K})}_{u,\fn}\prec   \Psi_t(\fn,k;[s,t]). \)}
In particular, if we choose $k$ large enough depending on $\fn$ such that $(\ell_t^d/\ell_s^d)^{\fn-1}\cdot  (W^d\ell_s^d\eta_s\sqrt{\kappa})^{1-k/4} \le (W^d\ell_s^d\eta_s\sqrt{\kappa})^{1/2}$, we get
$$\sup_{u\in[s,t]}\wh\Xi^{({\cal L-K})}_{u,\fn}\prec \Psi_t(\fn,k;[s,t]) \lesssim (W^d\ell_s^d\eta_s\sqrt{\kappa})^{1/2}.$$
Together with \eqref{rela_XILXILK}, it implies the estimate \eqref{eq:Xiiter} under condition \eqref{con_st_ind}, thereby completing the proof of the $G$-loop bound \eqref{Eq:LGxb}.
\end{proof}

\subsection{Step 4: Sharp $(\cL-\cK)$-loop estimate}\label{sec:step4}

To prove \eqref{Eq:L-KGt-flow}, we can once again apply the estimate \eqref{am;asoi222} established in Step 3. Then, we can invoke the $G$-loop bound \eqref{Eq:LGxb}, which allows us to choose the parameters \smash{$\Xi^{(\cal L)}_{v,\fn+1}=1$ and $\Xi^{(\cal L)}_{v,2\fn+2}=1$}. 
Furthermore, to control \smash{$\Xi^{+}_{v}$}, we apply the same argument as the one used below \eqref{eq;Cauchy-int}. More precisely, by employing the sharp $G$-loop bounds from \eqref{Eq:LGxb}, and repeating the argument below \eqref{eq;Cauchy-int}, we obtain the following improvement of \eqref{eq:cont-est_pure2}: 
\be\label{eq:cont-est_pure3}\max_{[a]}\left|\avg{\p{G_u(z)-M_u(z)}E_{[a]}}\right|\prec (W^d\ell_u^d\eta_u)^{-1} \quad \text{uniformly in $z\in \gamma$.}
\ee
Plugging this estimate into \eqref{eq;Cauchy-int} and controlling the integral yields
$$\left[\cal P\circ (\mathcal{L} - \mathcal{K})^{(\fn)}_{u, \boldsymbol{\sigma}_+}\right]_{[a_1]}  \prec \ell_u^{-d} \big(W^d\eta_u\big)^{-\fn}\, ,$$
since $c$ is arbitrary. Thus, recalling \eqref{eq;Xip}, we can choose $\Xi^{+}_{v}=1$. So far, we have seen from \eqref{am;asoi222} that for any fixed $\fn\ge 2$: 
 \begin{align} \label{saww02}
\sup_{v\in[s,u]}\wh\Xi_{v,\fn}^{(\mathcal{L}-\mathcal{K})}  \prec  1+\sup_{v\in [s,u]}
  \left(\max_{r\in\qqq{2,\fn-1}}\Xi^{({\cal L-\cal K})}_{v,r}+\max_{r\in\qqq{2,\fn} }\frac{\Xi^{({\cal L-\cal K})}_{v,r}\Xi^{({\cal L-\cal K})}_{v,\fn-r+2}}{W^d\ell_v^d\eta_v\sqrt{\kappa}}\right).
\end{align}

On the other hand, by the rough bound \eqref{Eq:Gdecay_w} on $2$-$G$ loops, we have 
\be \label{saww02_ini}
\sup_{v\in [s,u]}\wh\Xi_{v,2 }^{(\mathcal{L}-\mathcal{K})}\prec \p{\frac{1-s}{1-u}}^{4}.
\ee
Hence, we can choose $\Xi_{v,2 }^{(\mathcal{L}-\mathcal{K})}\equiv |1-s|^4/|1-u|^{4}$ for $v\in[s,u]$. 
Then, taking $\fn=2$ in \eqref{saww02}, by using this parameter and the condition \eqref{con_st_ind}, we derive the following self-improving estimate for $2$-$G$ loops:  
$$\sup_{v\in[s,u]}\wh\Xi_{v,2}^{(\mathcal{L}-\mathcal{K})}  \prec  1+ \frac{1}{W^d\ell_u^d\eta_u\sqrt{\kappa}} \sup_{v\in [s,u]}
  \left(\Xi^{({\cal L-\cal K})}_{v,2}\right)^2 \prec 1+ \frac{\sup_{v\in [s,u]}
\Xi^{({\cal L-\cal K})}_{v,2}}{(W^d\ell_u^d\eta_u\sqrt{\kappa})^{1/2}} .$$ 
Iterating this estimate for $\OO(1)$ many times gives
\smash{$\sup_{v\in[s,u]}\wh\Xi_{v,2}^{(\mathcal{L}-\mathcal{K})}\prec 1=:\Xi_{v,2}^{(\mathcal{L}-\mathcal{K})}.$}
Starting from this initial bound, by applying  \eqref{saww02} inductively in $\fn$, we obtain that 
$$\sup_{v\in[s,u]}\wh\Xi_{v,\fn}^{(\mathcal{L}-\mathcal{K})}\prec 1$$
for any fixed $\fn\in\N$, which concludes the estimate \eqref{Eq:L-KGt-flow}.

\subsection{Step 5: Sharp 2-$G$ loop estimate}
 
To show that \eqref{Eq:Gdecay_flow} holds uniformly in $u\in[s,t]$, note we have established in Step 4 that 
$$
\max_{\bsig\in\{+,-\}^2}\max_{\ba\in(\Zn)^2}\left({\cal L}-{\cal K}\right)^{(2)}_{u,\bsig,\ba}\prec (W^d\ell_u^d\eta_u)^{-2} 
$$
uniformly in $u\in[s,t]$. It already implies that 
$$
\left({\cal L}-{\cal K}\right)^{(2)}_{u,\bsig,\ba} \prec {\cal T}_{u, D}(|[a_1]-[a_2]|) \quad \text{for}\quad |[a_1]-[a_2]|=\OO(\ell_u^*).
$$
On the other hand, when $|[a_1]-[a_2]|\ge 6 \ell_u^*$, \eqref{Eq:Gdecay_flow} is an easy consequence of \eqref{53}.

%\bibliographystyle{abbrv}
%\bibliography{band.bib}

\newpage
%\end{document}
\appendix
%\addtocontents{toc}{\protect\setcounter{tocdepth}{-1}}

\section{Proof of \Cref{ML:Kbound}}\label{subsec:estimates}

%\subsection{Sum zero property}\label{subsec:molecule}
In this appendix, we prove \Cref{ML:Kbound} using the tree representation formula from \Cref{tree-representation}.  %For the proof of \Cref{ML:Kbound}, 
We first introduce additional structural notions associated with canonical tree partitions, namely the concepts of the \emph{core} and \emph{molecules}. These notions enable us to decompose a canonical tree partition into smaller tree partitions, which in turn allows us to apply an inductive argument (based on estimates for these smaller components) effectively. 
% Before proceeding to the proof of the sum-zero property and the $\cal K$-loop bounds in \Cref{ML:Kbound}, we first introduce some decompositions of $\Gamma$ and related quantities in this subsection. 
% More specifically, given a tree graph with $\fn$ external vertices, we will show that it can be decomposed into a structure that is nearly a product of two tree graphs, each with strictly fewer external vertices. 
%Given $\Gamma \in \TSP\p{\mathcal{P}_{\ba}}$, if $R_k \cap R_\ell\ne \emptyset$ with $\{k,\ell\}\in \Z^{\mathrm{off}}_\fn$, then we will say $R_k$ and $ R_\ell$ are \emph{non-trivial neighbors}. 

%we can decompose it into what is almost a product of two tree graphs with strictly fewer external vertices.
\begin{definition}[Tree decomposition]\label{recursive-decomposition-lemma}
  Let $\Gamma \in \TSP\p{\mathcal{P}_{\ba}}$ and let $1 \leq k < l \leq \fn$ be such that $R_k$ and $ R_l$ are non-trivial neighbors. 
  Let $\p{[c], [c']} = R_k \cap R_l$ denote the shared edge. Then, we have the decomposition
  \begin{equation}\label{recursive-decomposition}
    \Gamma_{t,\bm{\sigma},\ba}^{\p{\fn}}
      = \frac{\p{\Gamma_c}^{\p{\fn+k-l+1}}_{t,\cutL^{[a]}_{kl}\p{\bm{\sigma}, \ba}}}{\Theta_{t,[a][c]}^{\p{\sigma_k, \sigma_l}}}
        \frac{\big(\Theta_t^{\p{\sigma_k, \sigma_l}}-1\big)_{[c][c']}}{m(\sig_k)m(\sig_l)}
        \frac{\p{\Gamma_{c'}}^{\p{l-k+1}}_{t,\cutR^{[a']}_{kl}\p{\bm{\sigma}, \ba}}}{\Theta_{t,[a'][c']}^{\p{\sigma_k, \sigma_l}}},
  \end{equation}
  where $[a]$ and $[a']$ are two auxiliary vertices, and the $\TSP$ graphs $\Gamma_c$ and $ \Gamma_{c'}$ are obtained as follows.
  \begin{enumerate}
    \item Since $\Gamma$ is a tree, removing the edge $\p{[c], [c']}$ results in two connected subtrees: one connected to $[c]$ and the other connected to $[c']$. We denote these subtrees as $\mathcal{T}_c$ and $\mathcal{T}_{c'}$, respectively.
    \item Then, we define $\Gamma_c \coloneqq \mathcal{T}_c \cup \p{[c], [a]} \in \TSP\big(\mathcal{P}_{\cutL^{[a]}_{k\ell}\p{\ba}}\big)$ and $\Gamma_{c'}$ analogously.
  \end{enumerate}
  %Moreover, this decomposition is unique in the sense that for each edge $\p{c, c'}$ in $\Gamma^{\p{M}}$, there is exactly one pair $\p{\Gamma_c, \Gamma_{c'}}$ satisfying (\ref{recursive-decomposition}), up to relabeling of vertices. 
  We refer readers to Figure \ref{recursive-decomposition-figure} for an illustration of this decomposition. 
  \end{definition}
%\begin{proof}
%By the definition of $\Gamma^{\p{M}}$, the indices $a$ and $a'$ are held constant and appear only in the edges $\p{a, c}$ and $\p{a', c'}$, respectively. 
%It is not hard to check that the charges of the vertices and $M$-edges in $\Gamma^{\p{M}}_c$ and $\Gamma^{\p{M}}_{c'}$ are consistent with those of $\Gamma^{\p{M}}$ itself, with which we readily conclude the proof.
  \begin{figure}[h]
    \centering
    \scalebox{0.9}{
    \begin{tikzpicture}
      \coordinate (ak) at (-2, 2);
      \coordinate (ak-1) at (2, 3);
      \coordinate (al-1) at (-2, -2);
      \coordinate (al) at (2, -3);
      
      \coordinate (b11) at (1, 0);
      \coordinate (b10) at (2, 1);
      \coordinate (b12) at (2, -1);
      
      \coordinate (b21) at (-1, 0);
      \coordinate (b20) at (-2, -1);
      \coordinate (b22) at (-2, 1);
      
      \fill (ak-1) circle (1pt) node[above]{$[a_{k-1}]$};
      \fill (ak) circle (1pt) node[above]{$[a_k]$};
      \fill (al) circle (1pt) node[below]{$[a_l]$};
      \fill (al-1) circle (1pt) node[below]{$[a_{l-1}]$};
      
      \fill (b11) circle (1pt) node[below, xshift=-4pt]{$[c']$};
      \fill (b12) circle (1pt) node[right]{$[c_1]$};
      \fill (b10) circle (1pt) node[right]{$[c_2]$};
      
      \fill (b21) circle (1pt) node[below, xshift=2pt]{$[c]$};
      \fill (b22) circle (1pt) node[left]{$[c_3]$};
      \fill (b20) circle (1pt) node[left]{$[c_4]$};
      
      \draw[dashed, arrows={->[scale=1.5]}, shorten >= 2pt] (ak-1) -- (ak) node[pos=0.5, above]{$\sigma_k$};
      \draw[dashed, arrows={->[scale=1.5]}, shorten >= 2pt] (al-1) -- (al)  node[pos=0.5, below]{$\sigma_l$};
      
      \draw[purple] (b11) -- (b21);
      
      \draw (b10) -- (b11) -- (b12);
      \draw (b22) -- (b21) -- (b20);
      
      \draw[dotted] (b12) -- (al);
      \draw[dotted] (b20) -- (al-1);
      \draw[dotted] (b10) -- (ak-1);
      \draw[dotted] (b22) -- (ak);
      
      %\node at (-1.75, 0) {$M$};
      %\node at (1.75, 0) {$M$};
    \end{tikzpicture}
    %$\quad\to\quad$
    \begin{tikzpicture}
      \coordinate (ak) at (-2, 2);
      \coordinate (ak-1) at (2, 3);
      \coordinate (al-1) at (-2, -2);
      \coordinate (al) at (2, -3);
      
      \coordinate (b11) at (1, 0);
      \coordinate (b10) at (2, 1);
      \coordinate (b12) at (2, -1);
      
      \coordinate (b21) at (-1, 0);
      \coordinate (b20) at (-2, -1);
      \coordinate (b22) at (-2, 1);
      
      \coordinate (shift) at (0, -4);
      \coordinate (ak-1) at ($(ak-1) + (shift)$);
      \coordinate (al) at ($(al) + (shift)$);
      \coordinate (b12) at ($(b12) + (shift)$);
      \coordinate (b10) at ($(b10) + (shift)$);
      \coordinate (d') at ($(b21) + (shift)$);
      \coordinate (c') at ($(b11) + (shift)$);
      
      \fill (ak-1) circle (1pt) node[above]{$[a_{k-1}]$};
      \fill (ak) circle (1pt) node[above]{$[a_k]$};
      \fill (al) circle (1pt) node[below]{$[a_l]$};
      \fill (al-1) circle (1pt) node[below]{$[a_{l-1}]$};
      \fill (c') circle (1pt) node[below, xshift=-4pt]{$[c']$};
      \fill (d') circle (1pt) node[left]{$[a']$};
      
      \fill (b11) circle (1pt) node[right]{$[a]$};
      \fill (b12) circle (1pt) node[right]{$[c_1]$};
      \fill (b10) circle (1pt) node[right]{$[c_2]$};
      
      \fill (b21) circle (1pt) node[below, xshift=2pt]{$[c]$};
      \fill (b22) circle (1pt) node[left]{$[c_3]$};
      \fill (b20) circle (1pt) node[left]{$[c_4]$};
      
      \draw[dashed, arrows={->[scale=1.5]}, shorten >= 2pt] (b11) -- (ak) node[pos=0.5, above right]{$\sigma_k$};
      \draw[dashed, arrows={->[scale=1.5]}, shorten >= 2pt] (al-1) -- (b11)  node[pos=0.5, below right]{$\sigma_l$};
      
      \draw[dashed, arrows={->[scale=1.5]}, shorten >= 2pt] (ak-1) -- (d') node[pos=0.5, above left]{$\sigma_k$};
      \draw[dashed, arrows={->[scale=1.5]}, shorten >= 2pt] (d') -- (al)  node[pos=0.5, below left]{$\sigma_l$};
      
      \draw[purple] (b11) -- (b21);
      \draw[purple] (c') -- (d');

      \draw (b10) -- (c') -- (b12);
      \draw (b22) -- (b21) -- (b20);
      
      \draw[dotted] (b12) -- (al);
      \draw[dotted] (b20) -- (al-1);
      \draw[dotted] (b10) -- (ak-1);
      \draw[dotted] (b22) -- (ak);
      
      %\node at (-1.75, 0) {$M$};
      %\node at ($(1.75, 0) + (shift)$) {$M$};
      \node[anchor=east] at (-3.0, -4.0) {$\to$};
    \end{tikzpicture}
    }
    \caption{On the left-hand side is $\Gamma$ with $\protect\p{[c], [c']} = R_k \cap R_l$. The purple edges represent the $\Theta_t$ and $\Theta_t-1$ edges written explicitly in (\ref{recursive-decomposition}). The top and bottom graphs on the right-hand side are $\Gamma_c$ and $\Gamma_{c'}$, respectively.}\label{recursive-decomposition-figure}
  \end{figure}
%\end{proof}

Given a canonical tree partition, we define its core as the tree subgraph obtained by trimming all of its external and boundary edges. 
\begin{definition}[Core]\label{core-definition}
  Let $\Gamma \in \TSP\p{\mathcal{P}_{\ba}}$ and let $\bfb = \p{[b_1], \ldots, [b_\fm]}$ be its internal indices. Denote by $\bm{d}=([d_1], \ldots, [d_{\mathfrak r}])$ the internal vertices connected to $[a_1], \ldots, [a_\fn]$, where each $[a_i]$ connects to $[d_{\fa(i)}]$ for some function $\fa:\qqq{\fn}\to \qqq{\mathfrak r}$. Then, we define the \textbf{core} of $\Gamma$ by
  \[
    \Sigmagen^{\p{\Gamma}}_{t, \bm{\sigma}, \bm{d}}
      \coloneqq \sum_{i=1}^\fm \mathbf 1_{[b_{i}] \not\in \bm{d}} \sum_{[b_{i}] \in \Zn} \frac{\Gamma^{\p{\fn}}_{t,\bm{\sigma},\ba}}{\prod_{i=1}^\fn \Theta_{t,[a_i][d_{\fa(i)}]}^{(\sigma_i,\sig_{i+1})}},
  \]
  where we have slightly abused notation: in the indicator function $\mathbf 1$, $[b_{i}]\in \bm{d}$ means that $[b_i]\notin\{[d_1],\ldots, [d_{\mathfrak r}]\}$. 
  Simply speaking, the core is constructed graphically by removing all the external edges (i.e., those corresponding to \smash{$\Theta_t$}) from \smash{$\Gamma^{\p{\fn}}_{t,\bm{\sigma},\ba}$}, and then summing over all the internal vertices, excluding the vertices $[d_1], \ldots, [d_{\mathfrak r}]$.
\end{definition}

Next, we will present a corresponding decomposition for the cores of canonical tree partitions. This decomposition is similar to \Cref{recursive-decomposition-lemma}, but we must also keep track of all pairs of non-trivial neighbors. Informally, when ``splitting" each graph through a ``long" internal edge $R_k \cap R_l$ between non-trivial neighbors, every other non-trivial neighbor pair will belong to either the ``left" graph or the ``right" graph, as these pairs are not permitted to cross. Hereafter, we refer to a $\Theta^{(\sig,\sig')}$-edge as a \emph{long edge} if $\sig\ne \sig'$; otherwise, it is called a \emph{short edge}. We adopt this terminology because, according to \eqref{prop:ThfadC} and \eqref{prop:ThfadC_short}, the typical decay scale of a \smash{$\Theta^{(\sig,\sig)}$} edge is \smash{$\hell_t$}, which is generally smaller than the decay scale $\ell_t$ associated with a \smash{$\Theta^{(+,-)}$}-edge. (However, they can be of the same order when $\sqrt{\kappa}\lesssim|1-t|$).

\begin{definition} \label{k-sigma-pi}
Let $\ba = \p{[a_1], \ldots, [a_\fn]}$ and $\Gamma \in \TSP\p{\mathcal{P}_{\ba}}$. 
Given any $\bm{\sigma} \in \set{+, -}^{\fn}$, define the subset of long internal edges (i.e., the boundary between two nontrivial neighbors of different charges) as 
\[
{\cal F}_{\text{long}}(\Gamma, \boldsymbol{\sigma}) := \left\{\{k,l\} \in \Z^{\mathrm{off}}_\fn: R_k\cap R_l\neq \emptyset ,\ \sigma_k\ne \sigma_l \right\},
\]
where we define the subset 
\[
\Z^{\mathrm{off}}_\fn \coloneqq \set{\set{k, \ell} \suchthat 1 \leq k < \ell \leq \fn,\, k - \ell \mod \fn \notin \{1,-1\}}.
\]
Given any subset \(\pi \subset \mathbb{Z}_\fn^{\mathrm{off}}\), we denote by
\[
\TSP(\mathcal{P}_{\ba}, \boldsymbol{\sigma}, \pi) := \big\{\Gamma \in \TSP(\mathcal{P}_{\ba}) : {\cal F}_{\text{long}}(\Gamma, \boldsymbol{\sigma}) = \pi\big\}
\]
the subset of \(\TSP(\mathcal{P}_{\ba})\) with \(\pi\) labeling the pairs of all non-trivial neighbors (note $\pi$ can be $\emptyset$). Then, we define $\cK^{\p{\pi}}$ and $\Sigma^{(\pi)}$ as
  \be\label{eq:defKpi}
    \cK^{\p{\pi}}_{t,\bm{\sigma},\ba}
      \coloneqq W^{-d(\fn-1)}\sum_{\Gamma \in \TSP\p{\mathcal{P}_{\ba}, \bm{\sigma}, \pi}} \Gamma^{\p{\fn}}_{t,\bm{\sigma},\ba} \, .
  \ee
  We analogously define the \emph{core with respect to $\pi$} as
  \[
    \Sigma^{\p{\pi}}_{t,\bm{\sigma},\bm{d}}
      \coloneqq \sum_{\Gamma \in \TSP\p{\mathcal{P}_{\ba}, \bm{\sigma}, \pi}} \Sigmagen^{\p{\Gamma}}_{t, \bm{\sigma}, \bm{d}} \, ,
  \]
  where $\Sigmagen^{\p{\Gamma}}_{t, \bm{\sigma}, \bm{d}}$ is defined in Definition \ref{core-definition}. When $\pi = \emptyset$, we call  $\Sigma^{\p{\emptyset}}$ a \textbf{molecule}.
\end{definition}

% We recall the notation
% \[
%   \pi \subseteq \Z^{\mathrm{off}}_n \coloneqq \set{\set{i, j} \suchthat 1 \leq i < j \leq n,\, \abs{i - j} \neq 1 \mod n}
% \]
% and $\TSP\p{\mathcal{P}_{\ba}, \bm{\sigma}, \bm{\pi}}$ from Definition 3.8 in \cite{Band1D} and define the corresponding $\Kgen^{\p{\pi}}$ and $\Sigmagen^{\p{\pi}}$ for our setting.
% \begin{definition}[$\Kgen^{\p{\pi}}$ and $\Sigmagen^{\p{\pi}}$]\label{k-sigma-pi}
%   Given a subset $\pi \subseteq \Z^{\mathrm{off}}_n$, we define
%   \[
%     \Kgen^{\p{\pi}}_{t,\bm{\sigma},\ba}
%       \coloneqq \sum_{\Gamma \in \TSP\p{\mathcal{P}_{\ba}, \bm{\sigma}, \pi}} \Gamma^{\p{M}}_{t,\bm{\sigma},\ba}.
%   \]
%   We also analogously define
%   \[
%     \Sigmagen^{\p{\pi}}
%       \coloneqq \sum_{\Gamma \in \TSP\p{\mathcal{P}_{\ba}, \bm{\sigma}, \pi}} \Sigmagen^{\p{\Gamma; M}}\p{t, \bm{\sigma}, \bm{d}},
%   \]
%   where $\Sigmagen^{\p{\Gamma; M}}\p{t, \bm{\sigma}, \bm{d}}$ is as in Definition \ref{core-definition}. When $\pi = \emptyset$, the core $\Sigmagen^{\p{\emptyset}}$ is called a \textbf{molecule}.
% \end{definition}
\begin{lemma}[Core decomposition, Lemma 4.25 of \cite{RBSO1D}]\label{core-decomposition}
Without loss of generality, suppose $\set{1, r} \in \pi$ and $\p{[c], [c']} = R_1 \cap R_r$. Then, we have the decomposition
  \begin{equation}\label{recursive-decomposition-Sigma}
\Sigma^{\p{\pi}}_{t,\bm{\sigma},\bm{d}} = \sum_{[c], [c']} \Sigma^{(\pi^{[c]})}_{t,\bm{\sigma}^{[c]},\bm{d}^{[c]}} 
\big(\Theta_t^{\p{+,-}}-1\big)_{[c][c']}
\Sigma^{(\pi^{[c']})}_{t,\bm{\sigma}^{[c']},\bm{d}^{[c']}}\, ,
  \end{equation}
  where $\pi_0 \coloneqq \pi \setminus \set{\set{1, r}}$, $\pi^{[c']} \coloneqq \set{\set{k, \ell} \in \pi_0 \mid 1 \leq k < \ell \leq r}$, and 
  \begin{align*}
 &\pi^{[c]} \coloneqq \set{\set{k - r + 1, \ell - r + 1} \mid \set{k, \ell} \in \pi_0,\ r \leq k < \ell \leq \fn}.\end{align*}
 Moreover, \smash{$\bm{\sigma}^{[c]}$ and $\bm{\sigma}^{[c']}$} are defined through
  \begin{align*}
    \big(\bm{\sigma}^{[c]}, \ba^{[c]}\big)
        \coloneqq \cutL^{[c]}_{1r}\p{\bm{\sigma}, \ba}, \quad
    &\big(\bm{\sigma}^{[c']}, \ba^{[c']}\big)
        \coloneqq \cutR^{[c']}_{1r}\p{\bm{\sigma}, \ba},
  \end{align*}
and $\bm{d}^{[c]}$ and $\bm{d}^{[c']}$ denote the subsets of internal vertices connected to $\ba^{[c]}$ and $\ba^{[c']}$, respectively. Note $\pi^{[c]}$ and $\pi^{[c']}$ are defined so that we have the partition
  \(
    \pi = (\pi^{[c]}+r-1) \cup \pi^{[c']} \cup \set{\set{1, r}}
  \), and $[c]$ (resp.~$[c']$) belongs to $\bm{d}^{[c]}$ (resp.~$\bm{d}^{[c']}$) as a vertex connected to $[a]$ (resp.~$[a']$).
\end{lemma}

Using the estimate (\ref{prop:ThfadC_short}), we can derive the following bound on pure loops where all charges are identical.

\begin{claim}[Pure loop estimate]\label{claim:pure}
  Let $\bm{\sigma} \in \set{+, -}^\fn$ be such that $\sigma_1 = \sigma_2 = \cdots = \sigma_\fn$. Then, we have
  \[
    \cal{K}^{(\fn)}_{t,\bm{\sigma},\ba} \prec \big(W^d\hell_t^d\omega_t\big)^{-(\fn-1)}\omega_t^{-(\fn-2)}.
  \]    
\end{claim}
\begin{proof}
It suffices to prove that for any $\Gamma\in \TSP(\cal P_{\ba})$, we have 
\[
\Gamma^{(\fn)}_{t,\bm{\sigma},\ba} \prec \big(\hell_t^d\omega_t\big)^{-(\fn-1)}\omega_t^{-(\fn-2)}.
\]
This in turn follows from the following estimate on the core of $\Gamma:$
\be\label{eq:est-core}
\sum_{[d_2],\ldots,[d_{\mathfrak r}]} \left|\Sigmagen^{\p{\Gamma}}_{t, \bm{\sigma}, \bm{d}}\right| \prec \omega_t^{-(\fn-3)}.
\ee
In fact, suppose \eqref{eq:est-core} holds, and assume that $[d_1]$ is connected with $[a_1]$ without loss of generality. Then, using \eqref{prop:ThfadC_short}, we obtain that for any constants $\tau, D>0$,
\begin{align*}
\Gamma^{(\fn)}_{t,\bm{\sigma},\ba} \prec \frac{\omega_t^{-(\fn-3)}}{(\hell_t^d\omega_t)^{\fn-1}} \sum_{|[d_1]-[a_1]|\le W^\tau\hell_t} \frac{1}{\omega_t \hell_t^d}+ W^{-D}\prec \frac{W^{d\tau}}{(\hell_t^d\omega_t)^{\fn-1}\omega_t^{\fn-2}}.
\end{align*}

To show the estimate \eqref{eq:est-core}, we note from \eqref{prop:ThfadC_short} that 
every edge in \smash{$\Sigmagen^{\p{\Gamma}}_{t, \bm{\sigma}, \bm{d}}$} contributes a factor of \smash{$(\hell_t^d\omega_t)^{-1}$}, and all vertices (including both external vertices in $\bm{d}$ and internal ones) are within a neighborhood of distance \smash{$\le W^\tau\hell_t$} from $[d_1]$, up to a small error $W^{-D}$. 
Let $k_1$ denote the total number of vertices and $k_2$ the total number of edges in the core \smash{$\Sigmagen^{\p{\Gamma}}_{t, \bm{\sigma}, \bm{d}}$}. By the tree structure of $\Gamma$, we have $k_1=k_2+1\le \fn-2$\footnote{To see the inequality, we first observe that $k_1=\fn-2$ in the case $\fn=3$. Next, we note that a canonical tree partition \smash{$\Gamma\in \TSP(\cal P_{([a_1],\ldots,[a_{\fn}])})$} can be obtained by adding a new edge to a canonical tree partition \smash{$\Gamma'\in \TSP(\cal P_{([a_1],\ldots,[a_{\fn-1}])})$}, connecting the newly added vertex $[a_{\fn}]$ with either a vertex or a midpoint of an edge in $\Gamma'$ (see \Cref{Def:slice} and \Cref{slice-decomposition}). Each such process increases the number of external vertices by 1 from $\Gamma'$ to $\Gamma$, while the number of internal vertices either remains unchanged or increases by 1.}.
Thus, we can bound the left-hand side (LHS) of \eqref{eq:est-core} by  
$$\sum_{[d_2],\ldots,[d_{\mathfrak r}]} \left|\Sigmagen^{\p{\Gamma}}_{t, \bm{\sigma}, \bm{d}}\right| \prec (\hell_t^d\omega_t)^{-k_2}\big(W^\tau\hell_t\big)^{d(k_1-1)} \lesssim W^{d(k_1-1)\tau} \omega_t^{-k_2} \lesssim W^{d(k_1-1)\tau} \omega_t^{-(\fn-3)} .$$
This concludes the proof since $\tau$ is arbitrary.
\end{proof}

Next, by repeatedly applying Ward's identity \eqref{WI_calK}, we obtain the following bound from the pure loop estimate.
\begin{lemma}\label{K-ward-bound}
For any $\bm{\sigma} \in \set{+, -}^\fn$, we have that 
  \be\label{eq:sum_of_K}
    \sum_{\br{a_2},\ldots,\br{a_\fn}} \mathcal{K}^{(\fn)}_{t,\bm{\sigma},\ba}= \sum_\pi \sum_{\br{a_2},\ldots,\br{a_\fn}} \mathcal{K}^{(\pi)}_{t,\bm{\sigma},\ba}
    \prec \sqrt{\kappa}\big(W^d\eta_t\big)^{-\fn+1}.
  \ee
\end{lemma}
\begin{proof}
We apply Ward's identity \eqref{WI_calK} repeatedly. In each step, we gain a factor of $(W^d\eta_t)^{-1}$ while reducing the number of external vertices by 1. We continue this process until each resulting loop is a pure loop. 
If a pure loop contains only one vertex $[a_1]$, the last application of Ward's identity provides an additional factor of $\sqrt{\kappa}$, alongside the $(W^d\eta_t)^{-1}$ factor:
$$\frac{1}{2\ii}\left(\cK^{(1)}_{t,+,[a]}-\cK^{(1)}_{t,-,[a]}\right)=\im m\asymp \sqrt{\kappa},$$  
where we used \eqref{eq:1loop} in the first step and \eqref{square_root_density} for $m=m(\sE)$ in the second step. This leads to \eqref{eq:sum_of_K}.

If a pure loop contains $k\ge 2$ external vertices, then applying \Cref{claim:pure} and \eqref{prop:ThfadC_short} yields a bound 
$$\OO_\prec \left[\big(W^d\hell_t^d\omega_t\big)^{-(k-1)}\omega_t^{-(k-2)}\cdot (W^\tau\hell_t)^{d(k-1)}\right]=\OO_\prec \left[W^{d(k-1)\tau}\big(W^d\omega_t\big)^{-(k-1)}\omega_t^{-(k-2)}\right]$$ 
on its summation over the $(k-1)$ vertices that are not $[a_1]$. Combined with the \smash{$(W^d\eta_t)^{-(\fn-k)}$} factor obtained from the previous $(\fn-k)$ applications of Ward's identity, we have the bound
\begin{align*}
   W^{d(k-1)\tau} \cdot \OO_\prec \bigg[\frac{1}{(W^d\eta_t)^{\fn-k}}\frac{1}{(W^d\omega_t)^{k-1}\omega_t^{k-2}}\bigg]&=  W^{d(k-1)\tau} \cdot \OO_\prec \left[\frac{1}{(W^d\eta_t)^{\fn-1}} \frac{\eta_t}{\omega_t} \right]\\
   &= W^{d(k-1)\tau} \cdot \OO_\prec \left[\sqrt{\kappa}(W^d\eta_t)^{-(\fn-1)} \right],  
\end{align*}
where we used $\eta_t\lesssim \sqrt{\kappa}\omega_t \le \omega_t^2$ by \eqref{eq:kappat} in the derivation. This concludes the proof since $\tau$ is arbitrary.
\end{proof}

We now state and prove a \emph{sum-zero property for $\Sigma^{\p{\pi}}$}, which will be the main tool for the proof of \Cref{ML:Kbound}.

\begin{lemma}[Sum zero property]\label{sum-zero}
Let $\fn \in 2\N$ with $\fn \geq 4$, and consider an alternating loop $\bm{\sigma}$ such that $\sig_k\ne \sig_{k+1}$ for all $k\in\qqq{\fn}$, where we adopt the convention that $\sig_{\fn+1}=\sig_1$.
  % \[
  %   \bm{\sigma}^{\p{\alt}}_k
  %     = \begin{cases}
  %         + & \text{if $k$ is odd}, \\
  %         - & \text{if $k$ is even}.
  %       \end{cases}
  % \]
Then, the single-molecule graph with $\pi=\emptyset$ (recall \Cref{k-sigma-pi}) satisfies the \textbf{sum-zero property},
  \begin{equation}\label{sum-zero-property}
    \frac{1}{n^d} \sum_{\bm{d}} \Sigma^{\p{\emptyset}}_{t, \bm{\sigma}, \bm{d}} \prec {\frac{1-t}{\omega_t^{\fn-2}}},
  \end{equation}
where recall that $\bm{d}=([d_1], \ldots, [d_{\mathfrak r}])$ denote the internal vertices connected to the external vertices in $\ba$. 
\end{lemma}
\begin{proof}
When $1-t\ge \sqrt{\kappa}$, we have $1-t\asymp \omega_t$, and a similar proof to that for \eqref{eq:est-core} leads to \eqref{sum-zero-property}. It remains to consider the case $1-t<\sqrt{\kappa}$. In this case, we will prove the following slightly stronger statement for all $\fn \in 2\N$ and $\pi\subset \Z^{\mathrm{off}}_\fn$ by induction: 
  \begin{equation}\label{sum-zero-inductive-hypothesis}
    \frac{1}{n^d} \sum_{\bm{d}} \Sigma^{\p{\pi}}_{t,\bm{\sigma},\bm{d}}
      \prec {{\eta_t}\kappa^{-(\fn-1)/2}} \, .
  \end{equation}  
If only the scenario $\pi = \emptyset$ is possible, then by \Cref{K-ward-bound}, we have that  
\be\label{sumzero_derv}\frac{1}{(1-t)^\fn} \sum_{\bm{d}} \Sigma^{\p{\emptyset}}_{t, \bm{\sigma}, \bm{d}} = W^{d(\fn-1)}\sum_{\ba} \mathcal{K}^{(\emptyset)}_{t,\bm{\sigma},\ba}  \prec {n^d\sqrt{\kappa}\eta_t^{-\fn+1}},\ee
where, in the first step, we used the fact that for any $[y]\in \Zn$, \be\label{eq:sumTheta}
\sum_{[x]}\Theta_t^{(+,-)}([x],[y])=(1-t)^{-1}. 
\ee
Together with the estimate $\eta_t\asymp (1-t)\sqrt{\kappa}$ given by \eqref{eq:kappat}, the equation \eqref{sumzero_derv} implies \eqref{sum-zero-inductive-hypothesis}. As a special case, this concludes the base case with $\fn = 4$ where only $\pi = \emptyset$ is possible. For the general case with $\fn > 4$, suppose we have proved \eqref{sum-zero-inductive-hypothesis} for all alternating loops of length $k<\fn$ and $\pi\subset \Z^{\mathrm{off}}_k$. 
  
We first consider the $\pi\ne\emptyset$ case. Without loss of generality, suppose $\set{1, r} \in \pi$, and let $\p{[c], [c']} = R_1 \cap R_r$. Applying \Cref{core-decomposition}, we can write that
  \[
    \sum_{\bm{d}} \Sigma^{\p{\pi}}_{t,\bm{\sigma},\bm{d}}
      = \sum_{[c], [c']} \sum_{\bm{d}^{[c]} \setminus \set{[c]}}\Sigma^{(\pi^{[c]})}_{t,\bm{\sigma}^{[c]},\bm{d}^{[c]}} 
\cdot \big(\Theta_t^{\p{\sigma_1, \sigma_r}}-1\big)_{[c][c']}\cdot \sum_{\bm{d}^{[c']} \setminus \set{[c']}}
\Sigma^{(\pi^{[c']})}_{t,\bm{\sigma}^{[c']},\bm{d}^{[c']}} ,
  \]
where $\bm{d}^{[c]} \setminus \set{[c]}$ (resp.~$\bm{d}^{[c']} \setminus \set{[c']}$) denotes the vector of vertices obtained by removing $[c]$ (resp.~$[c']$) from $\bm{d}^{[c]} $ (resp.~$\bm{d}^{[c']}$). 
Since $\bm{\sigma}$ is alternating, $r$ must be be even. Then, both $\bm{\sigma}^{[c]}$ and $\bm{\sigma}^{[c']}$ are alternating loops of lengths $\fn - r + 2 \in 2\N$ and $r \in 2\N$, respectively. Thus, we can apply the induction hypothesis and derive that 
  \begin{align*}
    \frac{1}{n^d} \sum_{\bm{d}} \Sigma^{\p{\pi}}_{t,\bm{\sigma},\bm{d}}
      &= \bigg(\frac{1}{n^d} \sum_{\bm{d}^{[c]}}\Sigma^{(\pi^{[c]})}_{t,\bm{\sigma}^{[c]},\bm{d}^{[c]}} \bigg)
        \cdot \frac{1}{n^d} \sum_{[c], [c']}  \big(\Theta_t^{\p{+,-}}-1\big)_{[c][c']}\cdot   
        \bigg(\frac{1}{n^d} \sum_{\bm{d}^{[c']}}\Sigma^{(\pi^{[c']})}_{t,\bm{\sigma}^{[c']},\bm{d}^{[c']}}  \bigg) \\
      &\prec {\frac{\eta_t}{(\sqrt{\kappa})^{\fn-r+1}} \cdot \frac{t}{1 - t} \cdot \frac{\eta_t}{(\sqrt{\kappa})^{r-1}}} \lesssim {\eta_t \kappa^{-(\fn-1)/2}},
  \end{align*}
  where in the first step, we used that the first and third factors on the right-hand side (RHS) do not depend on $[c]$ or $[c']$ due to the block translation invariance,  in the second step, we applied \eqref{eq:sumTheta}, and in the last step, we used that $\eta_t\asymp(1-t)\sqrt{\kappa}$.

  To deal with the $\pi = \emptyset$ case, we adopt a similar argument as in \eqref{sumzero_derv} and decompose \smash{$\mathcal{K}^{(\fn)}_{t,\bm{\sigma},\ba}$} as 
  \begin{align}
    \frac{W^{d(\fn-1)}}{N} \sum_{\ba} \mathcal{K}^{(\fn)}_{t,\bm{\sigma},\ba}
      &= \p{\sum_{\ba} \prod_{i=1}^\fn \Theta_{t,[a_i][d_{\fa(i)}]}^{\p{+,-}}} \bigg(\frac{1}{n^d} \sum_{\bm{d}}\Sigma^{\p{\emptyset}}_{t,\bm{\sigma},\bm{d}}
        + \frac{1}{n^d}\sum_{\pi \neq \emptyset}  \sum_{\bm{d}}\Sigma^{\p{\pi}}_{t,\bm{\sigma},\bm{d}}\bigg) \nonumber\\
      &= \p{1 - t}^{-\fn} \bigg(\frac{1}{n^d} \sum_{\bm{d}}\Sigma^{\p{\emptyset}}_{t,\bm{\sigma},\bm{d}}
        + \frac{1}{n^d}\sum_{\pi \neq \emptyset}  \sum_{\bm{d}}\Sigma^{\p{\pi}}_{t,\bm{\sigma},\bm{d}}\bigg).\label{eq:piempty}
  \end{align}
  By \Cref{K-ward-bound}, we have that 
  $$\frac{W^{d(\fn-1)}}{n^d}\sum_{\ba} \mathcal{K}^{(\fn)}_{t,\bm{\sigma},\ba}   \prec \frac{W^{d(\fn-1)}}{n^d} \cdot n^d \sqrt{\kappa}\big(W^d\eta_t\big)^{-\fn+1}=\sqrt{\kappa}\eta_t^{-\fn+1}.$$
  Plugging this bound and the estimate (\ref{sum-zero-inductive-hypothesis}) established for the $\pi\ne \emptyset$ case into \eqref{eq:piempty}, we conclude that 
  \[
    \frac{1}{n^d} \sum_{\bm{d}}\Sigma^{\p{\emptyset}}_{t,\bm{\sigma},\bm{d}}
      = -\frac{1}{n^d} \sum_{\pi \neq \emptyset} \sum_{\bm{d}}\Sigma^{\p{\pi}}_{t,\bm{\sigma},\bm{d}} + \p{1 - t}^\fn \OO_\prec \p{\sqrt{\kappa}\eta_t^{-\fn+1}}
     \prec \eta_t\kappa^{-(\fn-1)/2}.
  \]
  This completes the induction step and hence concludes the proof of \Cref{sum-zero}.
\end{proof}

%\subsection{Proof of \Cref{ML:Kbound}}\label{subsec:estimates}

%We are now ready to give the proof of \Cref{ML:Kbound} using the sum zero property in \Cref{sum-zero}. 
\begin{proof}[\bf Proof of Lemma \ref{ML:Kbound}]
%We first focus on the proof of \eqref{eq:bcal_k}.
Using the definitions \eqref{Kn2sol} and \eqref{Kn3sol} for $2$-$\cK$ and $3$-$\cK$ loops (noting that an alternating loop is not possible when $\fn=3$), we can easily show that \eqref{eq:bcal_k} holds for $\fn\in\{2,3\}$ by applying the estimates \eqref{prop:ThfadC} and \eqref{prop:ThfadC_short}. %It remains to consider the $\fn\ge 4$ case. 

For $\fn\ge 4$, it suffices to establish the following estimate for any $\pi $ for any \(\pi \subset \mathbb{Z}_\fn^{\mathrm{off}}\): 
  \begin{equation}\label{Kbound-inductive-hypothesis}
\sum_{\bm{d}\setminus\{[d_{\fa(1)}]\}} \p{\prod_{i=2}^\fn \Theta^{\p{\sigma_i, \sigma_{i+1}}}_{t,\br{a_i}[d_{\fa(i)}]}}\Sigma^{\p{\pi}}_{t, \bm{\sigma}, \bm{d}}
      \prec \frac{(\ell_t^d|1-t|\omega_t)^{-(\fn-2)}}{\min_{2 \leq i \leq \fn} \langle\br{a_i} - [d_{\fa(1)}]\rangle^{d}} \, ,
  \end{equation}
where we recall the notation $\fa$ defined in \Cref{core-definition}. Without loss of generality, we can assume here that either $\bsig$ is a pure loop or $\sig_1\ne \sig_2$ (if not, we can relabel the vertices to achieve this). Hence, if $\bsig$ is not an alternating loop, there must exist $j\in\qqq{2,\fn}$ such that $\sig_j=\sig_{j+1}$. 
Now, summing the estimate \eqref{Kbound-inductive-hypothesis} over $[d_1]$ and applying \eqref{prop:ThfadC} to \smash{$\Theta_{t,[a_1][d_{\fa(1)}]}^{(\sig_1,\sig_2)}$} implies that 
$$ \sum_{\Gamma \in \TSP\p{\mathcal{P}_{\ba}, \bm{\sigma}, \pi}} \Gamma^{\p{\fn}}_{t,\bm{\sigma},\ba} \prec \big({\ell_t^d|1-t|}\big)^{-\fn+1}\omega_t^{-\fn+2}.$$
Then, using \eqref{eq:defKpi} and summing over $\pi$, we can conclude \eqref{eq:bcal_k}. 

We will prove \eqref{Kbound-inductive-hypothesis} by induction in $\fn$.  
Suppose that \eqref{Kbound-inductive-hypothesis} holds for all $\cK$-loops of length $k<\fn$ and $\pi\subset \Z^{\mathrm{off}}_k$. 
We begin with the case $\pi = \emptyset$ and further divide the problem into two cases based on $\bm{\sigma}$: (i) $\bsig$ is an alternating loop; (ii) there exists $j\in\qqq{2,\fn}$ such that $\sig_j=\sig_{j+1}$. 
In case (ii), %\smash{$\Theta_{t,\br{a_j}[d_{\fa(j)}]}^{\p{\sigma_j, \sigma_{j+1}}}$} is a short edge. 
applying \eqref{prop:ThfadC} and \eqref{prop:ThfadC_short}, we obtain that for any constants $\tau,D>0$,
  \begin{align}
    \prod_{i=2}^\fn \Theta_{t,\br{a_i}[d_{\fa(i)}]}^{\p{\sigma_i, \sigma_{i+1}}} \prec \big(\ell_t^d|1-t|\big)^{-\fn+2}\big(\hell_t^d\omega_t\big)^{-1} \mathbf 1\big(|[a_j]-[d_{\fa(j)}]|\le W^\tau \hell_t\big) + W^{-D} .\nonumber %\label{eq:onesamecolor}
  \end{align}
 Using the fact that the core \smash{$\Sigma^{\p{\emptyset}}_{t, \bm{\sigma}, \bm{d}}$} consists of short edges only and the estimate \eqref{prop:ThfadC_short}, we get that for any constants $\tau,D>0$,  
\begin{align*}
   \Sigma^{\p{\emptyset}}_{t, \bm{\sigma}, \bm{d}}  \prec \Big|\Sigma^{\p{\emptyset}}_{t, \bm{\sigma}, \bm{d}}\Big| \mathbf 1\left(\big|[d_{\fa(j)}]-[d_{\fa(1)}]\big|\le W^\tau \hell_t\right)+ W^{-D}.
\end{align*}
Combining the above two estimates, we obtain that 
  \begin{align*}
    &\sum_{\bm{d}\setminus\{[d_{\fa(1)}]\}}  \p{\prod_{i=2}^\fn \Theta^{\p{\sigma_i, \sigma_{i+1}}}_{t,\br{a_i}[d_{\fa(i)}]}}\Sigma^{\p{\emptyset}}_{t, \bm{\sigma}, \bm{d}}\\
    &\prec \frac{1}{(\ell_t^d|1-t|)^{\fn-2}} \frac{\mathbf 1\big(|[a_j]-[d_{\fa(1)}]|\le 2W^\tau \hell_t\big)}{\hell_t^d\omega_t}  \sum_{\bm{d}\setminus\{[d_{\fa(1)}]\}}  \abs{\Sigma^{\p{\emptyset}}_{t, \bm{\sigma}, \bm{d}}} + W^{-D}   \nonumber\\
    &\prec W^{d\tau}(\ell_t^d|1-t|\omega_t)^{-(\fn-2)} \avg{[a_j]-[d_{\fa(1)}]|}^{-d},
  \end{align*}
 where we used the estimate \eqref{eq:est-core} in the second step. This concludes \eqref{Kbound-inductive-hypothesis} in case (ii) since $\tau$ is arbitrary. % and the fact that $x^d\exp(-cx/(\log W)^{2})\lesssim (\log W)^{2d}\prec 1$ for any $x\ge 0$.   
  
  %Since $\Sigma^{\p{\emptyset}}$ is a short-range operator, we obtain (\ref{Kbound-inductive-hypothesis}) in this case.
  
It remains to consider the most challenging case (i), to deal with which we will need to use the sum zero property established in \Cref{sum-zero}. For brevity, we introduce the following notations: 
  \[
    [s_{\fa(i)}] \coloneqq [d_{\fa(i)}] - [d_{\fa(1)}],
    \quad f\p{\br{a_i}, [s_{\fa(i)}] } \coloneqq \Theta^{\p{\sigma_i, \sigma_{i+1}}}_{t,\br{a_i}([d_{\fa(1)}]+[s_{\fa(i)}])}=\Theta^{\p{+,-}}_{t,\br{a_i}([d_{\fa(1)}]+[s_{\fa(i)}])}.
  \]
  Then, the LHS of (\ref{Kbound-inductive-hypothesis}) can be written as
  \[
    \sum_{\bm{d}\setminus\{[d_{\fa(1)}]\}} \p{\prod_{i=2}^\fn \Theta^{\p{+,-}}_{t,\br{a_i}[d_{\fa(i)}]}}\Sigma^{\p{\pi}}_{t, \bm{\sigma}, \bm{d}}
      = \sum_{\bm{s}\setminus\{[s_{\fa(1)}]\}} \p{\prod_{i=2}^\fn f\p{\br{a_i}, [s_{\fa(i)}]}} \Sigma^{\p{\emptyset}}_{t, \bm{\sigma}, \bm{d}},
  \]
  where $\bm{s}$ denotes $\bm{s}:=([s_{\fa(1)}],\ldots, [s_{\fa(\mathfrak r)}]). $ 
  We further decompose
  \[
    f\p{\br{a_i}, [s_{\fa(i)}]}
      = f_0\p{\br{a_i}, [s_{\fa(i)}]}
        + f_1\p{\br{a_i}, [s_{\fa(i)}]}
        + f_2\p{\br{a_i}, [s_{\fa(i)}]},
  \]
  where we have defined
  \begin{align*}
    f_0\p{\br{a_i}, [s_{\fa(i)}]}
      &\coloneqq f\p{\br{a_i}, \br{0}}, \ \
    f_1\p{\br{a_i}, [s_{\fa(i)}]}
      \coloneqq \frac{1}{2} \p{f\p{\br{a_i}, [s_{\fa(i)}]} - f\p{\br{a_i}, -[s_{\fa(i)}]}}, \\
    f_2\p{\br{a_i}, [s_{\fa(i)}]}
      &\coloneqq \frac{1}{2} \p{f\p{\br{a_i}, [s_{\fa(i)}]} + f\p{\br{a_i}, -[s_{\fa(i)}]}}
        - f\p{\br{a_i}, \br{0}} .
  \end{align*}
  By (\ref{prop:ThfadC}), (\ref{prop:BD1}), and (\ref{prop:BD2}), they satisfy the following estimates:
  \be\label{f012}
  \begin{aligned}
    f\p{\br{a_i}, [0]}
      \prec &\big(\ell_t^d |1-t|\big)^{-1},\quad
    f_1\p{\br{a_i}, [s_{\fa(i)}]}
      \prec {|[s_{\fa(i)}]|}\big/{\avg{\br{a_i} - [d_{\fa(1)}]}^{d-1}}, \\
    &f_2\p{\br{a_i}, [s_{\fa(i)}]}
      \prec {|[s_{\fa(i)}]|^2}\big/{\avg{\br{a_i} - [d_{\fa(1)}]}^{d}}.
  \end{aligned}\ee
  Then, the LHS of (\ref{Kbound-inductive-hypothesis}) can be decomposed as
  \[
    % \sum_{\bm{d}\setminus\{[d_{\fa(1)}]\}} \p{\prod_{i=2}^\fn \Theta^{\p{+,-}}_{t,\br{a_i}[d_{\fa(i)}]}}\Sigma^{\p{\pi}}_{t, \bm{\sigma}, \bm{d}}
    %   = 
      \sum_{ \xi_2, \ldots, \xi_\fn \in \{0,1,2\}} \sum_{\bm{s}\setminus\{[s_{\fa(1)}]\}} \p{\prod_{i=2}^\fn f_{\xi_i}\p{\br{a_i}, [s_{\fa(i)}]}} \Sigma^{\p{\emptyset}}_{t, \bm{\sigma}, \bm{d}}.
  \]
  We will prove (\ref{Kbound-inductive-hypothesis}) for each fixed sequence $\p{\xi_2, \ldots, \xi_\fn}\in \{0,1,2\}^\fn$. We divide the proof into the following cases. %which we do by breaking the problem up into the following cases.
  %\begin{enumerate}
    %\item 
    
    \medskip
\noindent(1)   If $\xi_i = 2$ for at least one $i$, then we get that for any constants $\tau,D>0$,
      \begin{align*}
        &\sum_{\bm{s}\setminus\{[s_{\fa(1)}]\}} \p{\prod_{i=2}^\fn f_{\xi_i}\p{\br{a_i}, [s_{\fa(i)}]}} \Sigma^{\p{\emptyset}}_{t, \bm{\sigma}, \bm{d}}\\
        &\prec  {\big({\ell_t^d |1-t|}\big)^{-\fn+2}} \cdot (W^\tau\hell_t)^2 \avg{\br{a_i} - [d_{\fa(1)}]}^{-d} \cdot \sum_{\bm{d}\setminus\{[d_{\fa(1)}]\}} \abs{\Sigma^{\p{\emptyset}}_{t, \bm{\sigma}, \bm{d}}} + W^{-D}\\
         & \prec {\big({\ell_t^d |1-t|}\big)^{-\fn+2}} \cdot (W^\tau\hell_t)^2  \avg{\br{a_i} - [d_{\fa(1)}]}^{-d} \cdot \omega_t^{-\fn+3} \\
         &\lesssim W^{2\tau}{\big({\ell_t^d |1-t|\omega_t}\big)^{-\fn+2}}   \avg{\br{a_i} - [d_{\fa(1)}]}^{-d}.
      \end{align*}
      In the first step, we applied \eqref{f012}; in the second step, we used the estimate \eqref{eq:est-core} along with the fact that the sum over the region with \smash{$|[s_i]|\ge W^\tau\hell_t $} is at most $W^{-D}$; in the third step, we used that $\hell_t^2\omega_t\lesssim 1$.  
      
\medskip
\noindent(2) Suppose $\xi_i = 1$ for at least two indices $i$. Without loss of generality, suppose $\xi_2=\xi_3=1$. Similar to case (1), applying \eqref{f012} and \eqref{eq:est-core}, we find that for any constants $\tau,D>0$,
      \begin{align}\label{eq:twoone}
    & \sum_{\bm{s}\setminus\{[s_{\fa(1)}]\}} \p{\prod_{i=2}^\fn f_{\xi_i}\p{\br{a_i}, [s_{\fa(i)}]}} \Sigma^{\p{\emptyset}}_{t, \bm{\sigma}, \bm{d}} \nonumber\\
     &\prec  \frac{\p{\ell_t^d |1-t| \omega_t}^{-\fn+3}(W^\tau\hell_t)^2 \mathbf 1(|[a_2]-[d_{\fa(1)}]|\le (\log W)^2\ell_t)}{\avg{[a_2]-[d_{\fa(1)}]}^{d-1}\avg{[a_3]-[d_{\fa(1)}]}^{d-1}} + W^{-D} ,
      \end{align}
      where we also used the exponential decay of the $\Theta_t$-propagators beyond the scale $\ell_t$ due to \eqref{prop:ThfadC}. By using $ \hell_t^2\omega_t\lesssim 1$ and $\ell_t^2|1-t|\lesssim 1$, we can derive \eqref{Kbound-inductive-hypothesis} from \eqref{eq:twoone} for both dimensions $d=1$ and $d=2$. 
      
\medskip
\noindent(3) If $\xi_k=1$ for some $2\le k\le \fn$ and all other $\xi_i$'s are equal to zero, then using $f_1([a_k],[s_{\fa(k)}])=-f_1([a_k],-[s_{\fa(k)}])$, we see that the corresponding term vanishes.

\medskip
\noindent(4) Finally, if $\xi_2 = \cdots = \xi_\fn = 0$, then using \eqref{f012}, the sum zero property \eqref{sum-zero-property}, and the exponential decay of the $\Theta_t$-propagators beyond the scale $\ell_t$, we get that for any constant $D>0$, 
      \begin{align*}
         &\sum_{\bm{s}\setminus\{[s_{\fa(1)}]\}} \p{\prod_{i=2}^\fn f_{\xi_i}\p{\br{a_i}, [s_{\fa(i)}]}} \Sigma^{\p{\emptyset}}_{t, \bm{\sigma}, \bm{d}}
          = \p{\prod_{i=2}^\fn f\p{\br{a_i}, \br{0}}} \sum_{\bm{d}\setminus\{[d_{\fa(1)}]\}} \Sigma^{\p{\emptyset}}_{t, \bm{\sigma}, \bm{d}} \\
          &\prec |1-t|\omega_t^{-\fn+2}\big({\ell_t^d |1-t|}\big)^{-\fn+1}\mathbf 1\p{|[a_2]-[d_{\fa(1)}]|\le (\log W)^2\ell_t} +W^{-D}\\
          &\prec  \big({\ell_t^d |1-t| \omega_t}\big)^{-\fn+2}\avg{[a_2]-[d_{\fa(1)}]}^{-d}.
      \end{align*}
      %which implies (\ref{Kbound-inductive-hypothesis}).
 % \end{enumerate}
  
Combining the above four cases, we see that (\ref{Kbound-inductive-hypothesis}) holds when $\pi = \emptyset$, which also establishes the base case $\fn=4$ for the induction argument.   
  For a general $\fn\ge 4$ and $\pi\ne \emptyset$, we apply \Cref{core-decomposition} and use the induction hypothesis. Specifically, suppose we have the decomposition \eqref{recursive-decomposition-Sigma}. By the induction hypothesis \eqref{Kbound-inductive-hypothesis}, we have 
  \begin{align*}
      \sum_{\bm{d}^{[c]}\setminus \{[c]\}} \p{\prod_{i=r}^{\fn} \Theta^{\p{\sigma_i, \sigma_{i+1}}}_{t,\br{a_i}[d_{\fa(i)}]}}\Sigma^{(\pi^{[c]})}_{t,\bm{\sigma}^{[c]},\bm{d}^{[c]}}
      &\prec \frac{\p{\ell_t^d|1-t|\omega_t}^{-\fn+r}}{\big(\min_{r \leq i \leq \fn} \avg{\br{a_i} - \br{c}}\big)^{d}} \, ,\\
      \sum_{\bm{d}^{[c']}\setminus \{[d_{\fa(1)}]\}} \p{\Theta^{\p{\sigma_r, \sigma_{1}}}_{t,[c][c']}\prod_{i=2}^{r-1} \Theta^{\p{\sigma_i, \sigma_{i+1}}}_{t,\br{a_i}[d_{\fa(i)}]}}    \Sigma^{(\pi^{[c']})}_{t,\bm{\sigma}^{[c']},\bm{d}^{[c']}}
      &\prec \frac{\p{\ell_t^d|1-t|\omega_t}^{-r+2}}{\big(\avg{[c] - [d_{\fa(1)}]} \wedge\min_{2 \leq i \leq r-1} \avg{[a_i] - [d_{\fa(1)}]}\big)^{d}} \, .
  \end{align*}
Then, with the decomposition \eqref{recursive-decomposition-Sigma}, the above two estimates, and the identity 
$$\Big(\Theta_t^{(\sig_1,\sig_r)}-1\Big)_{[c][c']}=\Big(tS^{\LK}\Theta_t^{(+,-)}\Big)_{[c][c']},$$ 
we can write the LHS of \eqref{Kbound-inductive-hypothesis} as 
\begin{align*}
 %&   \sum_{\bm{d}\setminus\{[d_{\fa(1)}]\}} \p{\prod_{i=2}^\fn \Theta^{\p{\sigma_i, \sigma_{i+1}}}_{t,\br{a_i}[d_{\fa(i)}]}}\Sigma^{\p{\pi}}_{t, \bm{\sigma}, \bm{d}}\\ 
& \sum_{\bm{d}^{[c]}\setminus \{[c]\}} \sum_{\bm{d}^{[c']}\setminus \{[d_{\fa(1)}]\}} \sum_{[b],[c]} \p{\prod_{i=r}^{\fn} \Theta^{\p{\sigma_i, \sigma_{i+1}}}_{t,\br{a_i}[d_{\fa(i)}]}}\Sigma^{(\pi^{[c]})}_{t,\bm{\sigma}^{[c]},\bm{d}^{[c]}} \cdot tS^{\LK}_{[c][b]}
\p{\Theta^{\p{\sigma_r, \sigma_{1}}}_{t,[b][c']}\prod_{i=2}^{r-1} \Theta^{\p{\sigma_i, \sigma_{i+1}}}_{t,\br{a_i}[d_{\fa(i)}]}}    \Sigma^{(\pi^{[c']})}_{t,\bm{\sigma}^{[c']},\bm{d}^{[c']}}\\
& \prec \sum_{[b],[c]:|[b]-[c]|\le 1}\frac{(\ell_t^d|1-t|\omega_t)^{-(\fn-2)}}{\big(\min_{r \leq i \leq \fn} \avg{\br{a_i} - \br{c}}\big)^{d}}
\frac{1}{\big(\langle{[b] - [d_{\fa(1)}]}\rangle \wedge\min_{2 \leq i \leq r-1} \langle{[a_i] - [d_{\fa(1)}]}\rangle\big)^{d}}\\
&\prec \big({\ell_t^d|1-t|\omega_t}\big)^{-\fn+2} \Big(\min_{2 \leq i \leq \fn} \avg{\br{a_i} - [d_{\fa(1)}]}\Big)^{-d},
\end{align*}  
which gives the inductive assumption \eqref{Kbound-inductive-hypothesis}. This completes the proof of \eqref{eq:bcal_k}. 
\end{proof}

\section{Proof of some results in \Cref{Sec:Stoflo}}\label{sec:main_appd}

\subsection{Proof of \Cref{lem:pf_step2}}\label{sec:pf_step2}

For the proof of \Cref{lem:pf_step2}, in addition to \Cref{lem:sum_Ndecay}, we will also need the following evolution kernel estimate in \Cref{TailtoTail}. It shows that given $s<t\le 1$, if a two-dimensional tensor $\cal A$ decays exponentially on the scale $\ell_s$, then \smash{${\cal U}^{(2)}_{s,t,\bsig}\circ \cal A$} decays on the scale $\ell_t$. % for $\bsig\in \{(+,-),(-,+)\}$. 

\begin{lemma}[Lemma 7.2 of \cite{Band1D}]
\label{TailtoTail}
For any $t\in[0,1]$ and constant $D>0$, recall the notation in \eqref{def_WTuD}. 
% $\ell\ge 0$, introduce the notation  
%  $${\cal T}_{t}(\ell) := (W^d\ell_t^d\eta_t)^{-2} \exp \big(- \left| \ell /\ell_t\right|^{1/2} \big).
% $$
For $\bsig\in\{+,-\}^2$ and $s\in [0,1]$, suppose ${\cal A}_{\ba}$ satisfies that 
$$
|{\cal A}_{\ba}|\le {\cal T}_{s,D}(|[a_1]-[a_2]|) ,\quad \forall \ba=([a_1],[a_2])\in(\Zn)^2.
$$
Then, for any $t\in[s,1)$, we have that 
\begin{align}
\label{neiwuj} 
\left({\cal U}^{(2)}_{s,t,\boldsymbol{\sigma}} \circ 
{\cal A}\right)_{\ba} & \prec
 {\cal T}_{t,D}( |[a_1]-[a_2]|)+ \p{\frac{1-s}{1-t}}^2W^{-D}, \quad {\rm if}\quad  |[a_1]-[a_2]|\ge \ell_t^*. %(\log W)^{3/2}\ell_t .
\end{align}
\end{lemma}
% \begin{proof}
% The proof of this lemma is nearly identical to that of \cite[Lemma 7.2]{Band1D}, utilizing the estimate \eqref{prop:ThfadC}. 
% \end{proof}

It is also straightforward to check the following simple properties.  

\begin{claim}[Lemma 5.6 of \cite{Band1D}]\label{lem_dec_calE_0}
For any large constant $D>0$ and small constant $\delta>0$, the following estimates hold for all $u\in [0,t]$ if $|[b] -[a]|\ge \delta \ell_t^*$:
 \be \label{Kell*}
\big|{\Theta}_{t,[a][b]}^{\bsig}\big|\le W^{-D} ,\ \    \left(\frac{1 - u m(\sig_1)m(\sig_2)S^\LK }{1 - t m(\sig_1)m(\sig_2)S^\LK }\right)_{[a] [b]} \le W^{-D},\ \ \forall \bsig\in \{+,-\}^2;
\ee
\be
{\cal L}^{(2)}_{t
, \bsig,([a], [b])} \prec {\cal J}^*_{t
,D}\cdot{\cal T}_{t,D}(|[a]-[b]|) ,\quad \text{for}\ \ \bsig\in\{(+,-),(-,+)\}. \label{auiwii}
\ee
Furthermore, for any constant $C>0$, we have 
\begin{equation}\label{Tell*}
    {\cal T}_{t,D} \left(\ell-C\ell_t^*\right)\prec   {\cal T}_{t,D} \left(\ell \right),\quad \forall \ell\ge 0.
\end{equation} 
\end{claim}

%Recalling \eqref{def_Ju}, to show \eqref{Eq:Gdecay_w}, it suffices to prove that for all $u\in [s,t]$ and any large constant $D>0$, \[\max_{0\le \ell \le n}   {\cal J} _{u,D} (\ell) \prec (|1-s|/|1-u|)^4 .\] 
%Since the proof of this bound is similar to that of equation (2.72) in \cite{Band1D}, we will outline the proof without providing all the details. 
%  \begin{align*}%\label{shoellJJ}
% \max_{0\le \ell \le n}   {\cal J} _{u,D} (\ell) \prec (|1-s|/|1-u|)^4 .
%   \end{align*}
%We prove this fact by analyzing equations \eqref{int_K-LcalE} and \eqref{int_K-L_ST}. 

When $\fn=2$, the second term on the RHS of \eqref{int_K-LcalE} vanishes and we have
\begin{align}
(\mathcal{L} - \mathcal{K})^{(2)}_{t, \boldsymbol{\sigma}, \ba} & =
    \left(\mathcal{U}^{(2)}_{s, t, \boldsymbol{\sigma}} \circ (\mathcal{L} - \mathcal{K})^{(2)}_{s, \boldsymbol{\sigma}}\right)_{\ba} + \int_{s}^t \left(\mathcal{U}^{(2)}_{u, t, \boldsymbol{\sigma}} \circ \mathcal{E}^{(2)}_{u, \boldsymbol{\sigma}}\right)_{\ba} \dd u   \nonumber\\
    &+ \int_{s}^t \left(\mathcal{U}^{(2)}_{u, t, \boldsymbol{\sigma}} \circ \cW^{(2)}_{u, \boldsymbol{\sigma}}\right)_{\ba} \dd u + \int_{s}^t \left(\mathcal{U}^{(2)}_{u, t, \boldsymbol{\sigma}} \circ \dd \cB^{(2)}_{u, \boldsymbol{\sigma}}\right)_{\ba} . \label{int_K-L_ST_n2}
\end{align}
By the induction hypothesis \eqref{Eq:Gdecay+IND} for $({ \cal L-\cal K})_{s,\bsig,\ba}^{(2)}$, using the evolution kernel estimates in \Cref{lem:sum_Ndecay,TailtoTail}, we can control the first term on the RHS of \eqref{int_K-L_ST_n2} by 
\begin{equation} \label{res_deccalE_0}
\left(\mathcal{U}^{(2)}_{s, t, \boldsymbol{\sigma}} \circ (\mathcal{L} - \mathcal{K})^{(2)}_{s, \boldsymbol{\sigma}}\right)_{\ba}\Big/ 
  {\cal T}_{t, D}(|[a_1]-[a_2]|) 
  \prec (\ell_t^d/\ell_s^d)^2\cdot {\bf 1}(|[a_1]-[a_2]|\le \ell_t^*) + 1,
\end{equation}
where the second term arises from applying \eqref{neiwuj}, while the first term results from applying \Cref{lem:sum_Ndecay}. 
%where the factor $\left(\eta_u / \eta_t\right)^2$ transforms into $ (\ell_t^d/\ell_s^d)^2$ due to the difference in prefactors between $ {\cal T}_{s, D}$ and  ${\cal T}_{t, D}$. 
For the remaining three terms on the RHS of \eqref{int_K-L_ST_n2}, we will use tacitly the following fact: for any fixed $\fn\ge 2$, \smash{$\ba\in (\Zn)^\fn$}, and function \smash{$f:(\Zn)^\fn\to \C$} of order $\OO(W^C)$ for a constant $C>0$, we have that for any large constant $D'>D$,
$$ \left| \left({\cal U}^{(\fn)}_{u,t, \boldsymbol{\sigma}} \circ f' \right)_\ba\right| \le W^{-D'},\quad  \text{where}\ \ f'(\mathbf b):=f(\mathbf b)\mathbf 1(\|\mathbf b-\mathbf a\|_\infty\ge \ell_t^*)\, .$$
This follows from the exponential decay of \smash{${\cal U}^{(\fn)}_{u,t, \boldsymbol{\sigma}}$} when $\|\mathbf b-\mathbf a\|_\infty = \max_i |[a_i]-[b_i]|\ge \ell_t^*$, by the estimate \eqref{prop:ThfadC}. 
Combining \Cref{lem:sum_Ndecay} with \eqref{res_deccalE_lk}, we can bound the second term on the RHS of \eqref{int_K-L_ST_n2} as 
\begin{align}
 \frac{\big({\cal U}^{(2)}_{u,t, \boldsymbol{\sigma}}\circ  {\cal E}^{(2)}_{u, \boldsymbol{\sigma}}\big)_\ba}{\cal T_{t,D}(|[a_1]-[a_2]|)}
 &\prec \p{\frac{1-u}{1-t}}^2 \frac{\max_{\| \textbf{b}-\ba\|_\infty\le \ell_t^*}{\cal E}^{(2)}_{u, \boldsymbol{\sigma}, \mathbf{b}}+W^{-D'}}{{\cal T_{t,D}(|[a_1]-[a_2]|)}} 
 %\nonumber\\
 %&\prec \frac{1-u}{(1-t)^2}  \left(W^d\ell_u^d \eta_u\sqrt{\kappa}\right)^{-1}\left({\cal J}^{*}_{u,D} \right)^{2} \max_{\| \textbf{b}-\ba\|_\infty\le \ell_t^*}   \cal T_{t,D}(|[b_1]-[b_2]|)+W^{-D'}\nonumber\\
\prec \frac{1-u}{(1-t)^2} \frac{\big({\cal J}^{*}_{u,D} \big)^{2}}{W^d\ell_u^d \eta_u\sqrt{\kappa}} ,\label{res_deccalE_1}
\end{align}
where in the derivation, we have used \eqref{Tell*} under the condition $\| \textbf{b}-\ba\|_\infty\le \ell_t^*$ and chosen $D'$ sufficiently large depending on $D$.  
Similarly, combining \Cref{lem:sum_Ndecay} with the estimates \eqref{res_deccalE_wG} and \eqref{res_deccalE_dif} yields that   
\begin{align}  \label{res_deccalE_2}
& \frac{\big(\mathcal{U}^{(2)}_{u,t, \boldsymbol{\sigma}} \circ  \cW^{(2)}_{u, \boldsymbol{\sigma}}\big)_{\mathbf{a}}}{
{\cal T}_{t, D}(|[a_1]-[a_2]|)} \prec   \frac{1-u}{(1-t)^2} \bigg[\left(\frac{\ell_u^d}{\ell_s^d}\right)^2 {\bf1}(|[a_1]-[a_2]|\le 3\ell_t^*) + \frac{\big({\cal J}^{*}_{u,D}\big)^{3 }}{\left(W^d \ell_u^d \eta_u\sqrt{\kappa}\right)^{1/3}}   \bigg],\\
&\frac{\Big(\big(
   \mathcal{U}^{(2)}_{u, t,  \boldsymbol{\sigma}}
   \otimes 
   \mathcal{U}^{(2)}_{u, t,  \overline{\boldsymbol{\sigma}}}
   \big)\circ
   \big( \mathcal{B} \otimes  \mathcal{B} \big)^{(4)}
   _{u, \bsig}
   \Big)_{\ba, \ba}}{
 {\cal T}_{t, D}^2(|[a_1]-[a_2]|) }\prec \frac{(1-u)^3}{(1-t)^4}\bigg[\left(\frac{\ell^d_u}{\ell^d_s}\right)^5{\bf1}(|[a_1]-[a_2]|\le 6\ell_t^*) +\frac{\big({\cal J}^{*}_{u,D} \big)^{3 }}{\left(W^d \ell_u^d \eta_u\sqrt{\kappa}\right)^{1/3}}  \bigg]. \label{res_deccalE_3}
\end{align}

%By the monotonically increasing property of $\ell_t$, $\eta_t^{-1}$, and $\cal T_{t,D}$ with respect to $t$, it is clear that if we replace the evolution operator \smash{$\cal U_{u,t,\bsig}^{(2)}$} in \eqref{res_deccalE_0}--\eqref{res_deccalE_3} with \smash{$\cal U_{u,t',\bsig}^{(2)}$} (note we do not replace other $t$'s in them), then these bounds hold uniformly in $t'\in [u,t]$. 

Now, we define the stopping time $\tau = T \wedge t$ with $T$ defined in \eqref{eq:def_TTT}. 
% \begin{align}
% T:=\inf \left\{u: \max_{0\le \ell \le n}   {\cal J} _{u,D} (\ell) \ge W^\e\left(|1-s| / |1-t|\right)^4\right\}
% \end{align}
% for an arbitrarily small constant $\e>0$. 
By the monotonically increasing property of $\ell_t^*$ and $(1-t)^{-1}$ with respect to $t$, it is clear that the estimates \eqref{res_deccalE_0}--\eqref{res_deccalE_3} still hold if we replace \smash{$\cal U_{u,t,\bsig}^{(2)}$} and ${\cal T}_{t, D}(|[a_1]-[a_2]|)$ on the LHS with \smash{$\cal U_{u,\tau,\bsig}^{(2)}$} and ${\cal T}_{\tau, D}(|[a_1]-[a_2]|)$, respectively, while keeping the $(1-t)^{-1}$ and $\ell_t^*$ factors on the RHS unchanged.  Then, using (the $\tau$-version of) \eqref{res_deccalE_3}, we can bound the quadratic variation of the martingale term in \eqref{int_K-L_ST} as:
\begin{align*}
& \int_{s}^\tau 
   \left(\left(
   \mathcal{U}^{(\fn)}_{u, \tau,  \boldsymbol{\sigma}}
   \otimes 
   \mathcal{U}^{(\fn)}_{u, \tau,  \overline{\boldsymbol{\sigma}}}
   \right) \;\circ  \;
   \left( \cB \otimes  \cB \right)^{(2\fn)}
   _{u, \boldsymbol{\sigma}  }
   \right)_{\ba, \ba}\dd u    
  \nonumber \\
\prec~ &  {\cal T}_{\tau, D}^2(|[a_1]-[a_2]|)\cdot  \int_s^\tau   \frac{(1-u)^3}{(1-t)^4}\bigg[\left(\frac{\ell^d_u}{\ell^d_s}\right)^5{\bf1}(|[a_1]-[a_2]|\le 6\ell_t^*) +\frac{\big({\cal J}^{*}_{u,D} \big)^{3 }}{\left(W^d \ell_u^d \eta_u\sqrt{\kappa}\right)^{1/3}}  \bigg]\dd u   \nonumber \\
\lesssim ~ &  {\cal T}_{\tau, D}^2(|[a_1]-[a_2]|)\cdot   \bigg[\p{\frac{1-s}{1-t}}^{5}{\bf 1}(|[a_1]-[a_2]|\le 6\ell_t^*)+1 \bigg],
 \end{align*}
where we have chosen ${\cal J}^{*}_{u,D}\le W^\e \left(|1-s| / |1-t|\right)^4$ for a small enough  constant $\e>0$ depending on $\fd\wedge\fc$ such that (recall the conditions \eqref{eq:small_para} and \eqref{con_st_ind}): \[({\cal J}^{*}_{u,D})^3\p{|1-s|/|1-t|}^4\le (W^d \ell_u^d \eta_u\sqrt{\kappa})^{1/3}.\] 
Combining this bound with \Cref{lem:DIfREP}, we get 
 \begin{align}\label{51}
& \int_{s}^\tau \left(\mathcal{U}^{(\fn)}_{u, \tau, \boldsymbol{\sigma}} \circ \dd \cB^{(\fn)}_{u, \boldsymbol{\sigma}}\right)_{\ba}  \prec  {\cal T}_{\tau, D}(|[a_1]-[a_2]|) \bigg [ \p{\frac{1-s}{1-t}}^{5/2}{\bf 1}(|[a_1]-[a_2]|\le 6\ell_t^*)+1 \bigg ].
 \end{align}
Similarly, using \eqref{res_deccalE_1} and \eqref{res_deccalE_2}, we can bound the third and fourth terms on the RHS of \eqref{int_K-L_ST} as  %\eqref{alu9_STime}, we have 
\begin{align}
&\int_{s}^\tau \left(\mathcal{U}^{(\fn)}_{u, \tau, \boldsymbol{\sigma}} \circ \mathcal{E}^{(\fn)}_{u, \boldsymbol{\sigma}}\right)_{\ba} \dd u  + \int_{s}^\tau \left(\mathcal{U}^{(\fn)}_{u, \tau, \boldsymbol{\sigma}} \circ \cW^{(\fn)}_{u, \boldsymbol{\sigma}}\right)_{\ba} \dd u \nonumber\\
&\prec \cal T_{\tau,D}(|[a_1]-[a_2]|)\cdot\bigg [ \p{\frac{1-s}{1-t}}^{2} {\bf 1}(|a_1-a_2|\le 3\ell_t^*)+1 \bigg ]\label{eq:bddE+W}\, .
\end{align}
Plugging the estimates \eqref{res_deccalE_0}, \eqref{51}, and \eqref{eq:bddE+W} into \eqref{int_K-L_ST} yields that 
\be\label{eq:est2KuptoT}
(\mathcal{L} - \mathcal{K})^{(\fn)}_{\tau, \boldsymbol{\sigma}, \ba}/ \cal T_{\tau,D}(|[a_1]-[a_2]|) \prec (|1-s|/|1-t|)^{5/2} {\bf 1}(|[a_1]-[a_2]|\le 6\ell_t^*) + 1.
\ee

By the induction hypothesis \eqref{Eq:Gdecay+IND} for $({ \cal L-\cal K})_{s,\bsig,\ba}^{(2)}$, we know that $\max_{0\le \ell \le n}   {\cal J} _{s,D} (\ell) \prec 1$, i.e., $T\ge s$ with high probability. Combining this with the estimate \eqref{eq:est2KuptoT} and applying a standard continuity argument (see e.g., the argument in \cite[Appendix A.5]{RBSO1D}), we can show that $T\ge t$ with high probability. Together with \eqref{eq:est2KuptoT}, it also concludes the bound \eqref{53}.

%have actually established the slightly stronger bound \eqref{53} than \eqref{Eq:Gdecay_w}.

%which concludes the proof of \eqref{Eq:Gdecay_w} when $u=t$ since $\e$ is arbitrary. It is clear that we have actually proved the estimate \eqref{Eq:Gdecay_w} for each fixed $u\in [s,t]$ (by renaming the variables $t$ and $u$ as variables e.g., $u$ and $v\in [s,u]$ in the above arguments). Applying a standard $N^{-C}$-argument shows that \eqref{Eq:Gdecay_w} holds uniformly in $u\in [u,t]$. 
%\end{proof}
%In fact, by \eqref{eq:est2KuptoT}, we have actually established the slightly stronger bound \eqref{53} than \eqref{Eq:Gdecay_w}.

\iffalse
following slightly stronger bound than \eqref{Eq:Gdecay_w}:
\begin{align*}%\label{53}
\left({\cal L}-{\cal K}\right)^{(2)}_{t, \bsig, \ba}/ \prec  {\cal T}_{t, D}(|[a_1]-[a_2]|)\cdot  \big [ (|1-s|/|1-t|)^{5/2}  \cdot {\bf 1}(|a_1-a_2|\le 6\ell_t^*)+1 \big ].
\end{align*}
This estimate will be used in Step 5 below. 
\fi

\subsection{Proof of \Cref{lem:STOeq_NQ}}\label{sec:pf_STOeq_NQ}
With the fast decay property shown in \Cref{lem_decayLoop}, we can use \eqref{def:XiL} and \eqref{def:XIL-K} to bound the quantities on the RHS of \eqref{int_K-LcalE} as follows. For the expression in \eqref{DefKsimLK}, we can use \eqref{eq:bcal_k} to bound it as 
\begin{align}\label{CalEbwXi1}
\Big[\OK^{(\lenk)} (\mathcal{L} - \mathcal{K})\Big]^{(\fn)}_{u, \boldsymbol{\sigma},\ba} &\prec 
 W^d\ell_u^d\cdot \frac{\sqrt{\kappa}}{(W^d\ell_u^d\eta_u)^{\lenk-1}}\cdot  \Xi^{(\cal L-\cal K)}_{u,n-l_{\mathcal{K}}+2}(W^d\ell_u^d\eta_u)^{-(n-l_{\mathcal{K}}+2)} \nonumber\\
&\lesssim (1-u)^{-1}(W^d\ell_u^d\eta_u)^{-\fn }\cdot  \max_{k\in\qqq{2,\fn-1}}\Xi^{(\cal L-\cal K)}_{u,k} 
\end{align}
for any $\lenk \in \qqq{3,\fn}$, where the factor $\ell_u^d$ comes from the summation over $[a]$, whose range is restricted by the fast decay property, and we also used  $\eta_u\asymp(1-u)\sqrt{\kappa}$ by \eqref{eq:kappat}. Similarly, for the expression \eqref{def_ELKLK}, we have 
\begin{align}
    \mathcal{E}_{u, \boldsymbol{\sigma},\fa}^{(\fn)}&\prec W^d\ell_u^d \sum_{k=2}^{\fn} \Xi^{(\cal L-\cal K)}_{u,k} (W^d\ell_u^d\eta_u)^{-k}
 \cdot \Xi^{(\cal L-\cal K)}_{u,\fn-k+2}(W^d\ell_u^d\eta_u)^{-(\fn-k+2)} \nonumber\\
  & \lesssim (1-u)^{-1}(W^d\ell_u^d\eta_u)^{-\fn }\cdot \max_{k\in \qqq{2, \fn} }\left(\Xi^{(\cal L-\cal K)}_{u,k} \Xi^{(\cal L-\cal K)}_{u,\fn-k+2}\right)\cdot(W^d\ell_u^d\eta_u\sqrt{\kappa})^{-1} ;\label{CalEbwXi2}
  \end{align}  
for the expression \eqref{def_EwtG}, we have that  
  \begin{align}
  \cW_{u, \boldsymbol{\sigma},\ba}^{(\fn)}&\prec W^d\ell_u^d \cdot  (W^d\ell_u^d\eta_u)^{-1}
 \cdot \left[\Xi^{\cal L }_{u,\fn+1}\cdot  \sqrt{\kappa}(W^d\ell_u^d\eta_u)^{-\fn}\right]\nonumber\\
 &\prec 
  (1-u)^{-1} (W^d\ell_u^d\eta_u)^{-\fn }\cdot  \Xi^{(\cal L)}_{u,\fn+1} , \label{CalEbwXi3} 
    \end{align}  
where we applied \eqref{Gt_bound_loop1} to $ \langle (G_u(\sigma_k)-M(\sig_k)) E_{[a]} \rangle$;  
for any \smash{$\ba,\ba'\in (\Zn)^n$}, by \Cref{def:CALE}, we have   
    \begin{align}
(\mathcal{B} \otimes \mathcal{B})^{(2\fn)}_{u, \boldsymbol{\sigma},\ba,\ba'}&\prec W^d\ell_u^d\cdot \Xi^{(\cal L)}_{u, {2\fn+2}} \cdot  \sqrt{\kappa}(W^d\ell_u^d\eta_u)^{-2\fn-1 }  \nonumber\\
&\prec (1-u)^{-1}(W^d\ell_u^d\eta_u)^{-2\fn }\cdot \Xi^{(\cal L)}_{u, {2\fn+2}}   .\label{CalEbwXi}
\end{align}
%Then, by combining the estimates \eqref{CalEbwXi1}--\eqref{CalEbwXi} with \eqref{sum_res_1}, we can readily establish the following lemma.   

We now analyze the equation \eqref{int_K-LcalE} with the above estimates \eqref{CalEbwXi1}--\eqref{CalEbwXi}. Combining them with \eqref{sum_res_1} and the estimate \eqref{Eq:L-KGt+IND} on \smash{$(\mathcal{L} - \mathcal{K})^{(\fn)}_{s, \boldsymbol{\sigma}, \ba}$}, we obtain  
\begin{align}\label{sahwNQ}
& (W^d\ell_t^d\eta_t)^{\fn}\cdot  (\mathcal{L} - \mathcal{K})^{(\fn)}_{t, \boldsymbol{\sigma}, \ba} \prec 
\frac{\ell_t^d}{\ell_s^d} +   \int_{s}^t  \frac{\ell_u^d}{\ell_s^d}\frac{\max_{k\in \qqq{2,\fn-1}}\Xi^{({\cal L-\cal K})}_{u,k}}{1-u} \dd u + \int_{s}^t\frac{\ell_u^d}{\ell_s^d} \frac{\Xi^{(\cal L)}_{u,\fn+1}}{1-u} \dd u 
    \nonumber \\
    &+  \int_{s}^t \frac{\ell_u^d}{\ell_s^d} \frac{1}{1-u}  \max_{ k\in\qqq{2,\fn} } \frac{\Xi^{({\cal L-\cal K})}_{u,k} \Xi^{({\cal L-\cal K})}_{u,\fn-k+2}}{W^d\ell_u^d\eta_u\sqrt{\kappa}} \dd u  
   + (W^d\ell_t^d\eta_t)^\fn \int_{s}^t \left(\mathcal{U}^{(\fn)}_{u, t, \boldsymbol{\sigma}} \circ  \dd  \mathcal{B}^{(\fn)}_{u, \boldsymbol{\sigma}}\right)_{\ba}.
\end{align}
By Lemma \ref{lem:DIfREP}, using \eqref{CalEbwXi} and \eqref{sum_res_1}, we obtain that  
\begin{align} \label{sahwNQ2}
(W^d\ell_t^d\eta_t)^\fn\int_{s}^t \left(\mathcal{U}^{(\fn)}_{u, t, \boldsymbol{\sigma}} \circ  \dd  \mathcal{B}^{(\fn)}_{u, \boldsymbol{\sigma}}\right)_{\ba} & \prec (W^d\ell_t^d\eta_t)^\fn \left\{\int_{s}^t 
   \left(\left(
   \mathcal{U}^{(\fn)}_{u,t,  \boldsymbol{\sigma}}
   \otimes 
   \mathcal{U}^{(\fn)}_{u,t,  \overline{\boldsymbol{\sigma}}}
   \right) \;\circ  \;
   \left( \cB \otimes  \cB \right)^{(2\fn)}
   _{u, \boldsymbol{\sigma}  }
   \right)_{\ba, \ba}\dd u\right\}^{1/2} \nonumber\\
   &\prec \left\{\int_{s}^t 
   \frac{\ell_u^d}{\ell_s^d} (1-u)^{-1}\Xi^{(\cal L)}_{u, {2\fn+2}}\dd u\right\}^{1/2} .
  \end{align} 
Plugging it into \eqref{sahwNQ} and performing the integral over $u$, we conclude \eqref{am;asoiuw}.
%\end{proof}

\subsection{Proof of \Cref{lem:STOeq_Qt_weak}}\label{sec:pf_STOeq_weak}
With the definition of \smash{$\dthn_u^{(\fn)}$} in \eqref{eq:sumzero_op} and the estimate \eqref{prop:ThfadC}, we can check directly that 
\be\label{eq:derv_Theta} \|\dthn_{t}^{(\fn)}\|_\infty \prec (\ell_t^d)^{-(\fn-1)},\quad 
\|\partial_t\dthn_{t}^{(\fn)}\|_\infty \prec (1-t)^{-1}(\ell_t^d)^{-(\fn-1)}. \ee
If $\bsig$ is not a pure loop, relabeling the indices in $\ba$ if necessary, we can assume that $\sig_i\ne \sig_{i+1}$ for some $i\in \qqq{2,\fn}$. Then, using the first estimate in \eqref{eq:derv_Theta} and the bound \eqref{jywiiwsoks}, we can derive that 
\be\label{eq:nonPureXi}(W^d\ell_t^d\eta_t)^\fn\left[\cal P\circ (\mathcal{L} - \mathcal{K})^{(\fn)}_{t, \boldsymbol{\sigma}}\right]_{[a_1]}\dthn_{t}^{(\fn)}\prec \Xi^{(\mathcal{L}-\mathcal{K})}_{u, \fn-1}. \ee
When $\bsig$ is a pure loop, with the first bound in \eqref{eq:derv_Theta} and the definition  \eqref{eq;Xip}, we can bound the LHS of \eqref{eq:nonPureXi} by $\Xi^+_{t}$. 

For the sum zero term \smash{${\cal Q}_t \circ (\mathcal{L} - \mathcal{K})^{(\fn)}_{t, \boldsymbol{\sigma}, \ba}$}, with Duhamel's principle, we can derive from \eqref{zjuii1} and \eqref{zjuii2} the following counterpart of \eqref{int_K-LcalE}:
 \begin{align}\label{int_K-L+Q}
 & {\cal Q}_t\circ  (\mathcal{L} - \mathcal{K})^{(\fn)}_{t, \boldsymbol{\sigma}, \ba}   = 
\left(\mathcal{U}^{(\fn)}_{s, t, \boldsymbol{\sigma}} \circ  {\cal Q}_s\circ (\mathcal{L} - \mathcal{K})^{(\fn)}_{s, \boldsymbol{\sigma}}\right)_{\ba}  + \sum_{l_\mathcal{K} = 3}^\fn \int_{s}^t \left(\mathcal{U}^{(\fn)}_{u, t, \boldsymbol{\sigma}} \circ  {\cal Q}_u\circ \Big[\mathcal{O}_{\cK}^{(\lenk)} (\mathcal{L} - \mathcal{K})\Big]^{(\fn)}_{u, \boldsymbol{\sigma}}\right)_{\ba} \dd u \nonumber \\
    &+ \int_{s}^t \left(\mathcal{U}^{(\fn)}_{u, t, \boldsymbol{\sigma}} \circ  {\cal Q}_u\circ \mathcal{E}^{(\fn)}_{u, \boldsymbol{\sigma}}\right)_{\ba} \dd u + \int_{s}^t \left(\mathcal{U}^{(\fn)}_{u, t, \boldsymbol{\sigma}} \circ  {\cal Q}_u\circ \mathcal{W}^{(\fn)}_{u, \boldsymbol{\sigma}}\right)_{\ba} \dd u + \int_{s}^t \left(\mathcal{U}^{(\fn)}_{u, t, \boldsymbol{\sigma}} \circ  {\cal Q}_u\circ \dd\mathcal{B}^{(\fn)}_{u, \boldsymbol{\sigma}}\right)_{\ba} 
    \nonumber \\
     &+ \int_{s}^t 
     \left(\mathcal{U}^{(\fn)}_{u, t, \boldsymbol{\sigma}} 
     \circ \left( \left[{\cal Q}_u , \thn^{(\fn)}_{u,\boldsymbol{\sigma}} \right]\circ (\mathcal{L} - \mathcal{K})^{(\fn)}_{u, \boldsymbol{\sigma} }\right) \right)_{\ba} \dd u
   - \int_{s}^t \left(\mathcal{U}^{(\fn)}_{u, t, \boldsymbol{\sigma}} 
     \circ \left\{\left[ {\cal P} \circ
      \left(\mathcal{L} - \mathcal{K}\right)^{(\fn)}_{u, \boldsymbol{\sigma}}\right] \partial_u\dthn_{u}^{(\fn)} \right\}\right)_{\ba} \dd u.
\end{align}
First, with the induction hypothesis \eqref{Eq:L-KGt+IND} at time $s$, the estimates established in \eqref{CalEbwXi1}--\eqref{CalEbwXi3}, \Cref{lem_+Q}, and the evolution kernel estimate \eqref{sum_res_2}, we can control the first 4 terms on the RHS of \eqref{int_K-L+Q} in a manner similar to that in \Cref{lem:STOeq_NQ}. 
It remains to control the last three terms on the RHS of \eqref{int_K-L+Q}. 

Using the definition of $\cal Q_t$ in \eqref{eq:sumzero_op} and the definition of \smash{$\thn^{(2)}_{u,\bsig}$} in \eqref{def:op_thn}, we can bound that
\begin{align}
	\big[{\cal Q}_u , \thn^{(\fn)}_{u,\boldsymbol{\sigma}} \big]\circ (\mathcal{L} - \mathcal{K})^{(\fn)}_{u, \boldsymbol{\sigma}, \ba}
	&= \thn^{(\fn)}_{u, \boldsymbol{\sigma}} \circ \left[ \left({\cal P} \circ \left( \mathcal{L} - \mathcal{K} \right)^{(\fn)}_{u,\bsig}\right) \cdot \dthn^{(\fn)}_u \right]_{\ba} - \left[\left( {\cal P} \circ \thn^{(\fn)}_{u, \boldsymbol{\sigma}} \circ (\mathcal{L} - \mathcal{K})^{(\fn)}_{u,\bsig}\right)  \cdot \dthn^{(\fn)}_u \right]_{\ba} \nonumber
	\\ 
	& \prec (1-u)^{-1}(\ell_u^d)^{-(\fn-1)}   \left\| {\cal P} \circ \left(  \mathcal{L} - \mathcal{K} \right)^{(\fn)}_{u,\boldsymbol{\sigma}} \right\|_{\infty},\label{eq:boundcommutator}
\end{align} 
where in the second step, we used that 
$ \|\thn_{u,\bsig}^{(\fn)}\|_{\infty\to \infty} \prec (1-u)^{-1}$ by \eqref{prop:ThfadC} and the first bound in \eqref{eq:derv_Theta}. Using the second bound in \eqref{eq:derv_Theta}, we get that
\be\label{eq:boundcommutator2}
\left\|\left[ {\cal P} \circ \left(\mathcal{L} - \mathcal{K}\right)^{(\fn)}_{u, \boldsymbol{\sigma}}\right] \partial_u\dthn_{u}^{(\fn)}\right\|_\infty \prec (1-u)^{-1} (\ell_u^d)^{-(\fn-1)}   \left\| {\cal P} \circ \left(  \mathcal{L} - \mathcal{K} \right)^{(\fn)}_{u,\boldsymbol{\sigma}} \right\|_{\infty}. 
\ee
We can bound $ \big\| {\cal P} \circ \left(  \mathcal{L} - \mathcal{K} \right)^{(\fn)}_{u,\boldsymbol{\sigma}} \big\|_{\infty}$ with a similar argument as above, which yields:
\[
\left\|\big[{\cal Q}_u , \thn^{(\fn)}_{u,\boldsymbol{\sigma}} \big]\circ (\mathcal{L} - \mathcal{K})^{(\fn)}_{u, \boldsymbol{\sigma}}\right\|_\infty+\left\|\left[ {\cal P} \circ \left(\mathcal{L} - \mathcal{K}\right)^{(\fn)}_{u, \boldsymbol{\sigma}}\right] \partial_u\dthn_{u}^{(\fn)}\right\|_\infty \prec (1-u)^{-1}(W^d\ell_u^d\eta_u)^{-\fn} \p{ \Xi^{(\mathcal{L}-\mathcal{K})}_{u, \fn-1} + \Xi^{+}_{u}}.\]
Plugging it into \eqref{int_K-L+Q}, applying \eqref{sum_res_2}, and performing the integral over $u$,  we obtain the desired bound.

Finally, we bound the martingale term in \eqref{int_K-L+Q}. Through a direct calculation, we see that its quadratic variation $[\cdot]_t$ takes the form
\begin{align} \label{UEUEdif+Q}
	& \left[\int_{s}^t \left(\mathcal{U}^{(\fn)}_{u, t, \boldsymbol{\sigma}} \circ  {\cal Q}_u\circ \dd\mathcal{B}^{(\fn)}_{u, \boldsymbol{\sigma}}\right)_{\ba}\right]_t
	 =  \int_s^t \bigg\{\left(\mathcal{U}^{(\fn)}_{u, t,  \boldsymbol{\sigma}}\otimes \mathcal{U}^{(\fn)}_{u, t,  \overline{\boldsymbol{\sigma}}}\right)\circ \sum_{x,y\in \ZL} S_{xy} \left|\cal Q_u\left( \partial_{xy}\cL^{(\fn)}_{u,\bsig}\right) \right|^2\bigg\}_{\ba,\ba}\dd u  \nonumber\\
	&\qquad \lesssim \int_{s}^t
	\left( \left(
	\mathcal{U}^{(\fn)}_{u, t,  \boldsymbol{\sigma}}
	\otimes 
	\mathcal{U}^{(\fn)}_{u, t,  \overline{\boldsymbol{\sigma}}}
	\right)\circ 
	\left( \mathcal{Q}_u\otimes \mathcal{Q}_u\right) \circ \left( \mathcal{B}\otimes  \mathcal{B} \right)_{u,  \boldsymbol{\sigma}}^{(2\fn)}\right)_{\ba, \ba}\dd u,
\end{align}
where $(\mathcal{B}\otimes  \mathcal{B})_{u,  \boldsymbol{\sigma}}^{(2\fn)}$ and $\mathcal{U}^{(\fn)}_{u, t,  \boldsymbol{\sigma}}\otimes \mathcal{U}^{(\fn)}_{u, t,  \overline{\boldsymbol{\sigma}}}$ are defined in \Cref{def:CALE} and \Cref{lem:DIfREP}, respectively, and the operator $\mathcal{Q} _u\otimes \mathcal{Q}_u$ is defined as follows: given any $(2\fn)$-dimensional tensor $\cal A$,
\begin{align}
&	\left(\left( \mathcal{Q} _u\otimes \mathcal{Q}_u \right)\circ {\cal A}\right)_{\ba, \textbf{b}} 
	:=  
	{\cal A }_{\ba, \textbf{b} }-
	\delta_{[a_1'][a_1]}\sum_{[a_2'],\ldots, [a_\fn']}{\cal A }_{ \ba' ,  \textbf{b}}\cdot \dthn^{(\fn)}_{u, \ba}
	-\delta_{[b_1'] [b_1]}
	\sum_{[b_2'],\ldots, [b_\fn']}{\cal A }_{\ba, \textbf{b}'}\cdot \dthn^{(\fn)}_{u, \textbf{b} }
	\nonumber \\  
&\qquad	+\delta_{[a_1'][a_1]}\delta_{[b_1'] [b_1]}
	\sum_{[a_2'],\ldots, [a_\fn']}\sum_{[b_2'],\ldots, [b_\fn']} {\cal A }_{\ba', \textbf{b}'}\cdot \left(\dthn^{(\fn)}_{u, \ba}  \dthn^{(\fn)}_{u, \textbf{b}}\right)  ,\label{eq:QQA} 
\end{align}
where we denote $\ba=([a_1],\ldots,[a_\fn])$, $\ba'=([a_1'],\ldots,[a_\fn'])$, $\mathbf b=([b_1],\ldots,[b_\fn])$, and $\mathbf b'=([b_1'],\ldots,[b_\fn'])$. By definition, we can verify that the tensor in \eqref{eq:QQA} satisfies the double sum zero property in \eqref{eq:doublesumzero}. 
Furthermore, similar to \Cref{lem_+Q}, we can check that if $\cal A$ satisfies the $(u, \e, D)$-decay property, then 
\begin{align}\label{norm_doubleQA}
	\left\|\left({\cal Q}_u\otimes {\cal Q}_u\right) \circ \cal A\right\|_{\infty} \le W^{C_\fn \e} \|\cal A\|_\infty +W^{-D+C_\fn}
\end{align}
for a constant $C_\fn$ that does not depend on $\e$ or $D$. Hence, we can apply the improved estimate \eqref{sum_res_2_doublezero} to the RHS of \eqref{UEUEdif+Q}.  
Recall that \smash{$(\mathcal{B}\otimes  \mathcal{B})_{u,  \boldsymbol{\sigma}}^{(2\fn)}$} is bounded as in \eqref{CalEbwXi}. Using \eqref{norm_doubleQA} and \eqref{sum_res_2_doublezero}, we can bound that 
\begin{align*}
	\left\|  \left(
	\mathcal{U}^{(\fn)}_{u, t,  \boldsymbol{\sigma}}
	\otimes 
	\mathcal{U}^{(\fn)}_{u, t,  \overline{\boldsymbol{\sigma}}}
	\right)\circ 
	\left( \mathcal{Q}_u\otimes \mathcal{Q}_u\right) \circ \left( \mathcal{B}\otimes  \mathcal{B} \right)_{u,  \boldsymbol{\sigma}}^{(2\fn)}\right\|_\infty \prec  (1-u)^{-1} \Xi^{(\cal L)}_{u, {2\fn+2}}\cdot   (W^d\ell_t^d\eta_t)^{-2\fn }.
\end{align*}
Plugging it into \eqref{UEUEdif+Q} and performing the integral over $u$,  we  obtain that  
$$\int_{s}^t \left(\mathcal{U}^{(\fn)}_{u, t, \boldsymbol{\sigma}} \circ  {\cal Q}_u\circ \dd\mathcal{B}^{(\fn)}_{u, \boldsymbol{\sigma}}\right)_{\ba} \prec (W^d\ell_t^d\eta_t)^{-\fn} \sup_{u\in[s,t]} \left(\Xi^{(\cal L)}_{u, {2\fn+2}}\right)^{1/2}
$$
by using the Burkholder-Davis-Gundy inequality. This concludes the proof.  
% On the other hand, using \eqref{prop:ThfadC}, we can bound \smash{$\dthn^{(\fn)}_{u, \ba}$ by $(\ell_u^d)^{-(\fn-1)}$}. Together with \eqref{jywiiwsoks}, it implies that 
% \begin{align*}
% \left[\cal P\circ (\mathcal{L} - \mathcal{K})^{(\fn)}_{u, \boldsymbol{\sigma}}\right]_{[a_1]} \dthn^{(\fn)}_{u, \ba} \prec (W^d\ell_u^{d}\eta_u)^{-\fn}\Xi^{(\mathcal{L}-\mathcal{K})}_{u, \fn-1}
% \end{align*}
% uniformly in $u\in [s,t]$, which leads to the following bound:
% \be\label{eq:Ward_typeP}
% (W^d\ell_t^{d}\eta_t)^{\fn} \left[\cal P\circ (\mathcal{L} - \mathcal{K})^{(\fn)}_{t, \boldsymbol{\sigma}}\right]_{[a_1]} \dthn^{(\fn)}_{t, \ba} \prec \Xi^{(\mathcal{L}-\mathcal{K})}_{u, \fn-1}.
% \ee
%\end{proof}

\subsection{Proof of \Cref{lem:STOeq_Qt}}\label{subsec:STOeq_Qt_pf}
From the above proof of \Cref{lem:STOeq_Qt_weak}, we see that the partial sum term \smash{$\big[\cal P\circ (\mathcal{L} - \mathcal{K})^{(\fn)}_{u, \boldsymbol{\sigma}}\big]_{[a_1]}\dthn^{(\fn)}_{t, \ba}$} and the martingale term in \eqref{int_K-L+Q} are already bounded. We still need to control the remaining 6 terms on the RHS of \eqref{int_K-L+Q}:  
\begin{align}\label{int_K-L+Q_pf}
\left(\mathcal{U}^{(\fn)}_{s, t, \boldsymbol{\sigma}} \circ  {\cal Q}_s\circ \cal A(s)\right)_{\ba}  +  \int_{s}^t \left(\mathcal{U}^{(\fn)}_{u, t, \boldsymbol{\sigma}} \circ  {\cal Q}_u\circ \cal B(u)\right)_{\ba} \dd u\, ,
\end{align}
where we abbreviate that $\cal A(s):=(\mathcal{L} - \mathcal{K})^{(\fn)}_{s, \boldsymbol{\sigma}}$ and 
\begin{align*}
&  \cal B(u):= \sum_{\lenk=3}^\fn \Big[\mathcal{O}_{\cK}^{(\lenk)} (\mathcal{L} - \mathcal{K})\Big]^{(\fn)}_{u, \boldsymbol{\sigma}}+\mathcal{E}^{(\fn)}_{u, \boldsymbol{\sigma}}+\mathcal{W}^{(\fn)}_{u, \boldsymbol{\sigma}}+\left[{\cal Q}_u , \thn^{(\fn)}_{u,\boldsymbol{\sigma}} \right]\circ (\mathcal{L} - \mathcal{K})^{(\fn)}_{u, \boldsymbol{\sigma} }+\left[ {\cal P} \circ
\left(\mathcal{L} - \mathcal{K}\right)^{(\fn)}_{u, \boldsymbol{\sigma}}\right] \partial_u\dthn_{u}^{(\fn)}.
\end{align*}
Note that $\E\cal A(s)$ and $\E\cal B(u)$ satisfy both the sum zero property and the symmetry \eqref{eq:A_zero_sym} by the translation invariance and parity symmetry of our model on the block level. Hence, applying the bound \eqref{sum_res_2_sym} instead of \eqref{sum_res_2} leads to the desired improvement for the expectation of the expression \eqref{int_K-L+Q_pf}. It remains to control the fluctuation of \eqref{int_K-L+Q_pf} after removing the expectation. This follows from a CLT-type cancellation mechanism, which yields an additional $\ell_s/\ell_t$ factor that cancels the $\ell_t/\ell_s$
prefactor in \eqref{am;asoi222_weak}. Since the argument is almost identical to those in \cite[Section 7]{Band2D} and \cite[Appendix A.10]{RBSO1D}---relying on the estimates in \Cref{lem_propTH} and the fast decay properties of the resolvent entries---we omit the details.
%We are left to deal with the the fluctuation of  \eqref{int_K-L+Q_pf} with the expectation removed. This follows from a CLT-type of argument, which provides a additional $\ell_s/\ell_t$ factor to cancel the $\ell_t/\ell_s$ prefactor in \eqref{am;asoi222_weak}. Since the argument is almost identical to those in \cite[Section 7]{Band2D} and \cite[Section 7]{RBSO1D} by using the estimates in \Cref{lem_propTH} and the fast decay properties of the resolvent entries, we omit the details.
%\end{proof}

\subsection{Step 6 of the Proof of \Cref{lem:main_ind}}\label{sec:pf_step6}

At this stage, we have the initial estimate \eqref{Eq:Gtlp_exp+IND} at time $s$, sharp local laws \eqref{Gt_bound_flow} and \eqref{Gt_avgbound_flow}, and sharp $G$-loop estimates \eqref{Eq:LGxb}--\eqref{Eq:Gdecay_flow}. We now use them to establish \eqref{Eq:Gtlp_exp_flow}.  
%First, we have established a sharp averaged local law in \eqref{Gt_avgbound_flow}. Next, 
First, we establish an improved averaged local law for the expectation $\mathbb E\langle (G_u-M) E_{[a]}\rangle$: 
%in \eqref{Gt_avgbound_flow} bound on its expectation: 
\begin{align}\label{res_ELK_n=1} 
  \max_{[a]}  \left|\mathbb E\langle (G_u-M) E_{[a]}\rangle\right|
  \prec \kappa^{-1/2}(W^d\ell_u^d\eta_u)^{-2}  .
  \end{align} 
Under Definitions \ref{def_flow} and  \ref{Def:stoch_flow}, we can write $G_u$ and $M=mI_N$ as
$$ G_u = (H_u - z_u)^{-1},\quad m(\sE)= -( \sE + m(\sE))^{-1} = -( z_u(\sE) + u m(\sE))^{-1},$$
which gives the following relation:
\be\label{G-M}
G_u - M = - m(um + H_u) G_u.
\ee 
Plugging it into $\mathbb{E} \langle (G_u - M) E_{[a]} \rangle$ gives
\begin{align}
    \mathbb{E} \langle (G_u - M) E_{[a]} \rangle = - \mathbb{E} \langle m(um+H_u) G_u E_{[a]} \rangle 
    =um \sum_{[b]} \mathbb{E}\langle G_u E_{[a]} \rangle S^{\LK}_{[a][b]}\langle (G_u - M) E_{[b]} \rangle  \nonumber\\
    = um^2 \sum_{[b]} S^{\LK}_{[a][b]} \E\langle (G_u - M) E_{[b]} \rangle + u m \sum_{[b]} S^{\LK}_{[a][b]}\mathbb{E}\left[\langle (G_u-M) E_{[a]}\rangle \langle (G_u - M) E_{[b]} \rangle\right], \label{eq:G-MG-M}
\end{align}
where in the second step, we applied Gaussian integration by parts to the entries of $H_u$. 
Applying \eqref{Gt_avgbound_flow}, we can bound the second term on the RHS of \eqref{eq:G-MG-M} by $\OO_\prec((W^d\ell_u^d\eta_u)^{-2})$. 
% \be\label{eq:second_uGM}
% uW^d \sum_{[b]} \mathbb{E}\langle (G-M) E_{[a]}ME_{[b]} \rangle \langle (G - M) E_{[b]} \rangle\prec (W^d\ell_u^d\eta_u)^{-2}. 
% \ee
Hence, we can rewrite equation \eqref{eq:G-MG-M} as 
\begin{align*}
\sum_{[b]} \p{1-um^2S^{\LK}}_{[a][b]} \mathbb{E} \langle (G_u - m) E_{[b]} \rangle = \OO_\prec \left((W^d\ell_u^d\eta_u)^{-2}\right).
\end{align*}
Solving this equation and using that \smash{$\|\Theta_{u}^{(+,+)}\|_{\infty\to \infty}\prec \omega_t^{-1}\le \kappa^{-1/2}$} by \eqref{prop:ThfadC_short}, we conclude \eqref{res_ELK_n=1}.

\begin{proof}[\bf Step 6: Proof of \eqref{Eq:Gtlp_exp_flow}]
Taking the expectation of both sides of equation \eqref{int_K-LcalE} when $\fn=2$, we get that  
\begin{align}\label{Eexpint_K-L}
 \mathbb E (\mathcal{L} - \mathcal{K})^{(2)}_{t, \boldsymbol{\sigma}, \ba}  =
 &\left(\mathcal{U}^{(2)}_{s, t, \boldsymbol{\sigma}} \circ \E (\mathcal{L} - \mathcal{K})^{(2)}_{s, \boldsymbol{\sigma}}\right)_{\ba} + \int_{s}^t \left(\mathcal{U}^{(2)}_{u, t, \boldsymbol{\sigma}} \circ \E\mathcal{E}^{(2)}_{u, \boldsymbol{\sigma}}\right)_{\ba} \dd u  + \int_{s}^t \left(\mathcal{U}^{(2)}_{u, t, \boldsymbol{\sigma}} \circ \E\cW^{(2)}_{u, \boldsymbol{\sigma}}\right)_{\ba} \dd u  \, .
\end{align} 
At time $s$, by \eqref{Eq:Gtlp_exp+IND}, we have 
\begin{align}\label{jdasfuao01}
  \mathbb E\;  (\mathcal{L} - \mathcal{K})^{(2)}_{s, \boldsymbol{\sigma},\ba}\prec \kappa^{-1/2}(W^d\ell_s^d\eta_s)^{-3} \, .
\end{align}
For $\cal E^{(2)}_{u,\bsig,\ba}$ defined in \eqref{def_ELKLK}, using the 2-$G$ loop estimate in \eqref{Eq:Gdecay_flow}, we can bound it by 
\begin{align}\label{jdasfuao02}
 \cal E^{(2)}_{u,\bsig,\ba} \prec W^d\ell_u^d (W^d\ell_u^d\eta_u)^{-4}= \eta_u^{-1} (W^d\ell_u^d\eta_u)^{-3} \lesssim (1-u)^{-1}\cdot \kappa^{-1/2}(W^d\ell_u^d\eta_u)^{-3} \, . 
\end{align}
For $\E\cW^{(2)}_{u,\bsig,\ba}$ defined in \eqref{def_EwtG}, we can express it as 
\begin{align}\label{eq:calW}
\mathcal{W}^{(2)}_{t, \boldsymbol{\sigma}, \ba} = W^d \sum_{[a],[b]} \langle (G_u-M) E_{[a]}\rangle S^{\LK}_{[a][b]}{\cal L}^{(3)}_{u, (+,+,-),([b],[a_1], [a_2])}+c.c., 
 \end{align}
where $c.c.$ denotes the complex conjugate of the preceding term.
Using \eqref{Gt_bound_loop1}, \eqref{res_ELK_n=1}, and the estimates  \eqref{eq:bcal_k} and \eqref{Eq:L-KGt-flow} with $\fn=3$, we can bound \eqref{eq:calW} as   
\begin{align}
&\E\mathcal{W}^{(2)}_{t, \boldsymbol{\sigma}, \ba}= W^d \sum_{[a],[b]} \left(\E\langle (G_u-M) E_{[a]}\rangle S^{\LK}_{[a][b]} {\cal K}^{(3)}_{u, \bsig_3,\ba_3} +\E\langle (G_u-M) E_{[a]}\rangle S^{\LK}_{[a][b]} ({\cal L}-\cal K)^{(3)}_{u, \bsig_3,\ba_3}\right)+ c.c. \nonumber\\
&\prec W^d\ell_u^d (W^d\ell_u^d\eta_u)^{-4}= \eta_u^{-1} (W^d\ell_u^d\eta_u)^{-3} \lesssim (1-u)^{-1}\cdot \kappa^{-1/2}(W^d\ell_u^d\eta_u)^{-3}, \label{jdasfuao03}
\end{align}
where we denote $\bsig_3=(+,+,-)$ and $\ba_3=([b],[a_1],[a_2])$.

To control the RHS of \eqref{Eexpint_K-L}, we adopt a similar idea as in \Cref{sec:inductive_step}. First, if $\sig_1\ne \sig_2$, then using Ward's identity and \eqref{res_ELK_n=1}, we obtain that 
\be\label{eq:ELK2}
\br{\cal P\circ \mathbb E (\mathcal{L} - \mathcal{K})^{(2)}_{u, \boldsymbol{\sigma}}}_{[a_1]} = \frac{\im \E \qq{(G_u-M)E_{[a_1]}}}{W^d\eta_u}  \prec \kappa^{-1/2}(W^d\ell_u^d\eta_u)^{-2}(W^d\eta_u)^{-1}
\ee
uniformly in $u\in[s,t]$. 
If $\sig_1=\sig_2$, taking the expectation of \eqref{eq;Cauchy-int}, we get 
\begin{align}
 \E\left[\cal P\circ (\mathcal{L} - \mathcal{K})^{(2)}_{u, \boldsymbol{\sigma}}\right]_{[a_1]}  &=W^{-d}  \cdot \frac{1}{2\pi \ii}\oint_\gamma \frac{\E\avg{(G_u(z)-M_u(z))E_{[a_1]}}}{(z-z_u)^2}\dd z \, .\label{eq;Cauchy-int2}
\end{align}
With the estimate \eqref{eq:cont-est_pure3} established in Step 3, applying the argument below \eqref{G-M}, the same estimate \eqref{res_ELK_n=1} holds uniformly for $z\in\Gamma$. Plugging this into \eqref{eq;Cauchy-int2}, we see that the estimate \eqref{eq:ELK2} also holds in the case $\sig_1=\sig_2$.
Combining \eqref{eq:ELK2} with the estimate \eqref{prop:ThfadC}, we obtain that
\be\label{eq:EPL-K}
\br{\cal P\circ \mathbb E (\mathcal{L} - \mathcal{K})^{(2)}_{u, \boldsymbol{\sigma}}}_{[a_1]} \cdot (1-u)\Theta_{u,[a_1][a_2]}^{(+,-)}\prec \kappa^{-1/2}(W^d\ell_u^d\eta_u)^{-3}.
\ee

When $u=t$, the above argument yields the desired estimate for the partial sum term. It remains to deal with the sum zero term \smash{$\cal Q_t\circ \mathbb E (\mathcal{L} - \mathcal{K})^{(2)}_{t, \boldsymbol{\sigma}, \ba}$}. 
Taking the expectation of the equation \eqref{int_K-L+Q} with $\fn=2$ yields
\begin{align}\label{int_K-L+QE}
& {\cal Q}_t\circ  \E(\mathcal{L} - \mathcal{K})^{(2)}_{t, \boldsymbol{\sigma}, \ba}   = 
\left(\mathcal{U}^{(2)}_{s, t, \boldsymbol{\sigma}} \circ  {\cal Q}_s\circ \E(\mathcal{L} - \mathcal{K})^{(2)}_{s, \boldsymbol{\sigma}}\right)_{\ba} + \int_{s}^t \left(\mathcal{U}^{(2)}_{u, t, \boldsymbol{\sigma}} \circ  {\cal Q}_u\circ \E \mathcal{E}^{(2)}_{u, \boldsymbol{\sigma}}\right)_{\ba} \dd u \nonumber\\
& + \int_{s}^t \left(\mathcal{U}^{(2)}_{u, t, \boldsymbol{\sigma}} \circ  {\cal Q}_u\circ \E\mathcal{W}^{(2)}_{u, \boldsymbol{\sigma}}\right)_{\ba} \dd u  + \int_{s}^t 
\left(\mathcal{U}^{(2)}_{u, t, \boldsymbol{\sigma}} 
\circ \left( \left[{\cal Q}_u , \thn^{(2)}_{u,\boldsymbol{\sigma}} \right]\circ \E(\mathcal{L} - \mathcal{K})^{(2)}_{u, \boldsymbol{\sigma} }\right) \right)_{\ba} \dd u \nonumber\\
& - \int_{s}^t \left(\mathcal{U}^{(2)}_{u, t, \boldsymbol{\sigma}} \circ \left\{\left[ {\cal P} \circ \E\left(\mathcal{L} - \mathcal{K}\right)^{(2)}_{u, \boldsymbol{\sigma}}\right] \partial_u\dthn_{u}^{(2)} \right\}\right)_{\ba} \dd u.
\end{align}
As explained above \eqref{pqthlk}, all the tensors on the RHS satisfy the sum zero property \eqref{sumAzero} and the symmetry \eqref{eq:A_zero_sym}. Thus,  using the evolution kernel estimate \eqref{sum_res_2_sym}, \Cref{lem_+Q}, and the estimates \eqref{jdasfuao01}, \eqref{jdasfuao02}, and \eqref{jdasfuao03}, we can bound the first three terms on the RHS of \eqref{int_K-L+QE} by 
\begin{align}\label{eq:EKL_added}
 \left(\frac{\ell_s^d|1-s|}{\ell_t^d|1-t|}\right)^2 \frac{1}{\kappa^{1/2} (W^d\ell_s^d\eta_s)^{3}}+
 \int_{s}^t  \left(\frac{\ell_u^d|1-u|}{\ell_t^d|1-t|}\right)^2 \frac{(1-u)^{-1}}{\kappa^{1/2}(W^d\ell_u^d\eta_u)^{3}}
 \dd u 
 \prec  \kappa^{-1/2}(W^d\ell_t^d\eta_t)^{-3} \, .
 \end{align} 
It remains to control the last two terms on the RHS of \eqref{int_K-L+QE}. Combining the second bound in \eqref{eq:derv_Theta} (for $\fn=2$) with the estimate \eqref{eq:ELK2}, we get that
\be\label{eq:boundELKQ1}
\left\|\left[ {\cal P} \circ \E\left(\mathcal{L} - \mathcal{K}\right)^{(2)}_{u, \boldsymbol{\sigma}}\right] \partial_u\dthn_{u}^{(2)}\right\|_\infty \prec (1-u)^{-1}\cdot \kappa^{-1/2} (W^d\ell_u^d\eta_u)^{-3} \, . 
\ee
Second, using \eqref{eq:boundcommutator} and the estimate \eqref{eq:ELK2}, we can bound that
\begin{align}
   \big[{\cal Q}_u , \thn^{(2)}_{u,\boldsymbol{\sigma}} \big]\circ \E(\mathcal{L} - \mathcal{K})^{(2)}_{u, \boldsymbol{\sigma}, \ba} \prec (1-u)^{-1}\ell_u^{-d} \left\| {\cal P} \circ \E\left(  \mathcal{L} - \mathcal{K} \right)^{(2)}_{u,\boldsymbol{\sigma}} \right\|_{\infty}\prec (1-u)^{-1}\cdot \kappa^{-1/2} (W^d\ell_u^d\eta_u)^{-3}  .\label{eq:boundELKQ2}
\end{align} 
Combining the bound \eqref{sum_res_2_sym} with the estimates   \eqref{eq:boundELKQ1} and \eqref{eq:boundELKQ2},  
we can also get the above bound \eqref{eq:EKL_added} for the last two terms on the RHS of \eqref{int_K-L+QE}.
This concludes the proof of \eqref{Eq:Gtlp_exp_flow}.
 \end{proof}
 
\end{document}